\documentclass[12pt, oneside]{amsart} 

\usepackage[utf8]{inputenc}
\usepackage{geometry}
\usepackage[english]{babel}
\usepackage{amsmath,
	mathtools,
	amssymb,
	graphicx,
	enumerate,
	bbm,
	xcolor,
	amsthm,
	mathrsfs,
	stmaryrd,
    mathdots}
\usepackage[all]{xy}
\usepackage{hyperref}
\usepackage{tikz-cd}
\usetikzlibrary{decorations.pathmorphing}
\newtheorem{theorem}{Theorem}[section]
\newtheorem*{theorem*}{Theorem}
\newtheorem*{conjecture*}{Conjecture}
\newtheorem*{remark*}{Remark}
\newtheorem{corollary}[theorem]{Corollary}
\newtheorem{definition}[theorem]{Definition}
\newtheorem{lemma}[theorem]{Lemma}
\newtheorem{proposition}[theorem]{Proposition}
\newtheorem{conjecture}[theorem]{Conjecture}

\newtheorem{remark}[theorem]{Remark}
\newtheorem{hypothesis}[theorem]{Hypothesis}

\theoremstyle{plain}

\theoremstyle{remark}

\numberwithin{equation}{section}

\begin{document}
	\title{Special $L$-values and Selmer groups
	of Siegel modular forms of genus $2$}	
	\author{Xiaoyu ZHANG}
	\address{LAGA, Institut Galil\'{e}e, U. Paris 13, av. J.-B. Cl\'{e}ment, Villetaneuse 93430, France.}
	\email{zhang@math.univ-paris13.fr}
	\maketitle
	\begin{abstract}
		Let $p$ be an odd prime,
		$N$ a square-free
		odd positive integer prime to $p$,
		$\pi$ a $p$-ordinary
		cohomological irreducible
		cuspidal automorphic representation
		of $\mathrm{GSp}_4(\mathbb{A}_\mathbb{Q})$
		of principal level $N$ and Iwahori level at $p$.
		Using a $p$-integral version of
		Rallis inner product formula
		and modularity theorems
		for $\mathrm{GSp}_{4/\mathbb{Q}}$ and
		$\mathrm{U}_{4/\mathbb{Q}}$,
		we establish an identity
		between the
		$p$-part of the critical value
		at $1$ of the degree $5$
		$L$-function of $\pi$
		twisted by the non-trivial
		quadratic Dirichlet character $\xi$
		associated to the extension
		$\mathbb{Q}(\sqrt{-N})/\mathbb{Q}$
		and
		the $p$-part of the
		Selmer group of the degree $5$
		Galois representation associated
		to $\pi$ twisted by $\xi$,
		under certain
		conditions on the residual
		Galois representation.
	\end{abstract}
	
	\tableofcontents

	\section{Introduction}
	
	Let $p$ be an odd prime number.
	In this article
	we establish an identity
	between the $p$-part
	of the special value at $1$
	of the twist by an odd quadratic character
	$\xi$ of the standard $L$-function of a
	$p$-ordinary cuspidal
	automorphic forms $\varphi$ over
	$\mathrm{GSp}_4(\mathbb{A}_\mathbb{Q})$
	and the $p$-part of the
	Selmer group of
	the twist by $\xi$ of
	the $p$-adic standard Galois representation
	associated to $\varphi$,
	under certain conditions on the residual
	Galois representation.	
	The main tool is a
	$p$-integral Rallis inner product formula,
	as developed in this article.
	Denote by $M_\varphi$
	the rank $4$ symplectic motive associated to $\varphi$
	and
	$\rho_\mathrm{st}\colon
	\mathrm{GSp}_4\rightarrow
	\mathrm{SO}_5$
	the standard representation.
	Then our identity can be seen as
	supportive evidence for
	Bloch-Kato Tamagawa Number Conjecture
	for the motive
	$\rho_\mathrm{st}(M_\varphi)\otimes\xi$
	(\textit{cf}. \cite{BlochKato1990}),
	except that we consider automorphic periods
	instead of motivic periods.

	To state our result,
	we first prepare some notations.
	We fix an odd
	square-free integer $N>2$ prime to $p$.
	Let $\xi$ be the quadratic character
	associated to the quadratic imaginary extension
	$E=\mathbb{Q}(\sqrt{-N})/\mathbb{Q}$.
	Fix an embedding
	$\iota_p\colon
	\overline{\mathbb{Q}}
	\hookrightarrow
	\overline{\mathbb{Q}}_p$.
	Let $\pi=\otimes'_{v\leq\infty}\pi_v$ be an
	irreducible cuspidal
	automorphic representation of
	$\mathrm{GSp}_4(\mathbb{A}_\mathbb{Q})$
	and $\pi^\vee$ its contragredient.
	We assume that
	$\pi$ is ordinary for $\iota_p$
	and unramified outside $Np\infty$,
	that $\pi_\infty$ is in (anti)holomorphic discrete series,
	that $\pi$ is of principal level
	$N$ and  it has Iwahori level at $p$.
	For a finite place $\ell$ of $\mathbb{Q}$
	such that $\pi_\ell$ is unramified,
	$\pi_\ell$ has Satake parameter
	$S_\ell=\mathrm{diag}(\alpha_1,\alpha_2,\alpha_0/\alpha_1,
	\alpha_0/\alpha_2)\in\mathrm{GSp}_4(\mathbb{C})$.
	Then the conjugacy class of 
	$\rho_\mathrm{st}(S_\ell)$
	is the same as that of the diagonal matrix
	$\mathrm{diag}(1,\alpha_1/\alpha_2,\alpha_2/\alpha_1,
	\alpha_1\alpha_2/\alpha_0^2,\alpha_0^2/\alpha_1\alpha_2)$.
	The twisted standard $L$-factor of $\pi_\ell$
	by the character $\xi$
	is then
	(for $s\in\mathbb{C}$)
	\[
	L_\ell(s,\mathrm{St}(\pi)\otimes\xi)
	=
	(1-\xi(\ell)\ell^{-s})^{-1}
	(1-\xi(\ell)\ell^{-s}
	\frac{\alpha_1}{\alpha_2})^{-1}
	(1-\xi(\ell)\ell^{-s}
	\frac{\alpha_2}{\alpha_1})^{-1}
	(1-\xi(\ell)\ell^{-s}
	\frac{\alpha_1\alpha_2}{\alpha_0^2})^{-1}
	(1-\xi(\ell)\ell^{-s}
	\frac{\alpha_0^2}{\alpha_1\alpha_2})^{-1}.
	\]
	For any finite set $S$ of places of $\mathbb{Q}$
	containing the place $\infty$ and 
	all finite places $\ell$
	where $\pi_\ell$ is ramified,
	we define the partial $L$-function
	$L^S(s,\mathrm{St}(\pi)\otimes\xi)$
	as the product $\prod_{v\notin S}L_v(s,
	\mathrm{St}(\pi)\otimes\xi)$.
	Then it is known that $L^S(s,
	\mathrm{St}(\pi)\otimes\xi)$
	converges absolutely for
	$\mathrm{Re}(s)$ sufficiently large and
	admits a meromorphic continuation to the whole plane
	$\mathbb{C}$ with possibly simple poles
	(\cite{KudlaRallis1990}).

	Let $\mathcal{O}$ a sufficiently large extension of
	$\mathbb{Z}_p$ containing all the eigenvalues of
	$\pi$ for the Hecke operators.
	Let $\varpi$ be a uniformizer of
	$\mathcal{O}$.
	Fix then
	$p$-integral ordinary primitive
	(in terms of Fourier coefficients) 
	factorizable automorphic forms
	$\varphi\in\pi$ and
	$\varphi^\vee\in\pi^\vee$
	such that their pairing
	$\langle\varphi,\varphi^\vee\rangle\neq0$
	(\textit{cf}.
	Definitions
	\ref{nearly holomorphic automorphic forms}
	and
	\ref{antiholomorphic automorphci forms}).
	Let $\mathrm{GSO}_{6/\mathbb{Q}}$ be
	the
	similitude orthogonal group over $\mathbb{Q}$
	determined by the symmetric matrix
	$\mathrm{diag}(2,2,2,2N,2N,2N)$
	and $\mathrm{U}_{4/\mathbb{Q}}$
	be the unitary group over $E/\mathbb{Q}$
	determined by the Hermitian matrix
	$\mathrm{diag}(1,1,1,1)$.
	We then construct an algebraic
	$p$-integral
	theta section
	$\phi$ in the space of
	Schwartz-Bruhat functions
	$\mathcal{S}(\mathrm{M}_{4\times 6}
	(\mathbb{A}_\mathbb{Q}))$
	(\textit{cf}. Section 4).
	The theta series
	$\Theta_\phi$ associated to $\phi$
	sends the modular form
	$\varphi\otimes\varphi^\vee$
	by theta correspondence to a $p$-integral
	$\Theta_\phi(\varphi\otimes\varphi^\vee)$
	on $\mathrm{U}_4(\mathbb{A}_\mathbb{Q})$
	(via certain exceptional isogeny
	between $\mathrm{GSO}_6$ and $\mathrm{U}_4$,
	see Section
	\ref{isogeny from GSO(6) to U(4)}).
	Now we denote by
	$\mathfrak{c}(\pi)\in\mathcal{O}$
	the congruence number/ideal
	between
	$\Theta_\phi(\varphi)$
	and other cuspidal automorphic forms
	on $\mathrm{U}_4(\mathbb{A}_\mathbb{Q})$
	which are not theta lifts from
	$\mathrm{GSp}_4(\mathbb{A}_\mathbb{Q})$
	(\textit{cf}. (\ref{notation for congruence ideal})).
	To the automorphic representation $\pi$, 
	one can associate the symplectic
	Galois representation
	$\rho_\pi
	\colon\Gamma_{\mathbb{Q}}
	\rightarrow\mathrm{GSp}_4(\mathcal{O})$
	(\cite{Taylor1993,
		Laumon2005,
		Weissauer2005})
	and we write
	$\rho^\mathrm{st}_\pi
	=\rho_\mathrm{st}\circ\rho_\pi$.
	Then we can
	relate the congruence number
	$\mathfrak{c}(\pi)$
	to the Selmer group
	$\mathrm{Sel}(\mathbb{Q},
	\rho^\mathrm{st}_\pi
	\otimes\xi)$
	of the Galois representation
	$\rho^\mathrm{st}_\pi
	\otimes\xi$
	(in a way similar to
	\cite{HidaTilouine16}).
	Write
	$\chi(\rho^\mathrm{st}_\pi
	\otimes\xi)$
	for a generator of the
	Fitting ideal
	of the Selmer group
	$\mathrm{Sel}(\mathbb{Q},
	\rho^\mathrm{st}_\pi\otimes\xi)$
	viewed as an $\mathcal{O}$-module.
	Then our main result is
	(see Theorem \ref{Selmer group and L-value}
	for more details):
	\begin{theorem*}
		Assume the hypotheses as in 
		\cite[Theorem 7.3]{HidaTilouine16},
		i.e., ($N$-Min),
		($RFR^{(2)}$)
		and
		($BIG^{(2)}$),
		then we have the following identities,
		up to units in $\mathcal{O}$:
		\[
		\frac{L^{Np\infty}(1,
			\mathrm{St}(\pi)\otimes\xi)
		\widetilde{L}_{Np\infty}(1,
		\mathrm{St}(\pi)\otimes\xi)}
	    {P_{\pi^\vee}}
		=
		\mathfrak{c}(\pi)
		=
		\chi(\rho^\mathrm{st}_\pi
		\otimes\xi).
		\]
		where $\widetilde{L}_{Np\infty}(1,
		\mathrm{St}(\pi)\otimes\xi)$
		is the product
		of modified
		local $L$-factors at places dividing 
		$Np\infty$
		which depends on the local components
		at places $Np\infty$
		of
		$\varphi,\varphi^\vee$
		and $\phi$,
		and $P_{\pi^\vee}
		\in\mathbb{R}_{>0}$
		is a certain automorphic period of
		$\pi$.
	\end{theorem*}
    
    \begin{remark*}
    	Some remarks concerning the theorem
    	are in order:
    	\begin{enumerate}
    		\item 
    		The hypotheses in the above theorem,
    		($N$-Min),
    		($RFR^{(2)}$)
    		and
    		($BIG^{(2)}$)
    		are of Taylor-Wiles type.
    		See Hypothesis
    		\ref{hypotheses for R=T theorems}
    		for details;
    		
    		\item 
    		We have explicit formulas
    		for the $L$-factors
    		$\widetilde{L}_\ell
    		(1,\mathrm{St}(\pi)\otimes\xi)$
    		for $\ell|Np$.
    		See
    		Section \ref{summary}.
    		For the automorphic period
    		$P_{\pi^\vee}$,
    		see Lemma \ref{definition of periods};
    		
    		\item 
    		The identity in the above theorem is
    		reminiscent of a conjectural
    		Bloch-Kato formula
    		we recall in the below;   		
    	\end{enumerate}
    \end{remark*}
    
    Write $\mathrm{Tam}
    (\rho^\mathrm{st}_\pi\otimes\xi)
    \in \varpi^\mathbb{Z}$
    for the Tamagawa number of
    $\rho^\mathrm{st}_\pi\otimes\xi$
    (\textit{cf}.
    \cite[II.5.3.3]{FontainePerrinRiou1994})
    and $P^\mathrm{mot}_\varphi$
    for the Deligne period of the motive
    $\rho_\mathrm{st}(M_\varphi)\otimes\xi$.
    We write
    $H_f^1(\mathbb{Q},
    \rho^\mathrm{st}_\pi\otimes\xi)$
    for the Bloch-Kato Selmer group of the motive
    $\rho_\mathrm{st}(M_\varphi)\otimes\xi$
    using its $p$-adic realization
    $\rho^\mathrm{st}_\pi\otimes\xi$.
    Then the $p$-part
    of the Bloch-Kato Tamagawa Number Conjecture states
    as follows
    \begin{conjecture*}
    	We have the following identity,
    	up to units in $\mathcal{O}$:
    	\[
    	\frac{L^{Np\infty}(1,\mathrm{St}(\pi)\otimes\xi)
    	\widetilde{L}_{Np\infty}(1,
    	\mathrm{St}(\pi)\otimes\xi)}
    	{P^\mathrm{mot}_\varphi}
    	=
    	\chi(H_f^1(\mathbb{Q},\rho_\mathrm{st}
    	\circ\rho_\pi\otimes\xi))
    	\,
    	\mathrm{Tam}
    	(\rho^\mathrm{st}_\pi\otimes\xi).
    	\]
    \end{conjecture*}
    Here we identify the $L$-function of $\pi$ with
    that of
    the associated motive
    $\rho_\mathrm{st}(M_\varphi)\otimes\xi$.
    The Bloch-Kato Selmer group
    $H^1_f(\mathbb{Q},
    \rho^\mathrm{st}_\pi\otimes\xi)$
    is a subgroup of the Greenberg Selmer group
    $\mathrm{Sel}(\mathbb{Q},
    \rho^\mathrm{st}_\pi\otimes\xi)^\ast$
    of finite index
    (\textit{cf}. \cite[Theorem 3]{Flach1990}).
    The relation between the automorphic period
    $P_{\pi^\vee}$ and the motivic period
    $P_\varphi^\mathrm{mot}$
    is quite mysterious
    and we will not touch this topic in this article.
    There have been many works on the
    Bloch-Kato conjectures for various motives
    (\textit{cf}.
    \cite{Brown2007,BellaicheChenevier2009,Klosin2009,
    	Berger2015,
    	CalegariGeraghtyHarris2016}).
    When one wants to show the
    non-triviality of the
    Selmer group,
    the strategy in common used in these
    works is that of Ribet in \cite{Ribet1976}:
    suppose that $N_\pi$
    is a certain motive associated to
    an automorphic representation $\pi$
    of $G(\mathbb{A})$
    for some algebraic group $G_{/\mathbb{Q}}$.
    Suppose that the associated $L$-function
    $L(1,N_\pi)\neq0$ is critical
    in the sense of Deligne.
    To construct a non-trivial element in
    the Selmer group $H^1_f(N_\pi)$
    of $N_\pi$,
    one can try to find some larger
    algebraic group $G'_{/\mathbb{Q}}$
    such that
    the $L$-group $^LG$ of $G$
    maps to $^LG'$
    (as certain Levi subgroup of a parabolic subgroup
    of $^LG'$, for example)
    such that $\pi$ lifts to an automorphic
    representation $\pi'$ on $G'(\mathbb{A})$.
    Suppose that we have an irreducible
    $\mathfrak{p}$-adic Galois representation
    associated to $\pi'$.
    If $\mathfrak{p}$
    divides the normalized $L$-value
    $L(1,N_\pi)$,
    then one can,
    in favorable situations,
    construct another automorphic representation
    $\widetilde{\pi}'$ of $G'(\mathbb{A})$
    whose Hecke eigenvalues are congruent to those of
    $\pi'$
    modulo $\mathfrak{p}$.
    Now the modulo $\mathfrak{p}$ Galois representation
    associated to $\widetilde{\pi}'$
    is reducible and non-semisimple,
    which gives rise to a non-trivial element in
    the Selmer group
    $H^1_f(N_\pi)$
    (\textit{cf}.
    Introduction of \cite{Klosin2009}).
    Previously studied cases include
    $(G,G')=(\mathrm{GL}_1\times\mathrm{GL}_1,
    \mathrm{GL}_2)$
    (Ribet),
    $(\mathrm{GL}_2,\mathrm{GSp}_4)$
    (Urban),
    $(\mathrm{GL}_2,
    \mathrm{U}_{2,2})$
    (Skinner-Urban)
    and
    $(\mathrm{Res}_{E/F}(\mathrm{GL}_1),
    U(3))$
    (Bella\"{i}che-Chenevier).
    Along this strategy,
    this work can be seen as the case
    $(G,G')=
    (\mathrm{GSp}_4,\mathrm{U}_4)$.
    There are two main parts in this article.
    One part is to construct $\pi'$
    in a $p$-integral way:
    we construct a $p$-integral
    theta series $\Theta_\phi$.
    Then given a $p$-integral cuspidal Siegel
    modular form $\varphi\in\pi$,    
    the pairing between $\varphi$
    and $\Theta_\phi$ gives a
    $p$-integral modular form on
    $\mathrm{U}_4(\mathbb{A})$.
    Another part is to relate congruence ideals to
    Selmer groups.
    Once we can interpret the theta lift
    from $\pi$ to $\pi'$
    in terms of Galois representations,
    we can use the results similar to
    those in \cite{HidaTilouine16}
    to finish the proof.

    As indicated at the end of the last paragraph,
    we need a $p$-integral version of
    the theta lift from
    $\mathrm{GSp}_4$ to
    $\mathrm{GSO}_6$
    (and then transferred to $\mathrm{U}_4$).
    This is the so-called arithmetic
    Rallis inner product formula.
    There are already many works
    on this topic
    (\textit{cf}.
    \cite{BochererDummiganSchulzePillot2012,
    	HarrisKudla1992,Prasanna2006}).
    The $p$-integrality of the theta lift of
    a single automorphic form
    is also very useful
    and there are the works of
    Kudla and Millson
    (\textit{cf}. \cite{KudlaMillson1990}
    and the works cited therein)
    from an orthogonal group to a symplectic group
    (see also \cite{Berger2014}).
    We work in the other direction,
    from a (similitude) symplectic group
    to a (similitude) orthogonal group.
    Unlike Kudla and Millson
    who work with cycles on the symmetric space
    associated to the orthogonal group
    and construct theta series using
    cohomology classes,
    we construct explicitly
    local Schwartz-Bruhat functions and show
    that the associated theta series
    has $p$-integral Fourier coefficients.
    Then the pairing between a theta series and
    a $p$-integral Siegel modular form
    can be seen as a Serre duality pairing
    between the coherent
    cohomology groups
    $H^0$ and $H^3$ of the Siegel moduli scheme.
    This strategy is inspired from
    \cite{HarrisLiSkinner2005}
    in which the authors sketched a program
    to construct $p$-adic Rallis inner product formulas.
    The construction of such a $p$-adic family of
    theta series will be very interesting
    and is a future project of the author.
    The constructions used in this article can be easily
    generalized to other reductive dual pairs
    $(\mathrm{GSp}_{2n/F},\mathrm{GSO}_{m/F})$
    over a totally real number field $F$
    at least when $\mathrm{SO}_m$ is compact at
    $\infty$. Yet relating the congruence ideal to
    the Selmer group of certain Galois representation
    in this article uses some particular
    plethysm of representations of $\mathrm{GSp}_4$
    (purely group-theoretic).
    To generalize this part to
    other reductive dual pairs
    one needs some case-by-case study.

    As the reader can see,
    throughout the article
    we use heavily the results and ideas
    on the doubling method
    developed explicitly and arithmetically
    in \cite{Liu2015} for the
    symplectic groups (and
    \cite{EischenHarrisLiSkinner16}
    for the unitary groups).

    Let us give a brief description of the article.
    Section 2 gives some preliminary notions
    on Siegel modular forms and
    congruence ideals that will be used throughout
    this article.
    In Section 3 we formulate the theta correspondence
    for the reductive dual pair
    $(\mathrm{GSp}_4,\mathrm{GSO}_6)$.
    Section 4 is the principal part of
    the paper.
    In this section, we give explicit choice of
    local Schwartz-Bruhat functions in the
    theta correspondence and the associated
    Siegel section for the
    doubling method,
    and calculate explicitly
    the non-archimedean local zeta 
    doubling integrals
    and also show the non-vanishing of
    local Fourier coefficients.
    In Section 5,
    we use known results to write
    the transfer from automorphic representations
    of $\mathrm{GSp}_4(\mathbb{A})$
    to those of $\mathrm{U}_4(\mathbb{A})$
    in terms of Langlands parameters.
    This allows us to relate the
    Galois representations associated to
    $\pi$ and $\pi'$.
    In Section 6, we construct an explicit
    morphism between these two Galois representations
    and at the end we give the main result of the article.

	\paragraph*{\textbf{Notations}}\label{notations}
	\begin{enumerate}
		\item 
		We fix an odd rational prime $p$.
		We fix the following isomorphisms and inclusions
		of fields that are
		compatible with each other:
		\[
		\iota^\infty_p\colon
		\mathbb{C}\simeq
		\overline{\mathbb{Q}}_p,
		\quad
		\iota_\infty\colon
		\overline{\mathbb{Q}}
		\hookrightarrow
		\mathbb{C},
		\quad
		\iota_p
		=\iota_p^\infty\circ\iota_\infty
		\colon
		\overline{\mathbb{Q}}
		\hookrightarrow
		\overline{\mathbb{Q}}_p.
		\]

		\item 
		We denote by
		$\mathbb{A}
		=\mathbb{R}\times\mathbb{A}_\mathrm{f}$
		the ring of ad\`{e}les of $\mathbb{Q}$.
		For any place $v$ of $\mathbb{Q}$,
		we write $|\cdot|_v$ for the $v$-adic
		valuation on $\mathbb{Q}_v$
		such that $|\ell|_\ell=\ell^{-1}$
		and $|\cdot|_\infty$ is the absolute value.

		\item 
		We fix a square-free positive integer $N$ prime to $p$.
		We then write
		$E=\mathbb{Q}(\sqrt{-N})$.
		We define a Hecke character
		$\xi
		\colon
		\mathbb{Q}^\times\backslash\mathbb{A}^\times
		\rightarrow
		\{\pm1\},
		x\mapsto
		(x,-N)$.
		Here
		$(x,-N)=\prod_v(x_v,-N)_v$
		is the product of local Hilbert symbols
		for all places $v$ of
		$\mathbb{Q}$.

		\item 
		For each place $v$ of $\mathbb{Q}$,
		we fix an additive character
		$\mathbf{e}_v\colon\mathbb{Q}_v
		\rightarrow\mathbb{C}^\times$
		as follows:
		if $v=\infty$, then
		$\mathbf{e}_\infty(x)=\mathrm{exp}(2i\pi x)$;
		if $v\nmid\infty$, then
		$\mathbf{e}_v(x)=
		\mathbf{e}(-2i\pi\{x\}_v)$
		where $\{x\}_v$ is the fractional part of $x$ in $\mathbb{Q}_v$.	
		We then define a character
		of $\mathbb{A}$, trivial on $\mathbb{Q}$,
		as the tensor product
		of all these $\mathbf{e}_v$:
		\[
		\mathbf{e}:=
		\otimes_v\mathbf{e}_v
		\colon
		\mathbb{A}\rightarrow
		\mathbb{C}^\times.
		\]

		\item 
		For any algebraic group $G$ defined over $\mathbb{Q}$, we write
		$[G]$ for the quotient
		$G(\mathbb{Q})\backslash G(\mathbb{A})$.

		\item 
		Let $R$ be any commutative ring.
		We write $\mathrm{M}_{n\times m}(R)$
		to be the set of
		$R$-valued $n\times m$-matrices.
		We denote by $1_n$ the identity matrix of
		size $n\times n$.
		The Borel subgroup of
		$\mathrm{GL}_n$
		consisting of upper triangular matrices
		is denoted by $B_n$,
		its maximal torus by
		$T_n$ and
		unipotent radical by
		$N_n$.
		The transpose of
		$B_n$ is denoted by
		$B_n^-$,
		consisting of lower triangular matrices.
		For any $g\in\mathrm{M}_{n\times m}$,
		we write
		$g^t$ for its transpose.
		If $g\in\mathrm{GL}_n$,
		we write
		$g^{-t}$ for 
		$(g^t)^{-1}$.

		\item
		We write
		$J_{2n}=\begin{pmatrix}
		0 & 1_n \\
		-1_n & 0
		\end{pmatrix}$.
		The similitude symplectic group
		$G=\mathrm{GSp}_{2n}$
		over $\mathbb{Z}$
		is identified with the subgroup of
		$\mathrm{GL}_{2n}$
		consisting of matrices
		$g$ such that
		$g^tJ_{2n}g=\nu(g)J_{2n}$.
		Write ${g=\begin{pmatrix}
		A & B \\
		C & D
		\end{pmatrix}\in\mathrm{GSp}_{2n}}$
		in $n\times n$-blocks.
		We write
		$P_G$
		to be the subgroup of $G$
		consisting of $g$ such that $C=0$,
		$L_G$ the Levi subgroup of
		$P_G$
		consisting of those $g$ such that
		$B=0$,
		$B_G$
		the subgroup of $P_G$
		consisting of $g$ such that
		$A\in B_n$,
		similarly,
		$N_G$
		consisting of $g$ such that
		$A\in N_n$,
		$T_G$
		consisting of $g$ such that
		$A\in T_n$ and $B=0$.
		Let $Z_G$ be the center of $G$.
		We write the Lie algebras of
		$G$ and $T_G$ as
		$\mathfrak{g}_G$ and $\mathfrak{t}_G$ respectively.
		We fix a $\mathbb{C}$-basis for 
		the Lie algebra
		$\mathfrak{g}_G(\mathbb{C})
		=\mathfrak{gsp}_{2n}(\mathbb{C})$ 
		as follows
		\begin{align*}
		\eta^0
		&
		=1_{2n},
		&
		\eta_{i,j}
		&
		=
		E_{i,j}-E_{j+n,i+n},
		&
		(1\leq i,j\leq n),
		\\
		\mu_{i,j}^+
		&
		=
		E_{i,j+n}+E_{j,i+n},
		&
		\mu^-_{i,j}
		&
		=
		E_{i+n,j}+E_{j+n,i},
		&
		(1\leq i<j\leq n),
		\\
		\mu^+_{i,i}
		&
		=
		E_{i,i+n},
		&
		\mu^-_{i,i}
		&
		=
		E_{i+n,i}
		&
		(1\leq i\leq n).
		\end{align*}
		We fix a maximal compact subgroup
		$K_{G,\infty}$
		of $G(\mathbb{R})$
		consisting of matrices
		$\begin{pmatrix}
		A & B \\
		-\nu B & \nu A 
		\end{pmatrix}$
		such that
		$A+iB\in\mathrm{U}_n(\mathbb{R})$
		and $\nu=\pm1$.
		We write $\Gamma(N)$
		for the subgroup of
		$\mathrm{GSp}_{2n}(\mathbb{Z})$
		consisting of matrices $g$
		such that $g\equiv1(\mathrm{mod}\,N)$
		and $\Gamma=\Gamma(N,p^m)$ the subgroup of
		$\Gamma(N)$
		consisting of matrices
		$g$ such that
		$g\,(\mathrm{mod}\,p^m)\in
		N_{\mathrm{Sp}_{2n}}(\mathbb{Z}/p^m)$
		($m>0$).
		Denote by $\widehat{\Gamma}(N)$,
		resp., $\widehat{\Gamma}$
		the completion of
		$\Gamma(N)$, resp., $\Gamma$ in
		$\mathrm{GSp}_{2n}(\mathbb{A}_\mathrm{f})$.

		\item 
		We write $\rho_G$ for the half sum of the positive
		roots of $\mathfrak{gsp}_{2n}(\mathbb{C})$
		with respect to $T_G$,
		$\rho_{L_G}$
		for the half sum of the positive roots
		of the Levi subgroup $L_G$
		with respect to $T_G$.

		\item
		We write
		$G^1:=\mathrm{Sp}_{2n}$
		for the subgroup of
		$\mathrm{GSp}_{2n}$
		consisting of $g$ such that $\nu(g)=1$.
		We write
		$B_{G^1}$,
		$N_{G^1}$,
		$T_{G^1}$
		and
		$K_{G^1,\infty}$
		for the respective intersection
		of the groups $B_G$,
		$N_G$, $T_G$,
		$K_{G,\infty}$
		with $G^1$.

		\item 
		For a symmetric matrix
		$x\in\mathrm{Sym}_{n\times n}$,
		we write
		$u(x)=\begin{pmatrix}
		1_n & x \\
		0 & 1_n
		\end{pmatrix}\in
		\mathrm{GSp}_{2n}$.
		Similarly,
		for an invertible matrix
		$y\in\mathrm{GL}_n$,
		we write
		$m(y)=\begin{pmatrix}
		y & 0 \\
		0 & y^{-t}
		\end{pmatrix}\in
		\mathrm{GSp}_{2n}$.
		For a finite dimensional $F$-vector space $V$
		and any $y\in\mathrm{GL}_F(V)$,
		we also write
		$m(y)=\begin{pmatrix}
		y & 0 \\
		0 & y^\vee
		\end{pmatrix}$
		where $y^\vee\in\mathrm{GL}_F(V^\vee)$ 
		is the dual of $y$.
		For any $g\in\mathrm{GSp}_{2n}(\mathbb{Q}_v)$,
		the Iwasawa decomposition of
		$\mathrm{GSp}_{2n}(\mathbb{Q}_v)$
		gives a decomposition
		$g=\mathrm{diag}(1,\nu(g))m(y)u(x)k$
		for some $y,x$ and
		$k$ in the standard
		maximal compact subgroup
		$\mathrm{GSp}_{2n}(\mathbb{Z}_v)$ of 
		$\mathrm{GSp}_{2n}(\mathbb{Q}_v)$.
		Then we set
		$\mathbf{m}(g)=
		\mathrm{det}(a)^{n-1}\nu(g)^{-n(n-1)/2}$.

		\item 
		Let $\mathbb{H}_n$
		denote the Siegel upper half plane,
		consisting of
		$z\in\mathrm{M}_{n\times n}
		(\mathbb{C})$
		such that
		$z^t=z$ and $\mathrm{Im}(z)$ is positive definite.
		Let $\mathrm{GSp}^+_{2n}(\mathbb{R})$
		be the subgroup of $\mathrm{GSp}_{2n}(\mathbb{R})$
		consisting of $g$ with $\nu(g)>0$.
		The group
		$\mathrm{GSp}^+_{2n}(\mathbb{R})$
		acts on $\mathbb{H}_n$
		by
		$g\cdot z=
		(Az+B)(Cz+D)^{-1}$.
		For each $z\in\mathbb{H}_n$,
		we define a $1$-cocycle on 
		$\mathrm{GSp}^+_{2n}(\mathbb{R})$
		as
		$\mu(g,z):=Cz+D$.

		\item 
		We fix some sufficiently large
		extension of $\mathbb{Q}_p$
		containing 
		$\sqrt{2}$ and
		roots of unity
		$\mu_{N^2C(\pi)}$
		where $C(\xi)$
		is certain positive number depending on
		the $p$-ordinary
		automorphic representation
		(see the beginning of
		Section \ref{ramified place p}
		for its definition).
		Its ring of integers is denoted by
		$\mathcal{O}$
		with a fixed uniformizer
		$\varpi$.
	\end{enumerate}

	\section{Modular forms on $\mathrm{GSp}_4$}
	
	In this section we review the notions of Siegel modular variety
	associated to the symplectic groups
	$\mathrm{GSp}_4$ and
	$\mathrm{GSp}_8$.
	Many of the materials in this section are
	simply bookkeeping 
	from the references given in each subsection
	and we claim no originality,
	yet any errors in this section should be
	the author's of this article.

	\subsection{Siegel modular forms}
	In this subsection we review nearly holomorphic
	modular forms on
	$G=\mathrm{GSp}_{2n}$.
	We follow closely the treatment in
	\cite{Liu2015NearlyOverconvergent,Liu2015}.

	\subsubsection{Arithmetic modular forms
		and Hecke operators}
    A PEL datum is a tuple
	$\mathcal{P}=
	(L,
	\langle\cdot,\cdot
	\rangle,
	h)$
	where
	(\cite[1A]{Lan2012}):
	\begin{enumerate}\label{PEL hypothesis}		
		\item $L$ is a free $\mathbb{Z}$-module
		of finite rank,
		and 
		$\langle\cdot,\cdot
		\rangle
		\colon
		L\times
		L
		\rightarrow
		\mathbb{Z}(1)
		:=\mathrm{Ker}(\mathrm{exp}
		\colon
		\mathbb{C}
		\rightarrow
		\mathbb{C}^\times)$
		is a non-degenerate symplectic pairing
		such that for any
		$b\in\mathbb{Z}$
		and
		$x,y\in L$, we have
		$\langle
		bx,y
		\rangle
		=
		\langle
		x,by
		\rangle$;
		
		\item $h
		\colon
		\mathbb{C}
		\rightarrow
		\mathrm{End}_{\mathbb{R}}
		(L\otimes_\mathbb{Z}\mathbb{R})$
		is a homomorphism such that
		for any
		$z\in\mathbb{C}$
		and
		$x,y\in L\otimes\mathbb{R}$,
		we have
		$\langle
		h(z)x,
		y
		\rangle
		=
		\langle
		x,
		h(\overline{z})y
		\rangle$.
		Moreover, we require that 
		the new bilinear form
		$-\langle
		\cdot,
		h(i)\cdot
		\rangle$
		on $L$
		is symmetric and positive definite.
	\end{enumerate}

    Given such a datum $\mathcal{P}$,
    suppose that $L\simeq\mathbb{Z}^n$,
	one can then define the similitude symplectic group
	$G=G_\mathcal{P}$
	over $\mathbb{Z}$
	associated to the pairing
	$\langle\cdot,\cdot\rangle$ on $L$
	as follows:
	for any $\mathbb{Z}$-algebra $R$,
	the $R$-points of $G$ is given by
	\[
	G(R)=
	\{
	(g,\nu)
	\in
	\mathrm{GL}_{R}(L\otimes R)
	\times (L\otimes R)^\times|
	\langle
	gx,gy
	\rangle
	=
	\nu
	\langle
	x,y
	\rangle,
	\forall
	x,y\in
	L\otimes R
	\}
	\]
	Sometimes we also write an element in
	$G(R)$ as simply $g'=(g,\nu)$ and refer
	to $\nu(g')$ as the similitude factor of $g'$.
	Fixing a $\mathbb{Z}$-basis of $L$
	under which $\langle\cdot,\cdot\rangle$
	is of the form
	$J_{2n}$,
	one can identify $G$
	with $\mathrm{GSp}(2n)$.

	We have an algebraic stack
	$\mathbf{A}_{G,\widehat{\Gamma}}$
	over $\mathrm{Spec}(\mathbb{Z}[1/Np])$
	parameterizing the principally polarized
	abelian schemes over $\mathbb{Q}$ 
	of dimension $n$,
	level structure $\widehat{\Gamma}$
	(\cite[Chapter I.4.11]{FaltingsChai}).
	On the other hand,
	the complexification
	$\mathbf{A}_{G,\widehat{\Gamma}}(\mathbb{C})$
	is isomorphic to the
	Shimura variety
	$G^1(\mathbb{Z})
	\backslash
	\mathbb{H}_n\times G
	(\widehat{\mathbb{Z}})/\widehat{\Gamma}$
	which parameterizes the principally 
	polarized abelian varieties
	over $\mathbb{C}$ of dimension $n$,
	level structure $\widehat{\Gamma}$
	(\cite[Chapter I.6]{FaltingsChai}).
	We write
	$\widetilde{\mathbf{A}}_{G,\widehat{\Gamma}}$
	for the toroidal compactification of
	$\mathbf{A}_{G,\widehat{\Gamma}}$
	with boundary $C=C_{\widehat{\Gamma}}$.
	Let
	$\widetilde{\mathcal{A}}$
	be the universal semi-abelian scheme over $\widetilde{\mathbf{A}}_{G,\widehat{\Gamma}}$
	with the morphism
	${\mathbf{p}
		\colon
		\widetilde{\mathcal{A}}_{G,\widehat{\Gamma}}
		\rightarrow
		\widetilde{\mathbf{A}}_{G,\widehat{\Gamma}}}$.
	Then $\widetilde{\mathcal{A}}_{G,\widehat{\Gamma}}$ restricts
	to the universal abelian scheme 
	$\mathcal{A}_{G,\widehat{\Gamma}}$
	on $\mathbf{A}_{G,\widehat{\Gamma}}$
	and we still write
	$\mathbf{p}\colon
	\mathcal{A}_{G,\widehat{\Gamma}}
	\rightarrow
	\mathbf{A}_{G,\widehat{\Gamma}}$.

	We write
	${\omega(\widetilde{\mathcal{A}}_{G,\widehat{\Gamma}}
		/\widetilde{\mathbf{A}}_{G,\widehat{\Gamma}})
		=
		\mathbf{p}_*(
		\Omega^1_{\widetilde{\mathcal{A}}_{G,\widehat{\Gamma}}/
			\widetilde{\mathbf{A}}_{G,\widehat{\Gamma}}})}$
	for the sheaf of invariant differentials,
	locally free of rank $n$.
	The sheaf
	${
	\mathcal{H}^1_\mathrm{dR}
	(\mathcal{A}_{G,\widehat{\Gamma}}
	/\mathbf{A}_{G,\widehat{\Gamma}})
	=
	R^1\mathbf{p}_*
	(\Omega^1_{\mathcal{A}_{G,\widehat{\Gamma}}
		/\mathbf{A}_{G,\widehat{\Gamma}}})
    }$
	has a canonical extension
	$\mathcal{H}^1_\mathrm{dR}(\mathcal{A}_{G,\widehat{\Gamma}}
	/\mathbf{A}_{G,\widehat{\Gamma}})^\mathrm{can}$
	to $\widetilde{\mathbf{A}}_{G,\widehat{\Gamma}}$.
	One can show that
	$\mathcal{H}^1(\mathcal{A}_{G,\widehat{\Gamma}}
	/\mathbf{A}_{G,\widehat{\Gamma}})^\mathrm{can}$
	is locally free of rank $2n$ with the following Hodge filtration
	\[
	0
	\rightarrow
	\omega(\widetilde{\mathcal{A}}_{G,\widehat{\Gamma}}
	/\widetilde{\mathbf{A}}_{G,\widehat{\Gamma}})
	\rightarrow
	\mathcal{H}^1_\mathrm{dR}
	(\mathcal{A}_{G,\widehat{\Gamma}}
	/\mathbf{A}_{G,\widehat{\Gamma}})^\mathrm{can}
	\rightarrow
	\underline{\mathrm{Lie}}
	(\widetilde{\mathcal{A}}^t_{G,\widehat{\Gamma}}
	/\widetilde{\mathbf{A}}_{G,\widehat{\Gamma}})
	\rightarrow
	0.
	\]
	Moreover, $\mathcal{H}^1_\mathrm{dR}(\mathcal{A}_{G,\widehat{\Gamma}}
	/\mathbf{A}_{G,\widehat{\Gamma}})$
	has a symplectic pairing such that
	$\omega(\widetilde{\mathcal{A}}_{G,\widehat{\Gamma}}
	/\widetilde{\mathbf{A}}_{G,\widehat{\Gamma}})$
	is maximally isotropic.
	One can also show that the Gauss-Manin connection
	on $\mathcal{H}^1_\mathrm{dR}(\mathcal{A}_{G,\widehat{\Gamma}}
	/\mathbf{A}_{G,\widehat{\Gamma}})$
	extends to an integrable connection on
	$\mathcal{H}^1_\mathrm{dR}(\mathcal{A}_{G,\widehat{\Gamma}}
	/\mathbf{A}_{G,\widehat{\Gamma}})^\mathrm{can}$
	with log poles along the boundary $C_{\widehat{\Gamma}}$
	denoted as follows
	\[
	\nabla
	\colon
	\mathcal{H}^1_\mathrm{dR}(\mathcal{A}_{G,\widehat{\Gamma}}
	/\mathbf{A}_{G,\widehat{\Gamma}})^\mathrm{can}
	\rightarrow
	\mathcal{H}^1_\mathrm{dR}(\mathcal{A}_{G,\widehat{\Gamma}}
	/\mathbf{A}_{G,\widehat{\Gamma}})^\mathrm{can}
	\otimes_{\mathcal{O}_{\widetilde{\mathbf{A}}_{G,
				\widehat{\Gamma}}}}
	\Omega^1_{\widetilde{\mathbf{A}}_{G,\widehat{\Gamma}}}
	(\mathrm{log} C_{\widehat{\Gamma}}).
	\]

	We have the following right 
	$P_G$-torsor over 
	$\widetilde{\mathbf{A}}_{G,\widehat{\Gamma}}$
	\[
	\pi
	\colon
	T_\mathcal{H}
	=
	\underline{\mathrm{Isom}}_{
		\widetilde{\mathbf{A}}_{G,\widehat{\Gamma}}}
	(L\otimes_\mathbb{Z}\mathcal{O}_{
		\widetilde{\mathbf{A}}_{G,\widehat{\Gamma}}},
	\mathcal{H}^1_\mathrm{dR}
	(\mathcal{A}_{G,\widehat{\Gamma}}
	/\mathbf{A}_{G,\widehat{\Gamma}})^\mathrm{can})
	\rightarrow
	\widetilde{\mathbf{A}}_{G,\widehat{\Gamma}}
	\]	
	For any $\mathbb{Z}$-algebra $R$,
	a global section $f$ of $\pi_*(\mathcal{O}_{T_\mathcal{H}})$ 
	over $\widetilde{\mathbf{A}}_{G,
		\widehat{\Gamma}/R}$
	can be viewed as a functor assigning to each pair
	$(A,\epsilon)$ over an $R$-algebra $S$
	an element $f(A,\epsilon)\in S$,
	where $A$ is an element in
	$\mathbf{A}_{G,\widehat{\Gamma}}(S)$ and
	$\epsilon$ is a corresponding element in 
	$T_\mathcal{H}(S)$.
	We write
	$\mathrm{Rep}_\mathbb{Q}(P_G)$
	for the category of algebraic representations of 
	$P_G$ over 
	$\mathbb{Q}$-vector spaces.	
	With this torsor one can define the functor
	\[
	\mathcal{E}
	\colon
	\mathrm{Rep}_\mathbb{Q}(P_G)
	\rightarrow
	\mathrm{QCoh}(\widetilde{\mathbf{A}}_{G,\widehat{\Gamma}}),
	V
	\mapsto
	T_\mathcal{H}\times^{P_G}V.
	\]
	The image $\mathcal{E}(V)$ is a locally free sheaf over
	$\widetilde{\mathbf{A}}_{G,\widehat{\Gamma}}$.
	In particular, the similitude factor
	$\nu$ in $G$ defines an element 
	(again denoted by) $\nu$
	in $\mathrm{Rep}_\mathbb{Q}
	(P_G)$
	and thus an invertible sheaf
	$\mathcal{E}(\nu)$ over $\widetilde{\mathbf{A}}_{G,\widehat{\Gamma}}$.
	We will write
	$\mathcal{E}(V)(k)$
	for $\mathcal{E}(V)\otimes\mathcal{E}(\nu)^{\otimes k}$.
	
	Let $\mathfrak{g}_G$ be the Lie algebra
	of $G$.
	Then we write
	$\mathrm{Rep}_\mathbb{Q}(\mathfrak{g}_G,
	P_G)$
	for the category of 
	$(\mathfrak{g}_G,P_G)$-modules.
	More precisely, an object in
	$\mathrm{Rep}_\mathbb{Q}
	(\mathfrak{g}_G,P_G)$
	is an object $W$ in 
	$\mathrm{Rep}_\mathbb{Q}(P_G)$
	with an extra action of $\mathfrak{g}_G$ on $W$
	such that
	the restriction of this action to the Lie algebra
	$\mathfrak{p}_G$ of 
	$P_G$ agrees with the action
	of $\mathfrak{p}_G$
	induced from $P_G$.
	Moreover, we require that
	for any $g\in G$, $X\in\mathfrak{g}_G$
	and $w\in W$, there is the compatibility
	$(gXg^{-1})w=(\mathrm{Ad}(g)X)w$.
	For any $(\mathfrak{g}_G,P_G)$-module $V$,
	the Gauss-Manin connection
	$\nabla$ on
	$\mathcal{H}^1_\mathrm{dR}(\mathcal{A}_{G,\widehat{\Gamma}}
	/\mathbf{A}_{G,\widehat{\Gamma}})^{\mathrm{can}}$
	induces the Gauss-Manin connection on
	the sheaf
	(\cite[Proposition 2.2.3]{Liu2015NearlyOverconvergent})
	\[
	\nabla
	\colon
	\mathcal{E}(V)
	\rightarrow
	\mathcal{E}(V)
	\otimes_{\mathcal{O}_{
			\widetilde{\mathbf{A}}_{G,\widehat{\Gamma}}}}
	\Omega_{\widetilde{\mathbf{A}}_{G,\widehat{\Gamma}}}(\log C).
	\]
	One can show that this connection induces Hecke-equivariant
	maps on global sections.
	We next use $\nabla$
	to construct a differential operator.	
	Let $(\rho,W_\rho)$ be a finite dimension 
	algebraic representation
	of $\mathrm{GL}_n$. 
	We can associate 
	to it a $(\mathfrak{g}_G,P_G)$-module 
	$V_\rho$
	as follows
	(\cite[Section 2.3]{Liu2015NearlyOverconvergent}):
	write 
	$\underline{Z}$
	for the $n\times n$-matrix
	which is symmetric
	with entries $Z_{i,j}=Z_{j,i}$
	($1\leq i,j\leq n$).
	As a $\mathbb{Q}$-vector space, we set
	$V_\rho
	=
	W_\rho[\underline{Z}]$ to be 
	the space of polynomials in
	$Z_{i,j}$ with coefficients in $W_\rho$.
	We define an action of $P_G$ 
	on $V_\rho$ as follows:
	for any
	$\begin{pmatrix}
	a & b \\
	0 & d
	\end{pmatrix}\in P_G$
	and 
	$f(\underline{Z})$, we set
	$\begin{pmatrix}
	a & b \\
	0 & d
	\end{pmatrix}f(\underline{Z})
	:=af(a^{-1}b+a^{-1}\underline{Z}d)$.
	One can also define an action of $\mathfrak{g}_G$ on $V_\rho$
	as in \cite[Section 2.2]{Liu2015}
	and verify that
	$V_\rho$ becomes a $(\mathfrak{g}_G,P_G)$-module.	
	There is a natural filtration on $V_\rho$ respecting the action of 
	$P_G$
	given by the total degree of the polynomials in $V_\rho$:
	$V_\rho=\bigcup_{r\in\mathbb{N}} V_\rho^r$
	where
	$V_\rho^r
	=
	W_\rho[\underline{Z}]_{\mathrm{deg}\leq r}$.
	One can show that
	$\mathfrak{g}_GV_\rho^r\subset V_\rho^{r+1}$
	and the connection
	$\nabla$ restricts to
	$\nabla
	\colon
	\mathcal{E}(V_\rho^r)
	\rightarrow
	\mathcal{E}(V_\rho^{r+1})
	\otimes_{\mathcal{O}_{
			\widetilde{\mathbf{A}}_{G,\widehat{\Gamma}}}}
	\Omega_{
		\widetilde{\mathbf{A}}_{G,\widehat{\Gamma}}}(\mathrm{log} C)$.
	Let $\tau_n
	\colon
	\mathrm{GL}_n
	\rightarrow
	\mathrm{GL}(n(n+1)/2)$
	be the symmetric square representation of
	$\mathrm{GL}_n$.
	Then we get the following differential operator
	(\cite[Section 2.2]{Liu2015}):
	\[
	\nabla_\rho
	\colon
	\mathcal{E}(V_\rho^r)
	\xrightarrow{\nabla}
	\mathcal{E}(V_\rho^{r+1})
	\otimes_{\mathcal{O}_{
			\widetilde{\mathbf{A}}_{G,\widehat{\Gamma}}}}
	\Omega_{
		\widetilde{\mathbf{A}}_{G,\widehat{\Gamma}}}(\log C)
	\xrightarrow{\mathrm{K}-\mathrm{S}}
	\mathcal{E}(V_{\rho\otimes\tau_n})^{r+1}(-1)
	\rightarrow
	\mathcal{E}(V_{\rho\otimes\tau_n})^{r+1}.
	\]
	Here the map
	$\mathrm{K}-\mathrm{S}$
	is the Kodaira-Spencer isomorphism
	(\cite[Propositin 6.9(5)]{Lan2012}).
	Moreover, the composition of the first two maps is
	Hecke-equivariant.

	We next define the $q$-expansion of nearly
	holomorphic forms on the geometric side.
	This will give the integral structure on
	the space of automorphic forms.	
	Let $S_n=(L^+\otimes_\mathbb{Z}L^+)
	/(v\otimes v'=v'\otimes v)$ 
	be the symmetric quotient of
	$L^+\otimes_\mathbb{Z}L^+$.
	Write $S_{n,\geq0}$
	for the subset of $S_n$ consisting of
	elements $(v,v')$
	such that
	$f(v,v')\geq0$
	for any symmetric semi-positive definite
	bilinear form
	$f$ on
	$(L^+\times L^+)\otimes_\mathbb{Z}\mathbb{R}$.
	Write
	$\{s_1,s_2,\cdots,s_{n(n+1)/2}\}$
	for a $\mathbb{Z}$-basis of $S_n$ lying inside $S_{n,\geq0}$.
	We then set the Laurent power series
	$\mathbb{Z}((S_{n,\geq0}))
	=
	\mathbb{Z}[[S_{n,\geq0}]][1/s_1\cdots s_{n(n+1)/2}]$.
	There is a natural embedding
	$S_{n,\geq0}\hookrightarrow
	\mathbb{Z}[[S_{n,\geq0}]]$
	and we denote the image of
	$\beta\in S_{n,\geq0}$
	by $q^\beta$.	
	Now we have a natural map
	$\mathbb{Z}(e_1^+,\cdots,e_n^+)
	\rightarrow
	\mathbb{Z}(e_1^-,\cdots,e_n^-)
	\otimes S_n$
	which gives a period group
	$\mathbb{Z}(e_1^+,\cdots,e_n^+)\subset
	\mathbb{Z}(e_1^-,\cdots,e_n^-)
	\otimes
	\mathbb{G}_{m/\mathbb{Z}((S_{n,\geq0}))}$,
	principally polarized by the duality
	between
	$\mathbb{Z}(e_1^+,\cdots,e_n^+)$
	and 
	$\mathbb{Z}(e_1^-,\cdots,e_n^-)$
	given by the symplectic form on $L$.
	Mumford's construction
	\cite{FaltingsChai} gives an abelian variety
	$A_{/\mathbb{Z}((S_{n,\geq0}))}$
	with a canonical polarization
	$\lambda_\mathrm{can}$
	and a canonical basis
	$\omega_\mathrm{can}
	=(\omega_{1,\mathrm{can}},\cdots,\omega_{n,\mathrm{can}})$
	of
	$\omega(A/\mathbb{Z}((S_{n,\geq0})))$.
	We can then define the level structure 
	$\psi_{N,\mathrm{can}}$
	and filtration
	$\mathrm{fil}^+_{p^m,\mathrm{can}}$
	of this abelian variety
	$A\times_{\mathbb{Z}((S_{n,\geq0}))}
	\mathbb{Z}((\frac{1}{N}S_{n,\geq0}))
	[\zeta_N,\frac{1}{Np}]$
	from the following exact sequence
	\[
	0
	\rightarrow
	\mathbb{Z}(e_1^-,\cdots,e_n^-)
	\otimes
	\prod_l\lim\limits_{\overleftarrow{n}}\mu_{l^n}
	\rightarrow
	\prod_l
	T_l(A)
	\rightarrow
	\widehat{\mathbb{Z}}(e_1^+,\cdots,e_n^+)
	\rightarrow
	0.
	\]

	Denote by $D_{i,j}\in
	\mathrm{Der}(\mathbb{Z}((S_{n,\geq0})),\mathbb{Z}((S_{n,\geq0})))$
	the element dual to
	$\omega_{i,\mathrm{can}}\omega_{j,\mathrm{can}}$
	for all $1\leq i,j\leq n$
	and
	by
	$\delta_{i,\mathrm{can}}
	=
	\nabla(D_{i,i})\omega_{i,\mathrm{can}}$.
	For any $\beta\in S_{n,\geq0}$,
	one can show that
	$D_{i,j}(q^\beta)
	=
	(2-\delta_{i,j})\beta_{i,j}q^\beta$.
	From this one sees that
	$\delta_\mathrm{can}
	=
	(\delta_{1,\mathrm{can}},\cdots,\delta_{n,\mathrm{can}})$
	and
	$\omega_{\mathrm{can}}$
	forms a basis of
	$\mathcal{H}^1_\mathrm{dR}
	(A/\mathbb{Z}((S_{n,\geq0})))$
	compatible with the Hodge filtration
	and the symplectic form.
	For any
	$f\in
	H^0(\widetilde{\mathbf{A}}_{G,\widehat{\Gamma}},
	\mathcal{E}(V_\rho^r))
	=
	M_\rho(\widehat{\Gamma},\mathbb{Z},r)$,
	its evaluation at the Mumford test object
	\[
	(
	A_{/\mathbb{Z}((\frac{1}{N}S_{n,\geq0}))[\zeta_N,\frac{1}{Np}]},
	\lambda_\mathrm{can},
	\psi_{N,\mathrm{can}},
	\mathrm{fil}^+_{p^m,\mathrm{can}},
	\omega_{\mathrm{can}},
	\delta_\mathrm{can}
	)
	\]
	gives the polynomial $q$-expansion 
	$f(q,\underline{Z})$
	of $f$,
	which by definition is an element in
	the power series ring
	${\mathbb{Z}((\frac{1}{N}S_{n,\geq0}))[\zeta_N,\frac{1}{Np}]
		\otimes_\mathbb{Z}W_\rho(\underline{Z})_{\mathrm{deg}\leq r}}$.
	For the case 
	$\rho=\rho_{\underline{k}}$, apply 
	the element
	$\mathfrak{e}$ to
	$f(q,\underline{Z})$ and we get
	an element in
	$\mathbb{Z}[[\frac{1}{N}S_{n,\geq0}]][\zeta_N,\frac{1}{Np}]$.
	In summary, combining the above two successive
	operations on $f$, we get the following
	$q$-expansion map
	($F$ is a $\mathbb{Q}(\zeta_N)$-algebra)
	\begin{equation}\label{q-expansion}
	\varepsilon_q
	\colon
	H^0(\widetilde{\mathbf{A}}_{G,\widehat{\Gamma}},
	\mathcal{E}(V_\rho^r))
	\otimes_{\mathbb{Q}(\zeta_N)}F
	\rightarrow
	\mathbb{Z}[[\frac{1}{N}S_{n,\geq0}]]
	\otimes_\mathbb{Z}F.
	\end{equation}
	The map
	$\varepsilon_q$
	depends on the representation
	sheaf $\mathcal{E}(V_\rho^r)$
	yet we omit from the notation.
	By
	\cite[Lemma V.1.4]{FaltingsChai},
	we know that
	this map is injective.	
	For any subring $R$ of $F$,
	we write
	\[
	H^0(\widetilde{\mathbf{A}}_{G,\widehat{\Gamma},R},
	\mathcal{E}(V_\rho^r))
	:=
	\varepsilon_q^{-1}
	(\mathbb{Z}[[\frac{1}{N}
	S_{n,\geq0}]]
	\otimes_\mathbb{Z}R).
	\]
	This gives an integral structure of
	$H^0(\widetilde{\mathbf{A}}_{G,\widehat{\Gamma}},
	\mathcal{E}(V_\rho^r))
	\otimes_{\mathbb{Q}(\zeta_N)}F$.

	Let $\underline{k}=(k_1,\cdots,k_n)\in
	\mathbb{Z}^n$
	be a dominant character of
	the torus $T_n$ of
	$\mathrm{GL}_n$
	with respect to $B_n^-$
	(i.e.,
	$k_1\geq k_2\geq \cdots \geq k_n$).
	We write $(\rho_{\underline{k}},W_{\underline{k}})$
	for the irreducible algebraic representation of 
	$\mathrm{GL}_n$ associated to
	the character $\underline{k}$
	(of highest weight $\underline{k}$),
	which is defined as follows:
	for any 
	$\mathbb{Z}$-algebra $R$,
	\[
	W_{\underline{k}}(R)
	=
	\{
	f\in H^0(\mathrm{GL}_{n/R},\mathcal{O}_{\mathrm{GL}_n})|
	f(bh)
	=
	\underline{k}(b)f(h),
	\forall
	b\in B_n(R)
	\}
	\]
	An element $g\in \mathrm{GL}_n(R)$ acts on
	$f\in W_{\underline{k}}(R)$
	by $(gf)(h)=f(hg)$.

	\begin{definition}
		\label{nearly holomorphic automorphic forms}
		We fix a dominant character $\underline{k}$
		of $T_n$,
		an integer $r\geq0$,
		a nearly holomorphic 
		automorphic form of weight $\underline{k}$,
		degree $r$, level $\widehat{\Gamma}$
		on $\widetilde{\mathbf{A}}_{G,\widehat{\Gamma}}$
		is a global section of the sheaf
		$\mathcal{E}(V_{\underline{k}})^r
		:=\mathcal{E}(V_{\underline{k}}^r)$.
		The holomorphic automorphic form
		of weight $\underline{k}$, level $\widehat{\Gamma}$
		on $\widetilde{\mathbf{A}}_{G,\widehat{\Gamma}}$
		is a global section of the sheaf
		$\mathcal{E}(V_{\underline{k}}^0)$.
		
		More generally, for any $\mathbb{Z}$-algebra $R$,
		we write the space of
		$R$-valued, weight $\underline{k}$, degree $r$
		and level $\widehat{\Gamma}$ 
		nearly holomorphic automorphic forms on
		$\widetilde{\mathbf{A}}_{G,\widehat{\Gamma}}$
		as
		${
		M_{\underline{k}}(\widehat{\Gamma},R,r)
		=H^0
		(\widetilde{\mathbf{A}}_{G,\widehat{\Gamma}
			/R},\mathcal{E}(V_{\underline{k}}^r)).
	    }$.		
		The subspace of cuspidal forms is
		$S_{\underline{k}}(\widehat{\Gamma},R,r)
		=
		H^0(\widetilde{\mathbf{A}}_{G,\widehat{\Gamma}/R},
		\mathcal{E}(V_{\underline{k}}^r)
		(-C_{\widehat{\Gamma}}))$.

		Similarly we set
		$M_{\underline{k}}(\widehat{\Gamma},R)
		=
		H^0(\widetilde{\mathbf{A}}_{G,\widehat{\Gamma}/R},
		\mathcal{E}(V_{\underline{k}}))$
		and
		$
		S_{\underline{k}}(\widehat{\Gamma},R)
		=
		H^0(\widetilde{\mathbf{A}}_{G,\widehat{\Gamma}/R},
		\mathcal{E}(V_{\underline{k}})
		(-C_{\widehat{\Gamma}}))
		$.
	\end{definition}

	The moduli interpretation
	(\`{a} la Katz) of an element
	$f\in M_{\underline{k}}(\widehat{\Gamma},R,r)$
	is as follows:
	$f$ assigns in a functorial way to each
	tuple $(A_{/S},\lambda,\psi,\alpha)$
	an element
	$f(A_{/S},\lambda,\psi,\alpha)
	\in V_{\underline{k}}^r(S)$.
	Here $S$ is an $R$-algebra and
	$A_{/S}$ is an abelian scheme over $S$,
	$\lambda\colon
	A\rightarrow
	A^\vee$ is a principal polarization,
	$\psi$ is a principal level $N$ structure and
	$\alpha$ is a basis of
	$H^1_\mathrm{dR}(A/S)$
	respecting the Hodge filtration.

	For any algebraic representation
	$(\rho,W_\rho)$ of $\mathrm{GL}_n$,
	one can also consider the top degree 
	coherent cohomology
	$H^d(\widetilde{\mathbf{A}}_{G,\widehat{\Gamma}},
	\mathcal{E}(V_\rho^r))$
	where $d=n(n+1)/2$.
	Let $\rho^\vee$
	denote the dual representation of $\rho$.
	We have Serre duality
	\[
	\langle\cdot,\cdot\rangle^\mathrm{Ser}
	\colon
	H^0(\widetilde{\mathbf{A}}_{G,\widehat{\Gamma}}(\mathbb{C}),
	\mathcal{E}(V_\rho^0))
	\times
	H^d(\widetilde{\mathbf{A}}_{G,\widehat{\Gamma}}(\mathbb{C}),
	\mathcal{E}(V_{\rho^\vee}^0)
	\otimes\wedge^d\Omega_{
		\widetilde{\mathbf{A}}_{G,\widehat{\Gamma}}})
	\rightarrow
	\mathbb{C}.
	\]
	For any dominant weight $\underline{k}$
	of $T_n$,
	we write
	$\underline{k}^D:=
	(-k_n,-k_{n-1},\cdots,-k_1)+
	2(\rho_G-\rho_{L_G})$
	where recall that
	$\rho_G$, resp.,
	$\rho_{L_G}$
	denotes the half sum of positive roots of
	$G$, resp., the Levi subgroup $L_G$
	with respect to $T_G$.
	Then we have an isomorphism
	${\mathcal{E}(V_{\underline{k}^\vee}^0)
	\otimes\wedge^d\Omega_{
		\widetilde{\mathbf{A}}_{G,\widehat{\Gamma}}}
	\simeq
	\mathcal{E}(V_{\underline{k}^D}^0)}$
    (\cite[p.256]{FaltingsChai}).
    Note that
    $W_{2(\rho_G-\rho_{L_G})}$
    is the $1$-dimensional representation of
    $G$ sending $g$ to $\nu(g)^d$.
    The Serre duality can be concretely expressed as
    follows:
    for any $\phi\in
    H^0(\widetilde{\mathbf{A}}_{G,\widehat{\Gamma}}(\mathbb{C}),
    \mathcal{E}(W_\rho))$
    and
    $\rho'\in
    H^d(\widetilde{\mathbf{A}}_{G,\widehat{\Gamma}}(\mathbb{C}),
    \mathcal{E}(W_{\rho^\vee})\otimes
    \wedge^d\Omega_{
    	\widetilde{\mathbf{A}}_{G,\widehat{\Gamma}}})$,
    then
    one has
    \begin{equation}
    \label{Petersson product vs. Serre duality}
    \langle
    \phi,\phi'
    \rangle^\mathrm{Ser}
    =\int_{[G]}
    \phi(g)\phi'(g)|\nu(g)|^{-d}
    dg.
    \end{equation}

	Using the Serre duality,
	we define
	\begin{definition}
		\label{antiholomorphic automorphci forms}
		For any $\mathbb{Z}$-algebra $R$
		and any dominant weight $\underline{k}$
		of $T_n$,
		we write
		\[
		\widehat{M}_{\underline{k}}(\widehat{\Gamma},R)
		:=
		H^d(\mathbf{A}_{G,\widehat{\Gamma},R},
		\mathcal{E}(V_{\underline{k}^D}^0))
		:=
		\mathrm{Hom}_R(
		H^0(\mathbf{A}_{G,\widehat{\Gamma},R},
		\mathcal{E}(V_{\underline{k}}^0)),R
		).
		\]
		\[
		\widehat{S}_{\underline{k}}(\widehat{\Gamma},R)
		:=
		H^d_!(\mathbf{A}_{G,\widehat{\Gamma},R},
		\mathcal{E}(V_{\underline{k}^D}^0))
		:=
		\mathrm{Hom}_R(
		H^0(\mathbf{A}_{G,\widehat{\Gamma},R},
		\mathcal{E}(V_{\underline{k}}^0)),R).
		\]
	\end{definition}
    These $R$-modules give the integral structures of
    the top degree (cuspidal) cohomologies
    (\cite[Section 6.3]{EischenHarrisLiSkinner16}).

	\subsubsection{Automorphic forms}
	In this subsection, we associate
	automorphic forms to modular forms defined above.
	Fix an algebraic representation
	$(\rho,W_\rho)$ of $\mathrm{GL}_n$.	
	Recall that
	${\mathbf{A}_{G,\widehat{\Gamma}}(\mathbb{C})
	=G^1(\mathbb{Z})\backslash
	\mathbb{H}_n\times
	G(\widehat{\mathbb{Z}})/
	\widehat{\Gamma}}$.
    By the moduli interpretation mentioned above,
	to each $(z,k)\in\mathbb{H}_n\times
	G(\widehat{\mathbb{Z}})$,
	we can associate an abelian variety
	$A_{z,k}:=\mathbb{C}^n/(\mathbb{Z}^n\oplus
	z\mathbb{Z}^n)$,
	with a polarization
	$\lambda_z$ given by the 
	Hermitian matrix
	$\mathrm{Im}(z)^{-1}$
	for its Riemann form
	$E_z
	\colon
	\mathbb{C}^n\times\mathbb{C}^n
	\rightarrow
	\mathbb{R},
	(w_1,w_2)
	\mapsto
	-i\mathrm{Im}(w_1^t\mathrm{Im}(z)^{-1}\overline{w_2})$
	and a level structure
	$\psi_{z,k}
	\colon
	(\mathbb{Z}/N)^n
	\simeq
	A_{z,k}[N]=
	((\frac{1}{N}\mathbb{Z})^n
	\oplus z(\frac{1}{N}\mathbb{Z})^n)/
	(\mathbb{Z}^n\oplus
	z\mathbb{Z}^n)
	$
	where $k$ acts via its projection to
	$G(\mathbb{Z}/N)$
	(\cite[Section 2.1]{Pilloni2012}).
	For any 
	$\gamma=\begin{pmatrix}
	A & B \\
	C & D
	\end{pmatrix}
	\in G^1(\mathbb{Z})$,
	note that
	$\begin{pmatrix}
	A & B \\
	C & D
	\end{pmatrix}
	\begin{pmatrix}
	z \\
	1
	\end{pmatrix}
	(Cz+D)^{-1}
	=\begin{pmatrix}
	\gamma z \\
	1
	\end{pmatrix}$,
	thus the isomorphism
	$(Cz+D)
	\colon
	\mathbb{C}^n
	\rightarrow
	\mathrm{C}^n$
	sending $z'$ to
	$z'(Cz+D)^{-1}$
	induces an isomorphism
	$\phi_\gamma
	\colon
	A_{z,k}
	\rightarrow
	A_{\gamma z,\gamma k}$
	respecting the polarization.
	Let $w_1,w_2,\cdots,w_n$
	be the coordinates of $\mathbb{C}^n$
	and then
	$dw:=\{dw_1,\cdots,dw_n\}$
	form a trivialization of
	$H^1(A_{z,k}/\mathbb{C})$.
	For any
	modular form
	$f\in H^0(\widetilde{\mathbf{A}}_{G,\widehat{\Gamma}},
	\mathcal{E}(V_\rho^r))$,
	we can define a function
	$F_f
	\colon
	\mathbb{H}_n\times G(\widehat{\mathbb{Z}})
	\rightarrow
	V_\rho^r(\mathbb{C})$
	by
	$F_f(z,k):=
	f(A_{z,k},\lambda_z,\psi_{z,k},dw)$.
	Then one verifies that
	for any $\gamma\in G^1(\mathbb{Z})$,
	$F_f(\gamma z,\gamma k)
	=\rho(\mu(\gamma,z))
	F_f(z,k)$
	and for any $k'\in \widehat{\Gamma}$,
	$F_f(z,k)
	=F_f(z,kk')$
	(\cite[Section 2.4]{Pilloni2012}
	and \cite[Section 2.5]{Liu2015NearlyOverconvergent}).

	By strong approximation for
	$\mathrm{Sp}(2n)$,
	we have
	$ G(\mathbb{A})
	= G(\mathbb{Q})
	 G^+(\mathbb{R})
	 G(\widehat{\mathbb{Z}})$,
	so we write each element in
	$ G(\mathbb{A})$
	as $g=g_0g_\infty g_\mathrm{f}$.
	Now suppose that there is a real number
	$m(\rho)\in\mathbb{R}$
	such that
	$\rho_\mathbb{C}(\lambda\cdot1_n)
	=\lambda^{2m(\rho)}\mathrm{Id}_{W_\rho(\mathbb{C})}$
	for any $\lambda\in\mathbb{C}^\times$.
	Then to each
	$f\in H^0(\widetilde{\mathbf{A}}_{G,\widehat{\Gamma}},
	\mathcal{E}(V_\rho^r))$,
	one can associate a function
	\[
	\widetilde{\Phi}(f)
	\colon
	 G(\mathbb{A})
	\rightarrow
	V_\rho^r(\mathbb{C}),
	\quad
	g\mapsto
	\nu(g_\infty)^{m(\rho)}
	\rho(\mu(g_\infty,i1_n))^{-1}
	F_f(g_\infty\cdot i1_n,g_\mathrm{f}).
	\]
	One can then verify that
	for each
	$g_0\in G(\mathbb{Q})$,
	$a\in Z_G(\mathbb{A})$,
	$k_\infty\in K_{ G,\infty}^1$
	and $k\in \widehat{\Gamma}$,
	we have 
	${
	\widetilde{\Phi}(f)(g_0ag(k_\infty,k))
	=\rho(\mu(k_\infty,i1_n))^{-1}\widetilde{\Phi}(f)(g)
	}$.
    Now for any linear form
    $w^\vee\in W_\rho^\vee$,
    we define a new function
    $\Phi(f,w^\vee)$ by
    $\Phi(f,w^\vee)(g):=
    w^\vee(\widetilde{\Phi}(f)(g)|_{\underline{Z}=0})\in\mathbb{C}$
    (recall that
    $V_\rho^r=W_\rho[\underline{Z}]_{\mathrm{deg}\leq r}$).
    So we see that
    $\Phi(f,w^\vee)
    \in
    \mathcal{A}(Z_G(\mathbb{A})
    G(\mathbb{Q})\backslash
    G(\mathbb{A})/\widehat{\Gamma})$ is an automorphic form
    of trivial center character and level $\widehat{\Gamma}$.
    Moreover we fix a Hermitian form
    $\langle\cdot,\cdot\rangle_\rho^r$ on
    $V_\rho^r$ invariant under the action of
    $K_{G,\infty}^1$
    (unique up to a scalar),
    then we define a Petersson product on
    $H^0(\widetilde{\mathbf{A}}_{G,\widehat{\Gamma}},
    \mathcal{E}(V_{\rho}^r))$
    as follows:
    for any two cuspidal modular forms $f,f'$,
    $\langle f,f'\rangle
    :=
    \int_{\widetilde{\mathbf{A}}_{G,\widehat{\Gamma}}(\mathbb{C})}
    \langle\rho(\mathrm{Im}(z)^{1/2})F_f(z,k),
    \rho(\mathrm{Im}(z)^{1/2})F_{f'}(z,k)dzdk$.
    Then one verifies that
    (\cite[p.195]{AsgariSchmidt2001})
    for any $w^\vee\neq0$,
    there is a constant $c>0$
    such that
    $\langle f,f\rangle
    =c\int_{Z_{G}(\mathbb{A})
    	G(\mathbb{Q})\backslash
    	G(\mathbb{A})}
    |\Phi(f,w^\vee)(g)|dg$
    (we will take $c=1$ in the below).

    We next define some differential operators on
    $\mathcal{A}(Z_G(\mathbb{A})
    G(\mathbb{Q})\backslash
    G(\mathbb{A})/\widehat{\Gamma})$.
    Recall that
    $h
    \colon
    \mathbb{C}
    \rightarrow
    \mathrm{End}_\mathbb{R}(L\otimes\mathbb{R})$
    sends
    $x+iy$
    to
    $\begin{pmatrix}
    x1_n & y1_n \\
    -y1_n & x1_n
    \end{pmatrix}$.
    We then let
    $\mathbb{C}^\times$ act on $G(\mathbb{R})$
    by conjugation composed with $h$,
    i.e.,
    $(x+iy)\cdot g
    =
    h(x+iy)gh(x+iy)^{-1}$.
    This action induces an action of $\mathbb{C}^\times$
    on the complexification
    $\mathfrak{g}_{G,\mathbb{C}}$ of
    the Lie algebra
    $\mathfrak{g}_G$ of $G(\mathbb{R})$.
    We then write
    $\mathfrak{g}_{G,\mathbb{C}}^{a,b}$
    for the subspace of $\mathfrak{g}_{G,\mathbb{C}}$
    on which $z\in\mathbb{C}^\times$
    acts by the multiplication by the scalar
    $z^{-a}\overline{z}^{-b}$.
    It is easy to verify that
    we have the following decomposition
    of eigenspaces
    $\mathfrak{g}_{G,\mathbb{C}}
    =
    \mathfrak{g}^{-1,1}_{G,\mathbb{C}}
    \oplus
    \mathfrak{g}^{0,0}_{G,\mathbb{C}}
    \oplus
    \mathfrak{g}^{1,-1}_{G,\mathbb{C}}
    =:
    \mathfrak{g}_{G,\mathbb{C}}^+
    \oplus
    \mathfrak{t}_{G,\mathbb{C}}
    \oplus
    \mathfrak{g}_{G,\mathbb{C}}^-$.
    We define the Cayley element
    $\mathfrak{c}_n
    =
    \frac{1}{\sqrt{2}}
    \begin{pmatrix}
    1_n & i1_n \\
    i1_n & 1_n
    \end{pmatrix}$.
    Then one verifies that
    $\mathfrak{g}_{G,\mathbb{C}}^+$,
    resp.,
    $\mathfrak{g}_{G,\mathbb{C}}^-$,
    resp.,
    $\mathfrak{t}_{G,\mathbb{C}}$,
    is generated by
    $\widehat{\mu}_{i,j}^+:=
    \mathfrak{c}_n\mu_{i,j}^+\mathfrak{c}_n^{-1}$,
    resp.,
    $\widehat{\mu}_{i,j}^-
    :=\mathfrak{c}_n\mu_{i,j}^-\mathfrak{c}_n^{-1}$,
    resp.,
    $\widehat{\eta}_{i,j}
    :=\mathfrak{c}_n\eta_{i,j}\mathfrak{c}_n^{-1}$
    and $\widehat{\eta}^0:=\eta^0$.
    Suppose that the 
    symmetric square representation
    $(\tau_n,W_{\tau_n})$
    has a basis
    $\{X_{i,j}\}_{1\leq i,j\leq n}$
    such that $X_{i,j}=X_{j,i}$
    and
    $X_{i,j}$
    corresponds to
    $dw_idw_j=2i\pi dz_{i,j}$.
    Write $\{X_{i,j}^\vee\}_{1\leq i,j\leq n}$
    the basis of 
    $W_{\tau_n}^\vee$
    dual to
    $\{X_{i,j}\}_{1\leq i,j\leq n}$.
    Then we have
    (\cite[Proposition 2.5]{Liu2015}
    and \cite[Lemma 5]{AsgariSchmidt2001}):
    \begin{proposition}
    	The map
    	$\Phi(\cdot,w^\vee)
    	\colon
    	H^0(\widetilde{\mathbf{A}}_{G,\widehat{\Gamma}},
    	\mathcal{E}(V_\rho^r))
    	\rightarrow
    	\mathcal{A}(Z_G(\mathbb{A})
    	G(\mathbb{Q})\backslash
    	G(\mathbb{A})/\widehat{\Gamma})$
    	sending
    	$f$ to
    	$\Phi(f,w^\vee)$
    	preserves the norms of both sides and
    	satisfies the following commutative diagram:
    	\[
    	\begin{tikzcd}
    	H^0(\widetilde{\mathbf{A}}_{G,\widehat{\Gamma}}
    	(\mathbb{C}),
    	\mathcal{E}(V_\rho^r))
    	\ar[r,"\Phi(\cdot{,}w^\vee)"]
    	\ar[d,"4i\pi\nabla_\rho"]
    	&
    	\mathcal{A}(Z_G(\mathbb{A})
    	G(\mathbb{Q})\backslash
    	G(\mathbb{A})/\widehat{\Gamma})
    	\ar[d,"\widehat{\mu}_{i,j}^+"]
    	\\
    	H^0(\widetilde{\mathbf{A}}_{G,\widehat{\Gamma}}
    	(\mathbb{C}),
    	\mathcal{E}(V_{\rho\otimes\tau_n}^{r+1}))
    	\ar[r,"\Phi(\cdot{,}w^\vee\otimes X_{i,j}^\vee)"]
    	&
    	\mathcal{A}(Z_G(\mathbb{A})
    	G(\mathbb{Q})\backslash
    	G(\mathbb{A})/\widehat{\Gamma})
    	\end{tikzcd}
    	\]
    	Moreover, if $f$ is cuspidal, then so is
    	$\Phi(f,w^\vee)$;
    	$f$ is holomorphic,
    	resp., anti-holomorphic,
    	if and only if
    	$\mathfrak{g}_{G,\mathbb{C}}^-\Phi(f,w^\vee)=0$
    	resp.,
    	$\mathfrak{g}_{G,\mathbb{C}}^+\Phi(f,w^\vee)=0$,
    	for all $w^\vee\in W_\rho^\vee$.
    \end{proposition}

    This shows that 
    the map
    $\Phi(\cdot,w^\vee)$
    is a norm-preserving map from
    $H^0(\widetilde{\mathbf{A}}_{G,\widehat{\Gamma}},
    \mathcal{E}(V_{\rho}^r))$
    to the space
    $L^2(Z_G(\mathbb{A})
    G(\mathbb{Q})\backslash
    G(\mathbb{A}))$.    
    In particular, let 
    $(\rho_{\underline{k}},W_{\underline{k}})$
    be the irreducible representation of
	    $\mathrm{GL}_n$
    of dominant weight $\underline{k}$,
    then we define the element
    $\mathfrak{e}\in W_{\underline{k}}^\vee$
    to be
    $\mathfrak{e}(f):=f(1_n)$
    for any $f\in W_{\underline{k}}(R)$
    and any $\mathbb{Z}$-algebra $R$.
    We write
    $\mathcal{A}(Z_G(\mathbb{A})
    G(\mathbb{Q})\backslash G
    (\mathbb{A})/\widehat{\Gamma})_{\underline{k}}$
    as the $\rho_{\underline{k}}$-isotypic part
    of
    $\mathcal{A}(Z_G(\mathbb{A})
    G(\mathbb{Q})\backslash G
    (\mathbb{A})/\widehat{\Gamma})$
    as  $K_{G,\infty}$-representations.
    Then we have a norm-preserving injective map:
    \begin{equation}
    \label{correspondence between geometric and adelic forms}
    \Phi(\cdot,\mathfrak{e})
    \colon
    H^0(\widetilde{\mathbf{A}}_{G,\widehat{\Gamma}}
    (\mathbb{C}),
    \mathcal{E}(V_{\underline{k}}^r))
    \rightarrow
    \mathcal{A}(Z_G(\mathbb{A})
    G(\mathbb{Q})\backslash G
    (\mathbb{A})/\widehat{\Gamma})_{\underline{k}}.
    \end{equation}

	This map gives us the correspondence
	between the automorphic forms
	of geometric nature and those of adelic nature
	which is compatible with the differential operators
	on both sides.

	\subsubsection{Hecke operators}
	\label{Hecke operators on GSp4}
	We define adelic Hecke operators.
	We fix a prime $\ell$ of $\mathbb{Q}$.
	Then for any smooth functions
	$T\colon
	G(\mathbb{Q}_\ell)
	\rightarrow
	\mathbb{C}$
	and
	$f\colon
	G(\mathbb{A})
	\rightarrow
	\mathbb{C}$,
	we set
	\[
	(Tf)(g)
	:=
	\int_{G(\mathbb{Q}_\ell)}T(g')f(gg')dg',
	\forall
	g\in G(\mathbb{A}).
	\]
	(here we view $g'\in G(\mathbb{Q}_\ell)$
	also as an element in $G(\mathbb{A})$).
	Now let
	$T$ be the characteristic function of the double coset
	$G(\mathbb{Z}_\ell)MG(\mathbb{Z}_\ell)$
	with $M\in G(\mathbb{Q}_\ell)$
	and consider two functions
	$f_1,f_2\in
	L^2(Z_G(\mathbb{A})G(\mathbb{Q})\backslash G(\mathbb{A}))$,
	then it is easy to see
	\begin{lemma}
		We have the following identity:
		\[
		\langle Tf_1,f_2\rangle
		=
		\langle f_1,Tf_2\rangle.
		\]
		In other words,
		the operator $T$ is self-adjoint with respect to the
		Petersson product on
		$L^2(Z_G(\mathbb{A})G(\mathbb{Q})\backslash G(\mathbb{A}))$.
	\end{lemma}

    For any $\mathbb{Z}$-algebra $A$
    and any compact open subgroup 
    $K\subset G(\mathbb{Q}_\ell)$,
    we write
    $\mathcal{H}(G(\mathbb{Q}_\ell),K;A)$
    for the associative algebra generated over $A$
    by characteristic functions
    $\mathbf{1}(KMK)$
    of $KMK$ with $M\in G(\mathbb{Q}_\ell)$,
    where the multiplication is given by convolution
    $(T_1T_2)(g):=
    \int_{G(\mathbb{Q}_\ell)}
    T_1(gx)T_2(x^{-1})dx$.
    Here the Haar measure on
    $G(\mathbb{Q}_\ell)$
    is the one where
    $K$ has volume $1$.
    By the Cartan decomposition for
    $G(\mathbb{Q}_\ell)$,
    the algebra
    $\mathcal{H}(G(\mathbb{Q}_\ell),
    G(\mathbb{Z}_\ell),A)$
    is generated by the elements
    $T_{\ell,i}^{(n)}
    =\mathbf{1}(G(\mathbb{Z}_\ell)
    \mathrm{diag}(\ell\cdot1_{n-i},\ell\cdot1_i,
    \ell^2\cdot1_{n-i},\ell\cdot1_i)
    G(\mathbb{Z}_\ell))$
    for $i=1,\cdots,n$,
    $T_{\ell,0}^{(n)}
    =\mathbf{1}(G(\mathbb{Z}_\ell)
    \mathrm{diag}(1_n,\ell\cdot1_n)
    G(\mathbb{Z}_\ell))$
    and
    $(T_{\ell,n}^{(n)})^{-1}$.
    Let $T_G$ be the standard torus of $G$ consisting of
    diagonal elements in $G$.
    Then we can similarly define an algebra
    $\mathcal{H}(T_G(\mathbb{Q}_\ell),
    T_G(\mathbb{Z}_\ell),A)$.
    Writing
    $X_0:=\mathbf{1}(T_G(\mathbb{Z}_\ell)
    \mathrm{diag}(1_n,\ell\cdot1_n))$,
    $X_i:=\mathbf{1}(T_G(\mathbb{Z}_\ell)
    \mathrm{diag}(1_{i-1},\ell,1_{n-1},
    \ell^{-1},1_{n-i}))$
    for $i=1,\cdots,n$,
    then as $\mathbb{C}$-algebras,
    $\mathcal{H}(T_G(\mathbb{Q}_\ell),
    T_G(\mathbb{Z}_\ell),\mathbb{C})$
    is isomorphic to
    $\mathbb{C}[X_0^\pm,\cdots,X_n^\pm]$.
    For any element
    $T\in\mathcal{H}(G(\mathbb{Q}_\ell),
    G(\mathbb{Z}_\ell),\mathbb{C})$,
    we define its Satake transform 
    $\mathbf{S}(T)$ to be
    an element in
    $\mathcal{H}(T_G(\mathbb{Q}_\ell),T_G(\mathbb{Z}_\ell),
    \mathbb{C})$
    as follows:
    $\mathbf{S}(T)(g)
    :=|\delta_{B_G}(g)|^{1/2}
    \int_{N_G}T(gn)dn$.
    Here $\delta_{B_G}(\cdot)$
    is the modular character of
    $G$ with respect to the Borel subgroup
    $B_G$ of $G$,
    on the diagonal matrices, it is given by:
    $\delta_{B_G}
    (\mathrm{diag}(t_1,\cdots,t_n,
    t_0/t_1,\cdots,t_0/t_n))
    =t_0^{-n(n+1)/2}t_1^2t_2^4\cdots t_n^{2n}$.
    Write $W_G$ for the Weyl subgroup of $G$ with respect to
    the pair $(B_G,T_G)$.
    Then we have an isomorphism of
    algebras
    $\mathbf{S}
    \colon
    \mathcal{H}(G(\mathbb{Q}_\ell),
    G(\mathbb{Z}_\ell),\mathbb{C})
    \simeq
    \mathcal{H}(T_G(\mathbb{Q}_\ell),
    T_G(\mathbb{Z}_\ell),\mathbb{C})^{W_G}$
    (\cite[Section 3]{AsgariSchmidt2001}).

	We then define some Iwahori Hecke operators.
	We write
	$I_{G,\ell}$ for the Iwahori subgroup of $G(\mathbb{Z}_\ell)$
	consisting of matrices
	$g$ which is in $B_G(\mathbb{Z}_\ell/\ell)$
	modulo $\ell$.
	Then the dilating Iwahori Hecke algebra
	$\mathcal{H}^-(G(\mathbb{Q}_\ell),I_{G,\ell},A)$
	is the subalgebra of
	$\mathcal{H}(G(\mathbb{Q}_\ell),
	I_{G,\ell},A)$
	(which is no longer commutative)
	that is generated over $A$ by the elements
	$U_{\ell,i}^{(n)}
	=\mathbf{1}(I_{G,\ell}
	\mathrm{diag}(\ell\cdot1_i,1_{n-i},\ell\cdot1_i,\ell^2\cdot1_{n-i})
	I_{G,\ell})$
	for $i=1,\cdots,n$,
	$U_{\ell,0}^{}(n)
	=\mathbf{1}(I_{G,\ell}
	\mathrm{diag}(1_n,\ell\cdot1_n)
	I_{G,\ell})$
	and $(U_{\ell,n}^{(n)})^{-1}$.
	Similarly, we define
	$\mathcal{H}^-(T_G(\mathbb{Q}),
	T_G(\mathbb{Z}_\ell),A)$
	to be the subalgebra of
	$\mathcal{H}(T_G(\mathbb{Q}_\ell),
	T_G(\mathbb{Z}_\ell),A)$
	generated by the elements
	$T_G(\mathbb{Z}_\ell)
	\mathrm{diag}(\ell\cdot1_i,1_{n-i},\ell\cdot1_i,\ell^2\cdot1_{n-i})$
	for $i=1,\cdots,n$,
	$T_G(\mathbb{Z}_\ell)
	\mathrm{diag}(1_n,\ell\cdot1_n)$
	and
	${(T_G(\mathbb{Z}_\ell)
	\mathrm{diag}(\ell\cdot1_{2n}))^{-1}}$.
	Then one verifies that the map
	$\mathbf{S}^-
	\colon
	\mathcal{H}^-(G(\mathbb{Q}_\ell),I_{G,\ell},A)
	\rightarrow
	\mathcal{H}^-(T_G(\mathbb{Q}_\ell),
	T_G(\mathbb{Z}_\ell),A)$
	sending generators to the corresponding generators
	is an isomorphism of algebras
	(\cite[Lemma 4.1.5]{Casselman1995},
	\cite[Proposition 6.4.1]{BellaicheChenevier2009}).
	When $\ell=p$ and we let $U_{p,i}^{(n)}$
	act on the automorphic forms,
	we need to be careful about normalization
	of these operators.
	More precisely, for any
	dominant weight
	$\underline{k}\in\mathbb{Z}^n$ of
	$T_n$
	and
	$f\in
	\mathcal{A}(Z_G(\mathbb{A})G(\mathbb{Q})\backslash
	G(\mathbb{A}))_{\underline{k}}$,
	we define the action as(\cite[(2.5.2)]{Liu2015}):
	\[
	(U_{p,i}^{(n)}f)(g)
	:=
	p^{\langle \underline{k}-2\rho_{G,nc},\underline{i}
	\rangle}
	\int_{G(\mathbb{Q}_p)}
	U_{p,i}^{(n)}(h)f(gh)dh
	\]
	Here $\underline{i}=(i_0,i_1,\cdots,i_n)$
	is an element in $\mathbb{Z}^{n+1}$
	such that
	$i_1=i_2=\cdots=i_i=1$,
	$i_{i+1}=\cdots=i_n=0$,
	$i_0=2$ for $i=1,\cdots,n$
	and
	$0_1=\cdots=0_n=0$,
	$0_0=1$.
	Moreover $2\rho_{G,nc}$
	is the sum of non-compact positive roots of $G$
	with respect to $B_G$.
	More generally we write
	$C_n^+$ for the subset of $\mathbb{Z}^{n+1}$
	consisting of elements
	$\underline{a}=(a_0,a_1,\cdots,a_n)$
	such that $a_1\geq a_2\geq\cdots\geq a_n\geq0$.
	Then for any $\underline{a}\in C_n^+$,
	we write
	$p^{\underline{a}}
	=\mathrm{diag}(p^{a_1},\cdots,p^{a_n},
	p^{a_0-a_1},\cdots,p^{a_0-a_n})$.
	We define
	$U_{p,\underline{a}}^{(n)}$
	to be the characteristic function
	$\mathbf{1}(I_{G,p}
	p^{\underline{a}}I_{G,p})$
	and let it act on
	$\mathcal{A}(Z_G(\mathbb{A})G(\mathbb{Q})\backslash
	G(\mathbb{A}))_{\underline{k}}$
	by the same formula as above.
	We write
	$\mathbb{U}_p(f)$
	for the subspace of
	$\mathcal{A}(Z_G(\mathbb{A})G(\mathbb{Q})\backslash
	G(\mathbb{A}))$
	generated by $U_{p,\underline{a}}^{(n)}(f)$
	for all $\underline{a}\in C_n^+$.
	Then by \cite[Proposition 2.5.2]{Liu2015},
	the eigenvalues of these operators
	$U_{p,\underline{a}}^{(n)}$
	acting on $\mathbb{U}_p(f)$
	are all $p$-adic integers
	under the fixed isomorphism
	$\iota^\infty_p\colon
	\mathbb{C}\simeq
	\overline{\mathbb{Q}}_p$.
	We define
	$U_p^{(n)}:=\prod_{i=0}^nU_{p,j}^{(n)}$.
	For each nearly-holomorphic modular form
	$f\in
	\mathcal{A}(Z_G(\mathbb{A})G(\mathbb{Q})\backslash
	G(\mathbb{A}))_{\underline{k}}$,
	we define an operator
	$e$ acting on
	$\mathbb{U}_p(f)$
	as the limit
	$e:=\lim\limits_{r\rightarrow+\infty}U_p^{r!}$
	(the limit is taken with respect to the $p$-adic topology
	on $\mathbb{U}_p(f)$
	via $\mathbb{C}\simeq\overline{\mathbb{Q}}_p$).
	Then we call
	$e(f)$ the ordinary projection of $f$.
	It follows from the definition that
	each $U_{p,\underline{a}}^{(n)}$
	acts on $e(f)$
	by $p$-adic units.
	Moreover, by
	\cite[Proposition 2.5.3]{Liu2015},
	$e(f)$ is in fact a holomorphic modular form.

	\begin{definition}
		Let $R$ be a $\mathbb{Z}$-algebra 
		contained in $\mathbb{C}$.
		For a prime $\ell$ of $\mathbb{Q}$ such that
		$K_\ell=G(\mathbb{Z}_\ell)$,
		we write
		$\mathbb{T}_{\underline{k},\ell}(K_\ell,R)$
		for the
		$R$-subalgebra generated by
		the image of the spherical Hecke algebras
		$\mathcal{H}(G(\mathbb{Q}_\ell),
		G(\mathbb{Z}_\ell),R)$
		in
		$\mathrm{End}_\mathbb{C}
		(H^0(\widetilde{\mathbf{A}}_{G,K},\mathcal{E}
		(V_{\underline{k}})))$.
		Similarly, we write
		$\mathbb{T}^\mathrm{Iw}_{
			\underline{k},p}(I_{G,p},R)$
		for the $R$-subalgebra generated by the image of
		the $\mathbb{U}_p$-operators
		$U_{p,\underline{a}}^{(n)}$
		in the above endomorphism algebra.
		Finally we write
		\[
		\mathbb{T}_{\underline{k}}(\widehat{\Gamma},R)
		=
		(\bigotimes_{\ell\nmid Np}\,'
		\mathbb{T}_{\underline{k},\ell}(K_\ell,R))
		\bigotimes_R
		\mathbb{T}^\mathrm{Iw}_{\underline{k}}
		(I_{G,p},R).
		\]
	\end{definition}

	\subsection{Periods and congruence ideals}
	In this subsection we define 
	congruence ideals and some periods
	of automorphic forms on $[G]$.
	We follow closely the ideas of
	\cite[Section 6]{EischenHarrisLiSkinner16}.
	
	\begin{definition}
		Let $\pi$ be an irreducible cuspidal automorphic representation of
		$[G]$ with factorization
		$\pi=\pi_\infty\otimes\pi_\mathrm{f}$.
		We say that
		$\pi$ is holomorphic of type
		$(\underline{k},\widehat{\Gamma})$ if
		\[
		H^0(\mathfrak{g}_G^{1,-1}\oplus
		\mathfrak{t}_G,K_{G,\infty};
		\pi_\infty\otimes
		W_{\underline{k}})\neq0
		\text{ and }
		\pi_\mathrm{f}^{\widehat{\Gamma}}\neq0.
		\]
		
		Similarly, we say that
		$\pi$ is antiholomorphic of type
		$(\underline{k},\widehat{\Gamma})$
		if
		\[
		H^d(\mathfrak{g}_G^{1,-1}\oplus\mathfrak{t}_G,
		K_{G,\infty};
		\pi_\infty\otimes
		W_{\underline{k}^D})\neq0
		\text{ and }
		\pi_\mathrm{f}^{\widehat{\Gamma}}\neq0.
		\]
	\end{definition}
	
	Let $\pi$ be a holomorphic automorphic representation
	of $G(\mathbb{A})$ of type $(\underline{\kappa},
	\widehat{\Gamma})$.
	Then the action of $\mathbb{T}_{\underline{k}}
	(\widehat{\Gamma},\mathbb{C})$
	on $\pi^{\widehat{\Gamma}}$
	is given by a morphism, denoted by
	$\lambda_\pi\colon
	\mathbb{T}_{\underline{k}}(
	\widehat{\Gamma},\mathbb{C})
	\rightarrow
	\mathbb{C}$.
	For any subalgebra $R\subset\mathbb{C}$
	containing all the values of
	$\lambda_\pi$ in $\mathbb{C}$,
	we write
	\[
	S_{\underline{k}}(\widehat{\Gamma},R)[\lambda_\pi]
	\]
	for the $\lambda_\pi$-isotypic component of 
	$S_{\underline{k}}(\widehat{\Gamma},R)$.
	This is equivalent to saying that
	it is the localization of
	$S_{\underline{k}}(\widehat{\Gamma},R)$ at the prime ideal
	of $\mathbb{T}_{\underline{k}}
	(\widehat{\Gamma},R)$
	given by the kernel of
	$\lambda_\pi$.
	Then one has an embedding of the $1$-dimensional
	vector space as
	$\mathbb{T}_{\underline{k}}
	(\widehat{\Gamma},R)(\mathbb{C})$-modules
	\[
	\iota_\pi
	\colon
	H^0(\mathfrak{g}_G^{1,-1}\oplus
	\mathfrak{t}_G,
	K_{G,\infty},\pi^{\widehat{\Gamma}}\otimes_\mathbb{C}
	W_{\underline{\kappa}})
	\simeq
	\pi_\mathrm{f}^{\widehat{\Gamma}}
	\hookrightarrow
	S_{\underline{k}}(\widehat{\Gamma},\mathbb{C})(\pi).
	\]  
	We write $\mathfrak{m}_\pi$ for the maximal
	ideal of $\mathbb{T}_{\underline{k}}
	(\widehat{\Gamma},R)$
	given by the kernel of
	the composition
	$\mathbb{T}_{\underline{k}}
	(\widehat{\Gamma},R)
	\xrightarrow{\lambda_\pi}
	R
	\rightarrow
	R/\mathfrak{m}_R$
	where
	$\mathfrak{m}_R$ is the maximal ideal of $R$.
	We set
	\[
	S_{\underline{k}}^\mathrm{ord}(\widehat{\Gamma},R)[\pi]
	:=
	S_{\underline{k}}^\mathrm{ord}(
	\widehat{\Gamma},R)_{\mathfrak{m}_\pi}
	\bigcap
	S_{\underline{k}}^\mathrm{ord}(\widehat{\Gamma},
	R[1/p])
	[\lambda_\pi].
	\]
	
	\begin{lemma}
		Suppose as above $R\subset\mathbb{C}
		\simeq\overline{\mathbb{Q}}_p$ is 
		$p$-adically complete.
		We have
		\[
		S_{\underline{k}}^{\mathrm{ord}}(\widehat{\Gamma},
		R[1/p])
		[\lambda_\pi]
		=
		eS_{\underline{k}}(\widehat{\Gamma},
		R[1/p])
		[\lambda_\pi].
		\]
		
		Moreover, $\iota_\pi$ induces an isomorphism
		\[
		\iota_\pi
		\colon
		\pi_p^\mathrm{ord}\otimes\pi_S^{\widehat{\Gamma}}
		\simeq
		\pi_S^{\widehat{\Gamma}}
		\xrightarrow{\sim}
		S_{\underline{k}}^\mathrm{ord}(\widehat{\Gamma},\mathbb{C}).
		\]
		
		There is also an integral version:
		$\iota_\pi$ identifies
		$S_{\underline{k}}^\mathrm{ord}(\widehat{\Gamma},R)[\pi]$
		with an $R$-lattice in
		$\pi_p^\mathrm{ord}\otimes\pi_S^{\widehat{\Gamma}}$.
		It also identifies
		$S_{\underline{k}}^\mathrm{ord}(\widehat{\Gamma},R)_{
			\mathfrak{m}_\pi}$
		with an $R$-lattice in
		$\oplus_{\pi'}\big(
		(\pi')_p^\mathrm{ord}\otimes(\pi')_S^{\widehat{\Gamma}}
		\big)$
		where
		$\pi'$ runs through
		all ordinary holomorphic automorphic representations
		of $[G]$ of type
		$(\underline{k},\widehat{\Gamma})$
		such that
		$\lambda_{\pi'}
		\equiv
		\lambda_\pi(\mathrm{mod}\,  \mathfrak{m}_R)$.
	\end{lemma}
	
	Similarly, we can define
	the following $\mathbb{T}_{\widehat{\Gamma},\underline{k},R}$-modules
	\[
	\widehat{S}_{\underline{k}}(\widehat{\Gamma},R)
	:=
	\mathrm{Hom}_R(S_{\underline{k}}(\widehat{\Gamma},R),R),
	\quad
	\widehat{S}_{\underline{k}}^\mathrm{ord}(\widehat{\Gamma},R)
	:=
	\mathrm{Hom}_R(S_{\underline{k}}^\mathrm{ord}(\widehat{\Gamma},R),R).
	\]
	We use Serre duality to identify
	$\widehat{S}_{\underline{k}}(\widehat{\Gamma},R)$
	with
	\[
	H^d_{\underline{k}^D}(\widehat{\Gamma},R)
	:=
	\big\{
	\varphi\in
	H^d(\mathbf{A}_{G,\widehat{\Gamma}}(\mathbb{C}),
	\mathcal{E}(W_{\underline{k}^D}))|
	\langle S_{\underline{k}}(\widehat{\Gamma},R),\varphi
	\rangle_{\underline{k}}^\mathrm{Ser}\subset R
	\big\}.
	\]
	Using the Serre duality, we define 
	$S_{\underline{k}}^{\mathrm{ord},\perp}(\widehat{\Gamma},R)$
	to be the subspace of
	$H^d_{\underline{k}^D}(\widehat{\Gamma},R)$
	annihilating $S_{\underline{k}}^\mathrm{ord}(\widehat{\Gamma},R)$.
	Similarly, one can identify
	$\widehat{S}_{\underline{k}}^\mathrm{ord}
	(\widehat{\Gamma},R)$
	with
	\[
	H_{\underline{k}^D}^{d,\mathrm{ord}}
	(\widehat{\Gamma},R)
	=
	\big\{
	\varphi\in
	H^d(\mathbf{A}_{G,\widehat{\Gamma}},\mathcal{E}(W_{\underline{k}^D}))/
	S_{\underline{k}}^{\mathrm{ord},\perp}(\widehat{\Gamma},R)|
	\langle
	S_{\underline{k}}^\mathrm{ord}(\widehat{\Gamma},R),\varphi
	\rangle_{\underline{k}}^\mathrm{Ser}
	\subset R
	\big\}.
	\]

	Since $\pi$ is holomorphic of type
	$(\underline{k},\widehat{\Gamma})$,
	its contragredient $\pi^\vee$ is anti-holomorphic 
	of type $(\underline{k},\widehat{\Gamma})$.
	We have the following injection
	\[
	\widehat{\iota}_{\pi^\vee}
	\colon
	H^d(\mathfrak{g}_G^{1,-1}\oplus
	\mathfrak{t}_G,
	K_{G,\infty},\pi^{\vee,\widehat{\Gamma}}
	\otimes W_{\underline{k}^D})
	\hookrightarrow
	(\pi^{\vee})^{\widehat{\Gamma}}
	\xrightarrow{\sim}
	H^d_{\underline{k}^D}(\widehat{\Gamma},\mathbb{C})
	\simeq
	H^d(\mathbf{A}_{G,\widehat{\Gamma}}(\mathbb{C}),
	\mathcal{E}(W_{\underline{k}^D})).
	\]
	Similarly, we set
	\[
	H^{d,\mathrm{ord}}_{\underline{k}^D}(\widehat{\Gamma},R)
	[\pi^\vee]
	:=
	H^{d,\mathrm{ord}}_{\underline{k}^D}
	(\widehat{\Gamma},R)_{\mathfrak{m}_{\pi^\vee}}
	\bigcap
	H^{d,\mathrm{ord}}_{\underline{k}^D}(\widehat{\Gamma},
	R[1/p])
	[\lambda_{\pi^\vee}].
	\]
	
	The dual version of the above lemma is
	\begin{lemma}
		Let $R,E,\mathfrak{m}_R$ be as above.
		The map $\widehat{\iota}_{\pi^\vee}$
		induces an isomorphism
		\[
		\widehat{\iota}_{\pi^\vee}
		\colon
		\pi_S^{\widehat{\Gamma}}
		\xrightarrow{\sim}
		H^{d,\mathrm{ord}}_{\underline{k}^D}
		(\widehat{\Gamma},\mathbb{C})[\pi^\vee].
		\]
		Moreover,
		$\widehat{\iota}_{\pi^\vee}$ identifies
		$H^{d,\mathrm{ord}}_{\underline{k}^D}(\widehat{\Gamma},R)
		[\pi]$
		with an $R$-lattice in
		$\pi^{\widehat{\Gamma}}_S$.
		It also identifies
		$H^{d,\mathrm{ord}}(\widehat{\Gamma},R)_{
			\mathfrak{m}_{\pi^\vee}}$
		with an $R$-lattice in
		$\otimes_{\pi'}(\pi')^{\widehat{\Gamma}}_S$
		where
		$\pi'$ runs through
		all ordinary holomorphic automorphic representations
		of $[G]$ of type
		$(\underline{k},\widehat{\Gamma})$
		such that
		$\lambda_{\pi'}
		\equiv
		\lambda_\pi(\mathrm{mod}\,\mathfrak{m}_R)$.
	\end{lemma}
	
	One can also show that the Serre duality induces perfect
	$\mathbb{T}_{\widehat{\Gamma},\underline{k},R}^\mathrm{ord}$-equivariant
	pairings
	\[
	S^{\mathrm{ord}}_{\underline{k}}(\widehat{\Gamma},R)[\pi]
	\bigotimes_R
	H^{d,\mathrm{ord}}_{\underline{k}^D}(\widehat{\Gamma},R)
	[\pi^\vee]
	\rightarrow
	R,
	\quad
	S^{\mathrm{ord}}_{\underline{k}}(\widehat{\Gamma},R)_{\mathfrak{m}_\pi}
	\bigotimes_R
	H^{d,\mathrm{ord}}_{\underline{k}^D}
	(\widehat{\Gamma},R)_{\mathfrak{m}_{\pi^\vee}}
	\rightarrow
	R.
	\]

	\subsubsection{Congruence ideals}
	In this subsection
	we write
	$\mathbb{T}^\mathrm{ord}$,
	resp.,
	$\mathbb{T}^\mathrm{ord}_{\mathfrak{m}_\pi}$
	for $e\mathbb{T}_{\underline{k}}
	(\widehat{\Gamma},R)$,
	resp.,
	$e\mathbb{T}_{\underline{k}}
	(\widehat{\Gamma},R)_{\mathfrak{m}_\pi}$.
	We make the following Gorenstein hypothesis
	\begin{hypothesis}\label{Gorenstein hypothesis}
		The $R$-algebra $\mathbb{T}^\mathrm{ord}_{
			\mathfrak{m}_\pi}$
		is Gorenstein
		and the $\mathbb{T}^\mathrm{ord}_{
			\mathfrak{m}_\pi}$-module
		$S_{\underline{k}}^\mathrm{ord}
		(\widehat{\Gamma},R)_{
			\mathfrak{m}_\pi}$ 
		is a free module.
	\end{hypothesis}

	Note that $\mathbb{T}^\mathrm{ord}_{
		\mathfrak{m}_\pi}$
	is a reduced $R$-algebra,
	we have a decomposition of $R$-algebras
	\[
	\mathbb{T}^\mathrm{ord}_{\mathfrak{m}_\pi}[1/p]
	:=\mathbb{T}^\mathrm{ord}_{
		\mathfrak{m}_\pi}\otimes_RR[1/p]=
	R[1/p]\oplus X
	\]
	where the projection to
	$R[1/p]$ is induced by the character
	$\lambda_\pi\colon
	\mathbb{T}^\mathrm{ord}_{
		\mathfrak{m}_\pi}\rightarrow R$.
	We write
	$1_\pi$ for the idempotent element in
	$\mathbb{T}_{\mathfrak{m}_\pi}[1/p]$
	corresponding to the projection
	$\mathbb{T}^\mathrm{ord}_{
		\mathfrak{m}_\pi}[1/p]
	\rightarrow R[1/p]$.
	Let $M$ be a finite $\mathbb{T}^\mathrm{ord}$-module
	flat over $R$
	and put
	$M[1/p]:=M\otimes_RR[1/p]$.
	Define the following 
	$\mathbb{T}^\mathrm{ord}$-modules:
	\[
	M^\pi:=1_\pi M_{\mathfrak{m}_\pi},
	\quad
	M_\pi:=M\bigcap 1_\pi M[1/p],
	\quad
	M^X:=(1-1_\pi)M_{\mathfrak{m}_\pi},
	\quad
	M_X:=M\bigcap M^X.
	\]
	
	One can show that there are isomorphisms of
	$R$-modules
	\[
	\frac{M_{\mathfrak{m}_\pi}}{M_\pi+M_X}
	\simeq
	\frac{M^\pi}{M_\pi}
	\simeq
	\frac{M^X}{M_X}
	\simeq
	\frac{M^\pi+M^X}{M_{\mathfrak{m}_\pi}}.
	\]
	
	\begin{definition}
		The cohomological congruence ideal 
		$\mathfrak{c}^\mathrm{coh}(M,\pi)$ of
		$M$ is the annihilator in $R$ of the
		$R$-module
		$M^\pi/M_\pi$.
	\end{definition}
    \begin{lemma}
    	\label{idetification of various congruence ideals}
    	Assuming
    	Hypothesis \ref{Gorenstein hypothesis},
    	one has
    	\[
    	\mathfrak{c}^\mathrm{coh}
    	(\mathbb{T}^\mathrm{ord},\pi)
    	=
    	\mathfrak{c}^\mathrm{coh}
    	(S_{\underline{k}}^\mathrm{ord}
    	(\widehat{\Gamma},R),\pi)
    	=
    	\mathfrak{c}^\mathrm{coh}
    	(H_{\underline{k}^D}^{d,\mathrm{ord}}
    	(\widehat{\Gamma},R),\pi).
    	\]
    \end{lemma}
    \begin{proof}
    	\textit{cf}.
    	\cite[Section 2]{TilouineUrban2018}.
    \end{proof}
    \begin{definition}
    	We write
    	$\mathfrak{c}^\mathrm{coh}(\pi)$
    	for the ideals in the above lemma.
    \end{definition}
    \begin{remark}
    	The cohomological congruence ideal
    	defined here is the same as the
    	one given in
    	\cite[after Lemma 6.6.3]{EischenHarrisLiSkinner16}
    	for the case of $\mathrm{GSp}_{2n}$.
    	More precisely,
    	write
    	$S_{\underline{k}}^\mathrm{ord}
    	(\widehat{\Gamma},R)[\pi]^\perp
    	\subset
    	S_{\underline{k}}^\mathrm{ord}
    	(\widehat{\Gamma},R)_{\mathfrak{m}_\pi}$
    	for the orthogonal complement of the subspace
    	$S_{\underline{k}}^\mathrm{ord}(\widehat{\Gamma},R)
    	[\pi]
    	\subset
    	S_{\underline{k}}^\mathrm{ord}(\widehat{\Gamma},R)_{
    		\mathfrak{m}_\pi}$
    	under the pairing of Petersson product.
    	Then the ideal
    	$\mathfrak{c}^\mathrm{coh}(\pi)$
    	is equal to
    	the ideal of $R$
    	annihilating the $R$-module
    	$S_{\underline{k}}^\mathrm{ord}(\widehat{\Gamma},R)_{
    		\mathfrak{m}_\pi}/
    	(
    	S_{\underline{k}}^\mathrm{ord}(\widehat{\Gamma},R)
    	[\pi]
    	+
    	S_{\underline{k}}^\mathrm{ord}(\widehat{\Gamma},R)
    	[\pi]^\perp)$
    \end{remark}

    The cohomological congruence ideals
    can be generalized in the following way
    (\textit{cf}.
    \cite[Section 8.2]{HidaTilouine16}).
    Let $S,S'$ be non-zero 
    reduced algebra over a normal noetherian
    domain $R$ which are flat of finite type as
    $R$-modules.
    Let $f\colon S\rightarrow S'$ be a
    surjective homomorphism of
    $R$-algebras.
    Then we can decompose the total quotient ring
    $\mathrm{Q}(S)$ into a direct product of $R$-algebras
    $\mathrm{Q}(S)
    =
    \mathrm{Q}(S')\oplus X$.
    Then the congruence module of the morphism
    $f$ is defined as
    $ C_0(f):=
    S/\mathrm{Ker}(S\rightarrow X)\otimes_{S,f}S'$
    and the congruence ideal in $S'$
    of $f$ is defined as
    \begin{equation}
    \label{notation for congruence ideal}
    \mathfrak{c}(f)
    =
    \mathrm{Ann}_{S'}(C_0(f)).
    \end{equation}
    The meaning and transfer property
    of the congruence ideal
    $\mathfrak{c}(f)$ are given in
    \cite[Proposition 8.3,
    Lemma 8.5]{HidaTilouine16}.
    Then for the case
    $\lambda_\pi
    \colon\mathbb{T}_{\mathfrak{m}_\pi}
    \rightarrow
    R$,
    under Hypothesis
    \ref{Gorenstein hypothesis},
    we have
    (\textit{cf}.
    \cite[Section 2.1]{TilouineUrban2018}):
    \[
    \mathfrak{c}^\mathrm{coh}(\pi)
    =\mathfrak{c}(\lambda_\pi).
    \]

	\subsubsection{Periods of automorphic representations}
	\label{periods}    
	\begin{lemma}\label{definition of periods}
		Let
		$\pi$ be an irreducible holomorphic
		automorphic representation of $[G]$
		of type 
		$(\underline{k},\widehat{\Gamma})$.
		The following $R$-modules contained inside $\mathbb{C}$
		given by the values of the Petersson
		products
		\[
		L[\pi]=
		\langle
		S_{\underline{k}}^\mathrm{ord}(\widehat{\Gamma},R)[\pi],
		S_{\underline{k}}^\mathrm{ord}(\widehat{\Gamma},R)[\pi]
		\rangle,
		\quad
		L_\pi
		=
		\langle
		S_{\underline{k}}^\mathrm{ord}
		(\widehat{\Gamma},R)[\pi],
		S_{\underline{k}}^\mathrm{ord}
		(\widehat{\Gamma},R)_{\mathfrak{m}_\pi}
		\rangle
		\]
		are both of rank $1$ over $R$.
	\end{lemma}
	\begin{proof}
		The proof is the same as
		\cite[Proposition 2.4.9]{Harris2013}
		which is a simple application of
		Schur's lemma.
	\end{proof}

	We fix two positive real numbers
	as generators of these two $R$-modules
	$L[\pi]$ and $L_\pi$:
	\[
	P[\pi]=P_G[\pi],
	\quad
	P_\pi=P_{G,\pi}.
	\]

	We have a dual version of Lemma \ref{definition of periods}:
	\begin{lemma}
		\label{anti-periods}
		Assume Hypothesis \ref{Gorenstein hypothesis},
		then the following $R$-modules
		\[
		\widehat{L}[\pi^\vee]
		=
		\langle
		H^{d,\mathrm{ord}}_{
			\underline{k}^D}(\widehat{\Gamma},R)[\pi^\vee],
		H^{d,\mathrm{ord}}_{
			\underline{k}^D}(\widehat{\Gamma},R)[\pi^\vee]
		\rangle,
		\quad
		\widehat{L}_{\pi^\vee}
		=
		\langle
		H^{d,\mathrm{ord}}_{
			\underline{k}^D}(\widehat{\Gamma},R)
		[\pi^\vee],
		H^{d,\mathrm{ord}}_{
			\underline{k}^D}
		(\widehat{\Gamma},R)_{\mathfrak{m}_{\pi^\vee}}
		\rangle
		\]
		are both of rank $1$ over $R$.
	\end{lemma}
	
	We fix the following positive real numbers
	as generators
	of the $R$-module
	$\widehat{L}[\pi^\vee]$,
	resp.,
	$\widehat{L}_{\pi^\vee}$.
	\[
	\widehat{P}[\pi^\vee]
	=
	\widehat{P}_G[\pi^\vee],
	\quad
	\widehat{P}_{\pi^\vee}
	=
	\widehat{P}_{G,\pi^\vee}.
	\]	
	These numbers
	$P[\pi],P_\pi,
	\widehat{P}[\pi^\vee],
	\widehat{P}_{\pi^\vee}$
	are called the periods of
	$\pi$ and $\pi^\vee$.
	Then we have the following relation of
	congruence numbers and periods:
	\begin{lemma}
		we have the following identities,
		up to units in $R$:
		\begin{equation}\label{period relation}
		\mathfrak{c}^\mathrm{coh}(\pi)P_\pi
		=
		P[\pi],
		\quad
		\mathfrak{c}^\mathrm{coh}(\pi)
		\widehat{P}_{\pi^\vee}
		=
		\widehat{P}[\pi^\vee],
		\quad
		\widehat{P}_{\pi^\vee}
		P_\pi=1.
		\end{equation}
	\end{lemma}
	\begin{proof}
		The first identity follows the definitions
		(\textit{cf}.
		\cite[Lemma 6.6.3]{EischenHarrisLiSkinner16})
		The second one comes from Serre duality.
		The last one comes from the isomorphism
		$c_G$ as well as the relation between 
		Serre duality and
		Petersson products as given by
		(\ref{Petersson product vs. Serre duality}).
	\end{proof}
    \begin{remark}
    	In applications to Rallis inner product formula,
    	we will consider antiholomorphic automorphic
    	representations of $[G]$
    	instead of holomorphic ones.
    	The above results on congruence ideals and
    	periods are still valid in this case
    	up to the exchange of roles of
    	$\pi$ and its contragredient $\pi^\vee$.
    \end{remark}

    \subsubsection{Congruence ideals and Petersson products
    on compact groups}
    \label{congruence ideals on compact groups}
    In this subsection we digress to review the relation
    between congruence ideals and Petersson products.
    Let $G$ be an algebraic group over $\mathbb{Q}$
    compact at $\infty$.
    Then there is an equivalence between the theory
    of algebraic modular forms
    on $G$ and the theory of classical modular forms
    on $G$.
    We will use the former theory in this subsection.
    
    Fix a compact open subgroup $U$ of 
    $G(\mathbb{A}_\mathrm{f})$,
    then $G(\mathbb{Q})\backslash
    G(\mathbb{A}_\mathrm{f})/U$ is a finite set,
    which we denote by
    $\bigcup_{i\in I}G(\mathbb{Q})t_iU$.
    For any ring $R$,
    Let $S(U,R)$ be the set of $R$-valued
    modular forms on
    $G(\mathbb{A}_\mathrm{f})$ of level $U$,
    which can be identified with the set of functions
    $\{f\colon I\rightarrow R\}$.
    Suppose $R$ is a subring of $\mathbb{C}$,
    we define the Petersson inner product on
    $S(U,R)$ as
    \[
    \langle\cdot,\cdot\rangle
    \colon
    S(U,R)\times S(U,R)
    \rightarrow
    R,
    \quad
    (f,g)
    \mapsto
    \langle
    f,g
    \rangle
    =\sum_if(t_i)\overline{g(t_i^{-1})}.
    \]
    This is a perfect pairing,
    invariant under the action of $G(\mathbb{A}_\mathrm{f})$.
    
    Suppose now that
    $R$ is a finite extension of $\mathbb{Z}_p$
    with uniformizer $\varpi$
    (still viewed as a subring of $\mathbb{C}$).
    For any automorphic representation $\pi$ of
    $G(\mathbb{A})$,
    let $f,f'$ be two $R$-valued modular eigenforms on
    $G(\mathbb{A}_\mathrm{f})$
    in the representation $\pi$
    such that 
    $c(f)=f$ and $c(f')=-f'$.
    Assume that both $f$ and $f'$
    are primitive.
    Let $\mathcal{C}$ be an irreducible component of
    the ordinary Hida-Hecke algebra $\mathbb{T}$
    passing through $\pi$,
    $\lambda_\pi
    \colon\mathbb{T}
    \rightarrow
    \mathcal{C}
    \xrightarrow{\mathfrak{s}_\pi}
    R$.
    Then as in
    \cite[Sections 5.2 and 5.3]{HidaArithmetic},
    we can show that
    \begin{lemma}
    	\label{congruence ideals and Petersson products}
    	The Petersson product
    	$\langle
    	f,f'
    	\rangle$ generates the 
    	cohomological congruence ideal
    	$\mathfrak{c}(\lambda_\pi)
    	\subset R$.
    \end{lemma}

	\section{Siegel Eisenstein series and theta correspondence}
	\subsection{Vector spaces and groups}
	
	\subsubsection{Symplectic groups}

	Let $V=\mathbb{Q}^4$ be a $\mathbb{Q}$-vector space
	of dimension $4$, equipped with a non-degenerate symplectic form
	$\langle\cdot,\cdot\rangle$.
	Let $L\subset V$ be a lattice of $V$.
	Fix a $\mathbb{Z}$-basis of $L$
	as follows:
	\[
	\mathcal{B}=\mathcal{B}_L=\mathcal{B}_V
	=\{
	e_1^+,e_2^+,e_1^-,e_2^-
	\}
	\]
	such that
	$\langle e_i^+,e_j^-\rangle=\delta_{i,j}$
	for $i,j=1,2$.
	We define some submodules, resp., subspaces
	of $L$, resp., $V$
	as follows:
	\[
	L^+=\mathbb{Z}e_1^++\mathbb{Z}e_2^+,
	\quad
	L^-=\mathbb{Z}e_1^-+\mathbb{Z}e_2^-,
	\quad
	V^\pm=L^\pm\otimes\mathbb{Q}.
	\]
	We then define
	$L_1=L_2=L$,
	resp.,
	$L_3=L_4=L_1\oplus L_2$
	with symplectic form
	$\langle\cdot,\cdot\rangle_{L_1}
	=\langle\cdot,\cdot\rangle$,
	resp.,
	$\langle\cdot,\cdot\rangle_{L_2}=
	-\langle\cdot,\cdot\rangle$,
	$\langle\cdot,\cdot\rangle_{L_3}
	=\langle\cdot,\cdot\rangle_{L_4}
	=\langle\cdot,\cdot\rangle_{L_1}+
	\langle\cdot,\cdot\rangle_{L_2}$.
	Then we set $V_i=L_i\otimes\mathbb{Q}$
	for $i=1,\cdots,4$.
	We fix the following basis for $V_4$:
	\[
	\mathcal{B}_{V_4}
	=
	\Big\{
	(e_1^+,0),
	(e_2^+,0),
	(0,e_1^+),
	(0,e_2^+),
	(e_1^-,0),
	(e_2^-,0),
	-(0,e_1^-),
	-(0,e_2^-)
	\Big\}
	\]
	
	The isometry groups
	$G_1,G_2,G_4$
	for the symplectic spaces
	$V_1,V_2,V_4$
	can be identified under the given basis
	with the following groups:
	\[
	G_1
	=
	G_2
	=
	G
	=
	\Big\{
	g\in\mathrm{GL}_4|
	g^tJ_4
	g
	=\nu(g)J_4
	\Big\}
	\]
	\[
	G_4
	=
	\Big\{
	g\in\mathrm{GL}_8|
	g^tJ_8
	g
	=\nu(g)J_8
	\Big\}
	\]
	Finally we put
	$G_3
	=
	\{
	\mathrm{diag}(g_1,g_2)
	\in\mathrm{diag}(G_1,G_2)
	\subset\mathrm{GL}_8|
	\nu(g_1)=\nu(g_2)
	\}$.
	For each $G_i$
	($1\leq i\leq 4$), we define the subgroup
	\[
	G_i^1=
	\{
	g\in G_i|
	\nu(g)=1
	\}.
	\]
	
    We define an injective morphism
	$G_1^1\times G_2^1\hookrightarrow G_4^1$
	as follows:
	\[
	(
	\begin{pmatrix}
	a_1 & b_1 \\
	c_1 & d_1
	\end{pmatrix},
	\begin{pmatrix}
	a_2 & b_2 \\
	c_2 & d_2
	\end{pmatrix}
	)
	\mapsto
	\begin{pmatrix}
	a_1 & 0 & b_1 & 0 \\
	0 & a_2 & 0 & -b_2 \\
	c_1 & 0 & d_1 & 0 \\
	0 & -c_2 & 0 & d_2
	\end{pmatrix}.
	\]

    We will also use
    another maximal polarization of
    $V_4$ later on which is given as follows:
	\[
	V_4^d
	=
	\big\{
	(v,v)\in V_4|
	v\in V_1
	\big\},
	\quad
	V_{4,d}
	=
	\big\{
	(v,-v)\in V_4|
	v\in V_1
	\big\}.
	\]
	Write $P_{V'}\subset G_4$
	for the parabolic subgroup stabilizing a
	subspace $V'$ of $V_4$.
	Then 
	using the relation between
	$P_{V_4^d}$ 
	and $P_{V_4^+}$,
	one verifies easily that the action of
	$G_1^1\times G_2^1$ on
	the flag variety 
	$(P_{V_4^d}\cap G_4^1)\backslash G_4^1$
	satisfies the conditions in
	\cite[p.2]{GelbartPSRallis87}.
	This will be used in the doubling method
	in the following.

	\subsubsection{Orthogonal groups}
	
	Let $N\in\mathbb{N}$ be a positive integer prime to $p$.
	Let $N_1\in\mathbb{N}$ be a positive integer prime to
	$Np$ such that the diagonal matrices
	$\eta_U=\mathrm{diag}(N^2/2,N^2/2,N^2/2,N^2/2,
	N/N_1,N_1)$
	and
	$\eta_U'=\mathrm{diag}(2N,2N,2N,2,2,2)$
	are $\mathbb{Q}$-equivalent.
	Let $U=\mathbb{Q}^6$ be a $\mathbb{Q}$-vector 
	space of dimension $6$
	equipped with a non-degenerate symmetric bilinear form
	$\langle\cdot,\cdot\rangle_U$
	such that
	under a basis
	$\mathcal{B}_U
	=(u_1,\cdots,u_6)$
	of $U$
	the symmetric form
	is the matrix $\eta_U$.
	We then write
	$\mathrm{O}(U)$, resp.
	$\mathrm{GSO}(U)$, for the 
	(resp. special similitude)
	orthogonal group
	over $\mathbb{Q}$ defined by
	$\langle\cdot,\cdot\rangle_U$.
	Under this basis
	$\mathcal{B}_U$,
	then can by identified with
	\[
	\mathrm{O}(U)
	=
	\Big\{
	h\in\mathrm{GL}_6|
	h^t\eta_Uh=\eta_U
	\Big\},
	\]
	\[
	\mathrm{GSO}(U)
	=
	\Big\{
	h\in\mathrm{GL}_6|
	h^t\eta_Uh=\nu(h)\eta_U,
	\mathrm{det}(h)=\nu(h)^3
	\Big\}.
	\]

	\subsubsection{Tensor products}
	
	We then define the tensor products
	$W=V\otimes_\mathbb{Q}U$
	and
	$W_i=V_i\otimes_\mathbb{Q}U$
	for $1\leq i\leq4$.
	We give these vector spaces
	symplectic structures as follows: take
	$W$ as an example, 
	for elements $v\otimes u,v'\otimes u'\in W$,
	we define a symplectic form
	$\langle\cdot,\cdot\rangle_W$ on $W$ by
	\[
	\langle v\otimes u,v'\otimes u'\rangle_W
	=\langle v,v'\rangle_V\times\langle u,u'\rangle_U.
	\]
	The maximal polarizations on $V_i$
	induce maximal polarizations on $W_i$ in the obvious way:
	\[
	W_i=(V_i^+\otimes U)\oplus(V_i^-\otimes U),
	\quad
	\forall 1\leq i\leq4,
	\]
	\[
	W_4=(V_4^d\otimes U)\oplus(V_{4,d}\otimes U).
	\]
	
	Similarly, we write
	$\mathrm{Sp}(W_i)$,
	resp.
	$\mathrm{GSp}(W_i)$
	for the 
	(resp. similitude)
	algebraic symplectic groups over $\mathbb{Q}$
	defined by the symplectic forms on
	$W_i$.
	Note that $\mathrm{Sp}(V_4)$ and $\mathrm{O}(U)$
	act on $W_4$ by the embedding
	$\mathrm{Sp}(V_4)\times
	\mathrm{O}(U)
	\hookrightarrow
	\mathrm{Sp}(W_4)$.
	Similarly we have a natural map
	$\mathrm{GSp}(V_4)\times\mathrm{GSO}(U)
	\rightarrow
	\mathrm{GSp}(W_4)$.
	
	\begin{remark}\label{matrix identification}
		For the purpose of computations,
		we will use the following obvious matrix identifications
		of the above defined various vector spaces:
		\begin{enumerate}
			\item 
			we identify $V_4$ with 
			$\mathrm{M}_{8\times 1}(\mathbb{Q})$:
			the $k$-th basis element $e_k\in\mathcal{B}_{V_4}$
			corresponding to the column matrix
			with $1$ at the $k$-th coordinate and
			$0$ elsewhere
			and extend the map to the whole $V_4$
			by $\mathbb{Q}$-linearity.
			Suppose that
			$v,v'\in V_4$ correspond to 
			$X,X'\in\mathrm{M}_{8\times 1}
			(\mathbb{Q})$,
			then the symplectic form becomes
			$\langle v,v'\rangle_{V_4}
			=X^tJ_8X'$.
			We identify
			$V_1,V_2,V_4^\pm$ with
			$\mathrm{M}_{4\times 1}
			(\mathbb{Q})$
			in the same way;
			
			\item 
			we identify $U$ with
			$\mathrm{M}_{1\times 6}
			(\mathbb{Q})$:
			the $j$-th basis element $u_j\in\mathcal{B}_U$
			corresponding to the row matrix
			with $1$ at the $j$-th coordinate and 
			$0$ elsewhere.
			Suppose that $u,u'\in U$
			correspond to
			$Y,Y'\in\mathrm{M}_{1\times 6}
			(\mathbb{Q})$,
			then the symmetric form becomes
			$\langle u,u'\rangle_U
			=
			Y\eta_U(Y')^t
			=
			\mathrm{tr}
			(Y^t\eta_UY')$;
			
			\item 
			we then identify
			$W_4=V_4\otimes U$ with
			$\mathrm{M}_{8\times 6}
			(\mathbb{Q})$:
			the $(k,j)$-th basis element
			$e_k\otimes u_j$ in the basis
			$\mathcal{B}_{W_4}=\mathcal{B}_{V_4}\times
			\mathcal{B}_U$
			corresponding to the matrix
			with $1$ at the $(k,j)$-th coordinate and
			$0$ elsewhere.
			Suppose that
			$w,w'\in W_4$ correspond to
			$Z,Z'\in\mathrm{M}_{8\times 6}
			(\mathbb{Q})$, then
			the symplectic form on $W_4$ becomes
			$\langle w,w'\rangle_{W_4}
			=
			\mathrm{tr}
			(
			Z^tJ_8
			Z'\eta_U
			)$.
		\end{enumerate}		
	\end{remark}

    \begin{remark}\label{Haar measures}
    	We make the following convention on the Haar measures
    	on the groups $G(\mathbb{Q}_v)$
    	and the vector spaces $V(\mathbb{Q}_v)$ and the like
    	as follows:
    	(1) the measure on $V_4(\mathbb{Q}_\ell)$
    	is the usual measure $\mu$ with
    	$\mu(V_4(\mathbb{Z}_\ell))=1$,
    	i.e., the product measure
    	of standard measures on $\mathbb{Q}_p$
    	with volume of $\mathbb{Z}_\ell$ being $1$.
    	Same choices for
    	$V_1(\mathbb{Q}_\ell),
    	V_2(\mathbb{Q}_\ell)$;
    	(2) the measure on $V_4(\mathbb{R})$
    	is the measure induced from
    	the standard Lebesgue measure
    	on $\mathrm{M}_{8\times 1}
    	(\mathbb{R})$;
    	(3) the measure $\mu_U$
    	on $U(\mathbb{Q}_\ell)$
    	is the measure $\mu$ such that
    	$\mu(U(\mathbb{Z}_\ell))=|2^{-4}N^9|_\ell$;
    	(4) the measure $\mu_U$
    	on $U(\mathbb{R})$
    	is the measure $2^{-4}N^9\mu^\mathrm{Leb}$
    	where $\mu^\mathrm{Leb}$ is the standard
    	Lebesgue measure on
    	$U(\mathbb{R})$
    	(or equivalently the measure induced from the 
    	Lebesgue measure on 
    	$\mathrm{M}_{1\times 6}
    	(\mathbb{R})$);
    	(5) the measure $\mu_{W_4}(\mathbb{Q}_v)$
    	is the product measure of measures on
    	$V_4(\mathbb{Q}_v)$ and $U(\mathbb{Q}_v)$;
    	(6) the measure on
    	$\mathrm{Sp}(V_i)(\mathbb{Q}_\ell)$ 
    	($i=1,2,4$)
    	is the one with
    	volume of $\mathrm{Sp}(V_i)(\mathbb{Z}_\ell)$
    	being $1$.
    	Similar choice for
    	$\mathrm{O}(U)(\mathbb{Q}_\ell)$;
    	(7) the measure on 
    	$\mathrm{Sp}(V_i)(\mathbb{R})$ is taken as follows:
    	$\mathrm{Sp}(V_i)(\mathbb{R})$
    	acts on the Siegel upper half plane
    	$\mathbb{H}_{d(i)}$
    	with $d(i)=(\mathrm{dim}V_i)/2$
    	by linear transformation and the
    	subgroup fixing the point
    	$\sqrt{-1}\in\mathbb{H}_{d(i)}$
    	is a maximal compact subgroup
    	$K_{\mathrm{Sp}(V_i)(\mathbb{R})}$
    	of $\mathrm{Sp}(V_i)(\mathbb{R})$.
    	We take the (unique) Haar measure on
    	$K_{\mathrm{Sp}(V_i)(\mathbb{R})}$
    	to be the one with total volume being $1$.
    	We take the Haar measure on
    	$\mathbb{H}_{d(i)}$ to be
    	$\mathrm{det}(y)^{-d(i)-1}\prod_{1\leq i\leq j\leq d(i)}
    	dx_{i,j}dy_{i,j}$.
    	Then the Haar measure on 
    	$\mathrm{Sp}(V_i)(\mathbb{R})$
    	is the product of these two measures;
    	(8) the measures on the adelic points of
    	the above mentioned objects are
    	the product measures of the local ones.
    \end{remark}

	\subsection{Weil representations}\label{Weil representations}
	For any maximal polarization
	$V_4=V'\oplus V''$,
	which gives rise to a polarization of $W_4$ as
	$W_4=(V'\otimes U)\oplus(V''\otimes U)$,
	we will define the local and global
	Weil representations 
	in this subsection
	(\cite{Kudla96,MoeglinVignerasWaldspurger}).
	The materials in this subsection are already
	well-known facts and we will therefore
	be very brief.
	
	\subsubsection{Local Weil representations}
	
	We fix a place
	$v$ of $\mathbb{Q}$
	and write $F=\mathbb{Q}_v$
	for the local field at the place $v$.
	Let $V$ be a symplectic vector space over $F$
	with the symplectic form
	$\langle\cdot,\cdot\rangle=
	\langle\cdot,\cdot\rangle_V$.
	Then the Heisenberg group $H(V)$ associated to
	$V$ is the set
	$V\times F$ with the group law
	\[
	(v,t)\cdot
	(v',t')=(v+v',t+t'+\langle v,v'\rangle/2).
	\]
	
	The symplectic group $\mathrm{Sp}(V)$
	defined by the symplectic form 
	$\langle\cdot,\cdot\rangle$
	acts on the Heisenberg group by
	\[
	g(v,t):=(gv,t),
	\quad
	\forall g\in\mathrm{Sp}(V),(v,t)\in H(V).
	\]
	
	Let $\mathbf{e}_F=\mathbf{e}_v
	\colon F\rightarrow
	\mathbb{C}^\times$ be the exponential character.
	We can then define a pairing on $W$ by
	\[
	[\cdot,\cdot]\colon
	W\times W\rightarrow\mathbb{C}^\times,
	\quad
	(v,v')\mapsto
	[v,v']:=\mathbf{e}_F(\langle v,v'\rangle).
	\]
	For any subspace $V'$ of $V$, we write
	$(V')^\perp$ for the subspace
	of $V$ consisting of vectors
	$v$ such that $[v,V']=1$.

	Let $X\subset V$ be a maximal isotropic subspace
	of $V$
	(thus $X^\perp=X$) and we define
	$H(X)$ to be the Heisenberg group associated to
	$X$ (with the symplectic form induced from $V$,
	i.e. the trivial symplectic form).
	We write $\mathcal{S}_X$ to be the set of
	$\mathbb{C}$-valued functions 
	$f\colon H(V)\rightarrow\mathbb{C}$
	such that
	$f(h_Xh)=\mathbf{e}_X(h_X)f(h)$
	for any $h_X\in H(X)$ and $h\in H(V)$
	and moreover there exists some open subgroup
	$X'$ of $X$ such that
	$f(h(x',0))=f(h)$
	for any $x'\in X'$.
	Here
	$\mathbf{e}_X((x,t)):=\mathbf{e}_F(t)$.
	For any subspace $V'$ of $V$, we write
	$\mathcal{S}(V')$ for the space of
	$\mathbb{C}$-valued Schwartz-Bruhat functions
	on $V'$.
	There is an action of $H(V)$ on
	$\mathcal{S}_X$ given by
	\[
	\rho(h)f(h')
	=(\rho(h)f)(h')
	:=f(h^{-1}h'),
	\quad
	\forall h\in H(V),f(\cdot)\in\mathcal{S}_X.
	\]
	For any maximal isotropic subspace $Y$ of $V$ such that
	$V=X+Y=X\oplus Y$, we have an
	isomorphism of vector spaces
	\begin{equation}\label{Schwartz-Bruhat-1}
	\mathcal{S}_X
	\xrightarrow{\sim}
	\mathcal{S}(Y),
	\quad
	f\mapsto
	\Big(
	y\mapsto f((y,0))
	\Big),
	\end{equation}
	whose inverse is given by
	\begin{equation}\label{Schwartz-Bruhat-2}
	\mathcal{S}(Y)
	\xrightarrow{\sim}
	\mathcal{S}_X,
	f'\mapsto
	\Big(
	(x+y,t)\mapsto \mathbf{e}_F(t-\langle x,y\rangle/2)f'(y)
	\Big).
	\end{equation}
	
	These isomorphisms induce an action of
	$H(V)$ on $\mathcal{S}(Y)$ from that on
	$\mathcal{S}_X$. Explicitly, it is given by
	\[
	\rho((x+y,t))f'(y')
	:=\mathbf{e}_F(t-\langle x,y'+y\rangle/2)f'(y+y').
	\]
	Note that this action depends on the decomposition
	$V=X+Y$.
	
	Now for any element $g\in\mathrm{Sp}(V)$,
	there is an obvious isomorphism of vector spaces
	\[
	A(g)\colon
	\mathcal{S}_X
	\rightarrow
	\mathcal{S}_{gX},
	\quad
	f\mapsto
	(h\mapsto f(g^{-1}h))
	\]
	One verifies easily that
	\[
	\rho(h)A(g)f(h')
	=A(g)\rho(g^{-1}h)f(h'),
	\forall g\in\mathrm{Sp}(V),f\in \mathcal{S}_X.
	\]
	
	For any two maximal isotropic subspaces
	$X_1,X_2$ of $V$, one can define an intertwining operator
	\begin{equation}\label{intertwining operator}
	I_{X_1,X_2}\colon
	\mathcal{S}_{X_1}\xrightarrow{\sim}
	\mathcal{S}_{X_2},
	\quad
	f\mapsto
	\Big(
	h\mapsto \int_{X_2/X_{12}}f((x_2,0)h)dx_2
	\Big),
	\end{equation}
	where $X_{12}=X_1\bigcap X_2$.
	One verifies that
	$I_{X_1,X_2}(\rho(h)f)=\rho(h)(I_{X_1,X_2}f)$
	for $h\in H(W)$.
	Similarly, for maximal isotropic decompositions
	$V=X_1+Y_1=X_2+Y_2$, one can use the isomorphisms
	$\mathcal{S}_{X_i}\simeq\mathcal{S}(Y_i)$
	for $i=1,2$ to define intertwining operators
	\[
	I(Y_1,Y_2)
	\colon
	\mathcal{S}(Y_1)
	\xrightarrow{\sim}
	\mathcal{S}(Y_2).
	\]
	From this one gets the following commutative diagram
	\[
	\begin{tikzcd}
	\mathcal{S}_X \arrow[r,"A(g)"] \arrow[d,"I_{X,g^{-1}X}"'] 
	&
	\mathcal{S}_{gX} \arrow[d,"I_{gX,X}"]
	\\
	\mathcal{S}_{g^{-1}X} \arrow[r,"A(g)"]
	&
	\mathcal{S}_X
	\end{tikzcd}
	\]
	and we denote their composition by
	\[
	\Omega_X(g)\colon
	\mathcal{S}_X
	\xrightarrow{\sim}
	\mathcal{S}_X.
	\]
	
	Using the theorem of Stone and
	von Neumann (\cite[Theorem 1.1]{Kudla96}),
	one verifies that this is a projective representation of
	$\mathrm{Sp}(V)$ on $\mathcal{S}_X$,
	i.e. the map 
	$\Omega_X
	\colon
	\mathrm{Sp}(V)
	\rightarrow
	\mathrm{PGL}_\mathbb{C}(\mathcal{S}_X)$
	is a group homomorphism.
	Using the isomorphism between
	$\mathcal{S}_X$ and $\mathcal{S}(Y)$,
	one gets a projective representation of
	$\omega_Y$ of $\mathrm{Sp}(V)$
	on $\mathcal{S}(Y)$.
	The projective representation
	$\omega_Y$ induces a genuine representation
	$\widetilde{\omega}_Y$
	of the metaplectic group
	$\widetilde{\mathrm{Sp}}(V)$
	on $\mathcal{S}(Y)$.
	This 
	irreducible smooth representation
	is the Schr\"{o}dinger model
	of the Weil representation of $\widetilde{\mathrm{Sp}}(V)$
	(\cite[p.29]{MoeglinVignerasWaldspurger}).
	Here the metaplectic group can be defined as
	\[
	\widetilde{\mathrm{Sp}}(V)
	=
	\{
	(g,A_g)\in
	\mathrm{Sp}(V)\times\mathrm{GL}_\mathbb{C}
	(\mathcal{S}(Y))|
	A_g\rho(h)A_g^{-1}=\rho(gh),
	\forall h\in H(V)
	\}
	\]
	Note that this representation depends on the maximal 
	polarization of $V=X+Y$.
	Moreover, the metaplectic group
	$\widetilde{\mathrm{Sp}}(V)$ can also be viewed as
	the central extension of
	$\mathrm{Sp}(V)$
	by $\mathbb{C}^\times$.
	The map
	$\mathrm{Sp}(V)\rightarrow
	\widetilde{\mathrm{Sp}}(V),
	g\mapsto
	(g,\omega_Y(g))$
	is a section of the natural projection
	$\mathrm{pr}\colon
	\widetilde{\mathrm{Sp}}(V)
	\rightarrow
	\mathrm{Sp}(V)$.
	One can then define a $2$-cocycle
	on $\mathrm{Sp}(V)$ valued in $\mathbb{C}^\times$
	by
	\begin{equation}
	\label{2-cocycle}
	c_Y(g',g):=\omega_Y(g')\omega_Y(g)
	\omega_Y(g'g^{-1}).
	\end{equation}
	One can use this cocycle to define a group
	$\mathrm{Sp}(V)\times_{c_Y}\mathbb{C}^\times$
	by setting
	$(g,z)\cdot(g',z'):=
	(gg',zz'c_{Y}(g,g'))$.
	There is a close relation between the metaplectic group
	$\widetilde{\mathrm{Sp}}(V)$
	and the twofold cover
	$\mathrm{Mp}(V):=
	\mathrm{Sp}(V)\times_{r_Y}\mu_2$
	of $\mathrm{Sp}(V)$.
	Here $r_Y$ is Rao's $2$-cocycle
	on $\mathrm{Sp}(V)$ valued in
	$\mu_2\subset\mathbb{C}^\times$
	(\cite[Chapter I, Theorem 4.5]{Kudla96}).
	To relate these two cocycles,
	we define a map
	$\beta\colon
	\mathrm{Sp}(V)
	\rightarrow\mathbb{C}^\times$
	as in \cite[Chapter I, Theorem 4.5]{Kudla96},
	then one can show
	$c_Y(g,g')=\beta(gg')\beta(g)^{-1}\beta(g')^{-1}r_Y(g,g')$.
	Using this identity, one can define
	a map
	\[
	\mathrm{Sp}(V)\times_{r_Y}\mathbb{C}^\times
	\rightarrow
	\widetilde{\mathrm{Sp}}(V),
	\quad
	(g,z)
	\mapsto
	(g,z\beta(g)\omega_Y(g)),
	\]
	which is an isomorphism.
	Moreover, it is easy to see
	$\mathrm{Sp}(V)\times_{r_Y}\mathbb{C}^\times
	=\mathrm{Mp}(V)\times_{\mu_2}\mathbb{C}^\times$
	where the RHS is the contracted product.
	Thus $\widetilde{\mathrm{Sp}}(V)$
	is isomorphic to
	$\mathrm{Mp}(V)\times_{\mu_2}\mathbb{C}^\times$
	via the above isomorphism.

	Now we restrict ourselves to the special case
	$V=W_4=V_4\otimes U$ 
	with the polarization
	$W_4=W_4^++W_4^-$.
	Recall that we have
	a homomorphism
	$\iota\colon
	\mathrm{Sp}(V_4)\times
	\mathrm{O}(U)
	\hookrightarrow
	\mathrm{Sp}(W_4)$
	with restrictions
	$\iota_{V_4}:=\iota|_{\mathrm{Sp}(V_4)\times1}$
	and
	$\iota_U:=\iota|_{1\times\mathrm{O}(U)}$.
	By \cite[Chapter II, Corollary 3.3]{Kudla96},
	the morphism $\iota_{V_4}$
	lifts uniquely to a morphism
	of metaplectic groups
	$\widetilde{\iota}_{V_4}
	\colon
	\widetilde{\mathrm{Sp}}(V_4)
	\rightarrow
	\widetilde{\mathrm{Sp}}(W_4)$
	whose restriction to the center
	$\mathbb{C}^\times$
	sends $z$ to $z^{\mathrm{dim}(U)}=z^6$.
	More precisely, the map is given by
	\[
	\mathrm{Sp}(V_4)\times_{r_{V_4^-}}\mathbb{C}^\times
	\rightarrow
	\mathrm{Sp}(W_4)\times_{r_{W_4^-}}\mathbb{C}^\times,
	\quad
	(g,z)\mapsto
	(\iota_{V_4}(g),z^6\mu_U(g)).
	\]
	Here $\mu_U(g)$ is defined as in
	\cite[Chapter II, Proposition 3.2]{Kudla96}
	such that
	$r_{W_4^-}(\iota_{V_4}(g),\iota_{V_4}(g'))
	=
	r_{V_4^-}(g,g')
	\mu_U(gg')\mu_U(g)^{-1}\mu_U(g')^{-1}$.
	One can show that this map is a homomorphism
	and thus gives the homomorphism
	$\widetilde{\iota}_{V_4}$.
	Moreover, since
	$\mathrm{dim}(U)=6$ is even,
	the restriction of $\widetilde{\iota}_{V_4}$
	to $\mathrm{Mp}(V_4)
	=\mathrm{Sp}(V_4)\times_{r_{V_4^-}}\mu_2$
	factors through its quotient 
	$\mathrm{Sp}(V_4)$.
	Combining these maps, we get the following
	homomorphism
	\[
	\mathrm{Sp}(V_4)
	\rightarrow
	\widetilde{\mathrm{Sp}}(W_4),
	\quad
	g\mapsto
	\Big(
	\iota_{V_4}(g),\beta(\iota_{V_4}(g))\mu_U(g)\omega_{W_4^-}
	(\iota_{V_4}(g))
	\Big)
	\]

	On the other hand, one can show
	(\cite[p.39]{Kudla96}) that
	$r_{W_4^-}$ is trivial when restricted to
	$\mathrm{O}(U)$, thus
	the morphism
	$\iota_U$
	lifts to a morphism
	\[
	\mathrm{O}(U)
	\rightarrow
	\widetilde{\mathrm{Sp}}(W_4),
	\quad
	h\mapsto
	\Big(
	\iota_U(h),\omega_{W_4^-}(\iota_U(h))
	\Big).
	\]
	The lift is not unique
	yet we fix this one in this article.
	In summary, we get a representation
	$\omega_{W_4^-}$
	of 
	the reductive dual pair
	$\mathrm{Sp}(V_4)\times\mathrm{O}(U)$
	on $\mathcal{S}(W_4^-)$
	through the morphism
	$\mathrm{Sp}(V_4)\times\mathrm{O}(U)
	\rightarrow
	\widetilde{\mathrm{Sp}}(W_4)$.
	We next make explicit
	this representation for
	future computations.
	From now on we will only consider the representation
	$\omega_{W_4^-}
	|_{\mathrm{Sp}(V_4)\times\mathrm{O}(U)}$
	and for any 
	$g\in\mathrm{Sp}(V_4)$,
	resp.
	$h\in\mathrm{Sp}(U)$,
	we write
	$\omega_{W_4^-}(g)$, resp. $\omega_{W_4^-}(h)$
	instead of
	$\omega_{W_4^-}(\iota_{V_4}(g))$,
	resp. $\omega_{W_4^-}(\iota_U(h))$
	in the following.

	Write elements in $V_4$ in the form $(x,y)\in V_4^++V_4^-$.
	We view $g\in\mathrm{Sp}(V_4)$
	as a matrix of morphisms
	\[
	g=
	\begin{pmatrix}
	a & b \\
	c & d
	\end{pmatrix},
	\quad
	g\cdot(x,y)=(ax+by,cx+dy);
	\quad
	\]
	\[
	\text{where }
	a\in\mathrm{Hom}(V_4^+,V_4^+),
	\,
	b\in\mathrm{Hom}(V_4^-,V_4^+),
	\,
	c\in\mathrm{Hom}(V_4^+,V_4^-),
	\,
	d\in\mathrm{Hom}(V_4^-,V_4^-).
	\]
	
	It is easy to see that
	the stabilizer in $\mathrm{Sp}(V_4)$
	of the decomposition $V_4^++V_4^-$ is the set of matrices
	$m(a)$
	with $a\in\mathrm{Aut}(V_4^+)$
	and $a^\vee\in\mathrm{Aut}(V_4^-)$
	determined by
	$\langle ax,a^\vee y\rangle=\langle x,y\rangle$
	for any $x\in V_4^+,y\in V_4^-$.
	Any $g\in\mathrm{Sp}(V_4)$
	such that $g|_{V_4^+}=\mathrm{Id}_{V_4^+}$ is of the form
	$u(b)$
	with 
	$\langle bx,x'\rangle=\langle bx',x\rangle$
	for any $x,x'\in V_4^-$.
	Note that $\mathrm{Sp}(V_4)$ is generated by
	elements of the form
	$m(a)$,
	$u(b)$
	and $J_8$ 
	the longest Weyl element of $\mathrm{Sp}(V_4)$
	exchanging the fixed basis of $V_4^+$ and $V_4^-$.
	Under the basis $\mathcal{B}_{V_4}$,
	one can show that
	$a^\vee=a^{-t}$,
	$b$ is a symmetric matrix in
	and
	the longest Weyl element is $J_8$.
	For any element
	$w=(w_1,w_2,\cdots,w_8)\in W_4\simeq U^8$,
	we define the Gram matrix as
	\[
	(w,w):=(\langle w_i,w_j\rangle_U).
	\]
	For any $f\in\mathcal{S}(W_4^-)$
	and $y\in W_4^-$
	(\cite[Chapter II, Proposition 4.3]{Kudla96})
	one has
	\begin{equation}
	\label{general expression for Weil representation}
	\omega_{W_4^-}(
	\begin{pmatrix}
	a & b \\
	c & d
	\end{pmatrix})
	f(y)
	=
	\gamma^{\mathrm{Weil}}
	\int_{\mathrm{Ker}(c)\backslash W_4^+}
	\mathbf{e}(\frac{1}{2}\mathrm{tr}(ay,by)
	-\mathrm{tr}(by,cx)
	+\frac{1}{2}\mathrm{tr}(cx,dx))
	f(ay+cx)dx.
	\end{equation}
	Here$\gamma^\mathrm{Weil}$ is the Weil index,
	which is a complex number of absolute value $1$.
	We will discuss Weil index in more detail in Lemma 
	\ref{transfer of Siegel sections}.
	In particular, we have
	\begin{align}
	\label{expression for Weil representation}
	\omega_{W_4^-}
	(m(a))f(y)
	&
	=
	\xi(\mathrm{det}\, a)|\mathrm{det}\, a|_v^3f((a^ty),
	\\
	\omega_{W_4^-}
	(u(b)
	)f(y)
	&
	=
	\mathbf{e}(\mathrm{tr}((by,y))/2)
	f(y),
	\\
	\omega_{W_4^-}(J_8)f(y)
	&
	=
	\gamma^\mathrm{Weil}
	\int_{W_4^-}f(x)\mathbf{e}_F(\langle J_8x,y\rangle)dx,
	\\
	\omega_{W_4^-}(h)f(y)
	&
	=
	f(h^{-1}y),
	\forall h\in\mathrm{O}(U),
	\end{align}
	
	\begin{remark}\label{archimedean Weil representation}
		For
		$v=\infty$, since 
		$\xi$ is an odd character,
		we have
		$\xi(\mathrm{det}\, a)|\mathrm{det}\, a|^3_\infty
		=\mathrm{det}\, a^3$.
	\end{remark}

	The intertwining map
	$I_{W_4^-,W_{4,d}}$ between
	$\mathcal{S}_{W_4^-}$
	and $\mathcal{S}_{W_{4,d}}$
	induces an intertwining map
	between
	$\mathcal{S}(W_4^-)$
	and $\mathcal{S}(W_{4,d})$
	for the Weil representations of
	$\widetilde{\mathrm{Sp}}(W_4)$.
	More explicitly
	(\cite[p.182]{JianshuLi1992}),
	\[
	\delta
	\colon
	\mathcal{S}(W_4^-)
	\rightarrow
	\mathcal{S}(W_{4,d})
	\]
	\[
	\phi\mapsto
	\bigg(
	((x,y),-(x,y))\mapsto
	\int_{W_1^-}
	\phi(u+y,u-y)
	\mathbf{e}(
	2\langle
	u,x
	\rangle_{W_4})du
	\bigg).
	\]
	Here in the integral we identify
	$W_4^d$ with $W_1$
	by sending
	$((x',y'),(x',y'))$
	to $(x',y')$.
	For any function $f\in\mathcal{S}(W_1^-)$,
	the Fourier inversion theorem gives
	\[
	\int_{W_1^+}
	\mathbf{e}(\langle y,x'\rangle)dy
	\int_{W_1^-}
	f(x)\mathbf{e}(\langle x,y\rangle)dx
	=f(x').
	\]
	
	From this we have the inverse of $\delta$ as follows
	\[
	\delta^{-1}\colon
	\mathcal{S}(W_{4,d})
	\rightarrow
	\mathcal{S}(W_4^-)
	\]
	\[
	\phi'\mapsto
	\bigg(
	((y,y')\mapsto
	\int_{W_1^+}
	\phi'(x/2,(y-y')/2)
	\mathbf{e}(\langle x,(y+y')/2\rangle)dx
	\bigg).
	\]

	\subsubsection{Weil representations for similitude groups}
	We recall briefly how to extend the Weil
	representation $\omega_{W_4^-}$
	of $\mathrm{Sp}(V_4)\times\mathrm{O}(U)$
	to the similitude group
	$\mathrm{GSp}^+(V_4)\times\mathrm{GSO}(U)$.
	Here
	$\mathrm{GSp}^+(V_4)$
	is the subgroup of
	$\mathrm{GSp}(V_4)$ given by
	\[
	\mathrm{GSp}^+(V_4)
	=
	\{
	g\in\mathrm{GSp}(V_4)|
	\nu(g)\in\nu(\mathrm{GSO}(U))
	\}.
	\]
	The canonical reference is \cite{Roberts96}.
	See also 
	\cite[Section 2, Similitude theta correspondences]{GanTakeda11}.
	
    Suppose
	$F=\mathbb{Q}_v$ non-archimedean.
	We consider the subgroup
	$R_0=\{(g,h)\in R|
	\nu(g)\nu(h)=1\}$
	of the group
	$R=\mathrm{GSp}^+(V_4)\times\mathrm{GSO}(U)$.
	We define
	\[
	\omega_{W_4^-}(g,h)f(y)
	=
	|\nu(h)|^{-6}\omega_{W_4^-}(g_1,h)f(y)
	\]
	where
	$g_1=g\mathrm{diag}(\nu(g)^{-1},1)
	\in\mathrm{Sp}(V_4)$.
	It is easy to see that the scalar element
	$(\lambda,\lambda^{-1})\in R_0$ acts as
	$\omega_{W_4^-}(\lambda,\lambda^{-1})f
	=\xi(\lambda)^{-4}f$.
	Now we consider the compactly induced representation
	of $R$,
	$\mathrm{Ind}_{R_0}^R(\omega_{W_4^-})$.
	We will still denote this representation by
	$\omega_{W_4^-}$
	(\cite[Sections 2,3]{Roberts96}).

	\subsubsection{Global Weil representations}
	
	Let $W_4$ be as above, a vector space over
	$\mathbb{Q}$.
	We define the global Weil representation
	$\omega_{W_4^-}$
	of
	$\mathrm{GSp}^+(V_4,\mathbb{A})
	\times
	\mathrm{GSO}(U,\mathbb{A})$
	on $\mathcal{S}(W_4^-(\mathbb{A}))$
	as the restricted tensor product
	\[
	\omega_{W_4^-}
	:=\bigotimes_v{}'
	\omega_{W_4^-\otimes\mathbb{Q}_v}.
	\]

	\subsection{Theta lift}

	\subsubsection{Local theta lift}
	Recall $F=\mathbb{Q}_v$.
	The local theta lift is usually defined on the level of
	automorphic representations.
	Let $\pi^+$ be an irreducible admissible
	representation of $\mathrm{GSp}^+(V_4)$.
	We define
	$N(\pi^+)
	:=
	\bigcap_{\lambda}\mathrm{Ker}(\lambda)$,
	where $\lambda$ runs through the space
	$\mathrm{Hom}_{\mathrm{GSp}(V_4)}
	(\omega_{W_4^-},\pi^+)$.
	Then we set
	$S(\pi^+):=\omega_{W_4^-}/N(\pi^+)$.
	The space $S(\pi^+)$
	is a representation of
	$\mathrm{GSp}^+(V_4)\times\mathrm{GSO}(U)$.
	By \cite[Chapitre 2, Lemme III.4]{MoeglinVignerasWaldspurger},
	there is a smooth representation
	$\widetilde{\Theta}(\pi^+)$
	of
	$\mathrm{GSO}(U)$ unique up to isomorphisms
	such that
	$S(\pi^+)\simeq\pi^+\otimes\widetilde{\Theta}(\pi^+)$.
	
	Then the Howe duality conjecture says
	\begin{conjecture}
		For any irreducible admissible representation $\pi^+$of 
		$\mathrm{GSp}^+(V_4)$,
		either $\widetilde{\Theta}(\pi^+)$ vanishes
		or
		it is an admissible representation of
		$\mathrm{GSO}(U)$ of finite length.
		In the latter case,
		there exists a unique
		$\mathrm{GSO}(U)$-invariant submodule
		$\widetilde{\Theta}'(\pi^+)$
		of $\widetilde{\Theta}(\pi^+)$
		such that
		\[
		\Theta(\pi^+):=
		\widetilde{\Theta}(\pi^+)/\widetilde{\Theta}'(\pi^+)
		\]
		is an irreducible representation of $\mathrm{GSO}(U)$.
		If $\widetilde{\Theta}(\pi^+)=0$,
		we put
		$\Theta(\pi^+)=0$.
		Moreover, for two irreducible admissible representations
		$\pi_1^+$ and $\pi_2^+$ of $\mathrm{GSO}(U)$,
		if $\Theta(\pi_1^+)$ and
		$\Theta(\pi_2^+)$ are both
		non-zero and isomorphic, then
		$\pi_1^+$ is isomorphic to $\pi_2^+$(
		i.e. the map
		$\Theta\colon
		\pi^+\mapsto\Theta(\pi^+)$
		is injective on the isomorphic classes of
		irreducible admissible representations
		of $\mathrm{GSp}^+(V_4)$
		with non-zero images by $\widetilde{\Theta}$).
	\end{conjecture}

    In our situation, this conjecture can be proved using
    the results of
    \cite{Morimoto14,GanTakeda11b}.
    We will discuss this in more detail in
    Section \ref{section: Langlands functoriality}.

	\subsubsection{Global theta lift}
	To shorten notations, we shall write
	$(V_4^+)_U=V_4^+\otimes U$ and the like.
	For any
	$\phi\in \mathcal{S}(W_4^-(\mathbb{A}))$, 
	we can define the associated theta series:
	\[
	\Theta_\phi(g,h)=\sum_{w\in W_4^-(\mathbb{Q})}
	\big(
	\omega_v(g,h)\phi
	\big)
	(w),
	\forall
	(g,h)\in\mathrm{GSp}^+(V_4,\mathbb{A})
	\times\mathrm{GSO}(U,\mathbb{A})
	\]
	
	For any cuspidal automorphic form 
	$f\in\mathcal{A}_\mathrm{cusp}(\mathrm{GSp}^+(V_4,\mathbb{A}))$,
	the theta lift $\Theta_\phi(f)$ of $f$ to
	$\mathrm{GO}(U,\mathbb{A})$ is defined as
	\[
	\Theta_\phi(f)(h)
	=\int_{[\mathrm{GSp}^+(V_4)]_h}
	f(g)\Theta_\phi(g,h)dg
	\]
	where
	$[\mathrm{GSp}^+(V_4)]_h$
	is the subset of
	$[\mathrm{GSp}(V_4)]$
	consisting of elements
	$g$ such that
	$\nu(g)=\nu(h)$.
	
	Similarly, for any automorphic form
	$f'\in
	\mathcal{A}_\mathrm{cusp}(\mathrm{GSO}(U)(\mathbb{A}))$,
	its theta lift
	$\overline{\Theta}_\phi(f')$
	to
	$\mathrm{GSp}^+(V_4,\mathbb{A})$
	is defined as
	\[
	\overline{\Theta}_\phi(f')(g)
	=\int_{[\mathrm{GO}(U)]_g}
	f'(h)\overline{\Theta_\phi(g,h)}dh,
	\]
	where
	$[\mathrm{GO}(U)]_g$
	is the subset of
	$[\mathrm{GO}(U)]$
	consisting of elements $h$
	such that
	$\nu(h)=\nu(g)$.
	
	We can extend $\overline{\Theta}_\phi(f')$
	to the whole
	$\mathrm{GSp}(V_4,\mathbb{A})$
	as follows
	\[
	\overline{\Theta}_\phi(f')(g)
	=
	\begin{cases}
	\overline{\Theta}_\phi(f')(g) 
	& g\in\mathrm{GSp}(V_4,\mathbb{Q})\mathrm{GSp}^+(V_4,\mathbb{A}); \\
	0
	& \text{otherwise}.
	\end{cases}
	\]
	
	The maps $(\Theta_\phi,\overline{\Theta}_\phi)$
	are the theta correspondence between
	$\mathrm{GSp}^+(V_4)$ and $\mathrm{GSO}(U)$.
	
	Similarly, one can define theta correspondences
	between
	$\mathrm{GSp}^+(V_i)$
	and
	$\mathrm{GSO}(U)$
	for $i=1,2$.
	
	On the level of automorphic representations,
	the theta lift is defined as follows:
	let $\pi^+$ be an irreducible admissible automorphic
	representation of $[\mathrm{GSp}^+(V_4)]$.
	Then the theta lift $\Theta(\pi^+)$ of
	$\pi$ to
	$[\mathrm{GSO}(U)]$
	is the automorphic representation of
	$[\mathrm{GSO}(U)]$
	on the space
	\[
	\Theta(\pi^+)
	:=
	\{
	\Theta_\phi(f)|
	\forall \phi\in\mathcal{S}(W_4^-(\mathbb{A})),
	f\in\pi
	\}
	\]
	Similarly one can define the theta lift of automorphic representations
	in the other direction.

	\subsection{Siegel-Eisenstein series}
	We study Siegel-Eisenstein series and
	define local zeta integrals and
	Fourier coefficients.

	\subsubsection{Siegel-Eisenstein series}
	In this subsection, we will work with the group 
	$G=G_4=\mathrm{GSp}(V_4)$ and
	its parabolic subgroup
	$P^d=P_{V_4^d}$ stabilizing the subspace $V_4^d$ of $V_4$,
	$M^d=M_{P^d}$
	the Levi subgroup of $P^d$
	preserving the polarization
	$V_4^d\oplus V_{4,d}$,
	$N^d=N_{P^d}$
	the unipotent radical of $P_{V_4^d}$
	and the maximal torus
	$T^d=T_{P^d}$.
	The modulus character of $P^d$ is defined as
	\[
	\delta_{P^d}
	\colon
	P^d(\mathbb{A})
	\rightarrow
	\mathbb{C}^\times,
	\mathrm{diag}(1,\nu)m(A)u(B)
	\mapsto
	|\mathrm{det}\, A|^5
	|\nu|^{-10}.
	\]
	For any complex number $s\in\mathbb{C}$,
	we define a character $\xi_s$
	of $P^d(\mathbb{A})$ as follows
	\[
	\xi_s
	\colon
	P^d(\mathbb{A})
	\rightarrow
	\mathbb{C}^\times,
	\quad
	p
	=
	\mathrm{diag}(1,\nu)m(A)u(B)
	\mapsto
	\xi(\mathrm{det}\, A)
	\delta_{P^d}(p)^{s/5}
	=
	\xi(\mathrm{det}\, A)|\mathrm{det}\, A|^s|\nu|^{-2s}.
	\]
	We then define the normalized induction
	as follows
	\[
	\mathrm{Ind}^{G(\mathbb{A})}_{P^d(\mathbb{A})}(\xi_s)
	=
	\big\{
	f\colon
	G(\mathbb{A})
	\rightarrow
	\mathbb{C},
	\text{smooth}|
	f(pg)
	=
	\delta_{P^d}(p)^{1/2}\xi_s(p)f(g),
	\forall
	p\in P^d(\mathbb{A}),
	g\in G(\mathbb{A})
	\big\}.
	\]
	
	We also have the corresponding local version of inductions
	for each place $v$ of $\mathbb{Q}$,
	denoted by
	$\mathrm{Ind}^{G(\mathbb{Q}_v)}_{P^d(\mathbb{Q}_v)}(\xi_{s,v})$.
	Then one has the restricted tensor product
	\[
	\mathrm{Ind}^{G(\mathbb{A})}_{P^d(\mathbb{A})}(\xi_s)
	=
	\bigotimes_v{}'
	\mathrm{Ind}^{G(\mathbb{Q}_v)}_{P^d(\mathbb{Q}_v)}(\xi_{s,v})
	\]
	
	For any section
	$f(s)\in
	\mathrm{Ind}^{G(\mathbb{A})}_{P^d(\mathbb{A})}(\xi_s)$, we define
	the Siegel Eisenstein series associated to $f(s)$ as
	\[
	E(g,f(s))
	=
	\sum_{\gamma\in P^d(\mathbb{Q})\backslash G(\mathbb{Q})}
	f(s)(\gamma g).
	\]
	
	We can do the same thing for another parabolic subgroup
	$P^+=P_{V_4^+}$ of $G$
	which stabilizes the subspace $V_4^+$ of $V_4$,
	the Levi subgroup
	$M^+$
	of $P^+$
	preserving the polarization $V_4^+\oplus V_4^-$ of $V_4$
	and the unipotent radical 
	$N^+=N_{P^+}$ of $P^+$
	and the maximal torus
	$T^+=T_{P^+}$ of $P^+$.
	We can thus define the normalized induction
	$\mathrm{Ind}_{P^+(\mathbb{A})}^{G(\mathbb{A})}(\xi_s)$
	and Siegel Eisenstein series
	$E(g,f_{P^+}(s))$
	associated to a section
	$f_{P^+}(s)\in
	\mathrm{Ind}_{P^+(\mathbb{A})}^{G(\mathbb{A})}(\xi_s)$.

	Under the basis
	$\mathcal{B}_{V_4}$,
	we define
	$\mathcal{S}\in
	G_4(\mathbb{Q})$
	and its inverse as follows
	\[
	\mathcal{S}
	=
	\begin{pmatrix}
	1_2 & 0 & 0 & 1_2/2 \\
	-1_2 & 0 & 0 & 1_2/2 \\
	0 & 1_2 & 1_2/2 & 0 \\
	0 & 1_2 & -1_2/2 & 0
	\end{pmatrix},
	\mathcal{S}^{-1}
	=
	\begin{pmatrix}
	1_2/2 & -1_2/2 & 0 & 0 \\
	0 & 0 & 1_2/2 & 1_2/2 \\
	0 & 0 & 1_2 & -1_2 \\
	1_2 & 1_2 & 0 & 0
	\end{pmatrix}.
	\]

	Then $\mathcal{S}$ sends
	$V_4^+$ to $V_4^d$
	(so that
	$\mathcal{S}^{-1}P^d\mathcal{S}=P^+$,
	$\mathcal{S}W_4^+=W_4^d$
	and
	$\mathcal{S}W_4^-=W_{4,d}$).
	Thus we see that for any section
	$f_{P^+}(s)$
	in
	$\mathrm{Ind}_{P^+(\mathbb{A})}^{G(\mathbb{A})}(\xi_s)$,
	the function
	\[
	f^d_{P^+}(s)
	\colon
	G(\mathbb{A})
	\rightarrow
	\mathbb{C},
	\quad
	g\mapsto
	f_{P^+}(s)(\mathcal{S}^{-1}g)
	\]
	is a section in
	$\mathrm{Ind}_{P^d(\mathbb{A})}^{G(\mathbb{A})}(\xi_s)$.
	Moreover, it is easy to see that
	$E(g,f^d_{P^+}(s))
	=
	E(g,f_{P^+}(s))$
	for any
	$g\in G(\mathbb{A})$.

	Now fix
	$\phi_i\in\mathcal{S}(W_i^-(\mathbb{A}))$
	for $i=1,2$.
	Write
	$\phi^+=\phi_1\otimes\phi_2$ for the tensor product
	which is an element in
	$\mathcal{S}(W_4^-(\mathbb{A}))$.
	Recall that
	$G^+=\{
	g\in G|\nu(g)\in\nu(\mathrm{O}(U))
	\}$
	and we put
	$P^{++}:=P^+\bigcap G^+$.
	We then define a map
	\[
	f_{\phi^+}
	\colon
	G^+(\mathbb{A})
	\rightarrow
	\mathbb{C},
	\quad
	g
	\mapsto
	\big(
	\omega_{W_4^-}(g)\phi^+
	\big)
	(0).
	\]	
	One can verify that
	$f_{\phi^+}\in
	\mathrm{Ind}_{P^{++}(\mathbb{A})}^{G^+(\mathbb{A})}(\xi_{1/2})$.
	We can extend $f_{\phi^+}$ to a section
	$f_{\phi^+}$ in
	$\mathrm{Ind}_{P^+(\mathbb{A})}^{G(\mathbb{A})}(\xi_{1/2})$
	in a unique way as follows:
	note that
	$G=P^+G^+$,
	for any $g=pg^+\in P^+G^+$,
	we define
	$f_{\phi^+}(g):=\xi_{1/2}(p)f_{\phi^+}(g^+)$.
	One can verify that this is well-defined
	and thus we get a section
	$f_{\phi^+}\in
	\mathrm{Ind}_{P^+(\mathbb{A})}^{G(\mathbb{A})}(\xi_{1/2})$
	from $\phi^+$.
	This procedure applies to any other maximal parabolic subgroup
	other than $P^+$
	and we will extend the sections from
	$G^+$ to $G$ always in this way without further
	comment.

	We define
	$\phi^d=\delta(\phi^+)
	\in
	\mathcal{S}(W_{4,d}(\mathbb{A}))$.
	Then the function
	\[
	f_{\phi^d}
	\colon
	G^+(\mathbb{A})
	\rightarrow
	\mathbb{C},
	\quad
	g
	\mapsto
	\omega_{W_{4,d}}(g)\phi^d(0)
	\]
	is in
	$\mathrm{Ind}_{P^d(\mathbb{A})}^{G(\mathbb{A})}(\xi_{1/2})$.

	One has the following:
	\begin{lemma}\label{transfer of Siegel sections}
		The two sections
		$f_{\phi^d}$ and $f_{\phi^+}$ are related by
		\[
		f_{\phi^d}(g)
		=
		f_{\phi^+}(\mathcal{S}^{-1}g),
		\quad
		\forall
		g\in G(\mathbb{A}).
		\]
	\end{lemma}
	\begin{proof}
		This follows easily from the definition of
		the Weil representation $\omega_{W_4^-}$ of
		$\mathrm{Sp}(V_4)\times\mathrm{O}(U)$.
		
		Indeed, recall that in the morphism
		$\mathrm{Sp}(V_4)\rightarrow
		\widetilde{\mathrm{Sp}}(W_4)$,
		an element $g$ is sent to the element
		$\Big(
		\iota_{V_4}(g),\beta_{W_4^-}
		(\iota_{V_4}(g))\mu_{U,W_4^-}(g)\omega_{W_4^-}
		(\iota_{V_4}(g))
		\Big)$
		and by definition the Weil representation
		of $\mathrm{Sp}(V_4)$
		on the space $\mathcal{S}(W_4^-)$ is
		$\omega_{W_4^-}(g)
		=
		\beta_{W_4^-}
		(\iota_{V_4}(g))\mu_{U,W_4^-}(g)
		\omega_{W_4^-}
		(\iota_{V_4}(g))$.
		Here we add the subscript
		$_{W_4^-}$ to indicate the dependence of the functions
		$\beta$ and $\mu$ on the isotropic subspace
		$W_4^-$.
		Similarly, we can define the Weil representation
		$\omega_{W_{4,d}}$ of
		$\mathrm{Sp}(V_4)$
		on the space $\mathcal{S}(W_{4,d})$,
		which is given by
		$\omega_{W_{4,d}}(g)
		=
		\beta_{W_{4,d}}(\iota_{V_4}(g))
		\mu_{U,W_{4,d}}(g)
		\omega_{W_{4,d}}(\iota_{V_4}(g))$.
		
		Now by definition,
		\begin{align*}
		f_{\phi^d}(g)
		&
		=
		\omega_{W_{4,d}}(g)\phi^d(0)
		\\
		&
		=
		\beta_{W_{4,d}}(\iota_{V_4}(g))
		\mu_{U,W_{4,d}}(g)
		\omega_{W_{4,d}}(\iota_{V_4}(g))\phi^d(0)
		\\
		&
		=
		\beta_{W_{4,d}}(\iota_{V_4}(g))
		\mu_{U,W_{4,d}}(g)
		A(g)\circ
		I(W_{4,d},g^{-1}W_{4,d})
		\circ
		I(W_4^-,W_{4,d})\phi^+(0)
		\\
		&
		=
		\beta_{W_4^-}
		(\iota_{V_4}(\mathcal{S}^{-1}g))\mu_{U,W_4^-}(\mathcal{S}^{-1}g)
		A(\mathcal{S})
		\circ
		A(\mathcal{S}^{-1}g)
		\circ
		I(W_4^-,(\mathcal{S}^{-1}g)^{-1}W_4^-)\phi^+(0)
		\\
		&
		=
		\beta_{W_4^-}
		(\iota_{V_4}(\mathcal{S}^{-1}g))
		\mu_{U,W_4^-}(\mathcal{S}^{-1}g)
		A(\mathcal{S})
		\omega_{W_4^-}(\mathcal{S}^{-1}g)\phi^+(0)
		\\
		&
		=
		\beta_{W_4^-}
		(\iota_{V_4}(\mathcal{S}^{-1}g))\mu_{U,W_4^-}(\mathcal{S}^{-1}g)
		\omega_{W_4^-}(\mathcal{S}^{-1}g)\phi^+(0)
		\\
		&
		=
		f_{\phi^+}(\mathcal{S}^{-1}g).
		\end{align*}
		From the third line to the fourth line
		we used the fact that
		
		\begin{multline}\label{intertwining operator relation}
		\beta_{W_4^-}
		(\iota_{V_4}(\mathcal{S}^{-1}g))\mu_{U,W_4^-}(\mathcal{S}^{-1}g)
		\mathrm{Id}_{\mathcal{S}(W_4^-)} \\
		=
		\beta_{W_{4,d}}(\iota_{V_4}(g))
		\mu_{U,W_{4,d}}(g)
		I(g^{-1}\mathcal{S}W_4^-,W_4^-)\circ
		I(\mathcal{S}W_4^-,g^{-1}\mathcal{S}W_4^-)\circ
		I(W_4^-,\mathcal{S}W_4^-)
		\end{multline}

		To show this last identity, 		
		we need some preliminary results
		on the composition of intertwining operators
		defined in \ref{intertwining operator}.
		Use again the notations in Section.\ref{Weil representations},
		putting $V=W_4$
		over a local field $F=\mathbb{Q}_v$
		and $X_i$ maximal isotropic subspaces of
		$V$.
		The composition of intertwining operators
		$I_{X_2,X_3}\circ I_{X_1,X_2}$ is not exactly
		$I_{X_1,_3}$, but differs from the latter by
		a scalar in $\mathbb{C}^\times$.
		This scalar is given by the Maslov index and Weil character
		defined below.
		We write
		$\widehat{W}(F)$ to be the Grothendieck group
		of isometry classes of quadratic forms over $F$
		and define the Witt group $W(F)$ as
		$W(F)=\widehat{W}(F)/\mathbb{Z}H$ where
		$H$ is the standard split
		hyperbolic plane (of dimension $2$).
		Let $(Q,q)$ be a quadratic space over $F$ 
		with quadratic form $q$.
		Consider the pairing
		$\mathbf{e}_F\circ q\colon
		Q\times Q\rightarrow\mathbb{C}^\times$.
		Let $d\mu_q$ be a measure on $V$ self-dual with respect to
		the pairing $\mathbf{e}_F\circ q$.
		Choose any
		Schwartz-Bruhat function
		$h\in\mathcal{S}(Q)$ such that
		its Fourier transform is a positive measure and
		$h(0)=1$.
		Then the Weil character $\gamma^\mathrm{Weil}(q)$
		is defined to be
		(\cite[\textsection 14, Th\'{e}or\`{e}me 2]{Weil1964},
		\cite[Proposition 1.2.13]{WenWeiLi})
		\[
		\gamma^\mathrm{Weil}(q)
		=
		\lim\limits_{s\rightarrow0}
		\int_Q
		h(sx)\mathbf{e}_F(q(x,x)/2)d\mu_q(x).
		\]
		This gives the Weil character of the Witt group
		\[
		\gamma^\mathrm{Weil}
		\colon
		W(F)\rightarrow
		\mathbb{C}^\times,
		\quad
		(Q,q)\mapsto
		\gamma^\mathrm{Weil}(q).
		\]
		For any maximal isotropic subspaces
		$X_1,X_2,X_3$ of $V$,
		we define
		$K:=X_1\oplus X_2\oplus X_3$ and give it the following
		quadratic form $q_K$: for any $v,w\in K$,
		\[
		q_K(v,w)
		=q_K((v_1,v_2,v_3),(w_1,w_2,w_3))
		:=
		\frac{1}{2}
		(
		\langle v_1,w_2-w_3\rangle
		+
		\langle v_2,w_3-w_1\rangle
		+
		\langle v_3,w_1-w_2\rangle
		).
		\]
		Then the Maslov index 
		$\tau(X_1,X_2,X_3)$ of $X_1,X_2,X_3$
		is the equivalence class of the quadratic space
		$(K,q_K)$ in $W(F)$
		(\cite[Proposition 2.3.3]{WenWeiLi}).
		Now one can show that
		(\cite[Theorem 3.5.1]{WenWeiLi})
		\[
		I_{X_3,X_1}\circ I_{X_2,X_3}\circ I_{X_1,X_2}
		=
		\gamma^\mathrm{Weil}(-\tau(X_1,X_2,X_3))
		\cdot
		\mathrm{Id}_{\mathcal{S}_{X_1}}
		\colon
		\mathcal{S}_{X_1}\rightarrow
		\mathcal{S}_{X_1}.
		\]
		Using the fact that
		$I_{X_1,X_2}\circ I_{X_2,X_1}=\mathrm{Id}$
		(\cite[Corollary 3.4.3]{WenWeiLi}),
		one sees that
		\[
		I_{X_2,X_3}\circ I_{X_1,X_2}
		=
		\gamma^\mathrm{Weil}
		(-\tau(X_1,X_2,X_3))
		\cdot
		I_{X_1,X_3}
		\colon
		\mathcal{S}_{X_1}
		\rightarrow
		\mathcal{S}_{X_3}.
		\]
		Fix a maximal isotropic subspace
		$Y$ of $V$.
		One can show that for any
		$g_1,g_2\in\mathrm{Sp}(V)$,
		$c_Y(g_1,g_2)=
		\gamma^\mathrm{Weil}
		(-\tau(Y,g_1Y,g_1g_2Y))$
		(\textit{cf}. (\ref{2-cocycle})
		and \cite[Chapter I, Theorem 3.1]{Kudla96}).
		
		Now let's return to our problem and set
		$X_1=W_4^+$,
		$X_2=W_4^d$
		and
		$X_3=g^{-1}W_4^d$.
		One gets
		\[
		I_{W_4^d,g^{-1}W_4^d}\circ I_{W_4^+,W_4^d}
		=
		\gamma^\mathrm{Weil}
		(-\tau(W_4^+,W_4^d,g^{-1}W_4^d))
		\cdot
		I_{W_4^+,g^{-1}W_4^d}.
		\]
		Let's write
		$\gamma(g)
		=\gamma^\mathrm{Weil}
		(-\tau(W_4^+,W_4^d,g^{-1}W_4^d))$.
		
		By the basic properties of
		the Maslov index $\tau$
		(\cite[\textsection\textsection2.1 and 4.3.1]{WenWeiLi}),
		one has
		\begin{align*}
		\gamma^\mathrm{Weil}
		(-\tau(W_{4,d},\mathcal{S}^{-1}W_{4,d},g^{-1}W_{4,d}))
		&
		=
		\gamma^\mathrm{Weil}
		(\tau(g^{-1}W_{4,d},\mathcal{S}^{-1}W_{4,d},W_{4,d}))
		\\
		&
		=
		\gamma^\mathrm{Weil}
		(\tau(W_{4,d},g\mathcal{S}^{-1}W_{4,d},gW_{4,d}))
		\\
		&
		=
		c_{W_{4,d}}(g\mathcal{S}^{-1},\mathcal{S}).
		\end{align*}
		On the other hand,
		$c_{W_{4,d}}$
		is a $2$-cocycle, therefore
		\[
		c_{W_{4,d}}(g\mathcal{S}^{-1},\mathcal{S})
		c_{W_{4,d}}(g,\mathcal{S}^{-1})
		=
		c_{W_{4,d}}(g,1)
		c_{W_{4,d}}(\mathcal{S}^{-1},\mathcal{S}).
		\]
		By definition and
		\cite[Corollary 3.4.3]{WenWeiLi},
		one has
		$c_{W_{4,d}}(g,1)=1=
		c_{W_{4,d}}(\mathcal{S}^{-1},\mathcal{S})$.
		
		Moreover, by definition one sees that
		\[
		\beta_{W_4^-}(\iota_{V_4}(\mathcal{S}^{-1}g))
		=
		\beta_{W_{4,d}}(\iota_{V_4}(g\mathcal{S}^{-1})),
		\quad
		\mu_{U,W_4^-}(\mathcal{S}^{-1}g)
		=
		\mu_{U,W_{4,d}}(g\mathcal{S}^{-1}).
		\]

		Therefore (\ref{intertwining operator relation})
		is equivalent to
		\[
		\beta_{W_{4,d}}
		(\iota_{V_4}(g\mathcal{S}^{-1}))
		\mu_{U,W_{4,d}}(g\mathcal{S}^{-1})
		=
		\beta_{W_{4,d}}(\iota_{V_4}(g))
		\mu_{U,W_{4,d}}(g)
		c_{W_{4,d}}(g,\mathcal{S}^{-1}).
		\]
		
		By the morphism
		$\mathrm{Sp}(V_4)
		\rightarrow
		\widetilde{\mathrm{Sp}}(W_4)$,
		this is equivalent to
		\[
		\beta_{W_{4,d}}(\iota_{V_4}(\mathcal{S}^{-1}))
		\mu_{U,W_{4,d}}(\iota_{V_4}(\mathcal{S}^{-1}))
		=1.
		\]
		
		We can show this in the following way.
		First we compute the LHS of the above formula
		for each local case
		$F=\mathbb{Q}_v$.
		Using the Bruhat decomposition for parabolic subgroup
		$P_{V_{4,d}}$ of $\mathrm{Sp}(V_4)$,
		we see that
		(using the notation as in
		\cite[p.19]{Kudla96})
		$j(\mathcal{S}^{-1})=2$
		and
		$j(\iota_{V_4}(\mathcal{S}^{-1}))=2\times6=12$.
		In the same way,
		$x(\iota_{V_4}(\mathcal{S}^{-1}))
		=
		x(\mathcal{S}^{-1})^6
		=1\,
		\mathrm{mod} (F^\times)^2$.
		Thus we see
		\[
		\beta_{W_{4,d}}(\iota_{V_4}(\mathcal{S}^{-1}))
		=\gamma(\eta)^{-12}
		=\gamma(-1,\eta)^2
		=(-1,-1)_v
		\]
		where
		$(\cdot,\cdot)_v$
		is the Hilbert symbol in
		$F=\mathbb{Q}_v$.
		
		On the other hand,
		by \cite[p.35]{Kudla96},
		\[
		\mu_{U,W_{4,d}}(\mathcal{S}^{-1})
		=(\mathrm{det}(U),x(\iota_{V_4}(\mathcal{S}^{-1})))_v
		\gamma(\mathrm{det}(U),\eta)^{-2}
		=(\mathrm{det}(U),-1)_v.
		\]
		
		Taking the product of the above two expressions
		and using the product formula for
		Hilbert symbols,
		one gets
		the desired identity
		for $g\in\mathrm{Sp}(V_4)(\mathbb{A})$.
		For
		$g\in\mathrm{GSp}(V_4)(\mathbb{A})\backslash
		\mathrm{Sp}(V_4)(\mathbb{A})$,
		one can use the definition of the extension of
		$f_{\phi^?}$
		from
		$\mathrm{Sp}(V_4)(\mathbb{A})$
		to
		$\mathrm{GSp}(V_4)(\mathbb{A})$
		to finish the proof.

	\end{proof}

	For each place $v$ of
	$\mathbb{Q}$,
	we define local sections
	$\widetilde{f}_{\phi^d,v}\in
	\mathrm{Ind}^{G(\mathbb{Q}_v)}_{P^d(\mathbb{Q}_v)}
	(\xi_{1/2})$
	as
	$\widetilde{f}_{\phi^d,v}(g)
	:=f_{\phi^+,v}(\mathcal{S}^{-1}g\mathcal{S})$.
	Then we put
	$\widetilde{f}_{\phi^d}:=
	\otimes'_v\widetilde{f}_{\phi^d,v}
	\in
	\mathrm{Ind}_{P^d(\mathbb{A})}^{G(\mathbb{A})}(\xi_{1/2})$.
	From the above lemma, we get
	\begin{corollary}
		Let $f_{\phi^?}
		\in\mathrm{Ind}_{P^?(\mathbb{A})}^{G(\mathbb{A})}(\xi_{1/2})$
		be the section defined by
		$\phi^?$ as above
		for $?=d,+$,
		then we have
		the following identity
		\[
		E(g,f_{\phi^d})
		=
		E(g,\widetilde{f}_{\phi^d})
		=
		E(g,f_{\phi^+})
		\]	
	\end{corollary}
	\begin{remark}
		In the following we will use
		$f_{\phi^+}$ to compute the Fourier coefficients of
		the Eisenstein series
		$E(\cdot,f_{\phi^+})$
		while we will use
		$\widetilde{f}_{\phi^d}$
		to compute the local zeta integrals
		given in the next subsection.
	\end{remark}

	\subsubsection{Zeta integrals}
	Let $\pi=\otimes_v'\pi_v$ 
	be an irreducible automorphic representation
	on $G_1(\mathbb{A})$,
	$\pi^\vee$ its contragredient,
	which is isomorphic to the complex conjugate
	$\overline{\pi}$.
	We choose an irreducible
	$G_1^1(\mathbb{A})$-constituent $\pi^1$ of $\pi$
	that occurs in the space of automorphic forms on
	$G_1^1(\mathbb{A})$.
	Similarly for
	$\pi^{\vee,1}:=(\pi^\vee)^1$.
	We assume that
	$\pi^1$ contains the spherical vectors for
	$\widehat{\Gamma}$.
	We will see that
	the standard $L$-function
	does not depend on the choice $\pi^1$ 
	in the decomposition of $\pi|_{G_1^1(\mathbb{A})}$.
	Let $S_\pi$ be the places of $\mathbb{Q}$
	dividing $N$.	
	We fix non-zero unramified vectors
	$\varphi_{v,0}\in\pi_v$
	and
	$\varphi_{v,0}^\vee
	\in\pi_v^\vee$
	for each $v\notin S_\pi$.
	Assume moreover that there is factorization
	\[
	\pi^1\simeq
	\pi^1_\infty\otimes\pi^1_\mathrm{f}
	\simeq
	\pi^1_\infty
	\otimes
	(\pi^1)_\mathrm{f}^{S_\pi,p}
	\otimes
	\pi^1_p
	\otimes
	\pi^1_{S_\pi}.
	\]
	Similarly for $\pi^{\vee,1}$:
	\[
	\pi^{\vee,1}
	\simeq
	\pi^{\vee,1}_\infty
	\otimes
	\pi^{\vee,1}_\mathrm{f}
	\simeq
	\pi^{\vee,1}_\infty
	\otimes
	(\pi^{\vee,1})_\mathrm{f}^{S_{\pi,p}}
	\otimes
	\pi^{\vee,1}_p
	\otimes
	\pi^{\vee,1}_{S_\pi}.
	\]
	
	Fix factorizable vectors
	$\varphi=\otimes_v\varphi_v
	\in(\pi^1)^{\widehat{\Gamma}}$
	and
	$\varphi^\vee=\otimes_v\varphi_v^\vee\in
	(\pi^{\vee,1})^{\widehat{\Gamma}}$.
	We think of $\varphi$, resp. $\varphi^\vee$,
	as an automorphic form on $G_1(\mathbb{A})$,
	resp. $G_2(\mathbb{A})$
	as follows:
	for any
	$g=g_1g_2\in
	G_1(\mathbb{Q})G_1^1(\mathbb{A})$,
	we set
	$\varphi(g)$
	to be $\varphi(g_2)$,
	while for
	$g\in G_1(\mathbb{A})\backslash
	G_1(\mathbb{Q})G_1^1(\mathbb{A})$,
	we set
	$\varphi(g)=0$.
	Similarly for
	$\varphi^\vee$.
	Moreover, we assume that
	$\varphi_v=\varphi_{v,0}$
	and
	$\varphi_v^\vee=\varphi_{v,0}^\vee$
	for any $v\notin S_\pi$.
	
	As for the places $v|p\infty$, we will specify
	$\varphi_v$ and
	$\varphi_v^\vee$ later on.

	Let $Z_1$ be the center of $G_1$.
	We fix a global Haar measure $dg$ on
	$G_1(\mathbb{A})$
	invariant under left translation of
	$G_1(\mathbb{Q})$.
	Then we define
	\[
	\langle
	\varphi,
	\varphi^\vee
	\rangle
	:=
	\int_{Z_1(\mathbb{A})G_1(\mathbb{Q})\backslash G_1(\mathbb{A})}
	\varphi(g)\varphi^\vee(g)dg.
	\]
	This pairing between $\pi$ and $\pi_\vee$
	decomposes into a product of local pairings
	$\langle\cdot,\cdot,\rangle_v
	\colon\pi_v\otimes\pi_v^\vee
	\rightarrow
	\mathbb{C}$ that are
	$G_1(\mathbb{Q}_v)$-invariant and
	$\langle
	\varphi_v,
	\varphi_v^\vee
	\rangle_v=1$.

	Now for any section
	$f(s)\in\mathrm{Ind}_{P^d(\mathbb{A})}^{G(\mathbb{A})}
	(\xi_s)$,
	we define the global zeta integral for
	$f(s)$, $\varphi$ and $\varphi^\vee$ as follows
	\[
	Z(\varphi,\varphi^\vee,f(s))
	:=
	\int_{Z_3(\mathbb{A})G_3(\mathbb{Q})\backslash G_3(\mathbb{A})}
	E((g_1,g_2),f(s))
	\xi^{-1}_s(\mathrm{det}\, g_2)
	\varphi(g_1)\varphi^\vee(g_2)
	dg_1dg_2.
	\]
	
	One can show
	\begin{lemma}
		We have
		\begin{equation}\label{global zeta integral}
		Z(\varphi,\varphi^\vee,f(s))
		=
		\int_{G_1^1(\mathbb{A})}
		f(s)((g,1))
		\langle
		\pi(g)\varphi,
		\varphi^\vee
		\rangle
		dg.
		\end{equation}
	\end{lemma}

    \begin{proof}
    	The proof is almost the same as
    	\cite[pp.702-703]{Harris93}.
    	Recall $P^d_4$ is the stabilizer in
    	$G_4$ of $V_4^d$.
    	We put $P^{d,1}_4=P^d_4\cap G^1_4$.
    	Then
    	$P^{d,1}_4\backslash G_4^1
    	\simeq P^d_4\backslash G_4$.
    	As in \cite{GelbartPSRallis87},
    	$G_3$ acts by translation on 
    	the right on the flag variety
    	$P^d_4\backslash G_4$.
    	The orbits of this action are the same as
    	the orbits of the action 
    	$G_1^1\times G_2^1$
    	on the flag variety
    	$P^{d,1}_4\backslash G_4^1$.
    	So we see that the orbit
    	$P^{d,1}\cdot1\cdot G_3$ is the main orbit,
    	as defined in \cite[p.2]{GelbartPSRallis87}
    	and the other orbits are all negligible.
    	We put also
    	$G_{3,\gamma}(\mathbb{Q})
    	:=
    	G_3(\mathbb{Q})\bigcap\gamma^{-1}P_4^{d,1}
    	(\mathbb{Q})\gamma$,
    	in particular one has
    	$G_{3,1}(\mathbb{Q})
    	=
    	G_3^1(\mathbb{Q})$. 	
    	Therefore
    	we can unfold the global zeta integral 
    	in the lemma as follows
    	\begin{align*}
    	Z(\varphi,\varphi^\vee,f(s))
    	&
    	=
    	\int_{Z_3(\mathbb{A})G_3(\mathbb{Q})\backslash G_3(\mathbb{A})}
    	\sum_{\gamma\in P_4^d(\mathbb{Q})\backslash G_4(\mathbb{Q})}
    	f(s)(\gamma(g_1,g_2))
    	\xi_s^{-1}(\mathrm{det}\, g_2)
    	\varphi(g_1)\varphi^\vee(g_2)dg_1dg_2
    	\\
    	&
    	=
    	\sum_{\gamma\in P_4^d(\mathbb{Q})
    		\backslash G_4(\mathbb{Q})/G_3(\mathbb{Q})}
    	\int_{Z_3(\mathbb{A})G_{3,\gamma}
    		(\mathbb{Q})\backslash G_3(\mathbb{A})}
    	f(s)(\gamma(g_1,g_2))
    	\xi_s^{-1}(\mathrm{det}\, g_2)
    	\varphi(g_1)\varphi^\vee(g_2)dg_1dg_2.
    	\\
    	&
    	=
    	\int_{Z_3(\mathbb{A})G_3^1(\mathbb{Q})
    		\backslash G_3(\mathbb{A})}
    	f(s)((g_2,g_2)(g_2^{-1}g_1,1))
    	\xi_s^{-1}(\mathrm{det}\, g_2)
    	\varphi(g_2g_2^{-2}g_1)
    	\varphi^\vee(g_2)
    	dg_1dg_2.
    	\end{align*}
    	In the last identity we used the fact that
    	only the main orbit
    	$P_4^{d,1}(\mathbb{Q})\cdot1\cdot G_3(\mathbb{Q})$
    	contributes to the integral.
    	Write $G_1^d$ for the image of the map
    	$G_1\rightarrow G_3$
    	sending $g$ to $(g,g)$.  	
    	Then
    	we have an isomorphism
    	\[
    	G_3
    	=
    	\big\{
    	(g_2,g_2)(g_2^{-1}g_1,1)|
    	\nu(g_1)=\nu(g_2)
    	\big\}
    	=
    	G_1^dG_1^1
    	\simeq
    	G_1\times G_1^1.
    	\]
    	Note that
    	under the above isomorphism,
    	the subgroup
    	$G_3^1$, resp.,
    	$Z_3$ of $G_3$
    	is sent to the subgroup
    	$1\times G_1^1$,
    	resp.,
    	$1\times Z_1$
    	of $G_1\times G_1^1$.
    	Since
    	$G_1^d\subset P_4^d$,
    	we get
    	\[
    	f(s)((g_2,g_2)(g,1))\xi_s^{-1}(\mathrm{det}\, g_2)
    	=f(s)((g,1))\xi_s(\mathrm{det}\, g_2)\xi_s^{-1}(\mathrm{det}\, g_2)
    	=f(s)((g,1)).
    	\]
    	
    	The above zeta integral becomes
    	\begin{align*}
    	Z(\varphi,\varphi^\vee,f(s))
    	&
    	=
    	\int_{Z_3(\mathbb{A})G_3^d(\mathbb{Q})
    		\backslash G_1^d(\mathbb{A})G_1^1(\mathbb{A})}
    	f(s)((g,1))\varphi(g_2g)\varphi^\vee(g_2)dg_2dg
    	\\
    	&
    	=
    	\int_{G_1^1(\mathbb{A})}
    	f(s)((g,1))dg
    	\int_{Z_1(\mathbb{A})G_1(\mathbb{Q})\backslash G_1(\mathbb{A})}
    	\varphi(g_2g)\varphi^\vee(g_2)dg_2
    	\\
    	&
    	=
    	\int_{G_1^1(\mathbb{A})}
    	f(s)((g,1))
    	\langle
    	\pi(g)\varphi,
    	\varphi^\vee
    	\rangle
    	dg,
    	\end{align*}
    	which is the desired integral.

    \end{proof}

	For each place $v$ of
	$\mathbb{Q}$,
	fix a section
	$f_v(s)
	\in
	\mathrm{Ind}_{P_4^d(\mathbb{Q}_v)}
	^{G_4(\mathbb{Q}_v)}(\xi_{s,v})$,
	we then define the local zeta integral
	of $f_v(s)$,
	$\varphi_v$ and $\varphi_v^\vee$
	as follows:
	\begin{equation}\label{local zeta integral}
	Z_v(\varphi_v,\varphi_v^\vee,f_v(s))
	:=
	\int_{G_1^1(\mathbb{Q}_v)}
	f_v(s)((g_v,1))
	\langle
	\pi_v(g_v)\varphi_v,
	\varphi_v^\vee
	\rangle_v
	dg_v.
	\end{equation}
	
	Then it is clear that
	
	\begin{lemma}
		For factorizable vectors
		$f(s)
		=
		\bigotimes_v{}'
		f_v(s)
		\in
		\bigotimes_v{}'
		\mathrm{Ind}_{P_4^d(\mathbb{Q}_v)}
		^{G_4(\mathbb{Q}_v)}(\xi_{s,v})$,
		$\varphi\in\pi$
		and
		$\varphi^\vee\in\pi^\vee$,
		we have
          \[
          Z(\varphi,\varphi^\vee,f(s))
          =
          \prod_v
          Z_v(\varphi_v,\varphi_v^\vee,f_v(s)).
          \]		
	\end{lemma}

	\subsubsection{Fourier coefficients}
	Next we define the Fourier coefficients of the
	Eisenstein series:
	as above we fix a section
	$f(s)
	\in
	\mathrm{Ind}_{P_4^+(\mathbb{A})}^{G_4(\mathbb{A})}(\xi_s)$.
	
	\begin{definition}
		For each symmetric matrix
		$\beta\in\mathrm{Sym}_{4\times 4}(\mathbb{Q})$,
		the $\beta$-th Fourier coefficient
		of the Eisenstein series 
		$E(\cdot,f^d(s))
		=E(\cdot,f(s))$
		is defined as
		\begin{align*}
		E_\beta(g,f(s))
		&
		:=
		\int_{[\mathrm{Sym}_{4\times 4}]}
		E(
		u(x)g,
		f(s)
		)
		\mathbf{e}(-\mathrm{tr}\,\beta x)
		dx
		\\
		&
		=
		\int_{[\mathrm{Sym}_{4\times 4}]}
		E(
		u(x)g,
		f(s)
		)
		\mathbf{e}(-\mathrm{tr}\,
		u(\beta)^t
		u(x)
		)
		dx.
		\end{align*}
		
		If $f(s)=\otimes_v'f_v(s)$
		is factorizable,
		then the $\beta$-th local
		Fourier coefficient of
		$E(g,f(s))$ at $v$ is
		\begin{equation}
		E_{\beta,v}(g,f(s))
		=
		E_{\beta,v}(g_v,f_v(s))
		:=
		\int_{\mathrm{Sym}_{4\times 4}(\mathbb{Q}_v)}
		f_v(s)
		(J_8u(x)
		g_v
		)
		\mathbf{e}_v(-\mathrm{tr}\,beta x_v)
		d_vx_v.
		\end{equation}
	\end{definition}

	We call the coset
	$P^+_4J_8P^+_4=P_4^+J_8N^+_4$
	the big cell of $G_4$.
	Here $N_4^+$ is the unipotent radical of
	$P_4^+$.
	One can show
	\begin{lemma}
		Fix a section
		$f(s)\in\mathrm{Ind}_{P_4^+(\mathbb{A})}
		^{G_4(\mathbb{A})}$
		such that
		there is a finite place $v_0$ of $\mathbb{Q}$
		at which in the natural projection
		$G_4(\mathbb{A})
		\twoheadrightarrow
		G_4(\mathbb{Q}_{v_0})$,
		the image of the support
		of $f(s)$ in $G_4(\mathbb{A})$
		is contained in the big cell
		$P_4^+(\mathbb{Q}_{v_0})J_8P_4^+(\mathbb{Q}_{v_0})$.
		Then for any
		$g\in P_4^+(\mathbb{A})$,
		one has		
		\[
		E_\beta(g,f(s))
		=
		\int_{\mathrm{Sym}_{4\times 4}(\mathbb{A})}
		f
		(s)
		(J_8u(x)
		g
		)
		\mathbf{e}(-\mathrm{tr}\,\beta x)
		dx.
		\]
		
		If moreover
		$f(s)=\otimes_vf_v(s)$
		is factorizable,
		then
		\begin{equation}
		E_\beta(g,f(s))
		=
		\prod_vE_{\beta,v}(g_v,f_v(s)).
		\end{equation}
	\end{lemma}

    \begin{proof}
    	The proof is essentially contained in
    	\cite[Section 18.9]{Shimura1997}.  
    	For any
    	$g=(g_v)_v\in P_4^+(\mathbb{A})$,
    	$f(s)(\gamma g)\neq0$ implies that
    	$\gamma g_{v_0}\in
    	P_4(\mathbb{Q}_{v_0})J_8P_4(\mathbb{Q}_{v_0})$,
    	thus
    	$\gamma\in
    	P_4^+(\mathbb{Q})J_8N_4^+(\mathbb{Q})$.
    	Therefore we have
    	\[
    	E(g,f(s))
    	=
    	\sum_{\gamma\in J_8N_4^+(\mathbb{Q})}
    	f(s)(\gamma g)
    	=
    	\sum_{n\in N_4^+(\mathbb{Q})}
    	f(s)(J_8ng).
    	\]
    	
    	From this, we get the $\beta$-th
    	Fourier coefficient of
    	$E(g,f(s))$
    	as
    	\begin{align*}
    	E_\beta(g,f(s))
    	&
    	=
    	\int_{[N_4^+]}
    	\sum_{n\in N_4(\mathbb{Q})}
    	f(s)(J_8nn'g)
    	\mathbf{e}(-\mathrm{tr}\,
    	u(\beta)^tn')dn'
    	\\
    	&
    	=
    	\int_{N_4^+(\mathbb{A})}
    	f(s)(J_8n'g)
    	\mathbf{e}(-\mathrm{tr}\,
    	u(\beta)^tn')
    	dn',
    	\end{align*}
    	which finishes the proof
    	of the first part.
    	The second part follows easily from the first part.
    \end{proof}

	\subsection{Rallis inner product formula}
	\subsubsection{Siegel-Weil formula}
	Now we can state one of the main ingredients of this article:
	the Siegel-Weil formula.
	We fix a section
	$\phi^+\in\mathcal{S}(W_4^-(\mathbb{A}))$,
	$\Theta_{\phi^+}(g,h)$ the theta series associated to
	$\phi^+$, $f_{\phi^+}
	\in\mathrm{Ind}_{P^+_4(\mathbb{A})}^{G_4(\mathbb{A})}(\xi_{1/2})$
	the Siegel-Weil section associated to $\phi^+$,
	$E(g,f_{\phi^+})$ the Eisenstein series defined by
	$f_{\phi^+}$.
	One can add a complex variable to
	$f_{\phi^+}$ as follows:
	\begin{equation}\label{add parameter s}
	f_{\phi^+}(s)(g)=\omega_{W_4^-}(g)\phi^+(0)
	|\mathbf{m}(g)|^{s-1/2}
	\end{equation}

	Then $f_{\phi^+}=f_{\phi^+}(1/2)$.
	Moreover one can verify that
	$f_{\phi^+}(s)\in
	\mathrm{Ind}_{P^+_4(\mathbb{A})}^{G_4(\mathbb{A})}
	(\xi_s)$.
	In the same manner one can define the 
	Eisenstein series
	$E(g,f_{\phi^+}(s))$
	associated to $f_{\phi^+}(s)$.
	We denote by
	$[\mathrm{GSO}(U)]_g$
	the subset of
	$[\mathrm{GSO}(U)]$
	consisting of elements
	$h$ such that
	$\nu(h)=\nu(g)$.
	One has
	(\cite[Main Theorem]{KudlaRallis88})
	\begin{theorem}\label{Siegel-Weil formula}
		Suppose that
		$\phi^+$ is $K$-finite where
		$K$ is a maximal compact subgroup
		of $G_4(\mathbb{A})$.
		The Eisenstein series $E(g,f_{\phi^+}(s))$
		is holomorphic at $s=1/2$.
		Moreover
		for any $g\in G_4^+(\mathbb{A})$,
		\[
		\int_{[\mathrm{GSO}(U)]_g}
		\Theta_{\phi^+}(g,h)dh
		=
		E(g,f_{\phi^+}).
		\]		
	\end{theorem}
    \begin{proof}
    	In \cite{KudlaRallis88},
    	only the case of isometry groups
    	(i.e.,
    	$g\in G_4^1(\mathbb{A})$) is treated.
    	The holomorphicity of the Eisenstein series at
    	$s=1/2$ follows from the isometry group case.
    	The second part of the theorem also
    	follows from this special case,
    	as follows.
    	For any
    	$g\in G_4^+(\mathbb{A})$,
    	choose any
    	$h_g\in\mathrm{GSO}(U,\mathbb{A})$
    	such that
    	$\nu(h_g)=\nu(g)$.
    	We then define a new section
    	$\phi^+_g\in
    	\mathcal{S}(W_4^-(\mathbb{A}))$
    	by
    	$\phi^+_g(x)
    	:=
    	\omega_{W_4^-}(h_g)\phi^+(x)
    	=
    	\phi^+(h_g^{-1}x)$
    	for
    	$x\in W_4^-(\mathbb{A})$.    	
    	Recall that we write
    	$g_1:=g\,\mathrm{diag}(\nu(g)^{-1},1)$.
    	Then we see that
    	\begin{align*}
    	\int_{[\mathrm{SO}(U)]}
    	\Theta_{\phi^+}(g,hh_g)dh
    	&
    	=
    	\int_{[\mathrm{SO}(U)]}dh
    	\sum_{x\in W_4^-(\mathbb{Q})}
    	\omega_{W_4^-}(g,hh_g)\phi^+(x)
    	\\
    	&
    	=
    	|\nu(g)|^{-6}
    	\int_{[\mathrm{SO}(U)]}dh
    	\sum_{x\in W_4^-(\mathbb{Q})}
    	\omega_{W_4^-}(g_1,h)\phi^+_g(x)
    	\\
    	&
    	=
    	|\nu(g)|^{-6}
    	\int_{[\mathrm{SO}(U)]}
    	\Theta_{\phi^+_g}(g_1,h)dh.
    	\end{align*}
    	
    	This last integral is,
    	by Siegel-Weil formula for isometry groups,
    	equal to
    	$E(g_1,f_{\phi^+_g})$
    	since $g_1\in G_4^1(\mathbb{A})$.
    	Moreover, it is easy to see that
    	$E(g_1,f_{\phi^+_g})
    	=
    	E(g_1,f_{\phi^+})$.
    	Combining with the factor
    	$|\nu(g)|^{-6}$,
    	we see that the Siegel-Weil formula
    	holds for
    	$g\in G_4^+(\mathbb{A})$.
    \end{proof}

	\subsubsection{Doubling method}
	The definite reference for the doubling method is
	\cite{GelbartPSRallis87}.
	For the case of similitude groups, one can consult
	\cite{Li90,Harris93}.
	
	Let $f(s)
	=\otimes_vf(s)_v
	\in\mathrm{Ind}_{P^d(\mathbb{A})}^{G(\mathbb{A})}(s)$
	be a factorizable section and
	$E(g,f(s))$ be the Eisenstein series
	associated to $f(s)$.
	Let $\pi$ be a cuspidal
	automorphic representation of $G(\mathbb{A})$
	and
	$\varphi=\otimes_v\varphi_v$
	a factorizable element in $\pi$.
	Let $\pi^\vee$ be the contragredient of $\pi$ and
	$\varphi^\vee=\otimes_v\varphi^\vee_v$
	a factorizable element in $\pi^\vee$.
	Recall that  $S_\pi$ is a finite set of places 
	of $\mathbb{Q}$
	containing $\infty$ and the places where
	$\pi$ is ramified,
	or
	$f(s)_v$ is not a spherical section,
	or
	$\varphi_v$ or
	$\varphi_v^\vee$ is a not spherical vector.
	
	The doubling method says
	(\cite[Section 3]{Li90},\cite[Section 6.3]{Harris93})
	\begin{theorem}
		We have the following decomposition
		\[
		Z(\varphi,\varphi^\vee,f(s))
		=
		\prod_{v\in S_\pi}
		Z_v(\varphi_v,\varphi^\vee_v,f_v(s))
		\times
		L^{S_\pi}(s+1/2,\mathrm{St}(\pi)\otimes\xi)
		\times
		\langle
		\varphi,\varphi^\vee
		\rangle.
		\]
		Here $L^{S_\pi}(s,\mathrm{St}(\pi)\otimes\xi)$
		is the partial standard $L$-function
		of $\pi$ twisted by the character $\xi$.
	\end{theorem}

	\subsubsection{Rallis inner product formula}
	The combination of the Siegel-Weil formula
	and the doubling method gives the Rallis
	inner product formula,
	which relates
	Petersson inner product of an
	automorphic form to
	special $L$-value of the automorphic form.
	
	Let $\phi_1=\otimes_v\phi_{1,v}\in
	\mathcal{S}(W_1^-(\mathbb{A}))$
	and
	$\phi_2=\otimes_v\phi_{2,v}\in
	\mathcal{S}(W_2^-(\mathbb{A}))$
	be two factorizable sections.
	As above, we write
	$\phi^+=\phi_1\otimes\phi_2
	\in\mathcal{S}(W_4^-(\mathbb{A}))$
	for their tensor product,
	$f_{\phi^+}\in
	\mathrm{Ind}_{P^+(\mathbb{A})}^{G_4(\mathbb{A})}
	(\xi_{1/2})$,
	$\phi^d\in\mathcal{S}(W_{4,d}(\mathbb{A}))$,
	$f_{\phi^d}\in
	\mathrm{Ind}_{P^d(\mathbb{A})}^{G_4(\mathbb{A})}
	(\xi_{1/2})$
	as above.
	Let $\pi$ be an anti-holomorphic
	cuspidal automorphic representation of
	$G_4(\mathbb{A})$
	and
	$\varphi=\otimes_v\varphi_v\in\pi$
	a factorizable element.
	Similarly,
	$\varphi^\vee=\otimes_v\varphi_v^\vee\in\pi^\vee$.
	Then the Rallis inner product formula states
	
	\begin{theorem}
		Define the $\mathbb{C}$-bilinear
		inner product
		\[
		\langle
		\Theta_{\phi_1}(\varphi),
		\Theta_{\phi_2}(\varphi^\vee)
		\rangle
		=
		\int_{[\mathrm{O}(U)]}
		\Theta_{\phi_1}(\varphi)(h)
		\Theta_{\phi_2}(\varphi^\vee)(h)
		dh,
		\]
		then one has
		\[
		\langle
		\Theta_{\phi_1}(\varphi),
		\Theta_{\phi_2}(\varphi^\vee)
		\rangle
		=
		\prod_{v\in S_\pi}
		Z_v(\varphi_v,\varphi^\vee_v,f_v(s))
		\times
		L^{S_\pi}(s+1/2,\mathrm{St}(\pi)\otimes\xi)
		\times
		\langle
		\varphi,\varphi^\vee
		\rangle.
		\]
	\end{theorem}

	\section{Choice of local sections}
	\label{choice of local sections}
	In this section, we will choose the local sections in
	the factorizations
	$\phi_i=\otimes_v'\phi_{i,v}
	\in
	\mathcal{S}(W_i^-(\mathbb{A}))$
	for 
	$i=1,2$.
	Let $\underline{\kappa}
	\in
	\mathrm{Hom}_\mathrm{cont}(
	T^1_2(\mathbb{Z}_p),\overline{\mathbb{Q}}_p^\times)$
	be an arithmetic point, i.e.,
	$\underline{\kappa}$ is a product of an algebraic character
	$\underline{\kappa}_\mathrm{alg}
	=\underline{k}\in\mathbb{Z}^2$
	and a finite order character
	$\underline{\kappa}_\mathrm{f}
	=
	(\kappa_1,\kappa_2)$.
	We say that $\underline{\kappa}$ is admissible
	if $k_1\geq k_2\geq3$.
	Recall that we have four PEL moduli problems
	$P_i$ ($i=1,\cdots,4$)
	and to each problem
	we have defined similitude symplectic groups
	$G=G_1$, $G_2$,
	$G_3$ and $G_4$
	and their isometry subgroups
	$G_i^1$ ($i=1,\cdots,4$).

	\subsection{Archimedean place}
	First consider the archimedean place $v=\infty$.

	\subsubsection{Choice of sections}	
	We write
	$K_{G_i,\infty}$
	the maximal compact subgroup of
	$G_i(\mathbb{R})$.
	Note that for an irreducible
	cuspidal automorphic representation
	$\pi$ of
	$G(\mathbb{A})$
	whose archimedean component
	$\pi_\infty$ is a holomorphic discrete series
	of lowest $K_{G,\infty}$-type of highest
	admissible weight
	$\underline{k}$,
	the standard $L$-function
	$L(s,\mathrm{St}(\pi)\otimes\xi)$ is critical at $s=1$
	since $\xi(-1)=-1$
	(\cite[Appendix]{BoechererSchmidt2000}).

	We define an automorphy factor
	for any
	$g=
	\begin{pmatrix}
	A & B \\
	C & D
	\end{pmatrix}\in G_4(\mathbb{R})$
	as
	follows:
	recall that
	$\mu(g,i)=Ci+D$.
	We set
	$j(g,i)=
	\mathrm{det}(\mu(g,i))$,
	and
	$\widehat{j}(g,i)
	=
	j(g,i)\nu(g)^{-2}$.
	Using this, we then define 
	a function on
	$G_4(\mathbb{R})$ as follows:
	\[
	f_\infty^+(\xi_{1/2})(g)
	:=
	\widehat{j}(g,i)^{-3}.
	\]
	It is easy to show that
	$f^+_\infty(\xi_{1/2})
	\in
	\mathrm{Ind}_{P^+(\mathbb{R})}^{G_4(\mathbb{R})}(\xi_{1/2})$
	and moreover
	$f_\infty^+(\xi_{1/2})(\lambda\cdot1_8)=1$
	for any $\lambda\in\mathbb{R}^\times$.
	We consider a section
	$\phi_\infty^+\in\mathcal{S}(W_4^-(\mathbb{R}))$
	defined as follows:
	for any 
	$(w_1^-,w_2^-)\in 
	W_1^-(\mathbb{R})\times
	W_2^-(\mathbb{R})=
	W_4^-(\mathbb{R})$,
	\[
	\phi_\infty^+((w_1^-,w_2^-))
	=
	\mathrm{det}(\eta_U)^{\mathrm{dim}\,V_4^-}
	\mathbf{e}_\infty
	\Big(
	i\langle (w_1^-,w_2^-),
	\widetilde{J}_8((w_1^-,w_2^-))\rangle_{W_4}/2
	\Big)
	\]
	where
	$\widetilde{J}_8((e_i^\pm,0))=\mp(e_i^\mp,0)$
	and
	$\widetilde{J}_8((0,e_i^\pm))=\mp(0,e_i^\mp)$
	for $i=1,2$.
	Under the basis $\mathcal{B}_{W_4}
	=\mathcal{B}_{V_4}\times\mathcal{B}_{U}$
	of $W_4$,
	$\widetilde{J}_8$ is of the form
	$J_8\otimes1_6$.
	Let
	$f_{\phi^+_\infty}\in
	\mathrm{Ind}_{P^+(\mathbb{R})}^{G_4^1(\mathbb{R})}(\xi_{1/2})$
	be the Siegel section
	associate to $\phi_\infty^+$
	via Weil representation.
	We have the following formula	
	\begin{lemma}\label{calculation of scalar section}
		For any
		$g=\begin{pmatrix}
		A & B \\
		C & D
		\end{pmatrix}\in
		G_4^1(\mathbb{R})$,
		the section $f_{\phi^+_\infty}$ satisfies
		$f_{\phi^+_\infty}(g)
		=\mathrm{det}(Ci+D)^{-3}$.
	\end{lemma}
	
	\begin{proof}
		It suffices to prove the lemma for the open dense subset
		of $G_4^1(\mathbb{R})$
		consisting of matrices
		$g$
		with $\mathrm{det}(C)\neq0$
		(using the Bruhat decomposition for
		$G_4^1(\mathbb{R})$).
		Moreover, each such matrix $g$ 
 		is of the form
 		$m(A')u(B')J_8u(_D')$.
		First, for any 
		$p=m(A')u(B')\in P^+(\mathbb{R})$,
		$g\in G_4^1(\mathbb{R})$,
		by Remark \ref{archimedean Weil representation},
		one has the following
		$f_{\phi^+_\infty}(pg)
		=\omega_{W_4^+}(p)\omega_{W_4^+}(g)\phi_\infty^+(0)
		=\mathrm{det}(A')^3f_{\phi^+_\infty}(g)$.
		Thus it suffices to deal with matrices of the form
		$g=g_0g_2$
		where
		$g_2=u(-D')$.
		Note that
		$\omega_{W_4^+}(g_2)\phi^+_\infty(w)
		=\mathbf{e}_\infty(\langle w,g_2^{-1}w\rangle_W/2)
		\phi_\infty^+(w)$
		for any
		$w\in W_4^-(\mathbb{R})$.
		Thus we get
		(note Remark \ref{Haar measures})
		\begin{align*}
		\omega_{W_4^+}(g)\phi^+_\infty(0)
		&
		=
		\int_{W_4^-(\mathbb{R})}
		\mathbf{e}_\infty(\langle w,g_2^{-1}w\rangle_W/2)
		\phi_\infty^+(w)d\mu_{W_4}
		\\
		&
		=
		\mathrm{det}(\eta_U)^{\mathrm{dim}\,V_4^-}
		\int\mathrm{exp}
		(2i\pi
		(\langle w,g_2^{-1}w\rangle/2+
		i\langle w,g_0w\rangle/2)
		)
		dw.
		\end{align*}
		Now view elements $w$ as matrices
		in $\mathrm{M}_{4\times 6}
		(\mathbb{R})$,
		then
		$\langle w,g_2^{-1}w\rangle+i\langle w,g_0w\rangle
		=\mathrm{tr}(w^t(D'+i)w\eta_U)$.
		Write $w$ in column vectors
		$w=\begin{pmatrix}
		w_1 &
		w_2 &
		w_3 &
		w_4 &
		w_5 &
		w_6
		\end{pmatrix}$, then we have
		$\mathrm{tr}(w^t(i+D')w\eta_U)=
		\sum_{j=1}^6(\eta_U)_{j,j}w_j^t(i+D')w_j$.
		Now a simple calculation of Gaussian type integrals
		gives
		\[
		\omega_{W_4^+}(g)\phi^d_\infty(0)
		=
		\mathrm{det}(\eta_U)^{\mathrm{dim}\,V_4^-}
		\cdot
		\mathrm{det}(\eta_U)^{-\mathrm{dim}\,V_4^-}
		\mathrm{det}(i+D')^{-3}
		=\mathrm{det}(i+D')^{-3}
		\]
		which  concludes the proof.
	\end{proof}

    \begin{corollary}\label{scalar weight archimedean section}
    	We have the identity
    	\[
    	f_{\phi_\infty^+}=f_\infty^+.
    	\]
    \end{corollary}

	We write
	$\mathcal{D}_{\underline{k}}$ for
	the holomorphic discrete series
	$(\mathfrak{g}_{G^1},
	K_{G,\infty})$-module
	whose lowest
	$K_{G,\infty}$-type is of highest weight
	$\underline{k}$.
	We then write
	$\mathcal{D}_{\underline{k}}(\underline{k})$
	to be the lowest $K_{G,\infty}$-type in
	$\mathcal{D}_{\underline{k}}$.
	We denote by
	$\mathcal{D}_{\underline{k}}^\vee$ the 
	contragredient of
	$\mathcal{D}_{\underline{k}}$
	and
	$\mathcal{D}_{\underline{k}}^\vee(-\underline{k})$
	its highest $K_{G,\infty}$-type.
	Recall in Notation\ref{notations}(6),
	we fix a basis
	$\widehat{\mu}_{i,j}^+$
	($1\leq i\leq j\leq4$)
	of $\mathfrak{g}_{G_4^1}^+$.
	We set 
	$\widehat{\mu}_{i,j}^+:=\widehat{\mu}_{i,j}^+$
	for $i>j$
	and denote by
	$(\widehat{\mu}_{i,j}^+)_{i,j}$
	the $4\times4$-matrix
	and in the form of $2\times2$-blocks by
	$\begin{pmatrix}
	\widehat{\mu}_1^+ & \widehat{\mu}_0^+ \\
	(\widehat{\mu}_0^+)^t & \widehat{\mu}_2^+
	\end{pmatrix}$.
	Write then
	$\mathfrak{p}_i^{+,1}$
	for the subalgebra of
	$\mathfrak{g}_{G_4^1}^+$
	generated by the elements in
	$\widehat{\mu}_i^+$
	for $i=0,1,2$.
	Denote by $U(\mathfrak{g}_{G_4^1})
	\cdot f_{\phi_\infty^+}$
	the sub-$(\mathfrak{g}_{G_4^1},
	K_{G_4^1,\infty})$-module
	of the principle discrete series
	$I_{P_{V^d}(\mathbb{R})}
	^{G_4^1(\mathbb{R})}(1/2,1)$
	generated by $f_{\phi_\infty^+}$.
	For an $n\times n$-matrix
	$M$, we write
	$\mathrm{det}_l(M)$
	for the determinant of the
	$l\times l$-minor of $M$
	(upper-left $l\times l$-block).
	Given a dominant weight
	$\underline{k}\in\mathbb{Z}^3$,
	we define
	the following differential operator
	\[
	D^{\underline{k}}
	:=
	\mathrm{det}_1(\frac{1}{4i\pi}
	\widehat{\mu}_0^+)^{k_1-k_2}
	\mathrm{det}_2(\frac{1}{4i\pi}
	\widehat{\mu}_0^+)^{k_2-3}
	\]
	and then we put
	\[
	f_{\underline{k},\infty}^+
	:=D^{\underline{k}}
	f_\infty^+.
	\]

	We then define sections in
	$\mathcal{S}(W_i^-(\mathbb{R}))$ as follows:
	\begin{equation}
	\phi_{1,\infty}(w_1^-)
	=
	\mathbf{e}_\infty
	(
	i
	\langle
	w_1^-,g_0w_1^-
	\rangle
	),
	\quad
	\phi_{2,\infty}(w_2^-)
	=
	\mathbf{e}_\infty
	(
	i
	\langle
	w_2^-,g_0w_2^-
	\rangle
	)
	\end{equation}
	\begin{equation}
	\phi_{\underline{k},\infty}^+(w_4^-)
	:=
	D^{\underline{k}}
	\phi_\infty^+(w_4^-).
	\end{equation}	
	Here the action of the Lie algebra
	$\mathfrak{g}_{G_4^1}$
	on $\mathcal{S}(W_4^-(\mathbb{R}))$
	is induced from the Weil representation of
	$G_4^1(\mathbb{R})$
	on $\mathcal{S}(W_4^-(\mathbb{R}))$.	
	We deduce from Corollary\ref{scalar weight archimedean section}
	that the section
	$\phi^+_{\underline{k},\infty}=
	\mathcal{S}(W_4^-(\mathbb{R}))$
	gives rise to
	$f^+_{\underline{k},\infty}$.

	We next make explicit
	the action of the differential operator
	$D^{\underline{k}}$
	on the space
	$\mathcal{S}(W_4^-(\mathbb{R}))$.
	Let $\widehat{\mu}_{i,j+6}^+$
	be an entry in
	$\widehat{\mu}^+_0$
	($i,j=1,2$).
	Recall that
	$\mu_{i,j}
	=E_{i,j+6}+E_{j+2,i+4}$
	and
	$\widehat{\mu}_{i,j+6}^+
	=\mathfrak{c}
	\mu_{i,j+6}^+\mathfrak{c}^{-1}$.
	We write
	$\mu_{i,j+6}$
	in $4\times4$-blocks
	as
	$\begin{pmatrix}
	0 & B_{i,j} \\
	0 & 0
	\end{pmatrix}$.
	For the variable 
	$w_4^-=(w_{i,j})\in W_4^-(\mathbb{R})$
	(viewed as a $4\times6$-matrix),
	for each
	$f\in\mathcal{S}(W_4^-(\mathbb{R}))$,
	we define
	$\frac{\partial f}{\partial w_4^-}$
	to be the $4\times6$-matrix
	$(\frac{\partial f}{\partial w_{i,j}})$.
	We define a map
	from $\mathcal{S}(W_4^-(\mathbb{R}))$
	to
	$\mathrm{M}_{4\times 6}
	(\mathcal{S}(W_4^-(\mathbb{R})))$
	as
	$\mathbb{D}f(w_4^-)
	=(\sqrt{2\pi}w_4^-\sqrt{\eta}
	-\frac{\partial}{\sqrt{2\pi}\partial w_4^-}\sqrt{\eta}^{-1})f(w_4^-)$
	where recall that
	$\eta=\mathrm{diag}
	(N^2/2,N^2/2,N^2/2,N^2/2,
	N/N_1,N_1)$.
	We define $\mathbb{D}^t$
	to be the transpose of
	$\mathbb{D}$:
	$\mathbb{D}^tf=(\mathbb{D}f)^t$.
	\begin{lemma}
		For any 
		$f\in\mathcal{S}(W_4^-(\mathbb{R}))$
		and $i,j=1,2$,
		we have
		\[
		\widehat{\mu}_{i,j+6}^+f(w_4^-)
		=
		\frac{\sqrt{-1}}{4}
		\mathrm{tr}
		(\mathbb{D}^tB_{i,j}\mathbb{D}
		)f(w_4^-).
		\]
	\end{lemma}
    \begin{proof}
    	By definition,
    	we have
    	\[
    	\widehat{\mu}_{i,j+6}^+
    	=
    	\frac{\sqrt{-1}}{2}
    	\begin{pmatrix}
    	-B_{i,j} & 0 \\
    	0 & B_{i,j}
    	\end{pmatrix}
    	+
    	\frac{1}{2}
    	\begin{pmatrix}
    	0 & B_{i,j} \\
    	0 & 0
    	\end{pmatrix}
    	+
    	\frac{1}{2}
    	\begin{pmatrix}
    	0 & 0 \\
    	B_{i,j} & 0
    	\end{pmatrix}.
    	\]
    	Then it suffices to treat each term in the sum
    	using the expressions of 
    	Weil representation given in
    	(\ref{general expression for Weil representation}).
    \end{proof}
    Write $w_4^-=\begin{pmatrix}
    w_1^- \\
    w_2^-
    \end{pmatrix}$ with
    $w_1^-=
    \begin{pmatrix}
    w_{1;1}^- \\
    w_{1;2}^-
    \end{pmatrix}
    =(w_{1;k,l}^-)\in
    W_1^-(\mathbb{R})$
    and
    $w_2^-
    =
    \begin{pmatrix}
    w_{2;1}^- \\
    w_{2;2}^-
    \end{pmatrix}
    =(w_{2;k,l}^-)\in
    W_2^-(\mathbb{R})$
    for $k=1,2$ and $l=1,\cdots,6$.
    From the above lemma, we get the following
    \begin{corollary}
        For $i,j=1,2$, we have
    	\[
    	\widehat{\mu}_{i,j+6}^+\phi_{\infty}^+(w_4^-)
    	=
    	2i\pi\,\mathrm{tr}
    	((w_4^-)^tB_{i,j}w_4^-\eta)\phi_\infty^+(w_4^-)
    	=
    	4i\pi\,
    	\sum_{l=1}^6
    	\bigg(
    	w_{1;i,l}\sqrt{\eta_{l,l}}\phi_{1,\infty}^+(w_1^-)
    	\bigg)
    	\bigg(
    	w_{2;j,l}\sqrt{\eta_{l,l}}\phi_{2,\infty}^+(w_2^-)
    	\bigg).
    	\]
    \end{corollary}
    
    Define variables $X_{k;i,l}$
    for $k,i=1,2$ and $l=1,\cdots,6$,
    $Y_{i,j}=\sum_{l=1}^6X_{1;i,l}X_{2;j,l}$,
    $Y=(Y_{i,j})_{i,j=1,2}$,
    and
    $D_Y^{\underline{k}}
    =\mathrm{det}_1(Y)^{k_1-k_2}
    \mathrm{det}_2(Y)^{k_2-3}$.
    If we set
    $X_{k;i,l}=w_{k;i,l}\sqrt{\eta_{l,l}}$,
    then we have
	$D^{\underline{k}}\phi_\infty^+(w_4^-)
	=D_Y^{\underline{k}}\phi_\infty^+(w_4^-)$.
	Develop $D_Y^{\underline{k}}$
	as a sum of monomials on $X_{k;i,l}$'s:
	$D_Y^{\underline{k}}=\sum_{I}a_IX^I$
	where
	$I=(I_{i,j})_{i,j=1,2}$ with
	$I_{i,j}=(I_{i,j;1},\cdots,I_{i,j;6})\in(\mathbb{Z}_{\geq0})^6$,
	$X^I=\prod_{l=1}^6\prod_{i,j=1}^2
	(X_{1;i,l}X_{2;j,l})^{I_{i,j;l}}$
	and $a_I\in\mathbb{Z}$.
	We define
	$X_t^I:=\prod_{l=1}^6\prod_{i,j=1}^2X_{t;i,l}^{I_{i,j;l}}$
	for $t=1,2$
	(thus $X^I=X_1^IX_2^I$).
	Write $\mathfrak{I}_{\underline{k}}$
	for the finite set of
	$I$ such that $a_I\neq0$.
	For each $I\in\mathfrak{I}_{\underline{k}}$,
	we define
	\begin{align}
	\label{Schwartz section at infinity for +}
	\phi_{t,\infty;I}(w_t^-)
	&
	=\sqrt{a_I}X_t^I\phi_{t,\infty}(w_t^-)
	\in
	\mathcal{S}(W_t^-(\mathbb{R})),
	\text{ where }\,t=1,2,
	X_{k;i,l}=w_{k;i,l}\sqrt{\eta_{l,l}};
	\\
	\phi_{\infty;I}^+(w_1^-,w_2^-)
	&
	=
	\phi_{1,\infty;I}(w_1^-)
	\phi_{2,\infty;I}(w_2^-)
	\in
	\mathcal{S}(W_4^-(\mathbb{R})).
	\end{align}

	By definition, we have
	$\sum_{I\in\mathfrak{I}_{\underline{k}}}
	\phi_{\infty;I}^+(w_4^-)
	=
	\phi_{\underline{k},\infty}^+(w_4^-)$.
	Now for each $I\in\mathfrak{I}_{\underline{k}}$,
	we let
	\[
	f_{\phi_{\infty;I}^+}
	\in
	\mathrm{Ind}^{G_4(\mathbb{R})}_{P^+(\mathbb{R})}
	(\xi_{1/2})
	\]
	be the section associated to the section
	$\phi_{\infty,I}^+$.
	Thus
	$\sum_{I\in\mathfrak{I}_{\underline{k}}}
	f_{\phi_{\infty,I}^+}=f_{\underline{k},\infty}$.

	\subsubsection{Fourier coefficients}
	We next compute the local Fourier coefficients
	$E_{\beta,\infty}(g_\infty,f_{\phi_{\infty;I}^+})$
	for each $I\in\mathfrak{I}_{\underline{k}}$.
	For any $\mathbf{z}=\mathbf{x}+i\mathbf{y}\in\mathbb{H}_4$, 
	we write
	$g_\mathbf{z}=
	u(\mathbf{x})
	m(\sqrt{\mathbf{y}})
	\cdot
	1_\mathrm{f}
	\in G(\mathbb{A})$.
	Recall we write
	$\beta=\begin{pmatrix}
	\beta_1 & \beta_0 \\
	\beta_0^t & \beta_2
	\end{pmatrix}\in\mathrm{Sym}_{4\times 4}(\mathbb{Q})$.
	Then we have
	(\cite{Shimura1982} or
	\cite[Section 3.4]{Liu2015}):
	\[
	E_{\beta,\infty}(g_\mathbf{z},f^+_\infty)
	=
	\frac{16}{\Gamma_4(3)}
	\pi^{12}
	(\mathrm{det}\, \beta)^{1/2}
	\mathrm{det}(\mathbf{y})^{3/2}
	\mathbf{e}_\infty(\mathrm{tr}\,\beta\mathbf{z})
	\]
	where
	$\Gamma_m(s)=\pi^{m(m-1)/4}\prod_{j=0}^{m-1}\Gamma(s-j/2)$.
	Note that if $\beta$ is not positive definite,
	then
	the local Fourier coefficient 
	$E_{\beta,\infty}(g_\mathbf{z},f^+_\infty)$
	vanishes.

	As for $f^+_{\underline{k},\infty}$, we have
	(\cite[Proposition 4.4.1]{Liu2015})
	\[
	E_{\beta,\infty}(g_\mathbf{z},f^+_{\underline{k},\infty})
	=
	D^{\underline{k}}(2\beta_0)
	E_{\beta,\infty}(g_\mathbf{z},f_\infty^+).
	\]
	where
	$D^{\underline{k}}(2\beta_0)
	=\mathrm{det}_1(2\beta_0)^{k_1-k_2}
	\mathrm{det}_2(2\beta_0)^{k_2-3}$.

	We need the following lemma concerning
	the existence of certain symmetric matrix associated to
	$\beta$ equivalent to $\eta$:
	\begin{lemma}
		Assume that $\beta\in\mathrm{M}_{4\times 4}
		(\mathbb{Q})$
		is non-singular, then there exist
		$a,b\in\mathbb{Q}^\times$
		such that the matrix
		$\beta':=\mathrm{diag}(2\cdot\beta,a,b)\in
		\mathrm{M}_{6\times 6}
		(\mathbb{Q})$
		is equivalent to
		$\eta$
		(i.e., there exists a matrix
		$X\in\mathrm{GL}_6(\mathbb{Q})$
		such that
		$X^t\eta'X=\beta'$)
	\end{lemma}
    \begin{proof}
    	This follows easily from standard results 
    	on the theory of quadratic forms over $\mathbb{Q}$
    	and $\mathbb{Q}_v$.
    	Note that $\beta'$ is equivalent to $\eta$
    	over $\mathbb{Q}$
    	if and only if they are equivalent over each
    	$\mathbb{Q}_v$,
    	if and only if for each place $v$ of $\mathbb{Q}$,
    	their discriminants are equal
    	$\mathrm{det}(\beta')=\mathrm{det}(\eta)\in\mathbb{Q}_v^\times
    	/(\mathbb{Q}_v^\times)^2$
    	and their Hasse invariants are equal
    	$\mathcal{H}_v(\beta')=\mathcal{H}_v(\eta)$
    	(\cite[Chapter IV,Theorems 7 and 9]{Serre1973}).
        The first condition gives
        $ab=\mathrm{det}(\eta)\mathrm{det}(2\beta)
        \in\mathbb{Q}_v^\times/(\mathbb{Q}_v^\times)^2$,
    	thus
    	$\mathcal{H}_v(\beta')=\mathcal{H}_v(2\beta)
    	(\mathrm{det}(2\beta),-\mathrm{det}(\eta))_v(a,
    	-\mathrm{det}(\eta)\mathrm{det}(2\beta))_v$
    	(here $(\cdot,\cdot)_v$ is the Hilbert symbol over 
    	$\mathbb{Q}_v^\times$).
    	So it suffices to find an element
    	$a\in\mathbb{Q}^\times$
    	such that this last identity holds for all
    	$v$.
    	This follows from
    	\cite[Chapter III, Theorem 4]{Serre1973}
    	and thus we conclude the proof.
    \end{proof}
    In the following, when $\beta$ is non-singular, we
    fix one such choice of
    $a,b\in\mathbb{Q}^\times$
    and $X\in\mathrm{GL}_6(\mathbb{Q})$
    as in the lemma 
    such that $\mathrm{det}(X)=1$
    and denote them by
    $a_\beta,b_\beta$ and $X_\beta$.

	Now we turn to
	$f_{\infty;I}^+$.
	For each $\beta\in\mathrm{Sym}_{4\times 4}(\mathbb{R})$
	definite positive,
	we write
	$d\mu_\beta$
	for the induced measure on the
	compact closed subset
	$B_\beta$
	of $W_4^-(\mathbb{R})$
	consisting of elements
	$w_4^-$
	such that
	$w_4^-\eta(w_4^-)^t=\beta$.
	We then denote by
	$\mathrm{vol}(B_\beta)$
	the volume of $B_\beta$
	under $d\mu_\beta$.
	Note that
	$\mathrm{vol}(B_\beta)
	=(\mathrm{det}\,\beta)^{1/2}
	\mathrm{vol}(B_{1_4})$.
	For $\mathbf{z}=\mathbf{x}+i\mathbf{y}
	\in\mathbb{H}_4$,
	we write
	$\beta_\mathbf{y}=
	\sqrt{\mathbf{y}}\beta\sqrt{\mathbf{y}}$.
	We introduce some notations
	for the next lemma.
	To a partition
	$\nu=(\nu_1,\nu_2,\cdots,\nu_l)$
	of a positive integer $n$
	($\nu_1\geq \nu_2\geq\cdots\geq \nu_l\geq1$)
	of length $l(\nu):=l$,
	we associate an integer
	$z_\nu:=\prod_{i=1}^l\prod_{j=1}^{\nu_i}(5+2j-i)$.
	Moreover we write
	$h(\nu)$
	for the dimension 
	of the irreducible representation 
	over $\mathbb{C}$ of the
	permutation group
	$S_{2n}$ of $2n$ elements
	associated to the partition
	$\nu$.
	Then we set
	$d(n)=\gcd_\nu(\frac{(2n)!}{h(\nu)}z_\nu)$
	where $\nu$ runs through all the partitions of
	$n$.

	Then we have
	\begin{lemma}\label{local Fourier coefficient at infinity}
		For any $I\in\mathfrak{I}_{\underline{k}}$,
		the local Fourier coefficient
		\[
		E_{\beta,\infty}(g_\mathbf{z},f_{\phi_{\infty,I}^+})
		\in
		d(k_1+k_2-6)^{-1}
		a_I\mathbf{e}_\infty(\mathrm{tr}\,\beta\mathbf{z})
		(\mathrm{det}\,\mathbf{y})^{3/2}
		(\mathrm{det}2\beta\,
		\mathrm{det}\eta^{-1})^{1/2}
		\mathbb{Z}[\sqrt{\beta_\mathbf{y}},\sqrt{\eta}^{-1}].
		\]
		Here
		$\mathbb{Z}[\sqrt{\beta_\mathbf{y}}]$
		is the polynomial ring
		$\mathbb{Z}[Z]$
		on the matrix
		$Z=(Z_{i,j})_{i,j=1,\cdots,4}$
		by identifying
		$Z$ with $\sqrt{\beta_\mathbf{y}}$.
		Similarly
		$\mathbb{Z}[\sqrt{\eta}^{-1}]$
		is the polynomial ring
		$\mathbb{Z}[Z']$
		on the matrix
		$Z'=(Z'_{i,j})_{i,j=1,\cdots,6}$
		by identifying $Z'$ with
		$\sqrt{\eta}^{-1}$.
	\end{lemma}
    \begin{proof}
    	By definition, we have
    	\begin{align*}
    	E_{\beta,\infty}(g_\mathbf{z},f_{\phi_{\infty;I}^+})
    	&
    	=
    	\mathbf{e}_\infty(\mathrm{tr}\,\beta\mathbf{x})
    	(\mathrm{det}\,\mathbf{y})^{3/2}
    	\int_{W_4^-(\mathbb{R})}
    	\phi_{\infty;I}^+(\sqrt{\mathbf{y}}^tw_4^-)dw_4^-
    	\int_{\mathrm{Sym}_{4\times 4}(\mathbb{R})}
    	\mathbf{e}(\mathrm{tr}(w_4^-\eta(w_4^-)^tx/2-\beta x))dx
    	\\
    	&
    	=
    	\mathbf{e}_\infty(\mathrm{tr}\,\beta\mathbf{x})
    	(\mathrm{det}\,\mathbf{y})^{3/2}
    	\int_{w_4^-\eta(w_4^-)^t=2\beta}
    	\phi_{\infty;I}^+(\sqrt{\mathbf{y}}w_4^-)
    	d\mu_\beta(w_4^-)
    	\\
    	&
    	=a_I
    	\mathbf{e}_\infty(\mathrm{tr}\,\beta\mathbf{x})
    	(\mathrm{det}\,\mathbf{y})^{3/2}
    	\int_{w_4^-\eta(w_4^-)^t=2\beta}
    	X^I(\sqrt{\mathbf{y}}w_4^-)
    	\phi_\infty^+(\sqrt{\mathbf{y}}w_4^-)
    	d\mu_\beta(w_4^-)
    	\\
    	&
    	=
    	a_I\mathbf{e}_\infty(\mathrm{tr}\,\beta\mathbf{z})\,
    	\mathrm{det}(\mathbf{y})^{3/2}
    	(\mathrm{det}(2\beta)
    	\mathrm{det}(\eta)^{-1})^{1/2}
    	\int_{w_4^-(w_4^-)^t=1}
    	X^I(\sqrt{\beta_\mathbf{y}}w_4^-\sqrt{\eta}^{-1})
    	d\mu_1(w_4^-)
    	\end{align*}
    	Now consider the integral:
    	$\mathrm{vol}^I(Z)
    	=\int_{B_1}X^I(Zw_4^-)
    	d\mu_1(w_4^-)$.
    	We set $N=k_1+k_2-6$.
    	It suffices to show that
    	$d(N)\mathrm{vol}^I(Z)\in
    	\mathbb{Z}[Z,Z']$.
    	We define a map
    	$\psi\colon
    	\mathrm{O}_6(\mathbb{R})
    	\rightarrow
    	B_1$
    	by writing
    	$M=\begin{pmatrix}
    	M' \\
    	M''
    	\end{pmatrix}\in
    	\mathrm{O}_6(\mathbb{R})$
    	with
    	$M'\in\mathrm{M}_{4\times 6}
    	(\mathbb{R})$
    	and setting
    	$\psi(M)=M'$.
    	By the QR decomposition for rectangular matrices,
    	we see that $\psi$ is surjective.
    	Consider the measure $\mu_{\mathrm{O}_6}$,
    	resp. $\mu_{B_1}$, on
    	$\mathrm{O}_6(\mathbb{R})$,
    	resp., $B_1$,
    	induced from the Euclidean space
    	$\mathrm{M}_{6\times 6}
    	(\mathbb{R})$,
    	resp., $\mathrm{M}_{4\times 6}
    	(\mathbb{R})$
    	(inner products both given by
    	$\langle S_1,S_2\rangle:=\mathrm{tr}(S_1S_2^t)$).
    	We define the same map
    	$\psi\colon
    	\mathrm{M}_{6\times 6}
    	(\mathbb{R})
    	\rightarrow
    	\mathrm{M}_{4\times 6}
    	(\mathbb{R}),
    	w\mapsto w_4^-$.
    	Then we see that
    	$\mu_{B_1}$ is the induced measure from
    	$\mu_{\mathrm{O}_6}$
    	by $\psi$.
    	Thus we can write
    	$\mathrm{vol}^I(Z)
    	=\int_{\mathrm{O}_6(\mathbb{R})}X^I(Z\,\psi(wZ'))
    	d\mu_{\mathrm{O}_6}(w)$.
    	Develop $X^I(Z\psi(wZ'))$
    	as polynomials on 
    	$Z,Z'$:
    	$X^I(Z\psi(wZ'))=\sum_{S,T}Z^S(Z')^Tb_{S,T}^I(w)$
    	where
    	$S\in\mathrm{M}_{4\times 4}
    	(\mathbb{Z}_{\geq0})$
    	and
    	$T\in\mathrm{M}_{6\times 6}
    	(\mathbb{Z}_{\geq0})$.
    	We next show that
    	$\mathrm{vol}(b_{S,T}^I)
    	=\int_{\mathrm{O}_6(\mathbb{R})}b_{S,T}^I(w)
    	d\mu_{\mathrm{O}_6}(w)$ 
    	is an integer multiple of
    	$\frac{1}{d(N)}\int_{\mathrm{O}_6(\mathbb{R})}
    	d\mu_{\mathrm{O}_6}(w)$ for all $(S,T)$.
    	Note that each $b_{S,T}^I(w)$ lies in
    	$\mathbb{Z}[w]$ of total degree $N$.
    	So it suffices to show that for each monomial
    	$f(w)\in\mathbb{Z}[w]$ of total degree $N$
    	the integral
    	$\mathrm{vol}(f)=\int_{\mathrm{O}_6(\mathbb{R})}f(w)
    	d_{\mathrm{O}_6}(w)$ 
    	is a certain integer multiple of
    	$d(N)^{-1}\mathrm{vol}(1)$.
    	For this we may apply Weingarten formula
    	given in \cite{CollinsMatsumoto2009}.
    	We write $\mathcal{M}(2N)$
    	for the set of pair partitions of the set
    	$\{1,2,\cdots,2N\}$.
    	For a partition $\nu=(\nu_1,\cdots,\nu_l)$ of $N$,
    	let
    	$2\nu=(2\nu_1,\cdots,2\nu_l)$
    	be the partition of $2N$.
    	The Weingarten function
    	$\mathrm{Wg}^{\mathrm{O}_6}
    	(\mathfrak{m},\mathfrak{n})$
    	on $\mathcal{M}(2N)\times
    	\mathcal{M}(2N)$
    	is given by
    	$\frac{2^NN!}{(2N)!}
    	\sum_\nu h(2\nu)\frac{\omega^\nu(\mathfrak{m}^{-1}
    		\mathfrak{n})}{z_\nu}$
    	where
    	$\nu$ runs through all the partitions of
    	$N$ such that
    	$z_\nu\neq0$
    	(\cite[(3.7)]{CollinsMatsumoto2009})
    	and
    	$\omega^\nu(\mathfrak{m}^{-1}
    	\mathfrak{n})$
    	is a certain integer multiple
    	of $(2^NN!)^{-1}$
    	(\cite[(3.4)]{CollinsMatsumoto2009})
    	since the irreducible characters of the permutation groups
    	always take integer values by the Frobenius formula.
    	By \cite[Theorem 2.1]{CollinsMatsumoto2009},
    	$\mathrm{vol}(f)=
    	\mathrm{vol}(1)
    	\sum_{\mathfrak{m},
    	\mathfrak{n}\in\mathcal{M}(2N)}
        \mathrm{Wg}^{\mathrm{O}_6}
        (\mathfrak{m},\mathfrak{n})\delta(\mathfrak{m},
        \mathfrak{n},f)$
        where $\delta(\mathfrak{m},
        \mathfrak{n},f)$ is 
        certain function which is equal to $0$ or $1$
        (note that our measure
        $d\mu_{\mathrm{O}_6}$
        is a Haar measure,
        not necessarily normalized).
        Thus we conclude that
        $\mathrm{vol}(f)$
        is a certain integer multiple of
        $d(N)^{-1}\mathrm{vol}(1)$.
    \end{proof}

    \begin{corollary}
    	Supposing $\beta$ non-singular,
    	then we have
    	\[
    	E_{\beta,\infty}(1,
    	f_{\phi_{\infty,I}^+})
    	\in
    	d(k_1+k_2-6)^{-1}
    	a_I
    	\mathbf{e}_\infty(\mathrm{tr}\,i\beta)
    	\mathrm{det}\eta^{-1})^{1/2}
    	\mathbb{Z}[\frac{1}{\sqrt{N}}][X_\beta].
    	\]
    \end{corollary}

	\subsubsection{Local zeta integrals}
	We show in this subsection the non-vanishing of
	certain local zeta integrals.
	We know by \cite[Proposition 4.3.1]{Liu2015}:	
	\begin{proposition}
		\label{non-vanishing of zeta integral at infinity of Liu}
		The local zeta integral
		\[
		Z_\infty(\cdot,\cdot,f^+_{\underline{k},\infty})
		\colon
		\mathcal{D}_{\underline{k}}(\underline{k})
		\times
		\mathcal{D}_{\underline{k}}^\vee(-\underline{k})
		\rightarrow
		\mathbb{C}
		\]
		is a non-zero pairing.
		More precisely,
		let
		$\varphi_{\underline{k},\infty}$
		be a non-zero vector
		of highest weight in $\mathcal{D}_{\underline{k}}(\underline{k})$
		and
		$\varphi_{\underline{k},\infty}^\vee$ 
		its dual vector in
		$\mathcal{D}^\vee_{\underline{k}}(-\underline{k})$,
		then
		\[
		Z_\infty(\varphi_{\underline{k},\infty},
		\varphi_{\underline{k},\infty}^\vee,
		f^+_{\underline{k},\infty})\neq0.
		\]
	\end{proposition}
    The strategy of the proof is to show that
    (1) $f_{\underline{k},\infty}^+|_{U_{G_1,\infty}\times
    	U_{G_2,\infty}}$ has a non-zero projection to the
    space
    $\sigma_{\underline{k}}
    =\mathcal{D}_{\underline{k}}(\underline{k})
    \times
    \mathcal{D}_{\underline{k}}(\underline{k})$;
    (2) the pairing $Z_\infty$
    of
    $\mathcal{D}_{\underline{k}}(\underline{k})
    \times
    \mathcal{D}_{\underline{k}}^\vee(-\underline{k})$
    and
    $\sigma_{\underline{k}}$
    is non-zero.
    In fact, in \cite{Liu2015},
    the combination of Theorem 4.3.2
    and Lemma 4.3.4
    shows that $\sigma_{\underline{k}}$
    is in fact a direct summand of
    $R_{0,6}|_{\mathfrak{g}_{G_1}\times\mathfrak{g}_{G_2}}$
    and
    $f_{\underline{k},\infty}^+$
    actually lies in
    $\sigma_{\underline{k}}$.
    Here
    $R_{0,2k}$ is as in \cite[p.25]{Liu2015},
    which is the sub-$G_4^1(\mathbb{R})$-representation
    of the degenerate principal series
    $\mathrm{Ind}_{(P^+)^1(\mathbb{R})}^{G_4^1
    	(\mathbb{R})}(\xi_{1/2})$
    whose
    $\mathfrak{g}_{G_4^1,\mathbb{C}}$-finite
    part is irreducible and
    contains the
    $K_{G_4,\infty}$-type of scalar weight $k$
    (here we use the case $k=3$).

    \begin{corollary}
    	\label{non-vanishing of zeta integral at infinity}
    	For each $I\in\mathfrak{I}_{\underline{k}}$,
    	the local zeta integral
    	\[
    	Z_\infty(\cdot,\cdot,f_{\phi_{\infty,I}^+})
    	\colon
    	\mathcal{D}_{\underline{k}}(\underline{k})
    	\times
    	\mathcal{D}_{\underline{k}}^\vee(-\underline{k})
    	\rightarrow
    	\mathbb{C}
    	\]
    	is a non-zero pairing.
    \end{corollary}
    \begin{proof}
    	It suffices to show that,
    	according to the decomposition
    	$R_{0,6}|_{\mathfrak{g}_{G_1^1}\times
    	\mathfrak{g}_{G_2^1}}
        \simeq
        \oplus_{a_1\geq a_2\geq3}
        \mathcal{D}_{\underline{a}}\boxtimes
        \mathcal{D}_{\underline{a}}$
        (\cite[(4.3.3)]{Liu2015}),
        the vector
        $f_{\phi_{\infty;I}^+}$
        has non-zero component inside
        $\mathcal{D}_{\underline{k}}(\underline{k})
        \boxtimes
        \mathcal{D}_{\underline{k}}(\underline{k})$.
        We identify the universal
        holomorphic discrete series
        $M_3:=U(\mathfrak{g}_{G_4^1,\mathbb{C}})
        \otimes_{U(\mathfrak{t}_{G_4^1,\mathbb{C}}
        	\oplus\mathfrak{g}_{G_4^1,
        		\mathbb{C}}^-)}\mathrm{det}^3$
        with
        $R_{0,6}$.
        On the other hand,
        $M_3\simeq U(\mathfrak{g}_{G_4^1,\mathbb{C}}^+)
        \otimes_\mathbb{C}\mathrm{det}^3$
        and when restricted to
        $\mathfrak{g}_{G_1^1,\mathbb{C}}\times
        \mathfrak{g}_{G_2^1,\mathbb{C}}$,
        $U(\mathfrak{g}_{G_4^1,\mathbb{C}}^+)$ 
        is isomorphic to
        $U(\mathfrak{g}_{G_1^1,\mathbb{C}}
        \times
        \mathfrak{g}_{G_2^1,\mathbb{C}})
        \otimes_\mathbb{C}
        U(\mathfrak{p}_0^+)$.
        We then identify
        $U(\mathfrak{p}_0^{+,1})$
        with the polynomial ring
        $\mathbb{C}[Y]$
        with $Y=(Y_{i,j})_{i,j=1,2}$
        defined as above.
        Write $\mathbb{C}[Y]_r$
        for the subspace of
        $\mathbb{C}[Y]$ consisting of polynomials of total degree
        less than or equal to $r$.
        Then the action of
        $G_1^1(\mathbb{R})\times G_2^1(\mathbb{R})$
        (as well as its Lie algebra)
        on $U(\mathfrak{p}_0^{+,1})$
        translates to an action on
        $\mathbb{C}[Y]_r$ and
        $\mathbb{C}[Y]$ as
        $((g_1,g_2)f)(Y):=f(g_1^tYg_2)$.
        We can define a Hermitian product
        $\langle\cdot,\cdot\rangle_r$
        on $\mathbb{C}[Y]_r$
        as follows:
        for any monomials
        $f(Y)=\prod_{i,j}Y_{i,j}^{F_{i,j}}$
        and $g(Y)=\prod_{k,l}Y_{k,l}^{G_{k,l}}$
        in $\mathbb{C}[Y]_r$,
        we define
        \[
        \langle f,g\rangle_r
        =\frac{1}{r!}
        \prod_{i,j}
        \delta_{F_{i,j},G_{3-i,3-j}}\cdot F_{i,j}!.
        \]
        Then one extends
        $\langle\cdot,\cdot\rangle_r$
        to the whole
        $\mathbb{C}[Y]_r$
        by sesquilinearity.
        One can verify that
        for any $g_1\in G_1^1(\mathbb{R})$,
        $g_2\in G_2^1(\mathbb{R})$,
        we have
        $\langle (g_1,g_2)f,(g_1,g_2)g\rangle_r
        =
        \mathrm{det}(g_1)\mathrm{det}(g_2)
        \langle f,g\rangle_r$.
        We then extend these products
        $\langle\cdot,\cdot\rangle_r$
        to the whole $\mathbb{C}[Y]$
        (denoted by
        $\langle\cdot,\cdot\rangle$)
        by orthogonality and sesquilinearity.
        Therefore, the decomposition of
        $R_{0,6}|_{\mathfrak{g}_{G_1^1}\times
        \mathfrak{g}_{G_2^1}}$
        is an orthogonal  decomposition with respect to this
        Hermitian product.
        To show the non-vanishing of the zeta integral,
        it suffices to show that
        $\langle f_{\phi_{\infty;I}^+},
        f_{\underline{k},\infty}^+
        \rangle\neq0$.

        For this we need a refined version of the
        inner product defined above.
        Recall that
        $Y_{i,j}=\sum_{l=1}^6X_{1;i,l}X_{2;j,l}$.
        We then write
        $\mathbb{C}[\{X_{t;i,l}\}]$
        for the polynomial ring on the variables
        $X_{t;i,l}$.
        We define an action of the group
        $G_1^1(\mathbb{R})\times
        G_2^1(\mathbb{R})$
        on this polynomial ring as in the above case:
        write the matrix of variables
        $X^{(l)}=
        \begin{pmatrix}
        X_{1;1,l}X_{2;1,l} & X_{1;1,l}X_{2;2,l} \\
        X_{1;2,l}X_{2;1,l} & X_{1;2,l}X_{2;2,l}
        \end{pmatrix}$
        and we define the action as
        $((g_1,g_2)f)(X^{(1)},\cdots,X^{(6)})
        =f(g_1^tX^{(1)}g_2,\cdots,g_1^tX^{(6)}g_2)$.
        Similarly, one defines an inner product
        on each
        $\mathbb{C}[X^{(l)}]$
        exactly as in the case of
        $\mathbb{C}[Y]$.
        Then we extend these inner products on each
        $\mathbb{C}[X^{(l)}]$
        to their tensor product
        $\mathbb{C}[\{X_{t;i,l}\}]$
        in the obvious way.
        Now the natural embedding
        of inner product spaces 
        $\mathbb{C}[Y]
        \rightarrow
        \mathbb{C}[\{X_{t;i,l}\}]
        \simeq\otimes_{l=1}^6
        \mathbb{C}[X^{(l)}]$
        sending $Y_{i,j}$
        to 
        $\sum_{l=1}^6X_{1;i,l}X_{2;j,l}$
        is equivariant for the actions of
        $G_1^1(\mathbb{R})
        \times
        G_2^1(\mathbb{R})$.
        We then identify
        $f_{\phi_{\infty;I}^+}$
        with the polynomial
        $X^I=\prod_{l=1}^6\prod_{i,j=1}^2
        (X_{1;i,l}X_{2;j,l})^{I_{i,j;l}}$
        in
        $\mathbb{C}[\{X_{t;i,l}\}]$
        (and identifying
        $f_{\underline{k},\infty}^+$
        with the sum of all the above polynomials
        with scalar coefficient
        $a_I$
        for $I$ running through
        $\mathfrak{I}_{\underline{k}}$),
        thus a simple calculation shows that
        $\langle f_{\phi_{\infty;I}^+},
        f_{\underline{k},\infty}^+\rangle
        \neq0$.        
    \end{proof}

    From the proof of the above corollary,
    we see that
    for each $I\in\mathfrak{I}_{\underline{k}}$,
    $f_{\phi_{\infty;I}^+}$
    lies in some discrete series
    $\oplus_{\underline{k}'\in
    \mathrm{wt}_I}
    \mathcal{D}_{\underline{k}'}$
    where
    $\mathrm{wt}_I$
    is a finite subset of
    dominant weights of
    $T_2$ depending on
    $I$.

	\subsection{Unramified places}
	Now we consider a finite place $\ell\nmid pN$.
	
	We define
	the following sections
	\begin{equation}
	\label{Schwartz section at unramified finite places}
	\begin{array}{c}
	\phi_{i,\ell}(y_i^-)
	=1_{W_1^-(\mathbb{Z}_\ell)}(y_i^-)
	\\
	\phi^+_\ell
	=
	\phi_{1,\ell}\otimes\phi_{2,\ell}
	=
	1_{W_4^-(\mathbb{Z}_\ell)}
	\in
	\mathcal{S}(W_4^-(\mathbb{Q}_\ell)).
	\end{array}
	\end{equation}

	This gives rise to a 
	Siegel section
	$f_{\phi^+_\ell}\in
	\mathrm{Ind}_{P^{+,1}
		(\mathbb{Q}_\ell)}^{G_4^1(\mathbb{Q}_\ell)}
	(\xi_{1/2})$
	(note that the $\ell$-th component of the 
	Hecke character $\xi$ is in fact trivial).
	It is easy to see that
	$f_{\phi^+_\ell}(G_4^1(\mathbb{Z}_\ell))=\{
	1 \}$.
	We can extend $f_{\phi^+_\ell}$
	to a section (still denoted by) $f_{\phi_\ell^+}$ in
	$\mathrm{Ind}_{P^+(\mathbb{Q}_\ell)}^{G_4(\mathbb{Q}_\ell)}(\xi_{1/2})$
	by setting
	$f_{\phi^+_\ell}(\mathrm{diag}(1_4,\nu1_4))
	=
	\xi(\nu)
	|\nu|^{-1/2}$.
	
	Now the local zeta integral
	\begin{theorem}
		We have
		\begin{equation}\label{zeta integral unramified place}
		Z_\ell(\varphi_\ell^\mathrm{ur},(\varphi_\ell^\mathrm{ur})^\vee,
		f_{\phi^+_\ell})
		=
		\langle
		\varphi_\ell^\mathrm{ur},
		(\varphi_\ell^\mathrm{ur})^\vee
		\rangle_\ell
		L_\ell(1,\mathrm{St}(\pi)\otimes\xi).
		\end{equation}
	\end{theorem}

	The local Fourier coefficient is given as
	(\cite[Theorem 13.6, Proposition 14.9]{Shimura1997})
	\begin{theorem}
		We have
		\begin{equation}\label{Fourier coefficient unramified place}
		E_{\beta,\ell}(1_8,f_{\phi^+_\ell})
		=
		d_\ell(s,\xi)
		L_\ell(1,\xi\lambda_\beta)
		g_{\beta,\ell}(\xi(\ell)\ell^{-3})
		\end{equation}
		where
		$\lambda_\beta(\ell)
		=
		(\frac{\mathrm{det}(2\beta)}{\ell})$
		is the quadratic symbol
		and
		$g_{\beta,\ell}(t)$
		is element in $\mathbb{Z}[t]$
		whose constant term is $1$ and 
		is of degree $\leq 8\mathrm{val}_\ell(\mathrm{det}(2\beta))$.
	\end{theorem}

	\subsection{Generalized Gauss sums}
	We will need some results on Gauss sums
	for the following subsection on the choice of
	Schwartz-Bruhat functions over places
	dividing $Np$.
	For any integers
	$a,b,c\in\mathbb{Z}$
	with $c>0$,
	we define the generalized Gauss sum
	$G(a,b,c)$ as
	\[
	G(a,b,c)
	=
	\sum_{j=0}^{c-1}\mathrm{exp}(-2i\pi\frac{aj^2+bj}{c}).
	\]
	The generalized Gauss sum has the following properties
	\begin{enumerate}
		\item 
		if $(a,c)|b$, then
		\[
		G(a,b,c)
		=
		(a,c)G(\frac{a}{(a,c)},\frac{b}{(a,c)},\frac{c}{(a,c)});
		\]
		otherwise,
		$G(a,b,c)=0$;
		\item 
		We define a function
		$\varepsilon\colon
		1+2\mathbb{Z}
		\rightarrow
		\{1,i\}$
		with
		$\varepsilon(m)=1$ if $m\equiv1(\mathrm{mod}\,  4)$
		and $=i$ if $m\equiv3(\mathrm{mod}\,  4)$.
		Then assuming $(a,c)=1$
		and $2\nmid ac$,
		we have
		\[
		G(a,b,c)
		=
		\varepsilon(c)\sqrt{c}(\frac{-a}{c})
		\mathrm{exp}(2i\pi\frac{a_c^{-1}b^2}{c})
		\neq0
		\]
		where $(\frac{a}{c})$ is the Jacobi symbol
		and $a_c^{-1}\in\mathbb{Z}$ is any integer such that
		$4a_c^{-1}a\equiv1(\mathrm{mod}\, c)$.
		In particular, if
		$c=1$,
		$G(a,b,c)=1$.
	\end{enumerate}
	
	Using the generalized Gauss sums, 
	we can evaluate some integrals.
	Let $v=v_\ell$ be
	the $\ell$-adic valuation.
	For any $S\subset\mathbb{Q}_\ell$ open compact subgroup,
	we write
	$v(S)$ to be the integer such that
	$\ell^{v(S)}\mathbb{Z}_\ell=S$.
	For any $a,b,r\in\mathbb{Q}_\ell$,
	we set
	$a=a'\ell^{v(a)}$,
	$b=b'\ell^{v(b)}$
	and
	$r=r'\ell^{v(r)}$.
	Then we consider the following integral
	\[
	\int_{r+S}\mathbf{e}_\ell(ax^2+bx)dx
	=
	\int_{r'\ell^{v(r)}+\ell^{v(S)}\mathbb{Z}_\ell}
	\mathbf{e}_\ell(a'\ell^{v(a)}x^2+b'\ell^{v(b)}x)dx
	\]
	To relate the above integral to
	the generalized Gauss sum, we choose an integer
	$M\gg0$ such that
	$M-v(S)>0$.
	Then the above integral becomes
	\begin{align*}
	\int_{r+S}\mathbf{e}_\ell(ax^2+bx)dx
	&
	=
	\ell^{-M}
	\sum_{j=0}^{\ell^{M-v(S)}-1}
	\mathbf{e}_\ell(a(r+\ell^{v(S)}j)^2+b(r+\ell^{v(S)}j))
	\\
	&
	=
	\ell^{-M}
	\sum_{j=0}^{\ell^{M-v(S)}-1}
	\exp
	(
	-2i\pi
	(
	a\ell^{2v(S)}j^2+(2ar\ell^{v(S)}+b\ell^{v(S)})j+ar^2+br
	)
	)
	\\
	&
	=
	\ell^{-M}
	\exp(-2i\pi(ar^2+br))
	\sum_{j=0}^{\ell^{M-v(S)}-1}
	\exp
	(
	-2i\pi
	\frac{\ell^{M-v(S)}(a\ell^{2v(S)}j^2+(2ar\ell^{v(S)}+b\ell^{v(S)})j)}
	{\ell^{M-v(S)}}
	)
	\\
	&
	=
	\ell^{-M}
	\mathbf{e}_\ell
	(ar^2+br)
	G(a\ell^{M+v(S)},(2ar+b)\ell^M,\ell^{M-v(S)}).
	\end{align*}
	So we see that
	\begin{lemma}\label{vanishing of Gauss integrals}
		With the above notations,
		one has
		\begin{equation}\label{generalized Gauss sums}
		\min(v(aS^2),0)>v((2ar+b)S)
		\Leftrightarrow
		\int_{r+S}\mathbf{e}_\ell(ax^2+bx)dx
		=0.
		\end{equation}
		Here we understand $v(aS^2)$ as
		$v(a)+2v(S)$.
		In particular, if
		$r=0$ and $v(aS^2)\leq0,v(bS)$,
		then
		$\int_S\mathbf{e}_\ell(ax^2+bx)dx
		=\ell^{v(a)/2}\epsilon(\ell^{-v(a)})
		(\frac{-a\ell^{-v(a)}}{\ell})^{-v(a)}
		\mathbf{e}_\ell(-\frac{b^2}{4a})$.
	\end{lemma}

	We can generalize the above result to higher dimensions.
	Let $a\in\mathrm{Sym}_{n\times n}(\mathbb{Q}_\ell)$ 
	be a symmetric
	matrix (not necessarily non-singular),
	$b\in\mathrm{M}_{n\times 1}
	(\mathbb{Q}_\ell)$
	a column matrix.
	For any element $r\in\mathrm{M}_{n\times 1}
	(\mathbb{Q}_\ell)$
	and any open compact subgroup
	$S\subset\mathrm{M}_{n\times 1}
	(\mathbb{Q}_\ell)$
	of the form $\ell^{v(S)}\mathrm{M}_{n\times 1}
	(\mathbb{Z}_\ell)$
	for some integer
	$v(S)\in\mathbb{Z}$
	(the standard open compact subgroups of
	$\mathrm{M}_{n\times 1}
	(\mathbb{Q}_\ell)$),
	we want to evaluate the following integral,
	which is analogous to the above one
	\[
	\int_{r+S}\mathbf{e}_\ell(x^tax+b^tx)dx.
	\]
	
	First note that $a$ is $\mathbb{Z}_\ell$-equivalent to some diagonal
	matrix, i.e., there exists some invertible matrix
	$D\in\mathrm{GL}_n(\mathbb{Z}_\ell)$
	and a diagonal matrix
	$\Lambda=\mathrm{diag}(\lambda_1,\cdots,\lambda_n)
	\in\mathrm{Sym}_{n\times n}(\mathbb{Q}_\ell)$
	such that
	$a=D^t\Lambda D$
	(\cite[Chapter 8]{Cassels82}).
	So if we write
	\[
	r'
	=
	Dr
	=
	\begin{pmatrix}
	r_1' \\
	r_2' \\
	\vdots \\
	r_n'
	\end{pmatrix},
	\quad
	S'
	=
	DS
	=
	\begin{pmatrix}
	S_1' \\
	S_2' \\
	\vdots \\
	S_n'
	\end{pmatrix},
	\quad
	b'
	=
	Db
	=
	\begin{pmatrix}
	b_1' \\
	b_2' \\
	\vdots \\
	b_n'
	\end{pmatrix},
	\]
	then the integral becomes
	\[
	\int_{r+S}\mathbf{e}_\ell(x^tax+b^tx)dx
	=
	|\mathrm{det}\, D|_\ell
	\int_{r'+S'}\mathbf{e}_\ell(x'^t\Lambda x'+b'^tx')dx'
	=
	|\mathrm{det}\, D|_\ell
	\prod_{k=1}^n
	\int_{r'_k+S'_k}\mathbf{e}_\ell(\lambda_kx^2+b'_kx)dx.
	\]
	From the above discussion, we see that if there exists some
	$k=1,\cdots,n$ such that
	$\min(v(\lambda_kS_k^{'2}),0)
	>v((2\lambda_kr'_k+b'_k)S'_k)$,
	then
	$\int_{r+S}\mathbf{e}_\ell(x^tax+b^tx)=0$.
	We define
	$v(r')=\min_{k=1}^n(v(r'_k))$,
	then
	$v(r)=v(r')$
	and similarly
	$v(S)=v(S')$,
	$v(b)=v(b')$.
	Then we have the following lemma
	\begin{lemma}\label{non-vanishing of Gauss sums-higher dim}
		Notations as above, one has
		\[
		\min(v(\lambda_kS_k^{'2}),0)
		>v((2\lambda_kr_k'+b_k')S_k')
		\text{ for some }k
		\Rightarrow
		\int_{r+S}\mathbf{e}_\ell(x^tax+b^tx)dx=0.
		\]
	\end{lemma}

    We can again generalize the above results as follows.
    Let $a\in\mathrm{Sym}_{n\times n}(\mathbb{Q}_\ell)$
    be as above 
    $\mathbb{Z}_\ell$-equivalent to
    $\Lambda=
    \mathrm{diag}(\lambda_1,\cdots,\lambda_n)$
    while
    $b\in\mathrm{M}_{n\times m}
    (\mathbb{Q}_\ell)$
    be an $n\times m$-matrix.
    Let $\eta'=\mathrm{diag}(\eta'_1,\cdots,\eta'_m)$
    be a diagonal matrix with entries in
    $\mathbb{Q}_\ell^\times$.
    Let $r\in\mathrm{M}_{n\times m}
    (\mathbb{Q}_\ell)$
    and
    $S\in\mathrm{M}_{n\times m}
    (\mathbb{Q}_\ell)$
    be an open compact open subgroup.
    Then we consider the following integral:
    \[
    \int_{r+S}\mathbf{e}_\ell(\mathrm{tr}(x^tax\eta'+b^tx))
    dx.
    \]
    We write
    $r=(r^{(1)},\cdots,r^{(m)})$,
    $b=(b^{(1)},\cdots,b^{(m)})$
    and
    $x=(x^{(1)},\cdots,x^{(m)})$
    in column matrices.
	If we assume $S$ is of the form
	$S=(S^{(1)},\cdots,S^{(m)})$,
	then we can write
	the above integral as a product
	$\prod_{j=1}^m
	\int_{r^{(j)}+S^{(j)}}
	\mathbf{e}_\ell((x^{(j)})^tax^{(j)}\eta'_j
	+(b^{(j)})^tx^{(j)})dx^{(j)}$.
	As in the above lemma,
	we can easily give a non-vanishing criterion 
	for this integral:
	if
	$\min(v(\lambda_k(S_k^{(j)})^2),0)>
	v((2\eta'_j\lambda_kr_k^{(j)}+b_k^{(j)})S_k^{(j)})$
	for some
	$j$ and $k$,
	then the integral vanishes.

	\subsection{Ramified places dividing $N$}
	Now let's consider a finite place
	$\ell|N$.
	We write
	$N_\ell$
	for $\ell^{\mathrm{val}_\ell(N)}$.
	Recall that we have in fact
	$N_\ell=\ell$ since $N$ is square-free.
	Moreover our symmetric form is
	$\eta_U=\mathrm{diag}(N^2/2,N^2/2,N^2/2,N^2/2,N/N_1,N_1)$
	with $N_1\in\mathbb{Z}_{>0}$ prime to $Np$.

	\subsubsection{Choice of sections}
	Recall that we identify $W_1^-(\mathbb{Q}_\ell),
	W_2^-(\mathbb{Q}_\ell)$
	with the set of matrices
	$\mathrm{M}_{2\times 6}
	(\mathbb{Q}_\ell)$
	under the correspondence
	$e_k^-\otimes u_r\leftrightarrow E_{k,r}$
	for $k=1,2$ and $r=1,\cdots,6$.
	We fix the following matrix $b$
	and compact open subgroup $S$ of
	$\mathrm{M}_{4\times 6}
	(\mathbb{Q}_\ell)$:
	\[
	b=
	\frac{1}{N_\ell}
	\begin{pmatrix}
	1 & 0 & 0 & 0 & 0 & 0 \\
	0 & 1 & 0 & 0 & 0 & 0 \\
	0 & 0 & 1 & 0 & 0 & 0 \\
	0 & 0 & 0 & 1 & 0 & 0
	\end{pmatrix},
	\quad
	S=\frac{1}{N_\ell}(\mathbb{Z}_\ell^4,
	\mathbb{Z}_\ell^4,
	\mathbb{Z}_\ell^4,
	\mathbb{Z}_\ell^4,
	\mathbb{Z}_\ell^4,
	\mathbb{Z}_\ell^4).
	\]
	Moreover, we write
	$b=\begin{pmatrix}
	b_{[1]} \\ b_{[2]}
	\end{pmatrix}$
	and
	$S=\begin{pmatrix}
	S_{[1]} \\ S_{[2]}
	\end{pmatrix}$
	in $2\times6$-blocks.	
	We define elements in
	the following sections:
	\begin{equation}
	\label{Schwartz section at ramified place dividing N}
	\phi_{i,\ell}(w_i^-)
	=
	1_{S_{[i]}}(w_i^-)
	\mathbf{e}_\ell(\mathrm{tr}
	(b_{[i]}^tw_i^-))
	\in
	\mathcal{S}(W_i^-(\mathbb{Q}_\ell))
	\text{ where }i=1,2.
	\end{equation}

	As above, we set
	$\phi^+_\ell=
	\phi_{1,\ell}\otimes\phi_{2,\ell}$,
	$\phi^d_\ell=\delta(\phi^+_\ell)$,
	which give rise to
	$f_{\phi^+_\ell}
	\in
	\mathrm{Ind}_{P^+(\mathbb{Q}_\ell)}^{G_4(\mathbb{Q}_\ell)}(\xi_{1/2})$
	and
	$f_{\phi^d_\ell},
	\widetilde{f}_{\phi^d_\ell}
	\in
	\mathrm{Ind}_{P^d(\mathbb{Q}_\ell)}^{G_4(\mathbb{Q}_\ell)}(\xi_{1/2})$
	(as in the case of unramified places,
	we first define sections in
	$\mathrm{Ind}_{P^*(\mathbb{Q}_\ell)}
	^{G_4^1(\mathbb{Q}_\ell)}(\xi_{1/2})$, 
	and then extend them to
	$G_4(\mathbb{Q}_\ell)$
	for $*=+,d$).
	Then it is easy to see
	\begin{lemma}\label{support of Gauss sums}
		For any $a\in\mathrm{Sym}_{4\times 4}
		(\mathbb{Q}_\ell)$
		which is $\mathbb{Z}_\ell$-equivalent to
		the diagonal matrix
		$\Lambda=
		\mathrm{diag}(\lambda_1,\cdots,\lambda_4)$,
		the generalized Gauss sum
		$I=\int_S\mathbf{e}_\ell
		(\mathrm{tr}(x^tax\eta+b^tx))dx\neq0$
		implies
		$v(\lambda_i)\leq-v(N^2)$
		for all $i=1,\cdots,4$.
		In particular,
		$I\neq0$ implies that $\mathrm{det}(a)\neq0$.
	\end{lemma}
    \begin{proof}
    	Indeed, suppose that there is a matrix
    	$D\in\mathrm{GL}_4(\mathbb{Z}_\ell)$
    	such that
    	$a=D^t\Lambda D$.
    	We then write as above
    	$b'=Db=(b^{'(1)},\cdots,b^{'(6)})$,
    	$S=DS=(S^{(1)},\cdots,S^{(6)})$
    	and
    	$x=(x^{(1)},\cdots,x^{(6)})$.
    	We define
    	$I_j=
    	\int_{S^{(j)}}
    	\mathbf{e}_\ell
    	(\eta_j(x^{(j)})^t\Lambda
    	x^{(j)}+(b^{'(j)})^tx^{(j)})
    	dx^{(j)}$
    	for $j=1,\cdots,6$.
    	Then $I=\prod_{j=1}^6I_j$.
    	Suppose that
    	$v(\lambda_k)>-v(N^2)$
    	for some $k$.
    	Since $D\in\mathrm{GL}_4(\mathbb{Z}_\ell)$,
    	by the definition of $b$,
    	we see that for the $k$ above,
    	there exists some $j=1,2,3,4$
    	such that
    	$b^{(j)}_k\in\frac{1}{N_\ell}\mathbb{Z}_\ell^\times$.    	
    	Then 
    	$v(b_k^{(j)}S_k^{(j)})
    	=-v(N^2)<
    	\mathrm{min}(v(\eta_j\lambda_k
    	(S_k^{(j)})^2),0)
    	=\min(v(\lambda_k),0)$
    	thus
    	by
    	Lemma \ref{non-vanishing of Gauss sums-higher dim},
    	we see that
    	$I_j=0$ and therefore
    	$I=0$.
    \end{proof}

    We write
    $\Delta=\begin{pmatrix}
    0 & 1_2 \\
    1_2 & 0
    \end{pmatrix}
    \in G_1(\mathbb{Q})$.
    Then it is easy to see that
    \[
    \mathcal{S}^{-1}
    (g_1,1)\mathcal{S}
    =\begin{pmatrix}
    \frac{1}{2}(g_1+1) & \frac{1}{4}(g_1-1)\Delta \\
    \Delta(g_1-1) & \frac{1}{2}\Delta(g_1+1)\Delta
    \end{pmatrix}.
    \]
    Then we have
	\begin{corollary}\label{support of Siegel section}
		For an element
		$g=\begin{pmatrix}
		A & B \\
		C & D
		\end{pmatrix}
		\in G_4(\mathbb{Q}_\ell)$,
		$f_{\phi_\ell^+}(g)\neq0$
		implies that
		$\mathrm{det}(C^tD)\neq0$.
		For $g_1\in G_1^1(\mathbb{Q}_\ell)$,
		$\widetilde{f}_{\phi_\ell^d}((g_1,1))\neq0$
		implies that
		$g_1\in
		G_1^1(\mathbb{Z}_\ell)$,
		$g_1\equiv1_4(\mathrm{mod}\,N^2)$
		and
		$\mathrm{det}(g_1-1)\neq0$.
	\end{corollary}
    \begin{proof}
    	By (\ref{general expression for Weil representation}),
    	we have that
    	\begin{align*}
    	f_{\phi_\ell^+}(g)
    	&
    	=
    	\gamma^\mathrm{Weil}
    	\int_{\mathrm{Ker}(C)\backslash
    		W_4^+(\mathbb{Q}_\ell)}
    	\mathbf{e}_\ell
    	(\mathrm{tr}(\frac{1}{2}x^tC^tDx\eta+b^tx))
    	1_S(Cx)dx
    	\\
    	&
    	=
    	\gamma^\mathrm{Weil}
    	\int_S
    	\mathbf{e}_\ell
    	(\mathrm{tr}(\frac{1}{2}x^tC^tDx\eta+b^tx))dx.
    	\end{align*}
    	By the above lemma,
    	from $f_{\phi_\ell^+}(g)\neq0$ we have
    	$\mathrm{det}(C^tD)\neq0$.

    	We write
    	$g_1=
    	\begin{pmatrix}
    	A_1 & B_1 \\
    	C_1 & D_1
    	\end{pmatrix}$
    	in $2\times2$-blocks.
    	Then by definition,
    	$\widetilde{f}_{\phi_\ell^d}((g_1,1))
    	=f_{\phi_\ell^+}(\mathcal{S}^{-1}
    	(g_1,1)\mathcal{S})$.   	
    	So by the first part of the corollary,
    	if $\widetilde{f}_{\phi_\ell^d}((g_1,1))\neq0$,
    	then
    	$\mathrm{det}((g_1-1)(g_1+1))\neq0$.
    	Set $g=\mathcal{S}^{-1}(g_1,1)\mathcal{S}$,
    	since $\mathrm{det}(C)\neq0$,
    	we can write $g$ as a product
    	\[
    	g=u(B')m(A')J_4u(B')=
    	\begin{pmatrix}
    	-B'A' & A'-B'A'B'' \\
    	-(A')^{-t} & -(A')^{-t}B''
    	\end{pmatrix}.
    	\]
    	Note that
    	$B''=\frac{1}{2}(g_1-1)^{-1}(g_1+1)\Delta$.
    	Therefore
    	$f_{\phi_\ell^+}(g)\neq0$
    	if and only if
    	the integral
    	\[
    	\int_S\mathbf{e}_\ell
    	(\mathrm{tr}(\frac{1}{2}x^tB''x\eta+b^tx))dx\neq0.
    	\]
    	By the above lemma,
    	this implies that
    	$(B'')^{-1}\in N_\ell^2
    	\mathrm{Sym}_{4\times 4}(\mathbb{Z}_\ell)$.
    	Since
    	\[
    	g_1=(1+(2B''\Delta)^{-1})
    	(1-(2B''\Delta)^{-1})^{-1}.
    	\]
    	We deduce that
    	$g_1\in G_1(\mathbb{Z}_\ell)$,
    	$g_1\equiv1_4(\mathrm{mod}\,N_\ell^2)$
    	and
    	$\mathrm{det}(g_1-1)\neq0$.
    \end{proof}

	\subsubsection{Local zeta integral}    
    We next calculate
    the local zeta integral
    $Z_\ell(\varphi_\ell,
    \varphi_\ell^\vee,
    \widetilde{f}_{\phi_\ell^d})$.
    Recall that
    $\Gamma(N)_\ell$ is the subgroup of
    $G_1(\mathbb{Z}_\ell)$
    consisting of matrices 
    $g$ such that $g\equiv1_4(\mathrm{mod}\,N_\ell)$.
	\begin{proposition}
		\label{non-vanishing of local zeta integral at l}
		Assume that
		$\varphi$
		is invariant under right translation of
		$\Gamma(N)_\ell$.
		Then the local zeta integral at $\ell$ is a non-zero
		rational
		multiple of
		$\langle\varphi_\ell,\varphi_\ell^\vee\rangle$
		(explicitly given in
		Lemma \ref{non-vanishing of local integral at l}):
		\[
		\frac{Z_\ell(\varphi_\ell,\varphi_\ell^\vee,
			\widetilde{f}_{\phi^d_\ell})}
		{\langle\varphi_\ell,\varphi_\ell^\vee\rangle}
		\in\mathbb{Q}^\times.
		\]
	\end{proposition}
	\begin{proof}
		By the assumption we have
		\[
		Z_\ell(\varphi_\ell,\varphi_\ell^\vee,
		\widetilde{f}_{\phi^d_\ell})
		=\int_{G_1^1(\mathbb{Q}_\ell)}
		\widetilde{f}_{\phi^d_\ell}((g,1))
		\langle
		\pi(g)
		\varphi_\ell,
		\varphi_\ell^\vee
		\rangle
		dg
		=
		\langle
		\varphi_\ell,
		\varphi_\ell^\vee
		\rangle
		\int_{\Gamma(N)_\ell}
		\widetilde{f}_{\phi^d_\ell}((g_1,1))
		dg_1.
		\]
		So it suffices to treat the integral on
		the RHS.
		As in the preceding corollary,
		we write
		$A'=-\Delta(g_1-1)^{-t}$
		and
		$B=(B'')^{-1}=2((g_1-1)^{-1}(g_1+1)\Delta)^{-1}
		\in\mathrm{Sym}_{4\times 4}
		(N_\ell\mathbb{Z}_\ell)$
		for $g_1\in\Gamma(N)_\ell$,
		i.e.
		$g_1=2(1-(2B^{-1}\Delta)^{-1})^{-1}-1$.
		Then by change of variables from
		$g_1$ to $B$ and taking into account of the
		Jacobian $|\mathrm{det}(\partial g_1/\partial B)|=1$,
		we see that
		\begin{align*}
		\widetilde{f}_{\phi_\ell^d}((g_1,1))
		&
		=|\mathrm{det}(A')|^3
		\xi(\mathrm{det}(A'))
		\int_S\mathbf{e}_\ell
		(\mathrm{tr}(\frac{1}{2}x^tB^{-1}x\eta+b^tx))dx
		\\
		&
		=|\mathrm{det}(B)|^{-3}\xi(\mathrm{det}(B))
		\int_S\mathbf{e}_\ell
		(\mathrm{tr}(\frac{1}{2}x^tB^{-1}x\eta+b^tx))dx.
		\end{align*}
		Therefore the integral becomes
		\[
		\int_{\Gamma(N)_\ell}
		\widetilde{f}_{\phi_\ell^d}((g_1,1))dg_1
		=
		\int_{\mathrm{Sym}_{4\times 4}
			(\ell\mathbb{Z}_\ell)}
		|\mathrm{det}(B)|^{-3}
		\xi(\mathrm{det}(B))dB
		\int_S
		\mathbf{e}_\ell
		(\mathrm{tr}(\frac{1}{2}x^tB^{-1}x\eta+b^tx))dx.
		\]
		
		We first treat the inner integral in the above.
		Without loss of generality we can assume that
		$B^{-1}=\mathrm{diag}(B_1^{-1},\cdots,B_4^{-1})$.
		Using the calculation preceding
		Lemma \ref{vanishing of Gauss integrals},
		we see that
		for any
		$i=1,\cdots,4$
		and
		$j=1,\cdots,6$,
		\[
		\int_{S^{(j)}_i}\mathbf{e}_\ell
		(\mathrm{tr}(\frac{1}{2}(x^{(j)}_i)^t
		B_i^{-1}x^{(j)}_i\eta_j+(b^{(j)}_i)^tx^{(j)}_i))
		dx^{(j)}_i
		\]
		\[
		=
		\ell^{-\frac{1}{2}v_\ell(B_i)+\frac{1}{2}v_\ell(\eta_j)}
		\varepsilon(\ell^{v_\ell(B_i\eta_j)})
		\big(
		\frac{-2B_i\eta_j\ell^{-v_\ell(B_i\eta_j)}}{\ell}
		\big)^{v_\ell(B_i\eta_j)}
		\mathbf{e}_\ell(-\frac{(b^{(j)}_i)^2B_i}{4})
		\]
		Note that
		\begin{align*}
		\prod_{i,j}\varepsilon(\ell^{v_\ell(B_i\eta_j)})
		&
		=\prod_{i,j}\varepsilon(\ell^{v_\ell(B_i)})
		=\prod_i\varepsilon(\ell^{v_\ell(B_i)})^2
		\\
		&
		=\prod_i(\frac{-1}{\ell})^{v_\ell(B_i)}
		=(-1)^{v_\ell(\mathrm{det}(B))\frac{\ell-1}{2}},
		\end{align*}
		\begin{align*}
		\prod_{i,j}
		\bigg(
		(\frac{B_i^\circ}{\ell})(\frac{\eta_j^\circ}{\ell})
		\bigg)^{v_\ell(B_i\eta_j)}
		&
		=\prod_{i,j}(\frac{B_i^\circ}{\ell})^{v_\ell(\eta_j)}
		(\frac{\eta_j^\circ}{\ell})^{v_\ell(B_i)}
		=\prod_{i,j}(B_i,\eta_j)_\ell
		(-1)^{v_\ell(B_i)v_\ell(\eta_j)\frac{\ell-1}{2}}
		\\
		&
		=(\mathrm{det}(B),\mathrm{det}(\eta))_\ell
		(-1)^{v_\ell(\mathrm{det}(B))
			v_\ell(\mathrm{det}(\eta))
			\frac{\ell-1}{2}}.
		\end{align*}

		So their product gives
		(for any $x\in\mathbb{Q}_\ell^\times$
		we write
		$x^\circ\in\mathbb{Z}_\ell^\times$ for
		$x\ell^{-v_\ell(x)}$)
		\begin{align*}
		\int_S\mathbf{e}_\ell(\mathrm{tr}(
		\frac{1}{2}x^tB^{-1}x\eta_U+b^tx))dx
		&
		=
		(\prod_i\ell^{-3v_\ell(B_i)})
		(\prod_j\ell^{2v_\ell(\eta_j)})
		(\prod_i\mathbf{e}_\ell(-\frac{B_i}{\ell^2}))
		\\
		&
		\quad
		\times
		(\frac{-2}{\ell})^{6\sum_iv_\ell(B_i)+
			4\sum_jv_\ell(\eta_j)}
		\prod_{i,j}
		\bigg(
		(\frac{B_i^\circ}{\ell})
		(\frac{\eta_j^\circ}
		{\ell})
		\bigg)^{v_\ell(B_i\eta_j)}
		\\
		&
		=
		|\mathrm{det}(B)|^3 
		|\mathrm{det}(\eta)|^{-2}
		\mathbf{e}_\ell(-\mathrm{tr}\,\frac{B}{\ell^2})
		(\mathrm{det}(B),\mathrm{det}(\eta))_\ell
		\\
		&
		\quad
		\times
		(-1)^{v_\ell(\mathrm{det}(B))v_\ell(\mathrm{det}(\eta))
			\frac{\ell-1}{2}}
		(-1)^{v_\ell(\mathrm{det}(B))\frac{\ell-1}{2}}
		\end{align*}
		On the other hand,
		note that by
		\cite[Chapter II, Proposition 4.3]{Kudla96},
		\begin{align*}
		\xi(\mathrm{det}(B))
		&
		=
		(\mathrm{det}(B),(-1)^{6\times5/2}\mathrm{det}(\eta))_\ell
		\\
		&
		=
		(\mathrm{det}(B),\mathrm{det}(\eta))_\ell
		(\mathrm{det}(B),-1)_\ell
		\\
		&
		=
		(\mathrm{det}(B),\mathrm{det}(\eta))_\ell
		(-1)^{v_\ell(\mathrm{det}(B))\frac{\ell-1}{2}}.
		\end{align*}
		Taking into account of the fact
		$v_\ell(\mathrm{det}(\eta))\equiv1(\mathrm{mod}\,2)$,
		we get
		\begin{align*}
		\int_{\Gamma(N)_\ell}
		\widetilde{f}_{\phi_\ell^d}((g_1,1))dg_1
		&
		=|\mathrm{det}(\eta)|_\ell^{-2}
		\int_{\mathrm{Sym}_{4\times 4}
			(N^2_\ell\mathbb{Z}_\ell)
			\cap\mathrm{GL}_4(\mathbb{Q}_\ell)}
		\mathbf{e}_\ell(-\mathrm{tr}\,\frac{B}{\ell^2})
		(-1)^{v_\ell(\mathrm{det}(B))\frac{\ell-1}{2}}dB
		\\
		&
		=\ell^{-20}|\mathrm{det}(\eta)|_\ell^{-2}
		\int_{\mathrm{Sym}_{4\times 4}
			(\mathbb{Z}_\ell)}
		(-1)^{v_\ell(\mathrm{det}(B))\frac{\ell-1}{2}}dB
		\end{align*}
		($(-1)^{v_\ell(\mathrm{det}(B))\frac{\ell-1}{2}}$
		is understood to be
		$0$ if $\mathrm{det}(B)=0$).
		We treat this last integral in the
		following series of lemmas.
		The result is given in 
		Lemma \ref{non-vanishing of local integral at l}.
	\end{proof}

    Recall that we say two matrices
    $A,B\in\mathrm{Sym}_{n\times n}
    (\mathbb{Q}_\ell)$
    are $\mathbb{Z}_\ell$-equivalent
    (written as
    $A\sim B$) if
    there is an invertible matrix
    $D\in\mathrm{GL}(n,\mathbb{Z}_\ell)$
    such that $A=DBD^t$.
    Similarly,
    we say that
    $A,B\in\mathrm{Sym}_{n\times n}
    (\mathbb{Z}/\ell^N)$
    are $\mathbb{Z}/\ell^N$-equivalent
    (written as $A\sim_{N}B$)
    if there is an invertible matrix
    $D\in\mathrm{GL}(n,\mathbb{Z}/\ell^N)$
    such that
    $A=DBD^t$.
    We fix a non-square 
    $z=z_\ell\in\mathbb{Z}_\ell^\times$.
    Moreover, 
    by \cite[Chapter 8, Theorem 3.1]{Cassels82},
    any matrix
    $A\in\mathrm{Sym}_{n\times n}
    (\mathbb{Q}_\ell)$ is $\mathbb{Z}_\ell$-equivalent
    to a diagonal matrix 
    $\Lambda$ of the form
    $\Lambda=\mathrm{diag}(\ell^{k_1}\Lambda_1,
    \ell^{k_2}\Lambda_2,\cdots,
    \ell^{k_r}\Lambda_r,0)$
    with
    $k_1<k_2<\cdots<k_r$
    and
    $\Lambda_i=1_{n_i}$
    or
    $\Lambda_i=\mathrm{diag}(1_{n_i-1},z)$.
    For any $A\in\mathrm{Sym}_{n\times n}
    (\mathbb{Z}_\ell)$,
    we also write $A$ for its image in
    by the canonical projection
    $\mathrm{Sym}_{n\times n}
    (\mathbb{Z}_\ell)
    \rightarrow
    \mathrm{Sym}_{n\times n}
    (\mathbb{Z}/\ell^N)$.

    \begin{lemma}\label{Z_l-equivalence vs. Z/l-equivalence}
    	Suppose as above
    	$\Lambda=\mathrm{diag}(\ell^{k_1}\Lambda_1,
    	\cdots,\ell^{k_r}\Lambda_r)$
    	an invertible matrix
    	with
    	$k_1\geq0$.
    	Then a matrix
    	$A\in\mathrm{Sym}_{n\times n}
    	(\mathbb{Z}_\ell)$
    	is $\mathbb{Z}_\ell$-equivalent to
    	$\Lambda$ if and only if
    	$A$ is $\mathbb{Z}/\ell^{k_r+1}$-equivalent to
    	$\Lambda$.
    \end{lemma}
    \begin{proof}
    	It suffices to show that
    	if $A\sim \Lambda$, then for any
    	$A'\in\mathrm{Sym}_{n\times n}
    	(\mathbb{Z}_\ell)$,
    	$A+\ell^{k_r+1}A'\sim \Lambda$.
    	Then it suffices to treat the case
    	$A=\Lambda$.
    	We prove this case by induction on $r$.
    	
    	For the case $r=1$,
    	we first assume that
    	$\Lambda=1_n$.
    	Then for any $A'\in\mathrm{Sym}_{n\times n}
    	(\mathbb{Z}_\ell)$,
    	we define
    	$D(\ell^{k_r+1}A)=\sqrt{1_n+\ell^{k_r+1}A}
    	=\sum_{i=0}^\infty
    	\binom{1/2}{i}(\ell^{k_r+1}A)^i$
    	which is a matrix
    	in $\mathrm{GL}(n,\mathbb{Z}_\ell)$.
    	Then it is clear that
    	$\Lambda+\ell^{k_r+1}A'=D(\ell^{k_r+1}A)
    	\Lambda D(\ell^{k_r+1}A)^t$
    	thus $\Lambda+\ell^{k_r+1}A'\sim \Lambda$.
    	Next we assume that
    	$\Lambda=\mathrm{diag}(1_{n-1},z)$.
    	Then for any
    	$A'=\begin{pmatrix}
    	A'_1 & A'_2 \\
    	(A'_2)^t & A'_3
    	\end{pmatrix}\in\mathrm{Sym}_{n\times n}
    	(\mathbb{Z}_\ell)$
    	with $A'_1$ of size
    	$(n-1)\times(n-1)$,
    	we define a matrix
    	$D(\ell^{k_r+1}A)=\begin{pmatrix}
    	D(\ell^{k_r+1}A')_1 & D(\ell^{k_r+1}A')_2 \\
    	D(\ell^{k_r+1}A')_3 & D(\ell^{k_r+1}A')_4
    	\end{pmatrix}\in\mathrm{M}_{n\times n}
    	(\mathbb{Q}_\ell)$
    	with $D(\ell^{k_r+1}A)_1$ of size $(n-1)\times(n-1)$
    	as follows:
    	\begin{align*}
    	D(\ell^{k_r+1}A')_4
    	&
    	:=D(z^{-1}\ell^{k_r+1}A_3'),
    	\\
    	D(\ell^{k_r+1}A')_3
    	&
    	:=0,
    	\\
    	D(\ell^{k_r+1}A')_2
    	&
    	:=z^{-1}
    	D(\ell^{k_r+1}A')_4^{-1}
    	\ell^{k_r+1}A_2',
    	\\
    	D(\ell^{k_r+1}A')_1
    	&
    	:=
    	D(\ell^{k_r+1}A'_1-
    	(z+\ell A_3')^{-1}\ell^{k_r+1}A_2'
    	(\ell^{k_r+1}A_2')^t).
    	\end{align*}
    	Then one verifies that
    	$D(\ell^{k_r+1}A')$ is a matrix in
    	$\mathrm{GL}(n,\mathbb{Z}_\ell)$
    	and
    	$\Lambda+\ell^{k_r+1}A'=
    	D(\ell^{k_r+1}A')\Lambda D(\ell^{k_r+1}A')^t$,
    	therefore
    	$\Lambda+\ell^{k_r+1}A'\sim
    	\Lambda$
    	(note that the argument is independent of the size 
    	$n\times n$ of $\Lambda$).
    	
    	Suppose that we have proved the statement in the
    	beginning for $r-1>0$, i.e.,
    	for
    	$\Lambda''=\mathrm{diag}
    	(\ell^{k_2}\Lambda_2,
    	\cdots,\ell^{k_r}\Lambda_r)$
    	and any
    	$\widetilde{A}\in\mathrm{Sym}_{\widetilde{n}
    	\times\widetilde{n}}(\mathbb{Z}_\ell)$,
    	there is a matrix
    	$D_{\Lambda''}(\ell^{k_r+1}\widetilde{A})
    	\in\mathrm{GL}(n,\mathbb{Z}_\ell)$
    	such that
    	$\Lambda''+\ell^{k_r+1}\widetilde{A}
    	=D_{\Lambda''}(\ell^{k_r+1}\widetilde{A})
    	\Lambda''
    	D_{\Lambda''}(\ell^{k_r+1}\widetilde{A})^t$.
    	We write
    	$\Lambda=\mathrm{diag}(\Lambda',\Lambda'')$
    	with $\Lambda'$ of size
    	$n_1\times n_1$.
    	For any $A'=\begin{pmatrix}
    	A_1' & A_2' \\
    	(A_2')^t & A_3'
    	\end{pmatrix}\in
    	\mathrm{Sym}_{n\times n}
    	(\mathbb{Z}_\ell)$
    	with $A_1'$ of size
    	$n_1\times n_1$.
    	Then we define a matrix
    	$D(\ell^{k_r+1}A')
    	=\begin{pmatrix}
    	D(\ell^{k_r+1}A')_1 & D(\ell^{k_r+1}A')_2 \\
    	D(\ell^{k_r+1}A')_3 & D(\ell^{k_r+1}A')_4
    	\end{pmatrix}\in\mathrm{M}_{n\times n}
    	(\mathbb{Q}_\ell)$
    	with $D(\ell^{k_r+1}A')_1$ of size
    	$n_1\times n_1$
    	as follows:
    	\begin{align*}
    	D(\ell^{k_r+1}A')_4
    	&
    	:=
    	D_{\Lambda''}(\ell^{k_r+1}A_3'),
    	\\
    	D(\ell^{k_r+1}A')_3
    	&
    	:=0,
    	\\
    	D(\ell^{k_r+1}A')_2
    	&
    	:=
    	A_2'D(\ell^{k_r+1}A')_4^{-t}
    	\ell^{k_r+1}(\Lambda'')^{-1},
    	\\
    	D(\ell^{k_r+1}A')_1
    	&
    	:=
    	D_{\Lambda'}(\ell^{k_r+1}A_1'-
    	D(\ell^{k_r+1}A')_2\Lambda''
    	D(\ell^{k_r+1}A')_2^t).
    	\end{align*}
       	Then one verifies that
       	$D(\ell^{k_r+1}A')\in\mathrm{GL}(n,\mathbb{Z}_\ell)$
       	and
       	$\Lambda+\ell^{k_r+1}A'=
       	D(\ell^{k_r+1}A')\Lambda
       	D(\ell^{k_r+1}A')^t$
       	and therefore
       	$\Lambda+\ell^{k_r+1}A'\sim
       	\Lambda$,
       	which concludes the induction step.
    \end{proof}

    Now for any invertible
    $\Lambda\in
    \mathrm{Sym}_{4\times 4}
    (\mathbb{Z}_\ell)$ as above,
    we write
    $\mathcal{E}(\Lambda)$
    for the subset of
    $\mathrm{Sym}_{4\times 4}
    (\mathbb{Z}_\ell)$
    consisting of matrices
    $\mathbb{Z}_\ell$-equivalent to
    $\Lambda$.
    By the above lemma, we see that
    each $\mathcal{E}(\Lambda)$ is an open compact subset of
    $\mathrm{Sym}_{4\times 4}
    (\mathbb{Z}_\ell)$.
    We also write
    $\mathcal{Q}$
    for the set of all such
    $\Lambda\in\mathrm{Sym}_{4\times 4}
    (\mathbb{Z}_\ell)$.    
    We define
    \[
    I(\Lambda)
    =\int_{B\in\mathcal{E}(\Lambda)}
    (\frac{-1}{\ell})^{v_\ell(\mathrm{det}(B))}
    =
    (\frac{-1}{\ell})^{v_\ell(\mathrm{det}(\Lambda))}
    \int_{B\in\mathcal{E}(\Lambda)}dB.
    \]
    Then we have
    \[
    \int_{\mathrm{Sym}_{4\times 4}
    	(\mathbb{Z}_\ell)}
    (\frac{-1}{\ell})^{v_\ell(\mathrm{det}(B))}dB
    =\sum_{\Lambda\in\mathcal{Q}}I(\Lambda).
    \]

    Before going on, we need some results on the order of
    orthogonal groups over certain finite rings.
    Let $\Lambda_1\in\mathrm{GL}(n_1,\mathbb{Z}/\ell)$
    be a symmetric matrix.
    The orthogonal group we consider is
    $\mathrm{O}(\Lambda_1,
    \mathbb{Z}/\ell^{s+1})$
    for any integer $s\geq0$.
    We have a canonical projection
    \[
    \mathrm{pr}_s\colon
    \mathrm{O}(\Lambda_1,
    \mathbb{Z}/\ell^{s+1})
    \rightarrow
    \mathrm{O}(\Lambda_1,
    \mathbb{Z}/\ell^s),
    \]
    whose kernel
    $\mathrm{pr}_s$ can be identified with
    the subset of
    $\mathrm{M}_{n_1\times n_1}
    (\mathbb{Z}/\ell)$
    consisting of matrices
    $D$ such that
    $D\Lambda_1+\Lambda_1D^t=0$.
    So we see that
    $\#(\mathrm{pr}_s)=\ell^{n_1(n_1-1)/2}$.
    Therefore by induction on $s$,
    \[
    \#\mathrm{O}(\Lambda_1,\mathbb{Z}/\ell^{s+1})
    =\ell^{sn_1(n_1-1)/2}
    \#\mathrm{O}(\Lambda_1,\mathbb{Z}/\ell).
    \]

    As for the order of the orthogonal group
    $\mathrm{O}(\Lambda_1,\mathbb{Z}/\ell)$,
    we record results from
    \cite[Section 3.7]{Wilson2009}.
    \begin{enumerate}
    	\item assume $\Lambda_1$ is of odd dimension
    	$n_1=2m+1$.
    	Then we know that
    	\[
    	\#\mathrm{O}(\Lambda_1,\mathbb{Z}/\ell)
    	=2\ell^{m^2}
    	(\ell^2-1)(\ell^4-1)\cdots(\ell^{2m}-1).
    	\]
    	
    	\item assume that
    	$n_1=2m$ is even and
    	$\ell\equiv1(\mathrm{mod}\,4)$
    	or
    	$n_1=2m$ with $m\equiv0(\mathrm{mod}\,2)$
    	and
    	$\ell\equiv3(\mathrm{mod}\,4)$,
    	then
    	\begin{align*}
    	\#\mathrm{O}(1_{2m},\mathbb{Z}/\ell)
    	&
    	=2\ell^{m(m-1)}
    	(\ell^2-1)(\ell^4-1)\cdots(\ell^{2m-2}-1)(\ell^m-1),
    	\\
    	\#\mathrm{O}(\mathrm{diag}(1_{2m-1},z),
    	\mathbb{Z}/\ell)
    	&
    	=
    	2\ell^{m(m-1)}
    	(\ell^2-1)\cdots(\ell^{2m-2}-1)(\ell^m+1).
    	\end{align*}
    	
    	\item assume that 
    	$n_1=2m$ is even with
    	$m\equiv1(\mathrm{mod}\,2)$
    	and
    	$\ell\equiv3(\mathrm{mod}\,4)$,
    	then
    	\begin{align*}
    	\#\mathrm{O}(1_{2m},\mathbb{Z}/\ell)
    	&
    	=
    	2\ell^{m(m-1)}
    	(\ell^2-1)\cdots(\ell^{2m-2}-1)(\ell^m+1),
    	\\
    	\#\mathrm{O}(\mathrm{diag}(1_{2m-1},z),
    	\mathbb{Z}/\ell)
    	&
    	=
    	2\ell^{m(m-1)}
    	(\ell^2-1)(\ell^4-1)\cdots(\ell^{2m-2})(\ell^m-1).
    	\end{align*}
    \end{enumerate}
    Now we can evaluate each $I(\Lambda)$.
    We have
    \begin{lemma}\label{Gauss sum for k_1>=0}
    	Assume $k_1\geq0$
    	in
    	$\Lambda=(\ell^{k_1}\Lambda_1,\cdots,
    	\ell^{k_r}\Lambda_r)$,
    	then the integral is
    	\[
    	I(\Lambda)=
    	(\frac{-1}{\ell})^{\sum_ik_in_i}
    	\frac{\ell^{6k_r-4}
    	(\ell-1)\cdots(\ell^4-1)}
        {\ell^{\sum_{j<i}n_in_j(k_r+1+k_j)+
        	\sum_ik_in_i(n_i+1)/2+k_rn_i(n_i-1)/2}
        	\prod_i\#\mathrm{O}(\Lambda_i,\mathbb{Z}/
        	\ell)}.
    	\]
    \end{lemma}
    \begin{proof}    	
    	Note that by the above lemma,
    	\begin{align*}
    	\mathrm{vol}(\mathcal{E}(\Lambda))
    	&
    	=
    	\ell^{-10(k_r+1)}
    	\#
    	\{
    	B\in\mathrm{Sym}_{4\times 4}
    	(\mathbb{Z}/\ell^{k_r+1})|
    	B\sim_{k_r+1}\Lambda
    	\}
    	\\
    	&
    	=\ell^{-10(k_r+1)}
    	\frac{\#
    		\mathrm{GL}_4(\mathbb{Z}/\ell^{k_r+1})}
    	{\#\{
    		D\in\mathrm{GL}_4(\mathbb{Z}/\ell^{k_r+1})|
    		\Lambda=D\Lambda D^t
    		\}}
    	\\
    	&
    	=\ell^{-10(k_r+1)}
    	\frac{\#\mathrm{GL}_4(\mathbb{Z}/\ell^{k_r+1})}
    	{\#\mathrm{O}(\Lambda,\mathbb{Z}/\ell^{k_r+1})}.
    	\end{align*}
    	
    	The order of
    	$\mathrm{GL}_4(\mathbb{Z}/\ell^{k_r+1})$
    	is given by
    	$\ell^{16k_r}(\ell^4-\ell^3)\cdots(\ell^4-1)$
    	(see for example \cite[Corollary 2.8]{Han2006}).

    	Next we calculate the order of
    	$\mathrm{O}(\Lambda,\mathbb{Z}/\ell^{k_r+1})$.
    	For any element
    	$D=(D_{i,j})\in\mathrm{GL}_4(\mathbb{Z}/\ell^{k_r+1})$
    	with
    	$D_{i,j}\in\mathrm{M}_{n_i\times n_j}
    	(\mathbb{Z}/\ell^{k_r+1})$,
    	we see that
    	$\widetilde{\Lambda}=D\Lambda D^t
    	=(\Lambda_{i,j})$
    	with
    	$\Lambda_{i,j}
    	=\ell^{k_1}D_{i,1}\Lambda_1D_{j,1}^t
    	+\cdots+
    	\ell^{k_r}D_{i,r}\Lambda_rD_{j,r}^t$.
    	If we equate $\widetilde{\Lambda}$
    	with $\Lambda$,
    	then
    	$\widetilde{\Lambda}_{1,1}=\Lambda_{1,1}$
    	implies that
    	$\ell^{k_1}(D_{1,1}\Lambda_1D_{1,1}^t
    	+\ell(\ast)-\Lambda_1)=0
    	(\mathrm{mod}\, \ell^{k_r+1})$
    	with $(\ast)\in
    	\mathrm{M}_{n_1\times n_1}
    	(\mathbb{Z}/\ell^{k_r+1})$.
    	For any fixed $(\ast)$, the number of
    	$D_{1,1}$
    	satisfying the above equation is the same as
    	$\ell^{k_1n_1(n_1+1)/2}$ times
    	the number of $D_{1,1}$
    	such that
    	$D_{1,1}\Lambda_1D_{1,1}^t=\Lambda_1$
    	by the above lemma,
    	which is
    	$\#\mathrm{O}(\Lambda_1,\mathbb{Z}/\ell^{k_r+1})$,
    	as given preceding the lemma.
    	Next $\widetilde{\Lambda}_{2,1}=\Lambda_{2,1}=0$
    	implies that
    	$\ell^{k_1}D_{2,1}\Lambda_1D_{1,1}^t+
    	\ell^{k_2}(\ast)=0
    	(\mathrm{mod}\,\ell^{k_r+1})$
    	with $(\ast)\in
    	\mathrm{M}_{n_2\times n_1}
    	(\mathbb{Z}/\ell^{k_r+1})$.
    	Therefore,
    	$\ell^{k_1}D_{2,1}=0
    	(\mathrm{mod}\,\ell^{k_2})$
    	and
    	$D_{2,1}=-\ell^{k_2-k_1}(\ast)
    	D_{1,1}^{-t}\Lambda_1^{-1}
    	(\mathrm{mod}\,\ell^{k_r-k_1})$.
    	Therefore the number of $D_{2,1}$
    	satisfying the above equation is
    	$\ell^{k_1n_2n_1}$.
    	Similarly, 
    	for $i=2,\cdots,r$,
    	one shows that
    	$\widetilde{\Lambda}_{i,1}=\Lambda_{i,1}=0$
    	implies that
    	$\ell^{k_1}D_{i,1}=0
    	(\mathrm{mod}\,\ell^{k_2})$
    	and therefore
    	the number of $D_{i,1}$
    	satisfying the above equation is
    	$\ell^{k_1n_in_1}$.
    	We can continue the above argument for any
    	$D_{i,j}$ and $D_{j,i}$
    	with $i>j$,
    	and the number of such
    	$D_{i,j}$
    	satisfying the corresponding equation is
    	$\ell^{k_jn_in_j}$
    	while the number of
    	such $D_{j,i}$
    	is $\ell^{(k_r+1)n_jn_i}$.
    	For any $i$,
    	the number of such
    	$D_{i,i}$
    	satisfying the corresponding equation is
    	$\ell^{k_in_i(n_i+1)/2}
    	\#\mathrm{O}(\Lambda_i,
    	\mathbb{Z}/\ell^{k_r+1})
    	=\ell^{k_in_i(n_i+1)/2+k_rn_i(n_i-1)/2}
    	\#\mathrm{O}(\Lambda_i,\mathbb{Z}/\ell)$.
    	So finally we get that
    	\[
    	\#\mathrm{O}(\Lambda,
    	\mathbb{Z}/\ell^{k_r+1})
    	=\prod_{i>j}\ell^{(k_r+1+k_j)n_jn_i}
    	\times
    	\prod_i
    	\ell^{k_in_i(n_i+1)/2+k_rn_i(n_i-1)/2}
    	\#\mathrm{O}(\Lambda_i,
    	\mathbb{Z}/\ell).
    	\] 	
    	Combining the above results,
    	we get the formula in the lemma.

    \end{proof}

    Combining the results in
    Lemmas \ref{Gauss sum for k_1>=0}
    we get the following:
	\begin{lemma}
		\label{non-vanishing of local integral at l}
		\[
		\int_{\mathrm{Sym}_{4\times 4}
			(\mathbb{Z}_\ell)}
		(\frac{-1}{\ell})^{v_\ell(\mathrm{det}(B))}dB
		=
		\frac{((\frac{-1}{\ell})\ell^5-1)(\ell-1)}
		{(\ell^5-1)((\frac{-1}{\ell})\ell-1)}.
		\]
	\end{lemma}
    \begin{proof}
    	For any $\Lambda=\mathrm{diag}
    	(\ell^{k_1}\Lambda_1,
    	\cdots,\ell^{k_r}\Lambda_r)\in\mathcal{Q}$,
    	we have seen that
    	\[
    	I(\Lambda)
    	=(\frac{-1}{\ell})^{\sum_ik_in_i}
    	\Pi(\Lambda)
    	\frac{\ell^{6k_r-4-16k_1}}
    	{\#\mathrm{O}(\Lambda,\mathbb{Z}/\ell^{k_r+1})}
    	(\ell-1)\cdots(\ell^4-1)
    	\]
    	We rewrite $k_1,\cdots,k_r$ as
    	$k_1=\alpha_1$,
    	$k_2=\alpha_1+\alpha_2$,
    	...,$k_r=\alpha_1+\cdots+\alpha_r$
    	with
    	$\alpha_1\geq0$,
    	$\alpha_i>0$ for $i>1$.
    	Now it suffices to take the sum of
    	$I(\Lambda)$
    	over all possible
    	$r=1,\cdots,4$, $\alpha_1\geq0,\cdots,
    	\alpha_r>0$ and $\Lambda_1,\cdots,\Lambda_r$.
    	We write
    	$\alpha=(\ell-1)\cdots(\ell^4-1)$.
    	The results are as follows:
    	   	
    	\begin{enumerate}
    		\item $r=1$, one has
    		\[
    		\sum_{r=1}I(\Lambda)
    		=\frac{\ell^6\alpha}
    		{(\ell^{10}-1)(\ell^2-1)(\ell^4-1)};
    		\]    		
    		\item $r=2$ and
    		$(n_1,n_2)=(1,3)$, one has
    		\[
    		\sum_{(r,n_1,n_2)=(2,1,3)}I(\Lambda)
    		=
    		\frac{\ell^2\alpha}
    		{(\ell^{10}-1)((\frac{-1}{\ell})\ell^6-1)(\ell^2-1)};
    		\]    		
    		\item $r=2$ and
    		$(n_1,n_2)=(2,2)$, one has
    		\[
    		\sum_{(r,n_1,n_2)=(2,2,2)}I(\Lambda)
    		=
    		\frac{\ell^4\alpha}
    		{(\ell^{10}-1)(\ell^3-1)(\ell^2-1)^2};
    		\]    		
    		\item $r=2$ and
    		$(n_1,n_2)=(3,1)$, one has
    		\[
    		\sum_{(r,n_1,n_2)=(2,3,1)}
    		I(\Lambda)
    		=
    		\frac{\ell^2\alpha}
    		{(\ell^{10}-1)((\frac{-1}{\ell})\ell-1)(\ell^2-1)};
    		\]    		
    		\item $r=3$ and
    		$(n_1,n_2,n_3)=(1,1,2)$, one has
    		\[
    		\sum_{(r,n_1,n_2,n_3)=(3,1,1,2)}
    		I(\Lambda)
    		=
    		\frac{\ell^2\alpha}
    		{(\ell^{10}-1)((\frac{-1}{\ell})\ell^6-1)(\ell^3-1)
    		(\ell^2-1)};
    		\]
    		\item $r=3$ and
    		$(n_1,n_2,n_3)=(1,2,1)$, one has
    		\[
    		\sum_{(r,n_1,n_2,n_3)=(3,1,2,1)}
    		I(\Lambda)
    		=
    		\frac{\ell^2\alpha}
    		{(\ell^{10}-1)((\frac{-1}{\ell})\ell^6-1)
    		((\frac{-1}{\ell})\ell-1)(\ell^2-1)};
    		\]    		
    		\item $r=3$ and
    		$(n_1,n_2,n_3)=(2,1,1)$, one has
    		\[
    		\sum_{(r,n_1,n_2,n_3)=(3,2,1,1)}I(\Lambda)
    		=
    		\frac{\ell^2\alpha}
    		{(\ell^{10}-1)(\ell^3-1)
    			((\frac{-1}{\ell})\ell-1)(\ell^2-1)};
    		\]    		
    		\item $r=4$, one has
    		\[
    		\sum_{r=4}I(\Lambda)
    		=
    		\frac{\alpha}
    		{(\ell^{10}-1)((\frac{-1}{\ell})\ell^6-1)
    		(\ell^3-1)((\frac{-1}{\ell})\ell-1)}.
    		\]
    	\end{enumerate}
    Taking the sum of all these numbers, one gets
    that if $(\frac{-1}{\ell})=1$, then
    $\sum_{\Lambda\in\mathcal{Q}}
    I(\Lambda)=1$,
    and if
    $(\frac{-1}{\ell})=-1$,
    then
    $\sum_{\Lambda\in\mathcal{Q}}
    I(\Lambda)
    =(\ell^5+1)(\ell-1)(\ell^5-1)^{-1}(\ell+1)^{-1}$,
    which is the desired formula in the lemma.
    \end{proof}

	\subsubsection{Local Fourier coefficient}
	We now turn to the Fourier coefficients
	\[
	E_{\beta,\ell}(1_8,f_{\phi^+_\ell})
	=
	\int_{\mathrm{Sym}_{4\times 4}
		(\mathbb{Q}_\ell)}
	f_{\phi^+_\ell}
	(J_8u(b))
	\mathbf{e}_\ell
	(
	-\mathrm{tr}\,\beta b
	)
	db.
	\]
	We need to show that
	$E_{\beta,\ell}(1_8,f_{\phi_\ell^+})$,
	as a function on $\beta$,
	is not identically zero
	and takes values in
	$\mathbb{Q}(\zeta_\ell)$.
	We write
	\[
	f(b)
	:=
	f_{\phi^+_\ell}
	(J_8u(b))
	=
	\int_{W_4^-(\mathbb{Q}_\ell)}
	\mathbf{e}_\ell
	(\mathrm{tr}(\frac{1}{2}wbw^t\eta_U))\phi^+_\ell(w)dw.
	\]
	Note that by Lemma \ref{support of Gauss sums},
	the support of $f(b)$ is contained in the subset of
	$b$ such that 
	$b^{-1}\in N_\ell^2
	\mathrm{Sym}_{4\times 4}
	(\mathbb{Z}_\ell)$
	and by the proof of
	Proposition
	\ref{non-vanishing of local zeta integral at l},
	\[
	f(b)
	=|\mathrm{det}(b)|^{-3}
	|\mathrm{det}(\eta_U)|^2
	(\frac{-1}{\ell})^{v_\ell(\mathrm{det}(b))}.
	\]
	We first show that
	\begin{lemma}
		$E_{\beta,\ell}(1_8,f_{\phi_\ell^+})$
		is continuous on 
		$\beta\in\mathrm{Sym}_{4\times 4}
		(\mathbb{Q}_\ell)$
		($\mathrm{Sym}_{4\times 4}
		(\mathbb{Q}_\ell)$
		is given the $l$-adic topology and
		$\mathbb{C}$ the Euclidean topology).
	\end{lemma}
    \begin{proof}
    	We write
    	$E_\beta$ for
    	$E_{\beta,\ell}(1_8,f_{\phi_\ell^+})$
    	and
    	$S_0=\mathrm{Sym}_{4\times 4}
    	(\mathbb{Z}_\ell)$
    	in this proof.
    	For any $\beta,\beta+\beta_0\in S_0$,
    	\[
    	|E_\beta-E_{\beta+\beta_0}|_\infty
    	\leq
    	\int_{b^{-1}\in N_\ell^2 S_0}
    	|\mathrm{det}(b)|_\ell^{-3}
    	|1-\mathbf{e}_\ell(-\mathrm{tr}(\beta_0b))|_\infty
    	db.
    	\]
    	Here $|\cdot|_\infty$
    	is the usual absolute value in $\mathbb{C}$.
    	Suppose that $\beta_0\in\ell^mS_0$
    	for $m\gg0$,
    	then
    	$1-\mathbf{e}_\ell(-\mathrm{tr}(\beta_0b))=0$
    	for any $b\in\ell^{-m}S_0$.
    	Thus the RHS of the above inequality is
    	equal to
    	\[
    	\int_{b^{-1}\in N_\ell S_0,
    		b\notin \ell^{-m}S_0}
    	|\mathrm{det}(b)|_\ell^{-3}db.
    	\]
    	It is easy to see this last integral
    	has a limit $0$ when $m\rightarrow+\infty$,
    	which gives the continuity of
    	$E_\beta$ on $\beta$.
    \end{proof}
    \begin{corollary}
    	The Fourier coefficient
    	$E_{\beta,\ell}(1_8,
    	f_{\phi_\ell^+})$
    	is not identically zero.
    \end{corollary}
    \begin{proof}
    	It suffices to show that
    	$E_{\beta,\ell}(1_8,f_{\phi_\ell^+})\neq0$
    	for $\beta=0$ by the above lemma.
    	By change of variables from
    	$b$ to $b'=b^{-1}$ in
    	$f(b)$, we see that
    	\begin{align*}
    	E_{0,\ell}(1_8,f_{\phi_\ell^+})
    	&
    	=
    	|\mathrm{det}(\eta_U)|^2
    	\left.
    	\int_{N_\ell^2\mathrm{Sym}_{4\times 4}
    		(\mathbb{Z}_\ell)}
    	\mathbf{e}_\ell(-\mathrm{tr}(\beta(b')^{-1}))
    	|\mathrm{det}(b')|
    	(\frac{-1}{\ell})^{v_\ell(\mathrm{det}(b'))}
    	db'
    	\right\vert_{\beta=0}
    	\\
    	&
    	=
    	|\mathrm{det}(\eta_U)|^2
    	|N_\ell|^{20}
    	\int_{\mathrm{Sym}_{4\times 4}
    		(\mathbb{Z}_\ell)}
    	|\mathrm{det}(b')|
    	(\frac{-1}{\ell})^{v_\ell(\mathrm{det}(b'))}
    	db'.
    	\end{align*}
    	So it suffices to show that the above integral
    	does not vanish.
    	The calculation of this integral is the same as
    	the one in
    	Lemma \ref{non-vanishing of local integral at l}.
    	For any
    	$\Lambda=\mathrm{diag}(\ell^{k_1}\Lambda_1,
    	\cdots,\ell^{k_r}\Lambda_r)
    	\in\mathcal{Q}$,
    	we write
    	\[
    	I_0(\Lambda)
    	=
    	\int_{B\sim\Lambda}
    	(\frac{-1}{\ell})^{v_\ell(\mathrm{det}(B))}
    	|\mathrm{det}(B)|dB
    	=
    	(\frac{-1}{\ell})^{v_\ell(\mathrm{det}(B))}
    	|\mathrm{det}(B)|
    	I(\Lambda).
    	\]
    	Then
    	$E_{0,\ell}(1_8,f_{\phi_\ell^+})
    	=|\mathrm{det}(\eta_U)|^2
    	|N_\ell|^{20}
    	\sum_{\Lambda}
    	I_0(\Lambda)$.
    	Moreover we have
    	(using the same notation as
    	Lemma \ref{non-vanishing of local integral at l}):
    	\begin{align*}
    	\sum_{r=1}I(\Lambda)
    	&
    	=\frac{\ell^6\alpha}
    	{(\ell^{14}-1)(\ell^2-1)(\ell^4-1)};
    	\\
    	\sum_{(r,n_1,n_2)=(2,1,3)}I(\Lambda)
    	&
    	=
    	\frac{\ell^2\alpha}
    	{(\ell^{14}-1)((\frac{-1}{\ell})\ell^9-1)(\ell^2-1)};
    	\\
    	\sum_{(r,n_1,n_2)=(2,2,2)}I(\Lambda)
    	&
    	=
    	\frac{\ell^4\alpha}
    	{(\ell^{14}-1)(\ell^5-1)(\ell^2-1)^2};
    	\\
    	\sum_{(r,n_1,n_2)=(2,3,1)}
    	I(\Lambda)
    	&
    	=
    	\frac{\ell^2\alpha}
    	{(\ell^{14}-1)((\frac{-1}{\ell})\ell^2-1)(\ell^2-1)};
    	\\
    	\sum_{(r,n_1,n_2,n_3)=(3,1,1,2)}
    	I(\Lambda)
    	&
    	=
    	\frac{\ell^2\alpha}
    	{(\ell^{14}-1)((\frac{-1}{\ell})\ell^9-1)(\ell^5-1)
    		(\ell^2-1)};
    	\\
    	\sum_{(r,n_1,n_2,n_3)=(3,1,2,1)}
    	I(\Lambda)
    	&
    	=
    	\frac{\ell^2\alpha}
    	{(\ell^{14}-1)((\frac{-1}{\ell})\ell^9-1)
    		((\frac{-1}{\ell})\ell^2-1)(\ell^2-1)};
    	\\
    	\sum_{(r,n_1,n_2,n_3)=(3,2,1,1)}I(\Lambda)
    	&
    	=
    	\frac{\ell^2\alpha}
    	{(\ell^{14}-1)(\ell^5-1)
    		((\frac{-1}{\ell})\ell^2-1)(\ell^2-1)};
    	\\
    	\sum_{r=4}I(\Lambda)
    	&
    	=
    	\frac{\alpha}
    	{(\ell^{14}-1)((\frac{-1}{\ell})\ell^9-1)
    		(\ell^5-1)((\frac{-1}{\ell})\ell^2-1)}.
    	\end{align*}
        Taking the sum of all these numbers,
        one gets
        \[
        \sum_{\Lambda\in\mathcal{Q}}
        I_0(\Lambda)
        =
        \frac{((\frac{-1}{\ell}\ell^6-1))
        (\ell^3-1)(\ell-1)}
        {(\ell^7-1)(\ell^5-1)
        ((\frac{-1}{\ell}\ell^2-1))}
        \neq0,
        \]
    	which concludes the proof of the lemma.    	
    \end{proof}
	By definition, we have
	\begin{align*}
	E_{\beta,\ell}(1_8,f_{\phi_\ell^+})
	&
	=
	\lim\limits_{n\rightarrow+\infty}
	\int_{\ell^{-n}\mathrm{Sym}_{4\times 4}
		(\mathbb{Z}_\ell)}
	\mathbf{e}_\ell(\mathrm{tr}
	(\frac{1}{2}w^t\eta_Uw-\beta)b)db
	\int_{W_4^-(\mathbb{Q}_\ell)}
	\phi_\ell^+(w)dw
	\\
	&
	=\lim\limits_{n\rightarrow+\infty}
	\ell^{10n}
	\int_{W_4^-(\mathbb{Q}_\ell)}
	\phi_\ell^+(w)1_{\ell^n\mathrm{Sym}_{4\times 4}
		(\mathbb{Z}_\ell)}
	(\frac{1}{2}w^t\eta_Uw-\beta)dw
	\\
	&
	=\lim\limits_{n\rightarrow+\infty}
	\ell^{10n}
	\int_S\mathbf{e}_\ell(b^tw)
	1_{\ell^n\mathrm{Sym}_{4\times 4}
		(\mathbb{Z}_\ell)}
	(\frac{1}{2}w^t\eta_Uw-\beta)dw.
	\end{align*}
	It is easy to see that
	\begin{lemma}
		For any positive integer $n>0$, one has
		\[
		\int_S\mathbf{e}_\ell(b^tw)
		1_{\ell^n\mathrm{Sym}_{4\times 4}
			(\mathbb{Z}_\ell)}
		(\frac{1}{2}w^t\eta_Uw-\beta)dw
		=\ell^{-24n}
		\sum_{w\in S/\ell^nS}
		\mathbf{e}_\ell(b^tw)
		1_{\ell^n\mathrm{Sym}_{4\times 4}
			(\mathbb{Z}_\ell)}
		(\frac{1}{2}w^t\eta_Uw-\beta)
		\]
	\end{lemma}
    We define the local density
    $\mathcal{D}_\beta(n)
    =
    \sum_{w\in S/\ell^nS}
    \mathbf{e}_\ell(b^tw)
    1_{\ell^n\mathrm{Sym}_{4\times 4}
    	(\mathbb{Z}_\ell)}
    (\frac{1}{2}w^t\eta_Uw-\beta)$.
    Then one has
    \begin{lemma}
    	Suppose
    	$\mathrm{det}(\beta)\neq0$
    	and
    	$\ell^h\beta^{-1}\in
    	\mathrm{Sym}_{4\times 4}
    	(\mathbb{Z}_\ell)$.
    	Then for any $n>h$,
    	\[
    	\ell^{-14n}
    	\mathcal{D}_\beta(n)
    	=\ell^{-14(n+1)}
    	\mathcal{D}_\beta(n+1).
    	\]
    \end{lemma}
    \begin{proof}
    	The proof is the same as
    	\cite[Lemma 5.6.1]{Kitaoka1983}.
    \end{proof}
	\begin{corollary}
		Suppose $\mathrm{det}(\beta)\neq0$,
		then
		\[
		E_{\beta,\ell}(1_8,f_{\phi_\ell^+})
		\in\mathbb{Z}_p[\zeta_\ell]
		\]
	\end{corollary}
    \begin{proof}
    	Indeed,
    	the above lemmas show that
    	$E_{\beta,\ell}(1_8,
    	f_{\phi_\ell^+})$
    	is equal to
    	$\ell^{-14n}
    	\mathcal{D}_\beta(n)$
    	for $n\gg0$
    	(depending on $\beta$ though).
    	It follows from definition that
    	$\mathcal{D}_\beta(n)
    	\in\mathbb{Z}[\zeta_\ell]$.
    	Since $\ell$ is different from $p$,
    	we see that
    	$E_{\beta,\ell}
    	(1_8,f_{\phi_\ell^+})$
    	lies in $\mathbb{Z}_p[\zeta_\ell]$.
    \end{proof}

    \begin{remark}
    	By Lemma \ref{local Fourier coefficient at infinity},
    	we see that
    	the local Fourier coefficient at $\infty$
    	vanishes
    	${E_{\beta,\infty}(g_\mathbf{z},
    	f_{\phi_{\infty,I}^+})=0}$
    	for any $I\in\mathfrak{I}_{\underline{\kappa}}$
    	when $\beta$ is singular,
    	thus so does the global
    	Fourier coefficient.
    	So there is no need to consider
    	the range of values of
    	$E_{\beta,\ell}(1_8,f_{\phi_\ell^+})$
    	for singular $\beta$
    	for $\ell|N$.
    \end{remark}

	\subsection{Ramified place $p$}
	\label{ramified place p}
	Now consider the case $\ell=p$.
	
	Recall that we have fixed an admissible character
	\[
	\underline{\kappa}
	=
	\underline{\kappa}_\mathrm{alg}
	\times
	\underline{\kappa}_\mathrm{f}
	=
	(k_0,k_1,k_2)
	\times
	(\kappa_0,\kappa_1,\kappa_2)
	\colon
	T_{G_1}(\mathbb{Z}_p)
	\rightarrow
	\overline{\mathbb{Q}}_p^\times.
	\]

	For any finite order character
	$\chi$ of $\mathbb{Z}_p^\times$,
	we write
	$p^{c'(\chi)}$
	for its conductor
	and
	$C(\chi)=p^{c(\chi)}=
	p^{\max(1,c'(\chi))}$.
	We also write
	$c(\underline{\kappa})=
	\max(c(\kappa_1),c(\kappa_2))$.
	In the following we also write
	$c$ for $c(\underline{\kappa})$
	when no confusion is possible.

	\subsubsection{Choice of local sections}
	We first choose sections in
	$\mathcal{S}(W_i^-(\mathbb{Q}_p))$
	for $i=1,2$.

	We write
	$\widetilde{\kappa}_2
	=\kappa_2(\frac{\cdot}{p})^{c'(\kappa_2)}$.
	We define a 
	Schwartz-Bruhat function 
	$\alpha_{2,\underline{\kappa},p}$
	on
	$\mathrm{Sym}_{2\times 2}
	(\mathbb{Q}_p)$
	whose Fourier transform is given by
	\[
	\widehat{\alpha}_{2,\underline{\kappa},p}
	(\beta_2)
	=
	1_{\mathrm{GL}_2(\mathbb{Z})}(\beta_2)
	\times
	(\kappa_1\kappa_2^{-1})(\mathrm{det}_1(\beta_2))
	\times
	(\widetilde{\kappa}_2)
	(\mathrm{det}_2(\beta_2)).
	\]
	Similarly, we define also a 
	Schwartz-Bruhat function
	$\alpha_{\underline{\kappa},p}$ on
	$\mathrm{Sym}_{4\times 4}
	(\mathbb{Q}_p)$
	whose Fourier transform is given by
	\[
	\widehat{\alpha}_{\underline{\kappa},p}
	(\begin{pmatrix}
	\beta_1 & \beta_0 \\
	\beta_0^t & \beta_2
	\end{pmatrix})
	=\widehat{\alpha}_{2,\underline{\kappa},p}
	(\beta_2)
	1_{p^c\mathrm{Sym}_{2\times 2}
		(\mathbb{Z}_p)}(\beta_1)
	1_{p^c\mathrm{M}_{2\times 2}
		(\mathbb{Z}_p)}(\beta_0).
	\]

	Moreover, we fix
	$w_0\in
	W_1^-(\mathbb{Z}_p)$
	such that
	its image in
	$W_1^-(\mathbb{Z}/p)$
	is a matrix of full rank
	($=2$) and
	$w_0\eta_Uw_0^t=0$
	(it is easy to see that such
	$w_0$ exists).
	We fix also some integer
	$s\in\mathbb{Z}$
	such that
	$s+c(\underline{\kappa})\leq0$
	(this gives the freedom in choosing 
	the support of the sections defined below.
	For application we can take
	$s=-c(\underline{\kappa})$).
	We define $\phi_{i,\underline{\kappa},p}
	\in\mathcal{S}(W_i^-(\mathbb{Q}_p))$
	for $i=1,2$
	as follows:
	\begin{gather}
	\label{Schwartz section at p}
	\phi_{1,\underline{\kappa},p}(y_1)
	=
	1_{w_0+C(\chi)W_1^-(\mathbb{Z}_p)}(y_1/p^s),
	\nonumber
	\\
	\phi_{2,\underline{\kappa},p}
	(y_2)
	=1_{W_2^-(\mathbb{Z}_p)}(y_2/p^s)
	\widehat{\alpha}_{2,\underline{\kappa},p}
	(\frac{1}{2}y_2\eta_Uy_2^t/2p^{2s})
	1_{p^c\mathrm{M}_{2\times 2}
		(\mathbb{Z}_p)}
	(y_2\eta_Uw_0^t/p^{2s}),
	\\
	\phi_{\underline{\kappa},p}^+(y)
	=\phi_{1,\underline{\kappa},p}(y_1)
	\phi_{2,\underline{\kappa},p}(y_2).
	\nonumber
	\end{gather}

	To determine the Siegel section
	$f_{\phi_{\underline{\kappa},p}}$
	associated to $\phi_{\underline{\kappa},p}$,
	we need some results on
	local densities.
	For any
	$\beta_2\in
	\mathrm{Sym}_{2\times 2}
	(\mathbb{Z}_p)
	\cap\mathrm{GL}(2,\mathbb{Z}_p)$,
	we define the following integral
	($c=c(\underline{\kappa})$):
	\[
	\alpha_{2,\underline{\kappa},p}'(\beta_2)=
	\int_{W_1^-(\mathbb{Q}_p)}
	\int_{W_2^-(\mathbb{Q}_p)}
	1_{W_2^-(\mathbb{Z}_p)}(y_2)
	1_{p^c}
	(\frac{1}{2}y_2\eta_Uy_2^t-\beta_2)
	1_{p^c}(y_2\eta_Uw_0^t)
	1_{w_0+p^cW_1^-(\mathbb{Z}_p)}(y_1)
	dy_2dy_1.
	\]
	Then one can show that
	\begin{lemma}
		The integral
		$\alpha_{2,\underline{\kappa},p}'(\beta_2)
		\neq0$
		if and only if
		$\mathrm{det}(\beta_2)
		\mathrm{det}(\eta_U)$ is a square in 
		$\mathbb{Z}_p^\times$
		(i.e.,
		$(\frac{\mathrm{det}(\beta_2)}{p})=
		(\frac{\mathrm{det}(\eta_U)}{p})$).
		Moreover, if
		$(\frac{\mathrm{det}(\beta_2)}{p})=
		(\frac{\mathrm{det}(\eta_U)}{p})$,
		we have
		$\alpha_{2,\underline{\kappa},p}'(\beta_2)
		=p^{-23c}(1-(\frac{-1}{p})\frac{1}{p})$.
	\end{lemma}
    \begin{proof}
    	We define a function
    	$1_0(y)$ on
    	$W_2^-(\mathbb{Z}/p^c)$
    	such that
    	$1_0(y)=1$ if
    	$y\equiv 0(\mathrm{mod}\,p^c)$
    	and $=0$ otherwise.
    	Clearly, we have
    	\[
    	\alpha_{2,\underline{\kappa},p}'(\beta_2)
    	=
    	p^{-24c}
    	\sum_{y_2\in W_2^-(\mathbb{Z}/p^c)}
    	1_0(y_2\eta_Uy_2^t-\beta_2)
    	1_{0}(y_2\eta_Uw_0^t).
    	\]

    	One can show that there exists an element
    	$\widetilde{w}_0\in
    	W_1^-(\mathbb{Z}_p)$
    	such that
    	$\widetilde{w}_0\eta_U
    	w_0^t$ has its image in
    	$\mathrm{GL}(2,\mathbb{Z}/p)$
    	in the projection 
    	$\mathrm{M}_{2\times 2}
    	(\mathbb{Z}_p)
    	\rightarrow
    	\mathrm{M}_{2\times 2}
    	(\mathbb{Z}/p)$.
    	Then the row vectors of
    	$w_0$ and $\widetilde{w}_0$
    	form a $\mathbb{Z}_p$-submodule
    	$U_1(\mathbb{Z}_p)$
    	of $U(\mathbb{Z}_p)$
    	on which the symmetric bilinear form $\eta_U$
    	is of the form
    	$\eta_{U_1}=\begin{pmatrix}
    	\ast & \ast \\
    	\ast & 0
    	\end{pmatrix}
    	\in\mathrm{GL}_4(\mathbb{Z}_p)$
    	(in $2\times2$-blocks)
    	under certain basis.
    	One can then show that this form
    	$\eta_{U_1}$
    	is $\mathbb{Z}_p$-equivalent to
    	the identity matrix $1_4$.
    	Now consider any
    	rank $2$-subspace
    	$U_2(\mathbb{Z}_p)$ of
    	$U(\mathbb{Z}_p)$
        whose vectors are orthogonal to
        the row vectors of
        $w_0$.
        Then the symmetric bilinear form
        $\eta_U$ on
        $U_2$ becomes
        $\eta_{U_2}\in\mathrm{GL}(2,\mathbb{Z}_p)$
        with
        $\mathrm{det}(\eta_{U_2})=
        \mathrm{det}(\eta_U)
        (\mathrm{mod}\,(\mathbb{Z}_p^\times)^2)$
        since $\mathrm{det}(\eta_{U_1})
        =1(\mathrm{mod}\,(\mathbb{Z}_p^\times)^2)$.
        Now the above sum is equal to
        \[
        \alpha_{2,\underline{\kappa},p}'(\beta_2)
        =
        p^{-24c}
        \sum_{y_1\in(U_2\otimes V_1^-)(\mathbb{Z}/p^c)}
        1_0(\frac{1}{2}y_2\eta_{U_2}y_2^t-\beta_2)
        =
        p^{-24c}
        \sum_{y_1\in\mathrm{GL}(2,\mathbb{Z}/p^c)}
        1_0(\frac{1}{2}y_2\eta_{U_2}y_2^t-\beta_2).
        \]
        
        We define local density
        for any $t\in\mathbb{Z}_{>0}$ as
        $A_t(\beta_2)
        =
        \#
        \{
        y_2\in\mathrm{GL}(2,\mathbb{Z}/p^t)|
        y_2\eta_{U_2}y_2^t=2\beta_2(\mathrm{mod}\,
        p^t)
        \}$.
        So $\alpha_{2,\underline{\kappa},p}'(\beta_2)
        =p^{-24c}A_c(\beta_2)$.
        It is clear that if
        $\mathrm{det}(\beta_2)\neq
        \mathrm{det}(\eta_{U_2})
        (\mathrm{mod}\,(\mathbb{Z}_p^\times)^2)$,
        then
        $A_t(\beta_2)=0$.
        Otherwise,
        $A_t(\beta_2)
        =\#\mathrm{O}(\eta_{U_2},\mathbb{Z}/p^t)
        =p^{t-1}
        \#\mathrm{O}(\eta_{U_2},\mathbb{Z}/p)$.
        Note that $N$ is a square in $\mathbb{Z}_p$ 
        by assumption,
        thus $\mathrm{det}(\eta_U)=1
        (\mathrm{mod}\,(\mathbb{Z}_p^\times)^2)$,
        and therefore
        we have
        $\#\mathrm{O}(1_2,\mathbb{Z}/p)$
        which is equal to
        $p-(\frac{-1}{p})$
        by the formulas given preceding
        Lemma \ref{Gauss sum for k_1>=0}.
        This gives the desired formula in
        the lemma.
    \end{proof}

    As a result, we have
    \[
    \alpha_{2,\underline{\kappa},p}'(\beta_2)
    =
    \frac{p^{-23c}}{2}
    (1-(\frac{-1}{p})\frac{1}{p})
    (1+(\frac{\mathrm{det}(\beta_2)}{p})).
    \]

	We define auxiliary characters
	$\kappa_i'=\kappa_i(\frac{\cdot}{p})$:
	for $i=1,2$
	and put
	$\underline{\kappa}'=(\kappa_1',\kappa_2')$.
	\begin{proposition}\label{Siegel section at p}
		Write
		$g=\begin{pmatrix}
		A & B \\
		C & D
		\end{pmatrix}
		\in G_4(\mathbb{Q}_p)$.
		Then
		$f_{\phi^+_{\underline{\kappa},p}}(g)$
		satisfies:
		if $\mathrm{det}(C)\neq0$,
		\[
		f_{\phi^+_{\underline{\kappa},p}}
		(g)
		=
		\xi(\mathrm{det}(C))
		|\mathrm{det}(C)|_p^{-3}
		\frac{p^{-24s-13c}}{2}
		(1-(\frac{-1}{p})\frac{1}{p})
		(\alpha_{\underline{\kappa},p}(p^{2s}C^{-1}D)+
		\alpha_{\underline{\kappa}',p}(p^{2s}C^{-1}D));
		\]
		otherwise,
		it is zero.
	\end{proposition}
    \begin{proof}
    	We can write
    	$\phi^+_{\underline{\kappa},p}(y)$
    	as a finite
    	linear combination of characteristic functions
    	$\sum_{i}a_i
    	1_{p^sr_i+p^{s+c}W_4^-(\mathbb{Z}_p)}(y)$
    	with $a_i\neq0$ and
    	$r_i\neq r_j$ for $i\neq j$.
    	Moreover,
    	the image of each $r_i$ in
    	$W_4^-(\mathbb{Z}/p)$ is
    	of rank $4$.
    	Then by the same reasoning as
    	in Corollary 
    	\ref{support of Siegel section},
    	we see that the function
    	$\omega_{W_4^-}(g)
    	\phi^+_{\underline{\kappa},p}(0)$
    	vanishes unless
    	$\mathrm{det}(C^tD)\neq0$ and
    	$\omega_{W_4^-}(J_8u(D))
    	\phi^+_{\underline{\kappa},p}(0)$
    	vanishes unless
    	$D\in p^{-(2s+c)}\mathrm{Sym}_{4\times 4}
    	(\mathbb{Z}_p)$.
    	By definition of
    	$f_{\phi^+_{\underline{\kappa},p}}$,
    	if $\mathrm{det}(C)\neq0$,
    	we have
    	\[
    	f_{\phi^+_{\underline{\kappa},p}}
    	(g)
    	=\xi(\mathrm{det}(C))
    	|\mathrm{det}(C)|_p^{-3}
    	\omega_{W_4^-}
    	(u(C^{-1}D))
    	\phi_{\underline{\kappa},p}^+(0).
    	\]
    	
    	We write
    	$f(D)=\omega_{W_4^-}
    	(J_8u(D))
    	\phi^+_{\underline{\kappa},p}(0)$.
    	To evaluate $f(D)$,
    	we first look at its Fourier transform
    	$\widehat{f}(\beta)$.
    	We have
    	\begin{align*}
    	\widehat{f}(\beta)
    	&
    	=\int_{W_4^-(\mathbb{Q}_p)}
    	\phi^+_{\underline{\kappa},p}(y)
    	\mathbf{e}_p(\frac{1}{2}
    	\mathrm{tr}\,y^tDy\eta_U)dy
    	\int_{p^{-(2s+c)}
    		\mathrm{Sym}_{4\times 4}
    		(\mathbb{Z}_p)}
    	\mathbf{e}_p(-\mathrm{tr}(\beta D))dD
    	\\
    	&
    	=
    	p^{10(2s+c)}
    	\int_{W_4^-(\mathbb{Q}_p)}
    	\phi^+_{\underline{\kappa},p}(y)
    	1_{p^{2s+c}\mathrm{Sym}_{4\times 4}
    		(\mathbb{Z}_p)}
    	(\frac{1}{2}y\eta_Uy^t-\beta)dy
    	\\
    	&
    	=
    	p^{10(2s+c)-24s}
    	\widehat{\alpha}_{2,\underline{\kappa},p}
    	(p^{-2s}\beta_2)
    	1_{p^c\mathrm{M}_{2\times 2}
    		(\mathbb{Z}_p)}
    	(p^{-2s}\beta_0)
    	1_{p^c\mathrm{Sym}_{2\times 2}
    		(\mathbb{Z}_p)}
    	(p^{-2s}\beta_1)
    	\\
    	&
    	\quad
    	\times
    	\int_{W_2^-(\mathbb{Z}_p)}
    	\int_{w_0+p^cW_1^-(\mathbb{Z}_p)}
    	1_{p^c\mathrm{Sym}_{2\times 2}
    		(\mathbb{Z}_p)}
    	(\frac{1}{2}y_2\eta_Uy_2^t-p^{-2s}\beta_2)
    	1_{p^c\mathrm{M}_{2\times 2}
    		(\mathbb{Z}_p)}
    	(y_2\eta_Uw_0^t)
    	dy_1dy_2
    	\\
    	&
    	=
    	p^{10(2s+c)-24s}
    	\widehat{\alpha}_{2,\underline{\kappa},p}
    	(p^{-2s}\beta_2)
    	1_{p^c\mathrm{M}_{2\times 2}
    		(\mathbb{Z}_p)}
    	(p^{-2s}\beta_0)
    	1_{p^c\mathrm{Sym}_{2\times 2}
    		(\mathbb{Z}_p)}
    	(p^{-2s}\beta_1)
    	\alpha_{2,\underline{\kappa},p}'(p^{-2s}\beta_2)
    	\\
    	&
    	=
    	\frac{p^{10(2s+c)-24s-23c}}{2}
    	(1-(\frac{-1}{p})\frac{1}{p})
    	(\widehat{\alpha}_{\underline{\kappa},p}(p^{-2s}\beta)+
    	\widehat{\alpha}_{\underline{\kappa}',p}(p^{-2s}\beta)).
    	\end{align*}
    	Therefore, taking the inverse Fourier transform,
    	we get
    	\[
    	f(D)=
    	\frac{p^{-24s-13c}}{2}
    	(1-(\frac{-1}{p})\frac{1}{p})
    	(\alpha_{\underline{\kappa},p}(p^{2s}D)+
    	\alpha_{\underline{\kappa}',p}(p^{2s}D)).
    	\]
    	This finishes the proof of the lemma.   	
    \end{proof}

    As for $f_{\phi_{\underline{\kappa},p}^d}
    \in\mathrm{Ind}^{G_4(\mathbb{Q}_p)}_{P^d(\mathbb{Q}_p)}(\xi)$,
    we have
    \begin{corollary}
    	\label{support of f^d at p}
    	For any
    	$g_1\in G_1^1(\mathbb{Q}_p)$,
    	if $g_1+1_4\notin
    	p^{-(2s+c)}
    	\mathrm{M}_{4\times 4}
    	(\mathbb{Z}_p)$,
    	$f_{\phi_{\underline{\kappa},p}^d}(\iota(g_1,1))
    	=0$.
    	Otherwise,
    	\[
    	f_{\phi_{\underline{\kappa},p}^d}(\iota(g_1,1))
    	=\frac{p^{-24s-13c}}{2}
    	(1-(\frac{-1}{p})\frac{1}{p})
    	(\alpha_{\underline{\kappa},p}(-p^{2s}
    	\frac{g_1+1}{2}\Delta)
    	+\alpha_{\underline{\kappa}',p}(-p^{2s}
    	\frac{g_1+1}{2}\Delta)).
    	\]
    	Recall that
    	$\Delta=\begin{pmatrix}
    	0 & 1_2 \\
    	1_2 & 0
    	\end{pmatrix}$.
    \end{corollary}
    \begin{proof}
    	By definition of
    	$f_{\phi_{\underline{\kappa},p}^d}$,
    	if $\mathrm{det}(\Delta(g_1-1))=0$,
    	then
    	$f_{\phi_{\underline{\kappa},p}^d}(\iota(g_1,1))=0$.
    	Otherwise,
    	\[
    	f_{\phi_{\underline{\kappa},p}^d}(\iota(g_1,1))
    	=
    	f_{\phi_{\underline{\kappa},p}}
    	(\begin{pmatrix}
    	\ast & \ast \\
    	\Delta(g_1-1) & \frac{1}{2}\Delta(g_1+1)\Delta
    	\end{pmatrix}).
    	\]
    	By the proof of the above lemma,
    	the non-vanishing of
    	$f_{\phi_{\underline{\kappa},p}^d}(\iota(g_1,1))$
    	implies that
    	$\frac{1}{2}(g_1-1)^{-1}(g_1+1)\Delta
    	\in p^{-(2s+c)}\mathrm{M}_{4\times 4}
    	(\mathbb{Z}_p)$.
    	This give the first part of the corollary.
    	Now assume that
    	$g_1+1\in p^{-(2s+c)}\mathrm{M}_{4\times 4}
    	(\mathbb{Z}_p)$,
    	recall also that
    	$-2s-c>0,-2s-2c\geq0$,
    	\[
    	\frac{1}{2}(g_1-1)^{-1}(g_1+1)\Delta
    	=
    	\frac{1}{2}\Delta
    	-\frac{1}{2}(1-\frac{g_1+1}{2})^{-1}\Delta
    	=-(\frac{g_1+1}{2}
    	+(\frac{g_1+1}{2})^2+\cdots)\Delta.
    	\]

    	It is easy to see that
    	$\alpha_{\underline{\kappa},p}(D)
    	=\alpha_{\underline{\kappa},p}(D+D')$,
    	$\alpha_{\underline{\kappa}',p}(D)
    	=\alpha_{\underline{\kappa}',p}(D+D')$
    	for any 
    	$D'\in\mathrm{Sym}_{4\times 4}
    	(\mathbb{Z}_p)$.
    	Now that
    	$p^{2s}(g_1+1)^n\in
    	p^{2s-n(2s+c)}\mathrm{M}_{4\times 4}
    	(\mathbb{Z}_p)$
    	and
    	$2s-n(2s+c)\geq0$
    	for $n\geq2$.
    	This show that
    	\[
    	\alpha_{\underline{\kappa},p}
    	(p^{2s}\frac{1}{2}(g_1-1)^{-1}(g_1+1)\Delta)
    	=
    	\alpha_{\underline{\kappa},p}
    	(-p^{2s}\frac{g_1+1}{2}\Delta)
    	\]
    	and similarly for
    	$\alpha_{\underline{\kappa}',p}$,
    	which proves the lemma.   	
    \end{proof}

	\subsubsection{Local zeta integral}
	Next we compute the local zeta integral
	$Z_p(\varphi_p,\varphi_p^\vee,
	f_{\phi^d_{\underline{\kappa},p}})$.
	The computation follows closely that of
	\cite[Section 5]{Liu2015}.
	Let's first 
	recall the notions of $\mathbb{U}_p$ operators
	and Jacquet modules.
	For any $\underline{t}=(t_1,t_2)\in
	\mathbb{Z}^2$, we set
	$p^{\underline{t}}=\mathrm{diag}
	(p^{t_1},p^{t_2},p^{-t_1},p^{-t_2})
	\in G_1(\mathbb{Q})$.
	Let $C^+$ be the set
	of $\underline{t}\in\mathbb{Z}^2$ such that
	$t_1\geq t_2\geq0$.
	Then for any $\underline{t}\in C^+$,
	we can define the 
	$\mathbb{U}_p$-operator
	$U_{p,\underline{t}}$
	acting on
	$f\in\mathcal{A}(
	G_1(\mathbb{Q})\backslash
	G_1(\mathbb{A})/\widehat{\Gamma})_{\underline{t}}$
	as follows:
	\[
	U_{p,\underline{t}}f(g)
	=p^{\langle
		\underline{t}'+2\rho_{\mathrm{Sp}(V),c},
		\underline{t}
		\rangle}
	\int_{N_{G_1}(\mathbb{Z}_p)}
	f(gup^{\underline{t}})du
	\]
	where
	$\rho_{\mathrm{Sp}(V),c}$
	is half the sum of the positive compact roots of
	$\mathfrak{g}_1$.
	We then define
	$U_{p,1}=U_{p,(1,0)}$,
	$U_{p,2}=U_{p,(1,1)}$,
	$U_p=U_{p,1}U_{p,2}$.
	Now for any nearly holomorphic form $\phi$,
	the limit
	$\lim\limits_{n\rightarrow+\infty}U_p^{n!}\phi$
	exists and is denoted by $e\phi$.
	This operator $e$ is in fact a polynomial in $U_p$.

    Let $\pi\simeq\otimes_v\pi_v$
	be an antiholomorphic
	irreducible cuspidal automorphic representation
	of $[G_1]$.
	Suppose that $\overline{\pi}_\infty
	\simeq\mathcal{D}_{\underline{k}}$
	is a holomorphic discrete series.
	We write $e\phi$ for the ordinary projection
	of any $\phi$ in
	$\pi$ or $\overline{\pi}$.
	It can be shown 
	\cite[Corollary 3.10.3]{Liu2015NearlyOverconvergent} that
	$e\pi$
	is inside the subspace of antiholomorphic forms in 
	$\pi$.
	Moreover,
	if $e\pi$ is non-zero,
	$\pi_p$ is isomorphic to a 
	composition factor of certain principal series,
	the projection of $e\pi$
	to $\pi_p$
	is of dimension $1$,
	the action of the $\mathbb{U}_p$ operators
	on $\bigcap_{\underline{c}\in C^+}
	U_{p,\underline{c}}(\pi_p)$
	is semi-simple
	(\textit{cf}. \cite[Section 5.5]{Liu2015}).
	When $e\pi$ is non-zero,
	there exist
	$\mathfrak{a}_1,\mathfrak{a}_2\in
	\mathcal{O}_{\overline{\mathbb{Q}}_p}^\times$
	such that
	$U_{p,\underline{t}}$
	acts on
	$e\pi$
	by the scalar 
	$\mathfrak{a}_1^{t_1}\mathfrak{a}_2^{t_2}$
	for all
	$\underline{t}=(t_1,t_2)\in C^+$.
	Now we set
	$\alpha_1=p^{-(k_1-1)}\mathfrak{a}_1$,
	$\alpha_2=p^{-(k_2-2)}\mathfrak{a}_2$
	and define
	characters
	$\theta_1$ and $\theta_2$
	of $\mathbb{Q}_p^\times$
	characterized by
	\[
	\theta_i(x)=
	\begin{cases*}
	\kappa_i(x) & if $x\in\mathbb{Z}_p^\times$ \\
	\alpha_i & if $x=p$.
	\end{cases*}
	\]
	Then by \cite[Section 5.5]{Liu2015},
	$\pi_p$ embeds in
	the normalized induction
	$\mathrm{Ind}^{G_1(\mathbb{Q}_p)}_{B_{G_1}(\mathbb{Q}_p)}
	(\theta)$
	where $\theta=(\theta_0,
	\theta_1,\theta_2)$
	is a character of
	$T_{G_1}(\mathbb{Q}_p)$
	(we do not precise $\theta_0$).
    Moreover their ordinary subspaces coincide,
	both being of dimension $1$.
	Moreover,
	$(\overline{\pi})_p=(\pi^\vee)_p$
	embeds in
	$\mathrm{Ind}^{G_1(\mathbb{Q}_p)}_{B_{G_1}(\mathbb{Q}_p)}
	(\theta^{-1})$.	
	We now take
	a vector
	$\varphi\in\pi$
	such that
	its ordinary projection
	$e\varphi=U_{p,\underline{c}}\varphi\neq0$
	for $\underline{c}\gg0$.
	Then it is easy to verify that
	the vector
	$\widetilde{\varphi}_{\underline{c}}(g)
	=
	\langle
	\varphi_{\mathrm{Sp}(V)\times\mathrm{Sp}(V')}
	(\mathcal{E}_{\kappa,\underline{c}},
	\mathfrak{e}_\mathrm{can})
	(\cdot,g),
	e\varphi
	\rangle$
	is also ordinary,
	therefore is a multiple of
	$e\varphi$:
	\[
	\widetilde{\varphi}_{\underline{c}}
	=C_{\kappa,\pi}e\varphi.
	\]
	We next evaluate this constant
	$C_{\kappa,\pi}
	\in\mathbb{C}$.
	The strategy is to
	pair the above two automorphic forms
	$\widetilde{\varphi}_{\underline{c}}$
	and $\varphi$ against some other form
	$\varphi'\in\overline{\pi}$
	and then use some local models to
	choose some special vectors
	$\varphi$ and $\varphi'$ to
	get the value of
	$C_{\kappa,\pi}$.
	By definition, we have
	\[
	\langle
	\widetilde{\varphi}_{\underline{c}},
	\varphi'
	\rangle
	=
	\int_{N_{G_1}(\mathbb{Z}_p)}
	\int_{G_1(\mathbb{A})}
	\int_{[G_1]}
	f_{\phi_{\underline{\kappa}}^d}
	(\iota(g',1))
	\varphi(gg'up^{\underline{c}})
	\varphi'(g)
	dgdg'du.
	\]
	We write
	the local zeta integrals 
	of the above global integral for places
	$v\nmid p$ as
	$\widetilde{L}_v(1,\mathrm{St}(\pi)\otimes\xi)
	\langle
	\varphi_v,\varphi'_v
	\rangle$ and
	the product of the local $L$-factors
	is denoted by
	$\widetilde{L}^p
	(1,\mathrm{St}(\pi)\otimes\xi)=
	\prod_{v\nmid p}\widetilde{L}_v
	(1,\mathrm{St}(\pi)\otimes\xi)$.
	Then we have
	\[
	\langle
	\widetilde{\varphi}_{\underline{c}},
	\varphi'
	\rangle
	=\widetilde{L}^p(1,\mathrm{St}(\pi)\otimes\xi)
	\int_{N_{G_1}(\mathbb{Z}_p)}
	\int_{G_1(\mathbb{Q}_p)}
	\int_{[G_1]}
	f_{\phi_{\underline{\kappa},p}^d}
	(\iota(g',1))
	\varphi(gg'up^{\underline{c}})
	\varphi'(g)
	dgdg'du.
	\]
	From this we see that to evaluate
	$C_{\kappa,\pi}$,
	it suffices to evaluate the quotient
	\[
	C_{\kappa,\pi}
	=
	\frac{
	\int_{N_{G_1}(\mathbb{Z}_p)}
	\int_{G_1^1(\mathbb{Q}_p)}
	\int_{[G_1]}
	f_{\phi_{\underline{\kappa},p}^d}
	(\iota(g',1))
	\varphi(gg'up^{\underline{c}})
	\varphi'(g)
	dgdg'du
    }{
    \int_{N_{G_1}(\mathbb{Z}_p)}
    \int_{[G_1]}
    \varphi(gup^{\underline{c}})
    \varphi'(g)
    dgdu
    }.
	\]
	Now we reduce the above integrals to local situations
	and use certain local model of
	$\pi_p$ to do the computation.
	We take
	$\varphi'_p\in
	\mathrm{Ind}^{G_1(\mathbb{Q}_p)}_{B_{G_1}(\mathbb{Q}_p)}
	(\theta^{-1})$
	supported on
	$B_{G_1}(\mathbb{Q}_p)
	J_4N_{G_1}(\mathbb{Q}_p)$.
	We then choose
	some sufficiently small open compact
	subgroup
	$K$ of $G_1(\mathbb{Q}_p)$
	such that
	the vectors
	$\int_{N_{G_1}(\mathbb{Z}_p)}
	\varphi'(gp^{-\underline{c}}u^{-1})du$
	and
	$\int_{G_1(\mathbb{Q}_p)}
	\int_{N_{G_1}(\mathbb{Z}_p)}
	f_{\phi_{\underline{\kappa},p}}(\iota(g',1))
	\varphi'(gp^{-\underline{c}}u^{-1}(g')^{-1})
	dudg'$
	for all
	$\underline{c}\gg0$
	are all fixed by translation of
	$K$.
	Then we take
	$\varphi\in
	\mathrm{Ind}^{G_1(\mathbb{Q}_p)}_{B_{G_1}(\mathbb{Q}_p)}
	(\theta)$
	supported on
	$B_{G_1}(\mathbb{Q}_p)K$
	and
	taking value $1$ on $K$.
	The local pairing between
	$\varphi$ and
	$\varphi'$
	is given by
	$\langle
	\varphi,\varphi'
	\rangle=
	\int_{G_1(\mathbb{Z}_p)}
	\varphi(g)\varphi'(g)dg$.
	So we see that
	\begin{align*}
	C_{\kappa,\pi}
	&
	=
	\frac{
		\int_{N_{G_1}(\mathbb{Z}_p)}
		\int_{G_1^1(\mathbb{Q}_p)}
		\int_{G_1(\mathbb{Z}_p)}
		f_{\phi_{\underline{\kappa},p}^d}
		(\iota(g',1))
		\varphi(gg'up^{\underline{c}})
		\varphi'(g)
		dgdg'du
	}{
		\int_{N_{G_1}(\mathbb{Z}_p)}
		\int_{G_1(\mathbb{Z}_p)}
		\varphi(gup^{\underline{c}})
		\varphi'(g)
		dgdu
	}
	\\
	&
	=
	\frac{
		\int_{N_{G_1}(\mathbb{Z}_p)}
		\int_{G_1^1(\mathbb{Q}_p)}
		\int_K
		f_{\phi_{\underline{\kappa},p}^d}
		(\iota(g',1))
		\varphi(g)
		\varphi'(gp^{-\underline{c}}u^{-1}(g')^{-1})
		dgdg'du
	}{
		\int_{N_{G_1}(\mathbb{Z}_p)}
		\int_K
		\varphi(g)
		\varphi'(gp^{-\underline{c}}u^{-1})
		dgdu
	}
	\\
	&
	=
	\frac{
		\int_{N_{G_1}(\mathbb{Z}_p)}
		\int_{G_1^1(\mathbb{Q}_p)}
		\int_K
		f_{\phi_{\underline{\kappa},p}^d}
		(\iota(g',1))
		\varphi'(p^{-\underline{c}}u^{-1}(g')^{-1})
		dg'du
	}{
		\int_{N_{G_1}(\mathbb{Z}_p)}
		\varphi'(p^{-\underline{c}}u^{-1})
		du
	}
	\\
	&
	=
	\int_{G_1^1(\mathbb{Q}_p)}
	f_{\phi_{\underline{\kappa},p}^d}
	(\iota(g',1))
	\varphi'((g')^{-1})
	dg'.
	\end{align*}
	
	By Corollary
	\ref{support of f^d at p},
	we can evaluate the last integral as follows.
	We denote
	\[
	I(\underline{\kappa})
	=\int_{G_1^1(\mathbb{Q}_p)}
	\alpha_{\underline{\kappa},p}
	(-p^{-2s}\frac{g+1}{2}\Delta)
	\varphi'(g^{-1})dg,
	\quad
	I(\underline{\kappa}')
	=\int_{G_1^1(\mathbb{Q}_p)}
	\alpha_{\underline{\kappa}',p}
	(-p^{-2s}\frac{g+1}{2}\Delta)
	\varphi'(g^{-1})dg
	\]
	Note that the support of $\varphi'$
	is $B_{G_1}(\mathbb{Q}_p)w_4
	N_{G_1}(\mathbb{Q}_p)$,
	we can write an element 
	$g\in G_1^1(\mathbb{Q}_p)$ 
	in this set as
	\[
	g=\begin{pmatrix}
	a & b \\
	c & d
	\end{pmatrix}
	=m(A)u(B)J_4u(B').
	\]
	As such,
	$\varphi'(g^{-1})
	=\theta^{-1}(-A^{-t})
	=\theta(c)$.
	Therefore, we have
	\begin{align*}
	I(\underline{\kappa})
	&
	=
	\int_{g+1\in p^{-2s}\mathrm{M}_{4\times 4}
		(\mathbb{Z}_p)}
	\alpha_{\underline{\kappa},p}(-\frac{p^{2s}}{2}
	\begin{pmatrix}
	b & a+1 \\
	d+1 & c
	\end{pmatrix})
	\theta(c)
	dg
	\\
	&
	=p^{-6c}
	\int_{g+1\in p^{-2s}\mathrm{M}_{4\times 4}
		(\mathbb{Z}_p)}
	1_{p^{-c}\mathrm{M}_{2\times 2}
		(\mathbb{Z}_p)}
	(-p^{2s}b)
	1_{p^{-c}\mathrm{M}_{2\times 2}
		(\mathbb{Z}_p)}
	(-p^{2s}(a+1))
	\alpha_{1,\underline{\kappa},p}(-\frac{p^{2s}}{2}c)
	\theta(c)dg
	\\
	&
	=
	p^{16s}\theta^{-1}(-2p^{-2s})
	\int_{\mathrm{M}_{2\times 2}
		(\mathbb{Q}_p)}
	\alpha_{2,\underline{\kappa},p}(c)\theta(c)
	dc.
	\end{align*}
	(similar formula for
	$I(\underline{\kappa}')$)
	The evaluation of
	this last integral is given in the next lemma,
	whose proof follows quite closely
	\cite[Section 5.7]{Liu2015}.
	For any finite order character
	$\chi$ of $\mathbb{Z}_p^\times$,
	we define
	$\chi^\circ=1$ if
	$\chi$ is the trivial character
	and
	$\chi^\circ=0$ otherwise.
	$G(\chi)$ is the Gauss sum of
	$\chi$.
	Recall that
	$\theta_i(p)=\alpha_i$,
	$\theta_i|_{\mathbb{Z}_p^\times}
	=\kappa_i$
	and
	$\widetilde{\kappa}_2
	=\kappa_2(\frac{\cdot}{p})^{c'(\kappa_2)}$.
	\begin{lemma}
		We have
		\begin{enumerate}
			\item 
			\begin{align*}
			\int_{\mathrm{M}_{2\times 2}
				(\mathbb{Q}_p)}
			\alpha_{2,\underline{\kappa},p}(c)\theta(c)
			dc
			&
			=
			\bigg(
			(1-\frac{1}{p})
			(-1)^{\kappa_1^\circ}
			\frac{1-\kappa_1^\circ\alpha_1^{-1}}
			{1-\kappa_1^\circ\alpha_1p^{-1}}
			G(\kappa_1)
			\alpha_1^{-c'(\kappa_1)}
			\bigg)
			\\
			&
			\quad
			\times
			\bigg(
			(1-\frac{1}{p})
			(-1)^{\kappa_2^\circ}
			\frac{1-\kappa_2^\circ p^{-1}}
			{1-\kappa_2^\circ\alpha_2p^{-1}}
			G(\widetilde{\kappa}_2)
			\alpha_2^{-c'(\kappa_2)}
			\varepsilon(p^{c'(\kappa_2)})
			p^{-c'(\widetilde{\kappa}_2)}
			\bigg);
			\end{align*}
			
			\item 
			\[
			\int_{\mathrm{M}_{2\times 2}
				(\mathbb{Q}_p)}
			\alpha_{2,\underline{\kappa}',p}(c)\theta(c)
			dc=0.
			\]
		\end{enumerate}		
	\end{lemma}
    \begin{proof}
    	By the support of $\varphi'$,
    	only those upper triangular matrices
    	$c$ contribute to the integral in the lemma.
    	Moreover,
    	$\alpha_{2,\underline{\kappa},p}(c)
    	=\alpha_{2,\underline{\kappa},p}(c')$
    	where $c-c'\in\mathrm{Sym}_{2\times 2}
    	(\mathbb{Z}_p)$.
    	Thus we can evaluate the integral
    	by considering only those diagonal
    	matrices
    	$c=\mathrm{diag}(c_1,c_2)$.
    	First we compute 
    	$\alpha_{2,\underline{\kappa},p}
    	(c)$.
    	Write a matrix
    	$\begin{pmatrix}
    	x & y \\
    	y & z
    	\end{pmatrix}\in
    	\mathrm{Sym}_{2\times 2}
    	(\mathbb{Z}_p)$: 
    	\begin{align*}
    	\alpha_{2,\underline{\kappa},p}(c)
    	&
    	=\int_{\mathrm{Sym}_{2\times 2}
    		(\mathbb{Z}_p)}
    	(\kappa_1\kappa_2^{-1})(x)
    	\widetilde{\kappa}_2(xz-y^2)
    	\mathbf{e}_p(-(c_1x+c_2z))
    	dxdydz
    	\\
    	&
    	=
    	\int_{\mathbb{Z}_p}
    	\int_{\mathbb{Z}_p}
    	\kappa_1(x)\mathbf{e}_p(-c_1x-c_2y^2x^{-1})dxdy
    	\int_{\mathbb{Z}_p}
    	\widetilde{\kappa}_2(z)\mathbf{e}_p(-c_2z)dz.
    	\end{align*}
    	First assume that
    	$\widetilde{\kappa}_2$ is a non-trivial character,
    	then
    	$\int_{\mathbb{Z}_p}
    	\widetilde{\kappa}_2(z)
    	\mathbf{e}_p(-c_2z)dz
    	=p^{-c'(\widetilde{\kappa}_2)}
    	G(\widetilde{\kappa}_2)
    	\widetilde{\kappa}_2^{-1}
    	(p^{c'(\widetilde{\kappa}_2)}c_2)$.
    	Now we have
    	\[
    	\alpha_{2,\underline{\kappa},p}(c)
    	=
    	p^{-c'(\widetilde{\kappa}_2)}
    	G(\widetilde{\kappa}_2)
    	\widetilde{\kappa}_2^{-1}
    	(p^{c'(\widetilde{\kappa}_2)}c_2)
    	p^{-c'(\widetilde{\kappa}_2/2)}
    	\varepsilon(p^{c'(\widetilde{\kappa}_2)})
    	(\frac{p^{c'(\widetilde{\kappa}_2)}c_2}
    	{p})^{c'(\widetilde{\kappa}_2)}
    	\int_{\mathbb{Z}_p}
    	\kappa_1(x)
    	(\frac{x}{p})^{c'(\widetilde{\kappa}_2)}
    	\mathbf{e}_p(-c_1x)dx.
    	\]
    	We then integrate on the variable $c_2$:
    	\[
    	\int_{\mathbb{Q}_p}
    	\alpha_{2,\underline{\kappa},p}
    	(\mathrm{diag}(c_1,c_2))
    	\theta_2(c_2)
    	dc_2
    	=
    	p^{-c'(\widetilde{\kappa}_2)/2}
    	G(\widetilde{\kappa}_2)
    	\varepsilon(p^{c'(\kappa_2)})
    	\alpha_2^{-c'(\kappa_2)}
    	(1-\frac{1}{p})
    	\int_{\mathbb{Z}_p}
    	\kappa_1(x)
    	(\frac{x}{p})^{c'(\widetilde{\kappa}_2)}
    	\mathbf{e}_p(-c_1x)dx.
    	\]
    	Note that the product before the above
    	integral is exactly the
    	product in the second big bracket in the lemma
    	for the case
    	$\widetilde{\kappa}_2$ non-trivial.

    	Next we assume
    	$\widetilde{\kappa}_2=1$,
    	then
    	$\int_{\mathbb{Z}_p}
    	\widetilde{\kappa}_2(z)
    	\mathbf{e}_p(-c_2z)dz
    	=-p^{-1}1_{p^{-1}\mathbb{Z}_p}(c_2)
    	+(1-p^{-1})1_{\mathbb{Z}_p}(c_2)$.
    	Then the integration on $c_2$ gives
    	\begin{align*}
    	\int_{\mathbb{Q}_p}
    	\alpha_{2,\underline{\kappa},p}
    	(\mathrm{diag}(c_1,c_2))
    	\theta_2(c_2)dc_2
    	&
    	=
    	\bigg(
    	(-p^{-1/2})
    	\varepsilon(p)\theta_2^{-1}(p)
    	(1-\frac{1}{p})
    	1_{c'(\kappa_2(\frac{\cdot}{p}))=0}
    	\int_{\mathbb{Z}_p}\theta_1(x)
    	\mathbf{e}_p(-c_1x)dx
    	\bigg)
    	\\
    	&
    	\quad+
    	\bigg(
    	(1-\frac{1}{p})^2
    	\frac{1}{1-\alpha_2p^{-1}}
    	1_{c'(\kappa_2)=0}
    	\int_{\mathbb{Z}_p}
    	\theta_1(x)
    	\mathbf{e}_p(-c_1x)dx
    	\bigg).
    	\end{align*}
    	The term in the first big bracket corresponds to
    	the case
    	$\kappa_2=(\frac{\cdot}{p})$
    	whiles
    	the term in the second bracket corresponds to
    	the case 
    	$\kappa_2=1$.
    	Moreover, the product before each integral
    	in each bracket
    	is exactly the formula given
    	in the lemma.

    	The integration on the variable
    	$c_1$ is exactly the same as
    	in the case of
    	$c_2$
    	(in fact simpler
    	than $c_2$).
    	The case of
    	$\underline{\kappa}'$
    	is similar to
    	$\underline{\kappa}$.
    	We thus omit the calculation.   	
    \end{proof}

    In summary we get the following
    \begin{corollary}
    	\label{non-vanishing of local zeta integral at p}
    	We set $s=-c=-c(\underline{\kappa})$,
    	then
    	for any
    	$\varphi\in\pi$
    	and $\varphi'\in\overline{\pi}$
    	such that
    	$\langle
    	e\varphi,\varphi'
    	\rangle\neq0$,
    	we have
    	\begin{align*}
    	\frac{Z_p((e\varphi)_p,\varphi'_p,
    		f_{\phi^d_{\underline{\kappa},p}})}
    	{\langle
    		(e\varphi)_p,
    		\varphi'_p
    		\rangle}
    	&
    	=
    	(1-(\frac{-1}{p})\frac{1}{p})
    	p^{-5c}
    	\alpha_1^{-2c}\alpha_2^{-2c}
    	\kappa_1(\frac{1}{2})\kappa_2(\frac{1}{2})
    	\\
    	&
    	\quad
    	\times
    	\bigg(
    	(1-\frac{1}{p})
    	(-1)^{\kappa_1^\circ}
    	\frac{1-\kappa_1^\circ\alpha_1^{-1}}
    	{1-\kappa_1^\circ\alpha_1p^{-1}}
    	G(\kappa_1)
    	\alpha_1^{-c'(\kappa_1)}
    	\bigg)
    	\\
    	&
    	\quad
    	\times
    	\bigg(
    	(1-\frac{1}{p})
    	(-1)^{\kappa_2^\circ}
    	\frac{1-\kappa_2^\circ p^{-1}}
    	{1-\kappa_2^\circ\alpha_2p^{-1}}
    	G(\widetilde{\kappa}_2)
    	\alpha_2^{-c'(\kappa_2)}
    	\varepsilon(p^{c'(\kappa_2)})
    	p^{-c'(\widetilde{\kappa}_2)}
    	\bigg).
    	\end{align*}
    \end{corollary}

	\subsubsection{Local Fourier coefficients}
	Next we compute the Fourier coefficients
	$E_{\beta,p}(1,f_{\phi^+_{\underline{\kappa},p}})$.
	It is easy to see that by
	Proposition \ref{Siegel section at p}:
	\begin{lemma}
		The local Fourier coefficient is
		\[
		E_{\beta,p}(1,f_{\phi^+_{\underline{\kappa},p}})
		=
		\frac{p^{-4s-13c}}{2}
		(1-(\frac{-1}{p})\frac{1}{p})
		(\widehat{\alpha}_{\underline{\kappa},p}
		(p^{-2s}\beta)+
		\widehat{\alpha}_{\underline{\kappa}',p}
		(p^{-2s}\beta)),
		\forall
		\beta\in\mathrm{Sym}_{4\times 4}
		(\mathbb{Q}).
		\]
	\end{lemma}

	\subsection{Summary}
	\label{summary}
	Recall that we fix an arithmetic point
	$\underline{\kappa}=(\underline{\kappa}_\mathrm{f},
	\underline{k})
	\in\mathrm{Hom}_\mathrm{cont}
	(T_{G_1^1}(\mathbb{Z}_p),\overline{\mathbb{Q}}_p^\times)$
	($k_1\geq k_2\geq3$).
	Let $\pi$ be an anti-holomorphic cuspidal automorphic representation
	of $G_1(\mathbb{A})$
	of type
	$(\underline{k},\widehat{\Gamma})$.
	Moreover, we fix $\varphi=\otimes_v\varphi_v\in\pi$
	to be a factorisable automorphic form such that
	$\varphi_\infty$ lies in the highest
	$K_{G_1^1,\infty}$-type 
	of $\pi_\infty$.
	We fix also a vector
	$\varphi^\vee\in\pi^\vee=\overline{\pi}$
	such that
	$\langle\varphi,\varphi^\vee\rangle\neq0$.
	For each place $v$ of $\mathbb{Q}$,
	we have chosen theta sections
	$\phi_{i,\underline{\kappa},v}\in\mathcal{S}(W_i^-(\mathbb{Q}_v))$
	($i=1,2$)
	as in the following
	\begin{enumerate}
		\item 
		if $v=\infty$,
		we fix one
		$I\in\mathfrak{I}_{\underline{\kappa}}$,
		then for
		$X_{k;j,l}=w_{k;j,l}
		\sqrt{\eta_{l,l}}$
		(\textit{cf}. (\ref{Schwartz section at infinity for +})),
		\[
		\phi_{i,\infty;I}(w_i^-)
		=\sqrt{a_I}
		X^I_i
		\phi_{i,\infty}(w_i^-);
		\]
		
		\item 
		if	$\ell\nmid Np$, then
		(\textit{cf}. 
		(\ref{Schwartz section at unramified finite places})):
		\[
		\phi_{i,\underline{\kappa},\ell}(w_i^-)
		=
		1_{W_i^-(\mathbb{Z}_\ell)}(w_i^-);
		\]
		
		\item 
		if	$\ell|N$, then
		(\textit{cf}.
		(\ref{Schwartz section at ramified place dividing N})):
		\[
		\phi_{i,\underline{\kappa},\ell}(w_i^-)
		=
		1_{S_{[i]}}(w_i^-)
		\mathbf{e}_\ell
		(
		\mathrm{tr}
		(
		b^t_{[i]}w_i^-
		)
		);
		\]
		
		\item 
		if $\ell=p$, then
		(\textit{cf}.
		(\ref{Schwartz section at p})):
		\[
		\phi_{1,\underline{\kappa},p}(y_1)
		=
		1_{w_0+C(\chi)W_1^-(\mathbb{Z}_p)}(y_1/p^s),
		\]
		\[
		\phi_{2,\underline{\kappa},p}
		(y_2)
		=1_{W_2^-(\mathbb{Z}_p)}(y_2/p^s)
		\widehat{\alpha}_{2,\underline{\kappa},p}
		(\frac{1}{2}y_2\eta_Uy_2^t/2p^{2s})
		1_{p^c\mathrm{M}_{2\times 2}
			(\mathbb{Z}_p)}
		(y_2\eta_Uw_0^t/p^{2s}).
		\]
	\end{enumerate}
    For each place $v$,
    we form the tensor product
	$\phi^+_{\underline{\kappa},v}
	=\phi_{1,\underline{\kappa},v}
	\otimes
	\phi_{2,\underline{\kappa},v}
	\in
	\mathcal{S}(W_4^-(\mathbb{Q}_v))$
	and its image under the intertwining operator $\delta$
	is denoted by
	$\phi^d_{\underline{\kappa},v}\in
	\mathcal{S}(W_4^d(\mathbb{Q}_v))$.
	These sections give rise to Siegel sections
	$f_{\phi^?_{\underline{\kappa},v}}
	\in\mathrm{Ind}_{P^?(\mathbb{Q}_v)}^{G_4(\mathbb{Q}_v)}
	(\xi_{1/2})$
	($?=+,d$).
	Their restricted tensor products 
	over all $v$ are
	denoted by
	$\phi^?_{\underline{\kappa}}$,
	$f_{\phi^?_{\underline{\kappa}}}$.
	The latter then define Siegel Eisenstein series
	$E(\cdot,f_{\phi^?_{\underline{\kappa}}})$.
	We have also studied their local zeta integrals
	and local Fourier coefficients and 
	showed that they are not identically zero
	and gave explicit expressions for
	Fourier coefficient at $\infty$,
	zeta integral and Fourier coefficient at
	$\ell\nmid Np$
	and zeta integral and Fourier coefficient at $p$.
	We put the modified local Euler factors
	as follows
	(\textit{cf}.
	Corollary
	\ref{non-vanishing of zeta integral at infinity},
	Proposition
	\ref{non-vanishing of local zeta integral at l},
	Lemma 
	\ref{non-vanishing of local integral at l}
	and
	Corollary
	\ref{non-vanishing of local zeta integral at p}):
	\begin{align*}
	\label{explicit formula for local L-factors}
	L_\infty^\ast(1,\mathrm{St}(\pi)\otimes\xi)
	&
	=
	L_{\infty,I}^\ast(1,\mathrm{St}(\pi)\otimes\xi)
	=
	\frac{Z_\infty(\varphi_{\underline{k},\infty},
		\varphi_{\underline{k},\infty}^\vee,
		f_{\phi_{\infty,I}^d})}
	{\langle
		\varphi_{\underline{k},\infty},
		\varphi_{\underline{k},\infty}^\vee
		\rangle};
	\\
	L_\ell^\ast(1,\mathrm{St}(\pi)\otimes\xi)
	&
	=\ell^{-20}
	|\mathrm{det}(\eta_U)|_\ell^{-2}
	\frac{((\frac{-1}{\ell})\ell^5-1)(\ell-1)}
	{(\ell^5-1)((\frac{-1}{\ell})\ell-1)},\quad
	\text{for}\quad\ell|N;
	\\
	L_p^\ast(1,\mathrm{St}(\pi)\otimes\xi)
	&
	=
	(1-(\frac{-1}{p})\frac{1}{p})
	p^{-5c}
	\alpha_1^{-2c}\alpha_2^{-2c}
	\kappa_1(-1)\kappa_2(-1)
	\times
	\bigg(
	(1-\frac{1}{p})
	(-1)^{\kappa_1^\circ}
	\frac{1-\kappa_1^\circ\alpha_1^{-1}}
	{1-\kappa_1^\circ\alpha_1p^{-1}}
	G(\kappa_1)
	\alpha_1^{-c'(\kappa_1)}
	\bigg)
	\\
	&
	\quad
	\times
	\bigg(
	(1-\frac{1}{p})
	(-1)^{\kappa_2^\circ}
	\frac{1-\kappa_2^\circ p^{-1}}
	{1-\kappa_2^\circ\alpha_2p^{-1}}
	G(\widetilde{\kappa}_2)
	\alpha_2^{-c'(\kappa_2)}
	\varepsilon(p^{c'(\kappa_2)})
	p^{-c'(\widetilde{\kappa}_2)}
	\bigg).
	\end{align*}
	
	Then we write their product
	as
	$L_{Np\infty}^\ast(1,\mathrm{St}(\pi)\otimes\xi)
	=\prod_{v|Np\infty}
	L_v^\ast(1,\mathrm{St}(\pi)\otimes\xi)$.
	Recall that we have defined periods
	$\widehat{P}[\pi]$
	and $P_\pi$ for
	$\pi$ in Section \ref{periods}.
	We take
	$p$-integral modular forms
	$\varphi,\overline{\varphi^\vee}
	\in
	H^{3,\mathrm{ord}}_{\underline{\kappa}^D}
	(\widehat{\Gamma},\mathcal{O}_\pi)[\pi]$
	(\textit{cf}. Section \ref{periods}).
	Then by definition
	$\langle
	\varphi,\varphi^\vee
	\rangle/\widehat{P}[\pi]
	\in\mathcal{O}_\pi$.
	We put these results into the following theorem
	\begin{theorem}\label{arithmetic Rallis inner product}
		Let the notations be as above.
		Assume Hypothesis
		\ref{Gorenstein hypothesis},
		then we have
		the Rallis inner product formula
		\[
		\langle
		\Theta_{\phi_{1,\underline{\kappa}}}
		(\varphi),
		\Theta_{\phi_{2,\underline{\kappa}}}
		(\varphi^\vee)
		\rangle_{\mathrm{O}(U)}
		=
		\mathfrak{c}^\mathrm{coh}(\pi)
		\frac{\langle \varphi,\varphi^\vee\rangle}
		{\widehat{P}[\pi]}
		\frac{L^{Np\infty}(1,\mathrm{St}
			(\pi)\otimes\xi)
		L_{Np\infty}^\ast(1,\mathrm{St}
		(\pi)\otimes\xi)}
	    {P_\pi}.
		\]		
	\end{theorem}

	\subsection{Theta series}
	We have defined sections
	$\phi_{i,\underline{\kappa}}\in
	\mathcal{S}(W_i^-(\mathbb{A}))$
	as the restricted tensor product
	$\phi_{i,\underline{\kappa}}
	=\otimes_v
	\phi_{i,\underline{\kappa},v}$.
	In this subsection we study the 
	algebraicity of the Fourier coefficients of
	the theta series
	$\Theta_{\phi_{i,\underline{\kappa}}}$
	for $i=1,2$
	and
	$\Theta_{\phi^+_{\underline{\kappa}}}$:
	\begin{align*}
	\Theta_{\phi_{i,\underline{\kappa}}}(g_i,h)
	&
	=
	\sum_{w_i^-\in W_i^-(\mathbb{Q})}
	\big(
	\omega(g_i,h)\phi_{i,\underline{\kappa}}
	\big)
	(w_i^-),
	\quad
	(g_i,h)\in G_i(\mathbb{A})\times \mathrm{O}(U)(\mathbb{A}),
	\\
	\Theta_{\phi^+_{\underline{\kappa}}}(g,h)
	&
	=
	\sum_{w_4^-\in W_4^-(\mathbb{Q})}
	\big(
	\omega(g,h)\phi^+_{\underline{\kappa}}
	\big)
	(w_4^-),
	\quad
	(g,h)\in G_4(\mathbb{A})\times \mathrm{O}(U)(\mathbb{A}).
	\end{align*}
	
	We fix $h\in\mathrm{O}(U)(\mathbb{A})$.
	Recall that the Fourier coefficients 
	$\Theta_{\phi^+_{\underline{\kappa}},\beta}(g,h)$
	of the theta series
	$\Theta_{\phi^+_{\underline{\kappa}}}$ is defined as
	\begin{align*}
	\Theta_{\phi^+_{\underline{\kappa}},\beta}(g,h)
	&
	=
	\int_{[\mathrm{Sym}_{4\times 4}]}
	\Theta_{\phi^+_{\underline{\kappa}}}
	(
	u(x)g,h
	)
	\mathbf{e}(-\mathrm{tr}\,\beta x)
	dx
	\\
	&
	=
	\sum_{w^-\in W^-(\mathbb{Q})}
	\int_{[\mathrm{Sym}_{4\times 4}]}
	\omega
	(
	u(x)g,h
	)\phi^+_{\underline{\kappa}}(w^-)
	\mathbf{e}(-\mathrm{tr}\,\beta x)dx.
	\end{align*}
	The second equality comes from the fact that
	$\omega(\cdot,h)\phi^+_{\underline{\kappa}}(w^-)$
	for $w^-\in
	W^-(\mathbb{Q})$
	is invariant under the action of the rational unipotent
	matrices
	$u(x)$.
	
	Recall that for any element in 
	the Siegel upper half plane
	$\mathbf{z}=\mathbf{x}+i\mathbf{y}
	\in\mathbb{H}_4$,
	we have an adelic element
	$
	g_\mathbf{z}
	=u(\mathbf{x})
	m(\sqrt{\mathbf{y}})
	\times 1_\mathrm{f}
	\in
	G_4(\mathbb{A})$.
	We then compute
	$\Theta_{\phi^+_{\underline{\kappa}},\beta}(g_\mathbf{z},h)$.
	We have
	(viewing $\sqrt{\mathbf{y}},\mathbf{x}$ as adelic elements)
	\begin{align*}
	\omega
	(u(x)g_\mathbf{z},h
	)
	\phi^+_{\underline{\kappa}}(w^-)
	&
	=
	\mathbf{e}
	(
	\mathrm{tr}
	\Big((\sqrt{\mathbf{y}}^{-1}w^-)^t(\mathbf{x}+x)
	\sqrt{\mathbf{y}}^{-1}w^-\eta_U\Big)
	)
	\omega(1,h)
	\phi^+_{\underline{\kappa}}(\sqrt{\mathbf{y}}^{-1}w^-)
	\\
	&
	=
	\mathbf{e}
	(
	\mathrm{tr}
	\Big((\sqrt{\mathbf{y}}^{-1}w^-)^t\mathbf{x}
	\sqrt{\mathbf{y}}^{-1}w^-\eta_U\Big)
	)
	\mathbf{e}
	(
	\mathrm{tr}\Big((\sqrt{\mathbf{y}}^{-1}w^-)^tx
	\sqrt{\mathbf{y}}^{-1}w^-\Big)
	)
	\omega(1,h)\phi^+_{\underline{\kappa}}(\sqrt{\mathbf{y}}^{-1}w^-).
	\end{align*}
	Therefore we get
	(we put $1_\beta(x)=1$
	if $x=\beta$ and
	$=0$ otherwise.)
	\begin{align*}
	\int_{[\mathrm{Sym}_{4\times 4}]}
	\omega
	(
	u(x)
	g_\mathbf{z},h
	)
	\phi^+_{\underline{\kappa}}(w^-)
	\mathbf{e}(-\mathrm{tr}\,\beta x)dx
	&
	=
	\mathbf{e}(\mathrm{tr}
	\Big((\sqrt{\mathbf{y}}^{-1}w^-)^t\mathbf{x}
	\sqrt{\mathbf{y}}^{-1}w^-\eta_U\Big)
	)
	\omega(1,h)\phi^+_{\underline{\kappa}}
	(\sqrt{\mathbf{y}}^{-1}w^-)
	\\
	&
	\quad
	\times
	\int_{[\mathrm{Sym}_{4\times 4}]}
	\mathbf{e}
	(
	\mathrm{tr}
	\Big(
	(\sqrt{\mathbf{y}}^{-1}w^-)^t
	x\sqrt{\mathbf{y}}^{-1}w^-\eta_U-\beta x
	\Big)
	)dx
	\\
	&
	=
	\mathbf{e}(\mathrm{tr}(\beta\mathbf{x}))
	\omega(1,h)\phi^+_{
		\underline{\kappa}}(\sqrt{\mathbf{y}}^{-1}w^-)
	1_{\beta}
	(
	\sqrt{\mathbf{y}}^{-1}w^-\eta_U(\sqrt{\mathbf{y}}^{-1}w^-)^t
	).
	\end{align*}
	From this, we have
	\[
	\Theta_{\phi^+_{\underline{\kappa}},\beta}(g_\mathbf{z},h)
	=
	\mathbf{e}
	(
	\mathrm{tr}(\beta\mathbf{x})
	)
	\sum_{w^-\in W_4^-(\mathbb{Q})}	
	\omega(1,h)\phi^+_{\underline{\kappa}}
	(\sqrt{\mathbf{y}}^{-1}w^-)
	1_{\beta}
	(
	\sqrt{\mathbf{y}}^{-1}w^-\eta_U(\sqrt{\mathbf{y}}^{-1}w^-)^t
	).
	\]
	The condition
	$\sqrt{\mathbf{y}}^{-1}w^-\eta_U(\sqrt{\mathbf{y}}^{-1}w^-)^t
	=
	\Big(
	(\sqrt{\mathbf{y}}^{-1}\times1_\mathrm{f})
	w^-
	\Big)\eta_U
	\Big(
	(\sqrt{\mathbf{y}}^{-1}\times1_\mathrm{f})
	w^-
	\Big)^t
	=
	\beta$
	means that
	(at the archimedean place)
	$\sqrt{\mathbf{y}}^{-1}w^-\eta_U
	(\sqrt{\mathbf{y}}^{-1}w^-)^t=\beta$
	and
	(at the non-archimedean place)
	$w^-\eta_U(w^-)^t=\beta$.
	This shows that
	$\sqrt{\mathbf{y}}\beta\sqrt{\mathbf{y}}^t=\beta$
	and thus
	$\mathrm{tr}(\mathbf{y}\beta)=\mathrm{tr}(\beta)$.
	We then compute each factor
	$\omega_v(1,h)\phi^+_{\underline{\kappa},v}
	(\sqrt{\mathbf{y}}^{-1}w^-)$
	of
	$\omega(1,h)\phi_{\underline{\kappa}}^+
	(\sqrt{\mathbf{y}}^{-1}w^-)$
	for all places $v$ of $\mathbb{Q}$.
	
	\begin{enumerate}
		\item 
		For $v=\infty$
		(\textit{cf}.
		\ref{Schwartz section at infinity for +}),
		\begin{align*}
		\omega_\infty(1,h_\infty)
		\phi^+_{\underline{\kappa},\infty}
		(\sqrt{\mathbf{y}}^{-1}w^-)
		&
		=
		a_I
		(X_1h_\infty)^I
		(X_2h_\infty)^I
		\mathbf{e}_\infty
		(i
		\langle
		\sqrt{\mathbf{y}}^{-1}w^-,
		J_4\sqrt{\mathbf{y}}^{-1}w^-
		\rangle
		)
		\\
		&
		=
		a_I
		(X_1h_\infty)^I
		(X_2h_\infty)^I
		\mathbf{e}_\infty(i\,\mathrm{tr}(\beta\mathbf{y}));
		\end{align*}

		\item 
		For $\ell\nmid Np\infty$,
		\[
		\omega_\ell(1,h_\ell)
		\phi^+_{\underline{\kappa},\ell}(w^-)
		=
		1_{W_4^-(\mathbb{Z}_\ell)}(w^-h_\ell).
		\]
		
		\item 
		For $\ell|N$
		(\textit{cf}.
		(\ref{Schwartz section at ramified place dividing N})),
		\begin{align*}
		\omega_\ell(1,h_\ell)
		\phi^+_{\underline{\kappa},\ell}(w^-)
		&
		=
		\phi^+_{1,\underline{\kappa},\ell}(w_1^-h_\ell)
		\phi^+_{2,\underline{\kappa},\ell}(w_2^-h_\ell)
		\\
		&
		=
		1_{S_{[1]}}
		(w_1^-h_\ell)
		\mathbf{e}_\ell(\mathrm{tr}(b_{[1]}^tw_1^-h_\ell))
		\times
		1_{S_{[2]}}
		(w_2^-h_\ell)
		\mathbf{e}_\ell(\mathrm{tr}(b_{[2]}^tw_2^-h_\ell))
		\in
		\mathbb{Z}[\mu_{N_\ell^2}].
		\end{align*}
		
		\item 
		For $\ell=p$
		(\textit{cf}.
		\ref{Schwartz section at p}),
		\[
		\omega_p(1,h_p)\phi^+_{\underline{\kappa},p}(w^-)
		=
		\phi_{\underline{\kappa},p}^+(w^-h_p)
		\in
		\mathbb{Z}[\mu_{C(\underline{\kappa})}].
		\]
	\end{enumerate}
	Note that for each fixed $h$,
	the summation in
	$\Theta_{\phi^+_{\underline{\kappa}},\beta}
	(g_\mathbf{z},h)$ is always
	a finite sum.
	Thus
	\begin{proposition}
		\label{algebraicity of Fourier coefficients of theta series}
		The Fourier coefficients
		$\Theta_{\phi^+_{\underline{\kappa}},\beta}
		(g_\mathbf{z},h)
		\mathbf{e}(-\mathrm{tr}\,\beta\mathbf{z})$
		and
		$\Theta_{\phi^+_{i,\underline{\kappa}},\beta_i}(g_{\mathbf{z}_i},h)
		\mathbf{e}(-\mathrm{tr}\,\beta_i\mathbf{z}_i)$
		($i=1,2$)
		are in
		the cyclotomic ring
		$\mathbb{Z}_p[\mu_{N^2C(\underline{\kappa})}]$.
	\end{proposition}
    \begin{definition}
    	\label{algebraic theta series}
    	For $i=1,2$,
    	we write
    	$\Theta_{
    		\phi_{i,\underline{k}}^+,
    		\beta_i}^\mathrm{alg}(h)$
    	for the preimage of
    	$\Theta_{\phi^+_{i,\underline{\kappa}},\beta_i}
    	(g_{\mathbf{z}_i},h)$
    	in
    	$H^0(\widetilde{\mathbf{A}}_{
    		G_i,\widehat{\Gamma}},
    	\mathcal{E}(W_i))$
    	for some algebraic representation
    	$W_i$ of
    	$\mathrm{GL}_2$ by
    	the map
    	$\Phi(\cdot,\mathfrak{e})$
    	(\textit{cf}.
    	(\ref{correspondence between geometric and adelic forms})).
    \end{definition}
    \begin{remark}
    	Though we do not explicate
    	the representation
    	$W_i$,
    	we know at least that it contains a copy of
    	$W_{\underline{k}}$
    	by the non-vanishing of the
    	archimedean local zeta integral
    	(\textit{cf}.
    	Corollary
    	\ref{non-vanishing of zeta integral at infinity}).
    \end{remark}

	\section{Transfer from $\mathrm{GSp}_4$
	to $\mathrm{U}_4$}\label{section: Langlands functoriality}

	In this section, we use results of Atobe and Gan
	\cite{AtobeGan16}
	to deduce the functoriality from the symplectic group
	$\mathrm{GSp}_4$ to the unitary group $\mathrm{U}_4$.
	
	\subsection{Isogeny from $\mathrm{GSO}_6$ to $\mathrm{U}_4$}
	\label{isogeny from GSO(6) to U(4)}
	Suppose that $E/\mathbb{Q}$
	is a quadratic imaginary extension
	with
	$E=\mathbb{Q}(\eta)$
	where $\eta^2=-N$,
	$H\colon E^4\times E^4\rightarrow E$
	is a Hermitian form, positive definite at $\infty$.
	Suppose also that 
	$H$ is $E$-linear on the first variable
	and $E$-semilinear on th second variable.
	Suppose an $E$-basis of $E^4$ is $e_1,e_2,e_3,e_4$.
	We write $\mathrm{GSU}_4, \mathrm{SU}_4,\mathrm{U}_4$
	for the (similitude special) unitary algebraic group defined
	over $\mathbb{Q}$ by the Hermitian form $H$.
	
	\subsubsection{Exceptional morphism}
	Here we establish an exceptional isogeny from 
	a similitude special unitary groups 
	$\mathrm{GSU}_4$ to 
	a similitude special orthogonal group
	$\mathrm{GSO}_6$.
	The construction is taken from \cite{Garrett15}.
	We will define the similitude special orthogonal group
	$\mathrm{GSO}_6$ later on, 
	which depends on the Hermitian form $H$.
	
	We define a $\mathrm{SL}_4(E)$-invariant $E$-valued 
	symmetric form
	$\langle\cdot,\cdot\rangle$ 
	on $\wedge^2E^4$ as follows
	\[
	\langle x\wedge y,z\wedge w\rangle 
	e_1\wedge e_2\wedge e_3\wedge e_4
	=x\wedge y\wedge z \wedge w,
	\quad
	(\forall x,y,z,w\in E^4)
	\]
	We define a $E$-semilinear isomorphism
	\[
	E^4\rightarrow (E^4)^*=\mathrm{Hom}_E(E^4,E),
	\quad
	x\mapsto(y\mapsto H(y,x)).
	\]
	This induces a map 
	$\wedge^2E^4\rightarrow\wedge^2(E^4)^*\simeq(\wedge^2E^4)^*$.
	We also have an $E$-linear isomorphism
	\[
	\wedge^2E^4\rightarrow(\wedge^2E^4)^*,
	\quad
	u\mapsto(v\mapsto\langle v,u\rangle).
	\]	
	Combining these, we get the following 
	$E$-semilinear isomorphism
	\[
	J\colon
	\wedge^2E^4\xrightarrow{\langle\cdot,\cdot\rangle}
	(\wedge^2E^4)^*
	\rightarrow\wedge^2(E^4)^*\xleftarrow{\wedge^2H}
	\wedge^2E^4.
	\]	
	Note that $\mathrm{GSU}_4(\mathbb{Q})$,
	as a subgroup of $\mathrm{SL}_4(E)$,
	respects $\langle\cdot,\cdot\rangle$. 
	By definition, it respects also $\wedge^2H$.
	Thus $J$ commutes with $\mathrm{GSU}_4(\mathbb{Q})$.
	
	Suppose that under the basis $e_1,e_2,e_3,e_4$ of $E^4$,
	$H$ is of the form
	$H(e_i,e_j)=a_i^{-1}\delta_{i,j}$ with 
	$0<a_i\in\mathbb{Q}$
	such that
	$a_1a_2a_3a_4=1$.
	By examining $J(e_i\wedge e_j)$
	(evaluating at $e_k\wedge e_l$),
	we see that
	$J^2=1$.
	We then pick out the eigenspace 
	$U_J$ of $J$ of 
	eigenvalue $+1$.
	One can take the following $\mathbb{Q}$-basis
	of $U_J$:
	\[
	e_1\wedge e_2+a_3a_4e_3\wedge e_4,
	\quad
	\eta e_1\wedge e_2-\eta a_3a_4e_3\wedge e_4,
	\quad
	e_1\wedge e_3-a_2a_4e_2\wedge e_4,
	\]
	\[
	\eta e_1\wedge e_3+\eta a_2a_4e_2\wedge e_4,
	\quad
	e_1\wedge e_4+a_2a_3e_2\wedge e_3,
	\quad
	\eta e_1\wedge e_4-\eta a_2a_3e_2\wedge e_3.
	\]
	Then the symmetric form
	$\langle\cdot,\cdot\rangle|_{U_J}$
	is of the form
	${\mathrm{diag}(2a_3a_4,
	-2\eta^2a_3a_4,
	2a_2a_4,-2\eta^2a_2a_4,
	2a_2a_3,-2\eta^2a_2a_3)}$.
	For application in this article,
	we can take
	$a_1=a_2=a_3=a_4=1$.
	Moreover, it is easy to see that
	$\langle\cdot,\cdot\rangle|_{U_J}$ is
	$\mathbb{Q}$-equivalent to
	the quadratic form
	${\eta_U=\mathrm{diag}(N^2/2,
	N^2/2,N^2/2,N^2/2,N/N_1,N_1)}$
	with $0\neq N_1\nmid Np$
	(using simple properties of
	Hilbert symbols and density of
	$\mathbb{Q}$ in $\mathbb{A}$).
	Therefore,
	the orthogonal groups
	$\mathrm{O}(\langle\cdot,\cdot\rangle)$
	and
	$\mathrm{O}(\eta_U)$
	defined by these quadratic forms
	are isomorphic over $\mathbb{Q}$.
	We will thus identify these two groups and
	also their automorphic forms and
	Hecke algebras
	without further comment.

	\subsection{$L$-groups and $L$-parameters}
	In this subsection we review some facts on the $L$-groups
	and $L$-parameters of groups
	that will be used in the sequel.

	We define the orthogonal group
	$\mathrm{O}_k(\mathbb{C})$
	to be
	$\{
	g\in\mathrm{GL}_k(\mathbb{C})|
	g^t\mathfrak{w}_kg=\mathfrak{w}_k
	\}$
	where $\mathfrak{w}_k$ is the anti-diagonal matrix
	with $1$ on the anti-diagonal.
	One defines also
	the special orthogonal group
	$\mathrm{SO}_k(\mathbb{C})$,
	the similitude group
	$\mathrm{GO}_k(\mathbb{C})
	=\{
	g\in\mathrm{GL}_k(\mathbb{C})|
	g^t\mathfrak{w}_kg=\nu(g)
	\mathfrak{w}_k
	\}$ and
	the similitude special group
	$\mathrm{GSO}_k(\mathbb{C})
	=\{
	g\in\mathrm{GO}_k(\mathbb{C})|
	\mathrm{det}(g)=\nu(g)^{k/2}
	\}$
	if $k$ is even and
	$=\mathrm{GO}_k(\mathbb{C})$
	if $k$ is odd.

	The $L$-group of the similitude symplectic group
	$\mathrm{GSp}_4(F)$
	is $\mathrm{GSp}_4(\mathbb{C})$.
	Recall that the standard representation of
	$\mathrm{GSp}_4(\mathbb{C})$
	is given as follows
	(\cite[Appendix 7, pp.286-287]{RobertsSchmidt2007}):
	write
	$\mathbf{V}=(\mathbb{C}^4,
	\langle\cdot,\cdot\rangle)$
	for the symplectic vector space
	such that under the standard basis
	$\{e_1,e_2,e_3,e_4\}$
	the symplectic form is
	$J_4$.
	This gives rise to
	the similitude isometry group
	$\mathrm{GSp_4(\mathbb{C})}$.
	Consider the exterior square
	$\wedge^2\mathbf{V}$ of $\mathbf{V}$
	on which $\mathrm{GSp}_4(\mathbb{C})$
	by
	\[
	\rho(g)(v\wedge u)
	:=\nu(g)^{-1}(gv)\wedge(gu)
	\]
	for $g\in\mathrm{GSp}_4(\mathbb{C})$
	and $v\wedge u\in\wedge^2\mathbf{V}$.
	We define a bilinear form on
	$\wedge^2\mathbf{V}$
	as follows:
	for any $v\wedge u, v'\wedge u'\in\wedge^2\mathbf{V}$,
	we set
	$(v\wedge u,v'\wedge u')
	:=\langle v,v'\rangle\langle u,u'\rangle-
	\langle v,u'\rangle\langle v',u\rangle$
	and then extend to the whole $\wedge^2\mathbf{V}$
	by bilinearity.
	It is clear that this is a symmetric form
	and
	\[
	(v\wedge u)\wedge(v'\wedge u')
	=(v\wedge u,v'\wedge u')\cdot
	\,e_1\wedge e_2\wedge e_3\wedge e_4
	\in\wedge^4\mathbf{V}.
	\]
	Therefore an element
	$g\in\mathrm{GSp}_4(\mathbb{C})$
	is sent to an element in
	$\mathrm{SO}_6(\mathbb{C})$,
	the isometry group of
	$(\wedge^2\mathbf{V},(\cdot,\cdot))$.
	Consider the basis
	of $\wedge^2\mathbf{V}$:
	\[
	\{E_1=e_1\wedge e_2,
	E_2=e_1\wedge e_4,
	E_3=e_1\wedge e_3,
	E_4=-e_2\wedge e_4,
	E_5=e_2\wedge e_3,
	E_6=e_3\wedge e_4\}.
	\]
	It is easy to verify that
	the symmetric bilinear form of
	$\wedge^2\mathbf{V}$
	under this basis is
	the anti-diagonal matrix
	$\mathfrak{w}_6$.
	Note that $\mathrm{GSp}_4(\mathbb{C})$
	is generated by
	$J_4$,
	$m(a)\mathrm{diag}(1_2,\nu\cdot1_2)$
	for $a\in\mathrm{GL}_2(\mathbb{C})$
	and $\nu\in\mathbb{C}^\times$,
	and $u(b)$
	for $b$ a symmetric matrix.
	Then one can verify case by case that
	the subspace of $\wedge^2\mathbf{V}$
	generated by the vector
	$E_3-E_4
	=e_1\wedge e_3+e_2\wedge e_4$ 
	is invariant under the action of
	these elements,
	thus
	this subspace is invariant under the action of
	$\mathrm{GSp}_4(\mathbb{C})$.
	Now consider its orthogonal complementary
	subspace
	$\mathbf{X}$ of
	$\wedge^2\mathbf{V}$
	generated by the basis
	\[
	\{
	E_1'=E_1,
	E_2'=E_2,
	E_3'=(E_3+E_4)/\sqrt{2},
	E_4'=E_5,
	E_5'=E_6
	\}.
	\]
	Then
	the symmetric bilinear form
	$(\cdot,\cdot)|_{\mathbf{X}}$ is
	$\mathfrak{w}_5$ under this basis.
	Therefore, one obtains the standard representation of
	$\mathrm{GSp}_4(\mathbb{C})$:
	\[
	\rho_\mathrm{st}
	\colon
	\mathrm{GSp}_4(\mathbb{C})
	\rightarrow
	\mathrm{SO}_5(\mathbb{C}).
	\]
	By the definition of the action
	$\rho$,
	$\rho_\mathrm{st}$
	factors through
	$\mathrm{PGSp}_4(\mathbb{C})=
	\mathrm{PSp}_4(\mathbb{C})$.
	As in \cite[p.287]{RobertsSchmidt2007},
	one verifies that
	$\rho_\mathrm{st}$ is surjective.
	This representation is used in the definition of the
	standard $L$-functions of
	Siegel modular forms of genus $2$.
	Note that the similitude factor does not play a role in
	the standard representation
	of $\mathrm{GSp}_4(\mathbb{C})$.

	We can modify the
	action
	$\rho$ of
	$\mathrm{GSp}_4(\mathbb{C})$
	on $\wedge^2\mathbf{V}$
	without dividing out the similitude factor:
	\[
	\rho'(g)(v\wedge u)
	:=
	gv\wedge gu.
	\]
	In the same manner
	$\mathrm{GSp}_4(\mathbb{C})$
	acts on $\mathbf{X}$
	via $\rho'$.
	Thus we get a surjective morphism
	\[
	\rho'_\mathrm{st}
	\colon
	\mathrm{GSp}_4(\mathbb{C})
	\rightarrow
	\mathrm{GSO}_5(\mathbb{C})
	\]
	whose kernel is
	$\{\pm1_4\}$.
	Therefore one identifies
	$\mathrm{GSp}_4(\mathbb{C})$
	with the similitude spin group
	$\mathrm{GSpin}_5(\mathbb{C})$.
	Moreover the images of
	elements in the standard torus of
	$\mathrm{GSp}_4(\mathbb{C})$:
	\[
	\rho'_\mathrm{st}
	(\mathrm{diag}(t_1,t_2,t/t_1,t/t_2))
	=
	\mathrm{diag}(t_1t_2,tt_1/t_2,t,tt_2/t_1,t^2/(t_1t_2)).
	\]
	This representation will be used in expressing theta lift of
	unramified representations in terms of
	Langlands parameters.

	The dual group of the similitude special
	orthogonal group
	$\mathrm{GSO}_6(F)$
	is the similitude spin group
	$\mathrm{GSpin}_6(\mathbb{C})$
	(\cite[p.81]{Xu2018}).
	One can use the above construction to show
	that $\mathrm{GL}_4(\mathbb{C})$
	is isomorphic to
	$\mathrm{GSpin}_6(\mathbb{C})$.
	Indeed, the representation
	$\rho'$ of $\mathrm{GSp}_4(\mathbb{C})$
	on $\wedge^2\mathbf{V}$
	extends to
	$\mathrm{GL_4}(\mathbb{C})$.
	This time we get a surjective morphism
	$\rho'_\mathrm{st}
	\colon
	\mathrm{GL}_4(\mathbb{C})
	\rightarrow
	\mathrm{GSO}_6(\mathbb{C})$,
	whose kernel is $\{\pm1_4\}$
	and
	$\rho'_\mathrm{st}(g)=\mathrm{det}(g)$.

	Similarly,
	under the basis
	$\{E_1,\cdots,E_6\}$ of $\wedge^2\mathbf{V}$,
	we have:
	\[
	\rho'_\mathrm{st}
	(\mathrm{diag}(\lambda_1,\lambda_2,\lambda_3,\lambda_4))
	=
	\mathrm{diag}(\lambda_1\lambda_2,\lambda_1\lambda_4,
	\lambda_1\lambda_3,\lambda_2\lambda_4,
	\lambda_2\lambda_3,\lambda_3\lambda_4)
	\]
	and under the basis
	$\{E_1',\cdots,E_5',i(E_3-E_4)/\sqrt{2}\}$ of
	$\wedge^2\mathbf{V}$
	(note that under this basis,
	the symmetric bilinear form $(\cdot,\cdot)$
	becomes
	$\mathrm{diag}(J_5,1)$):	
	\[
	\rho'_\mathrm{st}
	(\mathrm{diag}(t_1,t_2,t/t_1,t/t_2))
	=
	\mathrm{diag}(t_1t_2,tt_1/t_2,t,tt_2/t_1,t^2/(t_1t_2),t).
	\]

	The $L$-group of $\mathrm{GSO}_6(F)=\mathrm{GSO}(U)$
	depends on whether $V$ is split over $F$ or not.
	If $U$ is split over $F$,
	i.e., $\mathrm{GSO}(U)\simeq\mathrm{GSO}_{3,3}(F)$,
	then
	$\,^L\mathrm{GSO}(U)
	=\mathrm{GSpin}_6(\mathbb{C})
	\times\Gamma_{E/F}$.
	If $U$ is not split over $F$,
	i.e., $\mathrm{GSO}(U)\simeq\mathrm{GSO}_{4,2}(F)$,
	then
	$^L\mathrm{GSO}_6(F)
	=\mathrm{GSO}_6(\mathbb{C})\rtimes
	\Gamma_{E/F}$
	where the Galois group 
	$\Gamma_{E/F}=\{1,c\}$ acts on
	$\mathrm{GSO}_6(\mathbb{C})$ by
	$c(g)c:=J_4g^{-t}J_4^{-1}$.
	Note that
	$c$ fixes each element of
	$\mathrm{GSp}_4(\mathbb{C})$
	up to a scalar
	(the similitude factor).
	The action of
	$\Gamma_{E/F}$ on
	$\mathrm{GSpin}_6(\mathbb{C})
	\simeq
	\mathrm{GL}_4(\mathbb{C})$
	descends to $\mathrm{GSO}_6(\mathbb{C})$ as follows.
	We write
	\[
	\epsilon
	=-\mathrm{diag}(1_2,\mathfrak{m}_2,1_2)\in
	\mathrm{O}_6(\mathbb{C})\backslash
	\mathrm{SO}_6(\mathbb{C}).
	\]
	Then we have
	$\rho'_\mathrm{st}(J_4)
	=-\epsilon\mathfrak{w}_6
	=-\mathfrak{w}_6\epsilon$.
	Moreover,
	for any $g\in\mathrm{GL}_4(\mathbb{C})$,
	$\rho'_\mathrm{st}(g^{-t})
	=\rho'_\mathrm{st}(g)^{-t}$.
	Thus
	\[
	\rho'_\mathrm{st}(c(g))
	=(\epsilon \mathfrak{w}_6)
	\rho'_\mathrm{st}(g)^{-t}
	(\epsilon \mathfrak{w}_6)^{-1}
	=\nu(\rho'_\mathrm{st}(g))
	\epsilon \rho'_\mathrm{st}(g)\epsilon^{-1}
	=\mathrm{det}(g)
	\epsilon \rho'_\mathrm{st}(g)\epsilon.
	\]
	In other words,
	\begin{lemma}\label{action of c and epsilon}
		The element 
		$c\in\Gamma_{E/F}$ acts on 
		$\mathrm{GSO}_6(\mathbb{C})$
		by conjugation by $\epsilon$
		up to the similitude factor.
		The conjugate action of
		$\epsilon$ on
		$\mathrm{GSO}_6(\mathbb{C})$
		lifts to the action $c$ on
		$\mathrm{GL}_4(\mathbb{C})$
		up to the determinant factor.
	\end{lemma}
	To have a uniform expression for the cases
	split and non-split, we
	write in both cases the $L$-group of
	$\mathrm{GSO}_6(F)$
	as
	\[
	\,^L\mathrm{GSO}_6(F)
	=
	\mathrm{GSpin}_6(\mathbb{C})
	\rtimes\Gamma_{E/F}.
	\]
	It should be kept in mind that
	when $U$ is split over $F$,
	this is a direct product.

	\begin{remark}
		\label{orthogonal and special orthogonal, adjoint representation}
		\begin{enumerate}
			\item 
			We discuss the relation between
			$\mathrm{O}_6(\mathbb{C})$ and
			$\mathrm{SO}_6(\mathbb{C})$
			and the related similitude groups
			as well as (similitude) spin/pin groups,
			which will be useful later on when
			we express the theta lift of unramified
			representations in terms of
			Langlands parameters.
			We define a morphism
			$\mu
			\colon
			\mathrm{O}_6(\mathbb{C})
			\rightarrow
			\{0,1\}$
			where
			$\mu(g)=0$ if $\mathrm{det}(g)=1$
			and
			$\mu(g)=1$ otherwise.
			Then we can define an isomorphism
			\[
			\mathrm{O}_6(\mathbb{C})
			\xrightarrow{\sim}
			\mathrm{SO}_6(\mathbb{C})\rtimes\{1,\epsilon\},
			\quad
			g\mapsto
			(g\epsilon^{\mu(g)},\epsilon^{\mu(g)}).
			\]
			Here $\epsilon$ acts on
			$\mathrm{SO}_6(\mathbb{C})$
			by conjugation.
			One can extend $\mu$ to
			$\mathrm{GO}_6(\mathbb{C})$,
			the Pin group
			$\mathrm{Pin}^+_6(\mathbb{C})$
			and the similitude Pin group
			$\mathrm{GPin}^+_6(\mathbb{C})$
			as follows:
			if $\mathrm{det}(g)=\nu(g)^3$,
			then
			$\mu(g)=0$,
			otherwise
			$\mu(g)=1$.
			Using this $\mu$, one obtains similar isomorphisms
			\[
			\mathrm{GO}_6(\mathbb{C})
			\simeq
			\mathrm{GSO}_6(\mathbb{C})
			\rtimes\{1,\epsilon\},
			\quad
			\mathrm{Pin}^+_6(\mathbb{C})
			\simeq
			\mathrm{Spin}_6(\mathbb{C})\rtimes\{1,\widetilde{\epsilon}\},
			\]
			\[
			\mathrm{GPin}^+_6(\mathbb{C})
			\simeq
			\mathrm{GSpin}_6(\mathbb{C})
			\rtimes\{1,\widetilde{\epsilon}\}
			\simeq
			\mathrm{GL}_4(\mathbb{C})\rtimes\Gamma_{E/F}
			\]
			Here 
			$\widetilde{\epsilon}
			\in\mathrm{Pin}^+_6(\mathbb{C})$
			is any element that is mapped to
			$\epsilon$ under the natural projection.

			\item 
			We define the adjoint representations 
			$\rho_\mathrm{ad}$ of
			$\mathrm{GSp}_4(\mathbb{C})\times
			\Gamma_{E/F}$
			and
			$\mathrm{GL}_4(\mathbb{C})\rtimes
			\Gamma_{E/F}$
			as follows.
			We define a symmetric bilinear form
			on $\mathrm{M}_{4\times4}(\mathbb{C})$
			by
			$(X,Y):=\mathrm{tr}(XY)$.
			The action of
			$\mathrm{GL}_4(\mathbb{C})$ on 
			$\mathrm{M}_{4\times4}(\mathbb{C})$
			is defined to be
			$\rho_\mathrm{ad}(g)(X)
			=g\cdot X:=
			gXg^{-1}$.
			The action of
			$c\in\Gamma_{E/F}$ is
			$\rho_\mathrm{ad}(c)(X)
			=-J_4X^tJ_4^{-1}$.
			It is clear that the action of
			$\mathrm{GL}_4(\mathbb{C})\rtimes
			\Gamma_{E/F}$ preserves the bilinear form.
			There is a subspace of
			$\mathrm{M}_{4\times4}(\mathbb{C})$
			of dimension $5$
			defined to be
			$\widetilde{\mathbf{V}}
			=\{
			X\in\mathrm{M}_{4\times4}(\mathbb{C})|
			J_4X^tJ_4^{-1}=X,
			\mathrm{tr}(X)=0\}$.
			It is easy to verify that
			$\widetilde{\mathbf{V}}$
			is invariant under
			$\mathrm{GSp}_4(\mathbb{C})
			\times\Gamma_{E/F}$.
			Moreover, this representation
			$\rho_\mathrm{ad}$ of
			$\mathrm{GSp}_4(\mathbb{C})$
			is isomorphic to the 
			standard representation
			$\rho_\mathrm{st}
			\colon
			\mathrm{GSp}_4(\mathbb{C})
			\rightarrow
			\mathrm{SO}_5(\mathbb{C})$.
			One sees by definition that
			$\rho_\mathrm{ad}(c)(X)=-X$
			for $X\in\widetilde{\mathbf{V}}$.
			This is compatible with the 
			conjugate action of
			$\epsilon$ on the subspace
			$\widetilde{\mathbf{V}}'$ of
			the Lie algebra $\mathfrak{so}_6(\mathbb{C})
			=\{
			X\in\mathrm{M}_{6\times6}(\mathbb{C})|
			X^t\mathfrak{w}_6=-\mathfrak{w}_6X
			\}$
			defined by
			$\widetilde{\mathbf{V}}'
			=\{
			X|\rho_\mathrm{st}'(J_4)
			X^t\rho'_\mathrm{st}(J_4)^{-1}=X
			\}$.
			Indeed, 
			let's write $w=\rho'_\mathrm{st}
			(J_4)=w^t$.
			Then for any
			$X\in\widetilde{\mathbf{V}}'$,
			$\epsilon X
			=\mathfrak{w}_6wX
			=\mathfrak{w}_6(wX)^t
			=\mathfrak{w}_6X^tw
			=-X\mathfrak{w}_6w
			=-X\epsilon$.
			Note that
			$c$, resp., $\epsilon$ acts on
			the $1$-dimensional subspace
			$\mathbb{C}\cdot1_4\subset
			\mathrm{M}_{4\times4}(\mathbb{C})$,
			resp.,
			$\mathbb{C}\cdot 1_6\subset
			\mathfrak{gso}_6(\mathbb{C})$
			as
			$\rho_\mathrm{ad}(c)(1_4)=-1_4$,
			$\rho_\mathrm{st}(\epsilon)(1_6)=1_6$.
		\end{enumerate}
	\end{remark}

	The dual group of $\mathrm{U}_4(F)$
	is
	$\mathrm{GL}_4(\mathbb{C})
	\simeq\mathrm{GSpin}_6(\mathbb{C})$.
	The $L$-group of $\mathrm{U}_4(F)$
	depends also on whether
	$V$ is split over $F$ or not.
	If $U$ is split over $F$,then
	$E\simeq F\oplus F$.
	In this case,
	$\mathrm{U}_4(F)\simeq\mathrm{GL}_4(F)$
	(also an isomorphism for both groups as
	algebraic groups),
	thus the $L$-group of
	$\mathrm{U}_4(F)$
	is just
	$\,^L\mathrm{U}_4(F)
	=\mathrm{GL}_4(\mathbb{C})\times
	\Gamma_{E/F}$.
	If $U$ is not split over $F$,
	then $E/F$ is a quadratic field extension.
	In this case,
	$^L\mathrm{U}(F)
	=\mathrm{GL}_4(\mathbb{C})
	\rtimes\Gamma_{E/F}$
	where
	$c$ acts on $\mathrm{GL}_4(\mathbb{C})$
	by
	$c(g)c:=J_4g^{-t}J_4^{-1}$
	(the same as in the case of
	$\mathrm{GSO}_6(F)$).
	As in the case of $\mathrm{GSO}_6(F)$,
	we write the $L$-group of
	$\mathrm{U}_4(F)$ as
	$\mathrm{GL}_4(\mathbb{C})
	\rtimes\Gamma_{E/F}$
	for both cases where the semi-direct product is understood
	to be direct product when $U$ is split over $F$.

	From the above discussion,
	we have the following identification of $L$-groups:
	\[
	\,^L\mathrm{U}_4(F)
	\simeq
	\,^L\mathrm{GSO}_6(F)
	\simeq
	\mathrm{GSpin}_6(\mathbb{C})
	\rtimes\Gamma_{E/F}
	\simeq
	\mathrm{GL}_4(\mathbb{C})
	\rtimes\Gamma_{E/F}.
	\]

    \subsection{Local theta lift from $\mathrm{GSp}^+(V)$
    to $\mathrm{GSO}(U)$}
    Let $F$ be the local field $\mathbb{Q}_v$.
    Now we consider the subgroup
    $\mathrm{GSp}^+(V)$ of $\mathrm{GSp}(V)$ over $F$
    consisting of elements $g$ such that
    $\nu(g)\in\nu(\mathrm{GSO}(U))$.
    
    One has the following
    (\cite[Proposition 2.3]{GanTakeda11})
    \begin{proposition}
    	Suppose that $\pi$ is a supercuspidal representation of
    	$\mathrm{GSp}^+(V)$, then
    	$\Theta(\pi)$ is either zero or is an irreducible representation
    	of $\mathrm{GSO}(U)$.
    	Moreover, if $\pi'$ is another supercuspidal
    	representation of $\mathrm{GSp}^+(V)$ such that
    	$\Theta(\pi)=\Theta(\pi')\neq0$,
    	then $\pi=\pi'$.
    \end{proposition}

    \subsubsection{Unramified local theta lift from 
    	$\mathrm{GSp}^+_4$
        to $\mathrm{GSO}_{4,2}$,
        non-archimedean case}
    Suppose $F$ is non-archimedean and
    that $-N$ 
    is a non-square in $F$,
    then $U$ is $F$-equivalent to the quadratic space
    $F(\sqrt{-N})\oplus\mathbb{H}^2$
    where $\mathbb{H}$ is the split hyperbolic plane.
    Thus
    $\mathrm{GSO}(U)\simeq\mathrm{GSO}_{4,2}$.

    The results for the theta lift can be found in
    \cite[Theorem 6.21]{Morimoto14}.
    Concerning the theta lift of unramified representations,
    we can proceed as follows.
    We define a map
    $\widetilde{\mu}
    \colon
    \mathrm{W}_F
    \rightarrow
    \{0,1\}$
    by
    $\widetilde{\mu}(g)=0$ if
    $\xi(g)=1$
    and
    $\widetilde{\mu}(g)=1$ otherwise.
    We then define a morphism of $L$-groups
    \[
    \iota
    \colon
    \,^L\mathrm{GSp}_4(F)
    =\mathrm{GSpin}_5(\mathbb{C})\times\Gamma_{E/F}
    \rightarrow
    \,^L\mathrm{GSO}_6(F)
    =\mathrm{GSpin}_6(\mathbb{C})\rtimes\Gamma_{E/F},
    \]
    \[
    (g,\sigma)
    \mapsto
    (\mathrm{diag}(g,\nu(g))_{\mathfrak{B}'},
    \sigma)
    \]
    where the subscript $\mathfrak{B}'$
    means that we consider
    the element
    $\mathrm{diag}(g\xi(\sigma),1)$ under the basis
    $\mathfrak{B}'=
    \{E_1',\cdots,E_5',(E_3-E_4)/\sqrt{2}\}$
    of $\wedge^2\mathbf{V}$.
    For $\xi(\sigma)=-1$,
    it is easy to see that
    under the basis
    $\mathfrak{B}=\{E_1,\cdots,E_6\}$ of
    $\wedge^2\mathbf{V}$,
    the element
    $\mathrm{diag}(\xi(\sigma)1_5,1)_{\mathfrak{B}'}$
    becomes
    $\mathrm{diag}(-1_2,-J_2,-1_2)_\mathfrak{B}
    =\epsilon$.

    One verifies that:
    \begin{lemma}
    	The map
    	$\iota$ is a morphism of groups.
    \end{lemma}
    \begin{proof}
    	It suffices to show that
    	$\widetilde{\epsilon}$ commutes with
    	each element in
    	the image of
    	$\mathrm{GSpin}_5(\mathbb{C})$
    	in $\mathrm{GPin}^+_6(\mathbb{C})$.
    	By Lemma \ref{action of c and epsilon},
    	$\widetilde{\epsilon}$ 
    	acts by conjugation on
    	$\mathrm{GSpin}_6(\mathbb{C})
    	=\mathrm{GL}_4(\mathbb{C})$
    	via the action of $c$.
    	Now for any
    	$g\in\mathrm{GSp}_4(\mathbb{C})$,
    	we have
    	$c(g)=J_4g^{-t}J_4
    	=g$.
    	We conclude that
    	$\widetilde{\epsilon}$
    	commutes with each element of
    	$\mathrm{GSpin}_5(\mathbb{C})$
    	viewed as subgroup of
    	$\mathrm{GPin}^+_6(\mathbb{C})$.
    \end{proof}

    We can interpret the above morphism in terms of
    classical groups:
    for
    $\widetilde{\mu}(\sigma)=1$,
    i.e., $\xi(\sigma)=-1$,
    we have
    (be careful about the definition of the
    $E_i$'s):
    \[
    (\rho'_\mathrm{st})^{-1}
    (\mathrm{diag}(-1_2,-J_2,-1_2)_\mathfrak{B}
    \widetilde{\epsilon}^{\widetilde{\mu}(\sigma)})
    =
    (\rho'_\mathrm{st})^{-1}(-1_6)
    =
    i\cdot1_4
    \in
    \mathrm{GL}_4(\mathbb{C}).
    \]
    We write $A$ for this element.
    Then the above morphism becomes
    \[
    \iota
    \colon
    \,^L\mathrm{GSp}_4(F)
    =
    \mathrm{GSp}_4(\mathbb{C})
    \times\Gamma_{E/F}
    \rightarrow
    \,^L\mathrm{GSO}_6(F)
    =
    \mathrm{GL}_4(\mathbb{C})\rtimes\Gamma_{E/F},
    \quad
     (h,\sigma)
    \mapsto
    (hA^{\widetilde{\mu}(\sigma)},\sigma).
    \]
    Note that the morphism
    $\iota$ is independent of whether 
    $U$ is split over $F$ or not.

    Now we have
    (\cite[Corollary 6.23]{Morimoto14}):
    \begin{theorem}\label{Langlands parameters for GSp^+4 and GSO6}
    	Let $\pi=\pi(s)$ be an unramified,
    	resp., twisted Steinberg
    	representation of $\mathrm{GSp}_4(F)$
    	corresponding to the semi-simple class
    	$s\in\mathrm{GSp}_4(\mathbb{C})\times
    	\Gamma_{E/F}$
        and $\pi_+(s)$ an irreducible constituent of
        $\pi|_G$.
        Then $\Theta^*(\pi(s)):=\Theta(\pi_+(s))$
        is the unramified,
        resp., twisted Steinberg
        representation
        $\mathrm{GSO}_6(F)$
        corresponding to the semi-simple class
        $\iota(s)\in
        \mathrm{GL}_4(\mathbb{C})
        \rtimes\Gamma_{E/F}$.
        More precisely,
        suppose that
        $\pi(s)$ is the unique irreducible unramified submodule of
        $\mathrm{Ind}_{B}(\chi_1,\chi_2,\chi)$
        and set
        $t_1=\chi_1(\omega_F)$,
        $t_2=\chi_2(\omega_F)$,
        $t=\chi(\omega_F)$,
        the Satake parameter 
        of $\pi$ is
        \[
        s=(\mathrm{diag}(t,t_1t,t_1t_2t,t_2t),\sigma)
        \in
        \mathrm{GSp}_4(\mathbb{C})\times
        \Gamma_{E/F},
        \]
        then the Satake parameter of
        $\Theta^*(\pi(s))$ is
        \[
        \iota(s)
        =
        s\cdot(A^{\widetilde{\mu}(\sigma)},1)
        =
        \begin{cases}
        (i\cdot\mathrm{diag}
        (t,t_1t,t_1t_2t,t_2t),\sigma) 
        & \mathrm{if}\, \xi(\sigma)=-1, \\
        (\mathrm{diag}(t,t_1t,t_1t_2t,t_2t),\sigma) 
        & \mathrm{otherwise.}
        \end{cases}
        \]

        In terms of $L$-factors, if $\pi$
        is unramified, we have
        \[
        L(s,\Theta^*(\pi),\mathrm{st})
        =
        L(s,\mathrm{St}(\pi)\otimes\xi)
        (1-\ell^{-s})^{-1}.
        \]
    \end{theorem}

    \subsubsection{Local theta lift from $\mathrm{GSp}^+_4$
    to $\mathrm{GO}_{3,3}$, non-archimedean case}
    Suppose that $F$ is non-archimedean and
    $-N$ is a square in
    $F$, then $U$ is $F$-equivalent to the quadratic space
    $\mathbb{H}^3$ and thus
    $\mathrm{GO}(U)\simeq\mathrm{GO}_{3,3}$.

    The results for the theta lift can be found in
    \cite[Theorem 8.3]{GanTakeda11}.
        Concerning the theta lifts of unramified representations,
    we have
    (\cite[Corollary 12.3]{GanTakeda11b}):
    \begin{theorem}\label{Langlands parameters for GSp4 and GSO6}
    	Let $\pi=\pi(s)$ be an unramified representation of
    	$\mathrm{GSp}_4$
    	corresponding to the semi-simple class
    	$s\in\mathrm{GSp}_4(\mathbb{C})
    	=\,^L\mathrm{GSp}_4$,
    	then
    	$\Theta(\pi(s))$ is the unramified representation of
    	$\mathrm{GSO}_{3,3}$
    	corresponding to the semi-simple class
    	$\iota(s)\in\mathrm{GL}_4(\mathbb{C})
    	\times\mathrm{GL}_1(\mathbb{C})
    	=\,^L\mathrm{GSO}_{3,3}$.
    	In terms of $L$-factors, we have
    	\[
    	L(s,\Theta(\pi),\mathrm{st})
    	=
    	L(s,\mathrm{St}(\pi)\otimes\xi)\zeta_F(s)
    	\]
    	where we note that the character $\xi$ is in fact a trivial
    	character
    	and $\zeta_F(s)=(1-\ell^{-s})^{-1}$ is the usual Euler factor.
    \end{theorem}

    \subsubsection{Archimedean theta lift from $\mathrm{GSp}^+_4$
    to $\mathrm{GO}_6$}
    For the case $F=\mathbb{R}$,
    the group $\mathrm{GO}(U)$ is just
    $\mathrm{GO}_6=\mathrm{GO}_6(\mathbb{R})$.
    Moreover, it is easy to see
    $\nu(\mathrm{GO}_6(\mathbb{R}))
    =\mathbb{R}^\times$
    and thus
    $\mathrm{GSp}_4^+=\mathrm{GSp}_4(\mathbb{R})$.
    The results for this theta lift can bu found in
    \cite{Paul05}.
    For the completeness of the article and 
    the convenience of the reader,
    we reproduce the relevant results of \cite{Paul05}.
    See also \cite[Section 7]{Morimoto14}.
    
    We first recall the classification of discrete series representations of
    $\mathrm{Sp}_4(\mathbb{R})$.
    We have the characters
    $e_1,e_2$ of $T_{\mathrm{Sp}_4}$.
    Then the set of roots of $\mathrm{Sp}_4$
    is $\Sigma=\{\pm e_1\pm e_2,\pm2e_1,\pm2e_2\}$.
    A set of positive compact roots is
    $\Sigma_\mathrm{c}^+
    =\{e_1-e_2\}$.
    Then positive root systems containing
    $\Sigma_\mathrm{c}^+$ are the following ones
    \[
    \Sigma_1^+
    =
    \{
    e_1-e_2, 2e_1, e_1+e_2, 2e_2
    \},
    \quad
    \Sigma_2^+
    =
    \{
    e_1-e_2, -2e_1, -e_1-e_2, -2e_2
    \},
    \]
    \[
    \Sigma_3^+
    =
    \{
    e_1-e_2, 2e_1, e_1+e_2, -2e_2
    \},
    \quad
    \Sigma_4^+
    =
    \{
    e_1-e_2, 2e_1, -e_1-e_2, -2e_2
    \}.
    \]
    The corresponding sets of dominant weights are
    ($i=1,2,3,4$)
    (recall that $\langle e_j,e_k\rangle=2\delta_{j,k}$)
    \[
    X_i
    =
    \{
    (\lambda_1,\lambda_2)
    \in\mathbb{Z}^2|
    \langle
    \lambda_1e_1+\lambda_2e_2,\alpha
    \rangle>0,
    \forall \alpha\in\Sigma_i^+
    \}.
    \]
    The Harish-Chandra parameter of the discrete series representations of
    $\mathrm{Sp}_4(\mathbb{R})$
    is then the union $\bigcup_iX_i$.
    Moreover,
    $X_1$ parameterizes
    holomorphic discrete series representations,
    $X_2$ parameterizes
    anti-holomorphic discrete series representations
    while
    $X_3\bigcup X_4$ parameterizes
    generic discrete series representations.
    
    We then recall the classification of discrete series representations
    of $\mathrm{SO}_6(\mathbb{R})$.
    We have the characters
    $e_1',e_2',e_3'$ of $T_{\mathrm{SO}_6}$.
    The set of roots of $\mathrm{SO}_6$
    is $\Sigma'=\{\pm e_i\pm e_j|1\leq i<j\leq3\}$.
    A set of positive compact roots is
    $\Sigma_\mathrm{c}'^{+}=\{e_i'\pm e_j'|1\leq i<j\leq 3\}$.
    Then the positive root system compatible containing
    $\Sigma_\mathrm{c}'^+$ is
    $\Sigma'^+=\Sigma_\mathrm{c}'^+$.
    The corresponding set of dominant weights are
    \[
    X'
    =
    \big\{
    (\lambda_1,\lambda_2,\lambda_3)\in\mathbb{Z}^3|
    \langle
    \lambda_1e_1'+\lambda_2e_2'+\lambda_3e_3',
    \alpha
    \rangle
    >0,
    \forall \alpha\in\Sigma'^+
    \big\}
    =
    \big\{
    (\lambda_1,\lambda_2,\lambda_3)\in\mathbb{Z}^3|
    \lambda_1>\lambda_2>|\lambda_3|
    \big\}.
    \]
    
    If $\pi$ is a discrete series of $\mathrm{Sp}_4(\mathbb{R})$
    in
    $\Sigma_1^+$, resp. $\Sigma_2^+$ of weight
    $(\lambda_1,\lambda_2)$
    (such that $\lambda_1>\lambda_2>0$,
    resp. $0>\lambda_1>\lambda_2$),
    then by \cite[Theorem 15]{Paul05},
    the theta lift
    $\Theta(\pi)$ to
    $\mathrm{SO}_6(\mathbb{R})$ is non-zero
    of weight $(\lambda_1,\lambda_2,0)$.
    If $\pi$ is in
    $\Sigma_3^+$
    of weight
    $(\lambda_1,\lambda_2)$
    (such that
    $\lambda_1>|\lambda_2|$ and $\lambda_2<0$),
    then again by
    \cite[Theorem 15]{Paul05},
    the theta lift $\Theta(\pi)$
    to $\mathrm{SO}_6(\mathbb{R})$
    vanishes.
    Similarly, if $\pi$ is in $\Sigma_4^+$,
    the theta lift
    $\Theta(\pi)$ vanishes too.
    From this we conclude that
    the theta lift of a (anti-)holomorphic discrete 
    series of
    $\mathrm{Sp}_4(\mathbb{R})$
    to $\mathrm{SO}_6(\mathbb{R})$
    is non-zero
    while the theta lift of generic discrete series vanishes.
    The theta lift for the corresponding similitude groups
    can be similarly discussed and the result is the same:
    the theta lift of (anti-)holomorphic discrete series of
    $\mathrm{GSp}_4(\mathbb{R})$
    to $\mathrm{GSO}_6(\mathbb{R})$
    is non-zero while
    theta lift of generic discrete series vanishes.

   \subsection{Transfer between $\mathrm{GSp}_4(F)$ and
   	$\mathrm{Sp}_4(F)$,
   	$\mathrm{U}_4(F)$ and $\mathrm{SU}_4(F)$}
   
   In this subsection, we review some results on the
   transfer of representations between
   $\mathrm{GSp}_4(F)$ and $\mathrm{Sp}_4(F)$,
   $\mathrm{U}_4(F)$ and $\mathrm{SU}_4(F)$,
   globally and locally
   (\cite[Section 3]{LabesseSchwermer86}).
   In order to have a uniform treatment, we write
   $G$ to be
   $\mathrm{Sp}_4(F)$
   (for $F$ local field),
   $\mathrm{Sp}_4(\mathbb{A}_F)$
   (for $F$ global field),
   $\mathrm{SU}_4(F)$
   or $\mathrm{SU}_4(\mathbb{A}_F)$,
   resp.,
   $\widetilde{G}$ to be
   $\mathrm{GSp}_4(F)$,
   $\mathrm{GSp}_4(\mathbb{A}_F)$,
   $\mathrm{U}_4(F)$
   or $\mathrm{U}_4(\mathbb{A}_F)$.

   \subsubsection{$L$-packets}
   
   Note that in any case,
   $G$ is a normal subgroup of
   $\widetilde{G}$.
   For any irreducible admissible representation
   $\pi$ of $G$ and any element $g\in\widetilde{G}$,
   we define a new representation
   $\pi^g$ of $G$ as follows:
   $\pi^g(h):=\pi(ghg^{-1})$ for any $h\in G$.
   We say that two irreducible admissible representations
   $\pi$ and $\pi'$ of $G$ are
   equivalent if
   $\pi'\simeq\pi^g$ for some $g\in\widetilde{G}$.
   We denote by
   $\mathscr{L}(G)$ the quotient set of the set of
   irreducible admissible representation of $G$ by 
   this equivalence relation
   and each element in $\mathscr{L}(G)$ is an
   $L$-packet of $G$.
   The usual equivalence relation on representations of
   $\widetilde{G}$ defines the
   $L$-packets of $\widetilde{G}$.
   We write the set of $L$-packets of $\widetilde{G}$
   as $\mathscr{L}(\widetilde{G})$.
   We also write
   $\mathscr{E}=(\widetilde{G}/G)^\vee$ for the 
   set of characters of the quotient
   $\widetilde{G}/G$.
   Moreover, we say two
   irreducible admissible representations
   $\pi$ and $\pi'$ of $\widetilde{G}$
   $\mathscr{E}$-equivalent if
   $\pi'\simeq\pi\otimes\chi$ 
   for some $\chi\in\mathscr{E}$.
   We then write
   $\mathscr{E}(\widetilde{G})$
   the set of $\mathscr{E}$-equivalent classes of 
   irreducible admissible representations of
   $\widetilde{G}$.
   Thus there is a canonical projection
   $L(\widetilde{G})
   \twoheadrightarrow
   \mathscr{E}(\widetilde{G})$.
   
   One can show the following
   \begin{lemma}
   	For an irreducible admissible representation 
   	$\widetilde{\pi}$ of
   	$\widetilde{G}$,
   	its restriction to $G$ depends only on its class in
   	$\mathscr{E}(\widetilde{G})$.
   	Moreover,
   	$\widetilde{\pi}|_G$
   	is a direct sum of irreducible admissible 
   	representations of $G$.
   \end{lemma}
   The proof is the same as for
   \cite[Lemma 3.2]{LabesseSchwermer86}.
   \begin{proof}
   	The first point follows from the definition of
   	$\mathscr{E}(\widetilde{G})$.
   	
   	For the second point, first assume that
   	$F$ is local. Let $\widetilde{Z}$
   	be the center of $\widetilde{G}$.
   	Note that $\widetilde{Z}G\backslash\widetilde{G}$
   	is compact and abelian
   	(as connected reductive algebraic groups,
   	we have an $F$-isogeny
   	$ZG\rightarrow\widetilde{G}\rightarrow1$).
   	By \cite[Theorem]{Silberger79},
   	one knows that
   	$\widetilde{\pi}|_{\widetilde{Z}G}$ is a direct sum of irreducible admissible
   	representations of $\widetilde{Z}G$.
   	Since $\widetilde{Z}$ acts on each irreducible factor in
   	$\widetilde{\pi}|_{\widetilde{Z}G}$ by scalars, we see that
   	$\widetilde{\pi}|_G$ is also a direct sum of
   	irreducible admissible representations of $G$.
   	For $F$ global, we can decompose
   	$\widetilde{\pi}$ into local components and apply
   	the above result.
   \end{proof}
   \begin{remark}
   	For $F$ local, if $\widetilde{\pi}$ is an unramified representation of
   	$\widetilde{G}$,
   	then its restriction $\widetilde{\pi}|_G$
   	has only one irreducible admissible submodule which is
   	unramified.	
   \end{remark}

   \subsubsection{The transfer}
   First we assume $F$ local.
   For an irreducible admissible representation
   $\widetilde{\pi}$ of $\widetilde{G}$,
   fix $\pi$ an irreducible submodule of the restriction
   $\widetilde{\pi}|_G$.
   Then $\widetilde{\pi}|_G$ is a direct sum of representations
   of the form $\pi^g$ for some $g\in\widetilde{G}$.
   Now we can define the following map
   \[
   R\colon
   \mathscr{E}(\widetilde{G})
   \rightarrow
   \mathscr{L}(G),
   \widetilde{\pi}
   \mapsto
   \pi.
   \]
   By the above discussion,
   $R$ is well-defined and one can show by
   Frobenius reciprocity that $R$ is bijective.

   Now assume $F$ global.
   We say an $L$-packet of $G$ is
   cuspidal if some element in this $L$-packet
   is a cuspidal representation.
   We write $\mathscr{L}_0(G)$ for the
   subset of $\mathscr{L}(G)$
   consisting of cuspidal $L$-packets.
   Similarly we define cuspidal equivalence classes of
   in $\mathscr{E}(\widetilde{G})$
   and denote by
   $\mathscr{E}_0(\widetilde{G})$
   the subset of
   $\mathscr{E}(\widetilde{G})$
   consisting of cuspidal classes.
   Then one can show that the restriction of
   $R$ to $\mathscr{E}_0(\widetilde{G})$
   is again bijective
   (\cite[Lemma 1]{Silberger79}):
   \[
   R
   \colon
   \mathscr{E}_0(\widetilde{G})
   \rightarrow
   \mathscr{L}_0(G).
   \]
   As above, we write
   $\widetilde{Z}$ for the center of $\widetilde{G}$ and
   $Z=G\cap\widetilde{Z}$.
   Fix a unitary Hecke character $\widetilde{\chi}$ of
   $\widetilde{Z}$
   and $\chi=\widetilde{\chi}|_Z$.
   We write $\rho_\chi$ for the representation of
   $G$ by right translation on the space of
   $\mathbb{C}$-valued cuspidal square-integrable functions
   $L_0^2([G],\chi)$
   on which $Z$ acts by the character $\chi$.
   One can extend $\rho_\chi$ to a representation
   $\rho_\chi'$ on the same space of the group
   $G':=Z\widetilde{G}(F)G$
   as follows:
   for any
   $z\gamma g\in Z\widetilde{G}(F)G$
   and $f\in L_0^2([G],\chi)$,
   we set
   $\rho_\chi'(z\gamma g)f(x)
   :=\widetilde{\chi}(z)f(\gamma^{-1}x\gamma g)$.
   In the same manner as $\rho_\chi$, one can define
   the representation
   $\rho_{\widetilde{\chi}}$ of $\widetilde{G}$
   on the space
   $L_0^2([\widetilde{G}],\widetilde{\chi})$
   of square-integrable functions by right translation
   such that $\widetilde{Z}$ acts on by the character
   $\widetilde{\chi}$.
   Then one has
   \[
   \rho_{\widetilde{\chi}}
   =
   \mathrm{Ind}_{G'}^{\widetilde{G}}(\rho'_\chi).
   \]
   Moreover, $\rho_\chi'$ is a direct sum of
   irreducible admissible representations of
   $G'$ (with multiplicity one) and each irreducible
   cuspidal representation $\widetilde{\pi}$ of $\widetilde{G}$
   of central character $\widetilde{\chi}$
   occurs (with multiplicity one) in
   $\mathrm{Ind}_{G'}^{\widetilde{G}}(\pi')$
   for some submodule $\pi'$ of $\rho'_\chi$.
   As above, for any such representation $\pi'$ of
   $G'$, let $\pi$ be an irreducible submodule
   of $\pi'|_G$. 
   Then $\pi'|_G$ is a direct sum of
   $\pi^g$ for some $g\in\widetilde{G}$.
   Then we see that
   $\pi$ and $R(\widetilde{\pi})$ are in the same
   $L$-packet of $G$.

   Now we restrict ourselves to automorphic representations.
   Denote by
   $\mathscr{EA}(\widetilde{G})$,
   resp., $\mathscr{EA}_0(\widetilde{G})$,
   the subset of
   $\mathscr{E}(\widetilde{G})$,
   resp., $\mathscr{E}_0(\widetilde{G})$,
   consisting of irreducible, resp., irreducible cuspidal,
   automorphic representations of $\widetilde{G}$.
   We denote also by
   $\mathscr{LA}(G)$,
   resp., $\mathscr{LA}_0(G)$,
   the subset of
   $\mathscr{L}(G)$ consisting of
   irreducible, resp., irreducible cuspidal
   automorphic representations of $G$.
   Using \cite[Propisition 2]{Langlands79},
   one has
   \begin{lemma}\label{bijection of L-packets}
   	The map $R$ defined above induces bijections
   	\[
   	R
   	\colon
   	\mathscr{EA}(\widetilde{G})
   	\rightarrow
   	\mathscr{LA}(G),
   	\mathscr{EA}_0(\widetilde{G})
   	\rightarrow
   	\mathscr{LA}_0(G).
   	\]
   \end{lemma}

   In summary, let
   $\widetilde{\chi}$ be a Hecke character of $\widetilde{Z}$
   and $\chi=\widetilde{\chi}|_Z$,
   then one has
   \begin{corollary}\label{extension of automorphic representations}
   	Any irreducible admissible 
   	(resp., cuspidal) automorphic representation
   	$\pi$ of $G$ of central character $\chi$ can be extended to
   	an irreducible admissible (resp., cuspidal)
   	automorphic representation
   	$\widetilde{\pi}$ of
   	$\widetilde{G}$ of central character $\widetilde{\chi}$.
   \end{corollary}
   \begin{remark}\label{always possible to extend representations}
   	Since one can always extend a Hecke character $\chi$
   	of $Z$ to a character $\widetilde{\chi}$
   	of $\widetilde{Z}$,
   	the above corollary shows that
   	we can always extend an irreducible admissible
   	(resp., cuspidal) automorphic representation of
   	$G$ of central character $\chi$
   	to an irreducible admissible (resp., cuspidal)
   	automorphic representation 
   	$\widetilde{\pi}$ of $\widetilde{G}$.
   	Moreover, for any irreducible cuspidal automorphic
   	representation $\widetilde{\pi}$ of $\widetilde{G}$,
   	all the irreducible submodules of
   	$\widetilde{\pi}|_G$ have the same
   	unramified local components and thus the same
   	partial $L$-function.
   \end{remark}

   For an automorphic form
   $f$ on $G(\mathbb{A}_F)$, we can extend it
   to an automorphic form $\widetilde{f}$
   on $\widetilde{G}(\mathbb{A}_F)$
   as follows:
   for
   $g=g_1g_2$
   with $g_1\in\widetilde{G}(F)$
   and $g_2\in G(\mathbb{A}_F)$,
   we set
   $\widetilde{f}(g)=f(g_2)$,
   otherwise,
   $\widetilde{f}(g)=0$.

    \subsection{Theta lift from $\mathrm{GSp}_4$
    	to $\mathrm{U}_4$}
    \label{lift from GSp4 to U4}
    In this subsection, we use the previous results
    on the theta lift of the reductive dual pair
    $(\mathrm{GSp}^+(V),\mathrm{GSO}(U))$
    to deduce the transfer/theta lift
    from $\mathrm{GSp}_4$ to $\mathrm{U}_4$.

    Let $\pi=\otimes'_v\pi_v$ be a 
    cohomological cuspidal
    irreducible automorphic representation of
    $\mathrm{GSp}_4(\mathbb{A})$
    of Iwahori level
    $Np^m$ and of trivial central character.
    Let $f\in\pi$ be an automorphic form
    which is $p$-integral.
    Write $f_+$ its restriction to
    $\mathrm{Sp}_4(\mathbb{A})$.
    The theta lift
    $\Theta_\phi(f_+)$ to
    $\mathrm{SO}_6(\mathbb{A})$
    can be viewed as an automorphic form on
    $\mathrm{SU}_4(\mathbb{A})$
    and then extended by zero to an automorphic form
    on $\mathrm{U}_4(\mathbb{A})$
    (see below Remark 
    \ref{always possible to extend representations}),
    which we denote by
    $\Theta^*_\phi(f)$.
    Then write
    $\Theta^*(\pi)
    =\{\Theta^*_\phi(f)|f\in\pi,
    \phi\in\mathcal{S}(W_4^-(\mathbb{A}))\}$
    for the transfer of the representation
    $\pi$ on
    $\mathrm{GSp}_4(\mathbb{A})$
    to $\mathrm{U}_4(\mathbb{A})$.
    In terms of automorphic representations,
    write $\pi^+$ for an irreducible submodule of
    $\pi_+$.
    Then
    $\pi_+=\oplus_{i\in I}(\pi^+)^{g_i}$
    and all these
    $(\pi^+)^{g_i}$ are in the same $L$-packet of
    $\mathrm{Sp}_4(\mathbb{A})$.
    Then $\Theta^*(\pi)=\sum_i\Theta((\pi^+)^{g_i})$
    with $g_i\in\mathrm{GSp}_4(\mathbb{A})$
    as representations of $\mathrm{U}_4(\mathbb{A})$
    (Lemma \ref{bijection of L-packets}
    and Corollary 
    \ref{extension of automorphic representations}).

    To see the relation of the Langlands parameters of
    unramified components of
    $\pi$ and $\Theta^*(\pi)$,
    we proceed as follows:
    for any unramified local component
    $\pi_v$ on
    $\mathrm{GSp}_4(\mathbb{Q}_v)$ of $\pi$,
    let $\pi_v^+$ be the unique
    irreducible admissible unramified submodule of
    $(\pi_v)_+$, the restriction of
    $\pi_v$ to
    $\mathrm{GSp}^+_4(\mathbb{Q}_v)$.
    Suppose that
    $\pi_v^+$ corresponds to the semi-simple class
    $s(\pi_v)\in\mathrm{GSp}_4(\mathbb{C})
    \rtimes\Gamma_{E_v/\mathbb{Q}_v}$,
    then
    for any irreducible submodule $\Theta^*(\pi)^+$ of
    $\Theta^*(\pi)$,
    its unramified local component
    $(\Theta^*(\pi)^+)_v$
    corresponds to the semi-simple class
    $\iota(s(\pi_v))\in
    \mathrm{GL}_4(\mathbb{C})
    \rtimes\Gamma_{E_v/\mathbb{Q}_v}$
    (Theorems 
    \ref{Langlands parameters for GSp^+4 and GSO6}
    and
    \ref{Langlands parameters for GSp4 and GSO6}).
    The same result holds for the Steinberg
    components of $\pi$ and
    $\Theta(\pi)$
    by the same theorems.

    We summarize the discussion into
    the following:
    \begin{theorem}\label{Satake parameter relation of theta lift}
    	Let $\pi$ be a 
    	cohomological cuspidal irreducible automorphic representation
    	of $\mathrm{GSp}_4(\mathbb{A})$
    	of Iwahori level $Np^m$.
    	If the automorphic representation $\Theta^*(\pi)$ of
    	$\mathrm{U}_4(\mathbb{A})$ is not zero,
    	let the representations 
    	$\rho_\pi
    	\colon
    	\Gamma_\mathbb{Q}
    	\rightarrow
    	\mathrm{GSp}_4(\mathbb{C})$,
    	resp.,
    	$\rho_{\Theta^*(\pi)}
    	\colon
    	\Gamma_\mathbb{Q}
    	\rightarrow
    	\mathrm{GL}_4(\mathbb{C})
    	\rtimes\Gamma_{E/\mathbb{Q}}$,
    	be
    	the Galois representations associated to
    	$\pi$, resp.,
    	$\Theta^*(\pi)$.
    	Then
    	$\rho_{\Theta^*(\pi)}(\sigma)
    	=
    	(\rho_\pi(\sigma),\overline{\sigma})$,
    	where $\overline{\sigma}$ is the image of
    	$\sigma$ under the projection
    	$\Gamma_{\mathbb{Q}}
    	\rightarrow
    	\Gamma_{E/\mathbb{Q}}$.
    	The adjoint representations of
    	$\rho_\pi$ and
    	$\rho_{\Theta^*(\pi)}$ are related by
    	\[
    	\mathrm{ad}(\rho_{\Theta^*(\pi)})|_{\widetilde{\mathbf{V}}}
    	=\mathrm{ad}(\rho_\pi)|_{\widetilde{\mathbf{V}}}
    	\times\xi.
    	\]
    	Here $\widetilde{\mathbf{V}}$
    	is the subspace of
    	$\mathrm{M}_{4\times4}(\mathbb{C})$
    	defined in Remark 
    	\ref{orthogonal and special orthogonal, adjoint representation}.
    \end{theorem}

    \section{Selmer groups, congruence ideals and $L$-values}
    In this section,
    we first establish morphisms between
    the Hecke algebras and universal deformation rings
    on $\mathrm{GSp}_4$ and $\mathrm{U}_4$.
    This permits us to relate
    Selmer groups to congruence ideals
    and finally using theta correspondence
    we get an identity of the characteristic element of
    the Selmer group and the special
    $L$-value. 
    
    \subsection{Hecke algebras and Galois representations}
    In this subsection we will establish homomorphisms
    of Hecke algebras and Galois representations
    on the groups
    $\mathrm{GSp}_4$ and $\mathrm{U}_4$
    using the theta lift in the preceding section.

    \subsubsection{Case: $\mathrm{GSp}_4$}
    Recall that we have
    defined Hecke operators on
    $G=\mathrm{GSp}_4$:
    for any $\mathbb{Z}$-algebra $R$,
    the spherical Hecke operator
    $
    \mathcal{H}(G(\mathbb{Q}_\ell),G(\mathbb{Z}_\ell),
    R)$,
    the dilating Iwahori operator
    $\mathcal{H}^-(G(\mathbb{Q}_\ell),
    I_{G,p},R)$.
    Then the abstract global Hecke algebra of
    $G(\mathbb{Q})$
    is the restricted tensor product:
    \[
    H^\mathrm{s}(R)
    =
    (\bigotimes_{\ell\nmid Np}\,'
    \mathcal{H}(G(\mathbb{Q}_\ell),G(\mathbb{Z}_\ell),
    R)
    \bigotimes
    \mathcal{H}^-(G(\mathbb{Q}_\ell),
    I_{G,p},R).
    \]
    Moreover, we define the Hecke polynomial as
    \[
    P_\ell(X)
    =
    X^4-T^{(2)}_{\ell,0}(\ell)X^3
    +(T^{(2)}_{\ell,0}(\ell)-T_{\ell,1}^{(2)}(\ell^2)-
    \ell^2T_{\ell,2}^{(2)}(\ell^2))X^2
    -\ell^3T_{\ell,0}^{(2)}(\ell)
    T_{\ell,1}^{(2)}(\ell^2)X
    +\ell^6T_{\ell,2}^{(2)}(\ell^2)^2.
    \]

    We have also defined the space of
    cuspidal automorphic forms
    $S_{\underline{k}}(\widehat{\Gamma},R)$
    on
    $\mathrm{GSp}_4(\mathbb{A})$
    of weight $\underline{k}$,
    with coefficients in $R$
    and invariant under $\widehat{\Gamma}$.
    To avoid possible confusion of notations
    with the case of $\mathrm{U}_4$,
    we add a superscript
    `s' to this space:
    $S^\mathrm{s}_{\underline{k}}(\widehat{\Gamma},
    R)$.
    Then $H^{\mathrm{s}}(R)$
    acts on the space
    $S_{\underline{k}}^\mathrm{s}(\widehat{\Gamma},
    R)$
    and 
    the image of
    $H^\mathrm{s}(R)$
    in
    $\mathrm{End}_R
    (S_{\underline{k}}^\mathrm{s}(\widehat{\Gamma},
    R))$
    is denoted by $\mathbb{T}^\mathrm{s}_{\underline{k}}
    (\widehat{\Gamma},R)$.
    Suppose moreover that $R$ is $p$-adically complete,
    then
    as in Section 2, we have the idempotent element
    $e\in \mathbb{T}^\mathrm{s}_{\underline{k}}
    (\widehat{\Gamma},R)$
    and also the ordinary parts
    $\mathbb{T}_{\underline{k}}^\mathrm{s,ord}
    (\widehat{\Gamma},R)
    :=e\mathbb{T}_{\underline{k}}^\mathrm{s}
    (\widehat{\Gamma},R)$
    and
    $S_{\underline{k}}^\mathrm{s,ord}(\widehat{\Gamma},R)
    :=eS_{\underline{k}}^\mathrm{s}(\widehat{\Gamma},R)$.   
    Finally we write
    $S^\mathrm{s,ord}(\widehat{\Gamma}^\infty,R)$
    for the space of ordinary $p$-adic modular forms
    on $\mathrm{GSp}_4(\mathbb{Q})$,
    which is the projective limit of
    $S^\mathrm{s,ord}(\widehat{\Gamma}(N,p^m),R)$
    for $m\in\mathbb{Z}_{>0}$
    (in \cite[Th\'{e}or\`{e}me 2.1(2)]{Pilloni2012},
    it is denoted by $\mathcal{V}^{ord,*}_{cusp}$).
    This is a finite free module over
    $\Lambda^\mathrm{s}$
    where
    $\Lambda^\mathrm{s}$ is the Iwasawa weight algebra
    given by
    \begin{align*}
    \Lambda^\mathrm{s}
    :=
    \mathcal{O}[[T(1+p\mathbb{Z}_p)]]
    &
    \simeq
    \mathcal{O}[[X_1,X_2,X_3]],
    \\
    \mathrm{diag}(1_{2-i},(1+p)
    \cdot1_i,(1+p)^2\cdot1_{2-i},(1+p)\cdot1_i)
    &
    \mapsto
    1+X_{i+1}
    \quad
    (i=0,1,2)
    \end{align*}
    We then define the big ordinary Hecke algebra
    $\mathbb{T}^\mathrm{s,ord}
    (\widehat{\Gamma}^\infty,\mathcal{O})$
    as the image of
    $H^{\mathrm{s}}
    \otimes_\mathbb{Z}\Lambda^\mathrm{s}$
    in the endomorphism algebra
    $\mathrm{End}_{\mathcal{O}}
    (S^{\mathrm{s,ord}}(\widehat{\Gamma}^\infty,\mathcal{O}))$
    (in \cite{Pilloni2012} below Th\'{e}or\`{e}me 6.1, 
    it is denoted by
    $\widetilde{\mathbb{T}}$).

    Next we recall several properties of the Galois
    representation associated to 
    an automorphic representation of
    $\mathrm{GSp}_4(\mathbb{A})$.
    Let $\pi=\otimes_v\pi_v$ be a cohomological cuspidal 
    irreducible automorphic representation
    of $\mathrm{GSp}_4(\mathbb{A})$
    of weight
    $\underline{k}=(k_0,k_1,k_2)$
    with $k_1\geq k_2$.
    Moreover, assume that
    $\pi$ is not CAP,
    $\pi_\infty$ is a discrete series
    holomorphic representations of
    $\mathrm{GSp}_4(\mathbb{R})$
    and $\pi_\ell$ is unramified outside $Np$ and
    Steinberg representation at $p$
    (by \cite[Tables 1 and 15]{RobertsSchmidt2007},
    $\pi_p$ has a unique vector fixed by the Iwahori subgroup
    $\mathrm{Iw}_p^{0,1}$ of 
    $\mathrm{GSp}_4(\mathbb{Z}_p)$).
    Then there is a $p$-adic Galois representation
    $\rho_\pi
    \colon
    \Gamma_\mathbb{Q}
    \rightarrow
    \mathrm{GSp}_4(\mathcal{O})$
    associated to
    $\pi$
    (\cite{Taylor1993,GenestierTilouine2005,
    	Jorza2012,Weissauer2005}).
    Then for any local component 
    $\pi_\ell$,
    we have
    an isomorphism of semi-simplified
    Weil-Deligne representations
    \begin{equation}\label{Weil-Deligne vs. Galois, GSp(4)}
    \mathrm{WD}(\rho_\pi|_{D_\ell})^\mathrm{ss}
    \simeq
    \iota(\mathrm{rec}(\pi_\ell\otimes|\cdot|^{-3/2})^\mathrm{ss}).
    \end{equation}
    Here $\mathrm{rec}$ is the local Langlands correspondence for
    $\mathrm{GSp}_4$,
    $D_\ell$ is the decomposition group at $\ell$,
    and $\iota\colon\mathbb{C}\simeq\overline{\mathbb{Q}}_p$
    is the isomorphism fixed since the beginning
    (\cite[Theorem A]{Jorza2012}).
    On the other hand,
    we can relate the Frobenius to the Hecke polynomial
    as follows
    (\cite[Theorem I]{Weissauer2005}): 
    let $X$ be a variable and $\mathrm{Fr}_\ell$
    the Frobenius element at $\ell$,
    if $\ell\nmid Np$, then
    \[
    \mathrm{det}(X-\mathrm{Fr}_\ell,\rho_\pi|_{D_\ell})
    =
    P_\ell(X).
    \]
    Moreover, $\rho_{\pi,p}:=\rho_\pi|_{D_p}$ is crystalline
    (\cite{Faltings1989}).
    We write
    $\rho_{\pi,p}^\mathrm{crys}
    =(\rho_{\pi,p}\otimes\mathrm{B}_\mathrm{crys})^{D_p}$
    where $\mathrm{B}_\mathrm{crys}$
    is Fontaine's $p$-adic period ring
    and the absolute Frobenius element
    $\mathrm{Fr}_p'$ acts on
    $\rho_{\pi,p}^\mathrm{crys}$.
    Then we have
    (\cite[Th\'{e}or\`{e}me 1]{Urban2005})
    \[
    \mathrm{det}(X-\mathrm{Fr}_p',\rho_{\pi,p}^\mathrm{crys})
    =
    P_p(X).
    \]
    Suppose that $\pi$ is ordinary at $p$
    (\textit{cf}. \cite{Pilloni2012} below
    Th\'{e}or\`{e}me 3.3,
    i.e., $\pi_p$ is spherical and 
    the entries in $\lambda_\pi(\mathrm{diag}
    (v_1u_1,v_2u_1,u_1u_1,v_1u_2))
    \in(\overline{\mathbb{Q}}^\times)^4/W_G$ 
    have $p$-adic valuations
    $\{0, k_2+1, k_1+2, k_1+k_2+3\}$).
    We write these numbers as
    $\alpha_i$ ($i=1,\cdots,4$)
    with $p$-adic valuations
    $0,k_2+1,k_1+2,k_1+k_2+3$.
    Write
    $\chi_p
    \colon
    D_p\rightarrow
    \mathbb{Z}_p^\times$
    for the $p$-adic cyclotomic character
    and
    for any $a\in\overline{\mathbb{Z}}_p^\times$,
    $\mathrm{ur}(a)$ denotes the unramified character of
    $D_p$ such that
    $\mathrm{ur}(a)(\mathrm{Fr}_p)=a$.
    If we assume that
    $\rho_\pi$ is irreducible, then
    $\rho_{\pi,p}$
    is conjugate to an upper triangular representation
    \begin{equation}\label{crystalline at p for GSp(4)}
    \rho_{\pi,p}
    \simeq
    \begin{pmatrix}
    \chi_1 & * & * & * \\
    0 & \chi_2 & * & * \\
    0 & 0 & \chi_3 & * \\
    0 & 0 & 0 & \chi_4
    \end{pmatrix}
    \end{equation}
    whose diagonal entries are
    characters defined as
    $\chi_1:=\mathrm{ur}(\alpha_1)$,
    $\chi_2:=\mathrm{ur}(\alpha_2p^{-k_2-1})\chi_p^{-k_2-1}$,
    $\chi_3:=\mathrm{ur}(\alpha_3p^{-k_1-2})\chi_p^{-k_1-2}$
    and
    $\chi_4:=\mathrm{ur}(\alpha_4p^{-k_1-k_2-3})\chi_p^{-k_1-k_2-3}$
    (\cite[Corollaire 1]{Urban2005}).   
    We write
    $\overline{\rho}_\pi
    \colon
    \Gamma_\mathbb{Q}
    \rightarrow
    \mathrm{GSp}_4(\mathbb{F})$
    to be
    the residual Galois representation of
    $\rho_\pi$.
    Then there is a filtration
    $(\overline{\mathrm{Fil}}_i)_{i=0}^4$ on
    $\overline{\rho}_{\pi,p}$
    such that
    $\overline{\mathrm{Fil}}_i/
    \overline{\mathrm{Fil}}_{i-1}=
    \overline{\chi}_i$.
    We define a unipotent matrix
    (corresponding to the matrix
    $\epsilon$ in \cite[Section 3.4.1]{Pilloni2012})
    \[
    \varepsilon
    =
    \begin{pmatrix}
    0 & 1 & 0 & 0 \\
    0 & 0 & 0 & 1\\
    0 & 0 & 0 & 0 \\
    0 & 0 & -1 & 0
    \end{pmatrix}.
    \]
    Then for $\ell|N$,
    the restriction $\overline{\rho}_\pi|_{I_\ell}$
    to the inertia subgroup $I_\ell$
    of $\overline{\rho}_\pi$
    has image in the unipotent subgroup
    $\mathrm{exp}(\mathbb{F}\varepsilon)
    \subset\mathrm{GSp}_4(\mathbb{F})$
    by the local-global compatibility
    of Galois representations associated to 
    Siegel automorphic forms
    (\cite[Theorems A and B]{Sorensen2010},
     \cite[Theorem A]{Jorza2012})
    and the classification of Iwahori-spherical representations
    (\cite[Table 15]{RobertsSchmidt2007}).

    We define deformation functors and
    their universal deformation rings for
    $\overline{\rho}_{\pi,p}$.
    We write
    $\mathrm{CNL}_\mathcal{O}$
    for the category of complete noetherian local
    $\mathcal{O}$-algebras
    $A$ with maximal ideal $\mathfrak{m}_A$
    such that $A/\mathfrak{m}_A=\mathbb{F}$.
    Similarly we write
    $\mathrm{AL}_\mathcal{O}$
    for the subcategory of
    $\mathrm{CNL}_\mathcal{O}$
    of (complete) artinian local $\mathcal{O}$-algebras.
    We define
    deformation functors as follows
    \[
    \mathcal{D}_\pi,
    \mathcal{D}_{\underline{k}''}
    \colon
    \mathrm{AL}_\mathcal{O}
    \rightarrow
    \mathrm{Sets}
    \]
    where
    $\mathcal{D}_\pi(A)$
    is the set of equivalence classes of
    liftings $\rho_A
    \colon
    \Gamma_\mathbb{Q}
    \rightarrow
    \mathrm{GSp}_4(A)$
    of $\overline{\rho}_{\pi,p}$
    such that
    \begin{enumerate}
    	\item 
    	$\rho_A$ is unramified outside $Np$;
    	
    	\item 
    	for each $\ell|N$, up to a conjugation,
    	the image
    	$\rho_A(I_\ell)$
    	of the inertia subgroup
    	is contained in the unipotent subgroup
    	$\mathrm{exp}(A\varepsilon)$;
    	
    	\item 
    	there is a filtration
    	$(\mathrm{Fil}^A)_i$ on $\rho_A$
    	lifting $(\overline{\mathrm{Fil}})_i$
    	and stable under $D_p$;
    	
    	\item 
    	the character
    	of $D_p$ on
    	$\mathrm{Fil}^A_0/\mathrm{Fil}^A_{-1}$ is unramified.
    \end{enumerate}
    For $\mathcal{D}_{\underline{k}''}(A)$,
    we require moreover
    as in \cite[Section 5.6]{Pilloni2012}
    that in (3) above the character $\chi_i'$ of
    $D_p$ via $\rho_A$ on
    $\mathrm{Fil}^A_i/\mathrm{Fil}^A_{i-1}$
    is of the form
    $\chi'_i(\sigma)=\mathrm{Art}(\sigma)^{k''_i}$
    lifting $\overline{\chi}_i$
    for $\sigma\in I_p$
    and $\mathrm{Art}(\sigma)
    \in\mathbb{Z}_p^\times$
    by the local Artin map
    (clearly $\mathcal{D}_{\underline{k}''}(A)$
    is nonempty only if
    $\underline{k}''
    \equiv(0,-k_2-1,k_1-2,-k_1-k_2-3)
    (\mathrm{mod}\, p-1)$,
    which we will assume in the following).
    Two liftings $\rho_A$ and $\rho_A'$
    are equivalent if there is a matrix
    $g\in1_4+\mathfrak{m}_A\cdot\mathfrak{gsp}_4(A)$
    such that $\rho_A=g\rho_A'g^{-1}$.   
    We have the following hypotheses
    \begin{hypothesis}
    	\label{hypotheses for R=T theorems}
    	\begin{enumerate}
    		\item 
    		($N$-Min)
    		For each
    		$\ell|N$,
    		the restriction to inertia group
    		$\overline{\rho}_\pi|_{I_\ell}$
    		contains a regular unipotent element;
    		
    		\item 
    		($\mathrm{RFR}^{(2)}$)
    		The images of
    		$\alpha_1$,
    		$\alpha_2p^{-k_2-1}$,
    		$\alpha_3p^{-k_1-2}$
    		and
    		$\alpha_4p^{-k_1-k_2-3}$
    		in $\mathcal{O}/\varpi$
    		are mutually distinct;
    		
    		\item 
    		($BIG^{(2)}$)
    		The image of the
    		residual Galois representation
    		$\overline{\rho}_\pi(\Gamma_{\mathbb{Q}})$
    		contains
    		$\mathrm{Sym}^3(\mathrm{SL}_2(\mathbb{F}_p))$.
    	\end{enumerate}	
    	
    \end{hypothesis}

    Then
    by \cite[Proposition 2.1]{Tilouine2006},
    The functors $\mathcal{D}_\pi$, resp.,
    $\mathcal{D}_{\underline{k}''}$
    is pro-representable,
    say by the couple
    $(\rho^\mathrm{s}_\pi, R^\mathrm{s}_\pi)$,
    resp.,
    $(\rho_{\underline{k}''}^\mathrm{s},
    R_{\underline{k}''}^\mathrm{s})$,
    where $R^\mathrm{s}_\pi$,
    resp., $R_{\underline{k}''}^\mathrm{s}$ ,
    is an object in
    $\mathrm{CNL}_\mathcal{O}$
    and $\rho^\mathrm{s}$, resp.,
    $\rho_{\underline{k}''}^\mathrm{s}$,
    is a Galois representation
    $\Gamma_\mathbb{Q}
    \rightarrow
    \mathrm{GSp}_4(R^\mathrm{s}_\pi)$,
    resp.,
    $\Gamma_{\mathbb{Q}}
    \rightarrow
    \mathrm{GSp}_4(R_{\underline{k}''}^\mathrm{s})$.
    We write $\mathfrak{m}_\pi$
    for the maximal ideal of
    $\mathbb{T}^\mathrm{s,ord}
    (\widehat{\Gamma}^\infty,\mathcal{O})$
    associated to the residual representation
    $\overline{\rho_\pi}$
    and set
    \[
    \mathbb{T}^\mathrm{s}_\pi
    :=\mathbb{T}^\mathrm{s,ord}
    (\widehat{\Gamma}^\infty,
    \mathcal{O})_{\mathfrak{m}_\pi},
    \quad
    S_\pi^\mathrm{s}
    :=S^\mathrm{s,ord}(\widehat{\Gamma}^\infty,
    \mathbb{Q}_p/\mathbb{Z}_p)_{\mathfrak{m}_\pi},
    \]
    \[
    \mathbb{T}_{\underline{k}''}^\mathrm{s}
    :=
    \mathbb{T}_{\underline{k}''}^\mathrm{s,ord}
    (\widehat{\Gamma},
    \mathcal{O}),
    \quad
    S_{\underline{k}''}^\mathrm{s}
    :=
    S_{\underline{k}''}^\mathrm{s,ord}(\widehat{\Gamma},\mathbb{Q}_p/\mathbb{Z}_p).
    \]

    Using pseudo-representations, one can construct a
    Galois representation
    $\rho_{\mathbb{T}_\pi^\mathrm{s}}
    \colon
    \Gamma_\mathbb{Q}
    \rightarrow
    \mathrm{GSp}(\mathbb{T}_\pi^\mathrm{s})$
    such that it is a lift of
    $\overline{\rho_\pi}$
    (\cite[Section 6.6]{Pilloni2012}).
    Thus by the definition of
    $R^\mathrm{s}_\pi$, we have a morphism
    $R^\mathrm{s}
    \rightarrow
    \mathbb{T}^\mathrm{s}_\pi$
    of $\Lambda^\mathrm{s}$-algebras
    (the $\Lambda^\mathrm{s}$-algebra structure on
    $R^\mathrm{s}$ is given as in \cite[Section 5.5]{Pilloni2012}).
    The method of Taylor-Wiles gives
    (\cite[Th\'{e}or\`{e}me 7.1]{Pilloni2012}):
    \begin{theorem}
    	Assume
    	Hypothesis
    	\ref{hypotheses for R=T theorems},
    	then the morphism
    	$R_\pi^\mathrm{s}
    	\rightarrow
    	\mathbb{T}_\pi^\mathrm{s}$
    	is an isomorphism
    	and $\mathbb{T}_\pi^\mathrm{s}$
    	is finite flat complete intersection
    	over $\Lambda^\mathrm{s}$.
    	The Hecke module
    	$S_\pi^\mathrm{s}$ is a finite free
    	$\mathbb{T}_\pi^\mathrm{s}$-module.
    	Moreover, the specialization map
    	${\mathbb{T}_\pi^\mathrm{s}
    	\otimes_{\Lambda^\mathrm{s},\underline{k}''}\mathcal{O}
    	\rightarrow
    	\mathbb{T}_{\underline{k}''}^\mathrm{s}}$
    	is an isomorphism.
    	In particular, we have an isomorphism
    	$R_{\underline{k}''}^\mathrm{s}
    	\simeq
    	\mathbb{T}_{\underline{k}''}^\mathrm{s}$.
    \end{theorem}
    \begin{proof}
    	All the parts in the theorem are proved
    	except the part that
    	$S_\pi^\mathrm{s}$
    	is finite free over $\mathbb{T}_\pi^\mathrm{s}$.
    	In \cite[Th\'{e}or\`{e}me 7.1]{Pilloni2012},
    	it is shown that the linear dual
    	$\mathrm{Hom}_\mathcal{O}(
    	S_\pi^\mathrm{s},\mathcal{O})$
    	is finite free over $\mathbb{T}_\pi^\mathrm{s}$.
    	Now that $\mathbb{T}_\pi^\mathrm{s}$
    	is complete intersection, thus is Gorenstein,
    	which implies that $S_\pi^\mathrm{s}$
    	is finite free over $\mathbb{T}_\pi^\mathrm{s}$.
    \end{proof}
    \begin{remark}\label{hypothesis on Gorenstein is satisfied}
    	From this theorem, we see that
    	Hypothesis
    	\ref{Gorenstein hypothesis}
    	is satisfied.
    \end{remark}

    \subsubsection{Case $\mathrm{U}_4$}
    \label{Hecke algebra of U(4)}
    In the case of unitary groups, we follow
    closely \cite[Section 2]{Geraghty2010}.
    As in \cite{Geraghty2010},
    we are only interested in defining Hecke operators
    at finite places $\ell$ 
    of $\mathbb{Q}$
    that split in $E/\mathbb{Q}$.
    Since in this case
    $U_4(\mathbb{Q}_\ell)$
    is isomorphic to
    $\mathrm{GL}_4(\mathbb{Q}_\ell)$,
    thus we are essentially defining Hecke operators
    of
    $\mathrm{GL}_4(\mathbb{Q}_\ell)$,
    as we will do in the following.
    Suppose that
    $\ell$ splits as
    $\mathfrak{l}\mathfrak{l}^c$ in $E$.
    Then we fix an isomorphism
    \[
    \iota_\mathfrak{l}
    \colon
    \mathrm{U}_4(\mathbb{Z}_\ell)
    \xrightarrow{\sim}
    \mathrm{GL}_4(\mathcal{O}_{E,\mathfrak{l}})
    \simeq
    \mathrm{GL}_4(\mathbb{Z}_\ell).
    \]
    We can use this isomorphism
    $\iota_\mathfrak{l}$ to carry over the
    Hecke theory of
    $\mathrm{GL}_4$ to
    $\mathrm{U}_4$.
    we write $G'
    =\mathrm{GL}_4(\mathbb{Q}_\ell)
    $
    and the maximal compact open subgroup
    $K'=\mathrm{GL}_4(\mathbb{Z}_\ell)$.
    We define the spherical Hecke algebra
    $H_\ell^\mathrm{u}$
    of $G'$ as the
    $\mathbb{Q}$-algebra of
    finite $\mathbb{Q}$-linear combinations of
    double cosets
    $K'\gamma K'$
    for $\gamma\in G'$.
    We have several distinguished Hecke operator
    in $H_\ell^\mathrm{u}$
    as follows
    (\cite[p.10]{Geraghty2010})
    \[
    T^\mathrm{u}_i(\ell)=K'\mathrm{diag}(\ell\cdot1_i,1_{4-i})K'
    \quad
    (i=1,2,3,4).
    \]
    As $\mathbb{Q}$-algebras,
    $H_\ell^\mathrm{u}
    \simeq
    \mathbb{Q}[T^\mathrm{u}_1(\ell),\cdots,
    T^\mathrm{u}_4(\ell),
    T^\mathrm{u}_4(\ell)^{-1}]$.
    Let $T_{G'}$ be the standard torus of
    $G'$
    and define
    $H_\ell^\mathrm{u}(T_{G'})$
    to be the $\mathbb{Q}$-algebra of
    finite $\mathbb{Q}$-linear combinations of
    (double) cosets of
    $T_{G'}(\mathbb{Z}_\ell)\gamma$
    for $\gamma\in T_{G'}$.
    Then the Satake isomorphism identifies
    $H_\ell^\mathrm{u}$
    with the subalgebra
    of $H_\ell^\mathrm{u}(T_{G'})$
    consisting of elements invariant under the
    Weyl group
    $W_{G'}$ of $T_{G'}$ in
    $G'$.
    We have an isomorphism of
    $H_\ell^\mathrm{u}(T_{G'})$
    with the polynomial algebra
    $\mathbb{Q}[w_1^{\pm1},\cdots,w_4^{\pm1}]$
    by sending
    $T_{G'}(\mathbb{Z}_\ell)
    \mathrm{diag}(\ell^{a_1},\cdots, \ell^{a_4})$
    to
    $w_1^{a_1}\cdots w_4^{a_4}$.
    One can identify each morphism of
    $\mathbb{C}$-algebras
    $\lambda'
    \colon
    H_\ell^\mathrm{u}\otimes_\mathbb{Q}\mathbb{C}
    \rightarrow
    \mathbb{C}$
    to the element 
    $\lambda'(w_1\cdots w_4)$
    in
    $(\mathbb{C}^\times)^4/W_{G'}$.
    This is the Satake parameter of
    $\lambda'$.
    Since each unramified representation
    $\pi'$ of $G'$ corresponds to a
    $\lambda'$,
    such a representation $\pi'$ has its Satake parameter
    $\lambda'(w_1\cdots w_4)$.
    Moreover, $\pi'$ can be embedded into a
    principal series representation
    $\mathrm{Ind}_{B_{G'}}^{G'}(\xi'_1,\cdots,\xi'_4)$
    where $\xi'_1,\cdots, \xi'_4$ are unramified characters of
    $\mathbb{Q}_\ell^\times$.
    Then we have
    $\lambda'(w_1\cdots w_4)
    =
    (\xi'_1,\cdots,\xi'_4)(\ell\cdot1_4)
    \in(\mathbb{C}^\times)^4/W_{G'}$.
    The Hecke polynomial $P^\mathrm{u}_\ell(X)$
    at $\ell$ is defined to be
    \[
    P_\ell^\mathrm{u}(X)
    :=
    X^4-T_1^\mathrm{u}(\ell)X^3
    +\ell T_2^\mathrm{u}(\ell)X^2
    -\ell^3 T_3^\mathrm{u}(\ell)X
    +\ell^6 T_4^\mathrm{u}(\ell).
    \]

    We define the Iwahori Hecke algebra
    $H_p^\mathrm{u,Iw}$
    of $G'=\mathrm{GL}_4(\mathbb{Q}_p)$
    for $\ell=p$
    as follows.
    For any integers $c\geq b\geq0$,
    we write
    $\mathrm{Iw}_p^{b,c}$ to be the subgroup of
    $K'$
    such that
    \[
    \mathrm{Iw}_p^{b,c}
    \equiv
    \begin{pmatrix}
    * & * & * & * \\
    0 & * & * & * \\
    0 & 0 & * & * \\
    0 & 0 & 0 & *
    \end{pmatrix}
    (\mathrm{mod}\, p^b),
    \quad
    \mathrm{Iw}_p^{b,c}
    \equiv
    \begin{pmatrix}
    1 & * & * & * \\
    0 & 1 & * & * \\
    0 & 0 & 1 & * \\
    0 & 0 & 0 & 1
    \end{pmatrix}
    (\mathrm{mod}\, p^c).
    \]
    We fix a dominant algebraic character
    $\underline{k}=(k_1,\cdots,k_4)$
    of $T_{G'}$ such that
    $k_1\geq k_2\geq k_3\geq k_4$.
    Let $w_{G'}$ be the longest element in
    $W_{G'}$.
    Then we define the following Hecke operators
    (\cite[p.10]{Geraghty2010})
    \[
    U_{\underline{k},i}^{b,c}
    =
    (w_{G'}\underline{k})
    (\mathrm{diag}(p\cdot1_i,1_{4-i}))^{-1}
    \mathrm{Iw}_p^{b,c}\mathrm{diag}(p\cdot1_i,1_{4-i})
    \mathrm{Iw}_p^{b,c},
    \quad
    (i=1,2,3,4).
    \]
    We define also the diamond operators:
    for each $u\in T'(\mathbb{Z}_p)$,
    we write
    $\langle u\rangle
    =\mathrm{Iw}_p^{b,c}u\mathrm{Iw}_p^{b,c}$.
    Then $H_{p,\underline{k}}^{\mathrm{u},b,c}$
    is the
    $\mathbb{Z}$-algebra generated by these
    $U_{\underline{k},i}^{b,c}$ and
    $\langle u\rangle$.

    We then define the global Hecke algebra
    $H^{\mathrm{u},p}$
    of $\mathrm{U}_4(\mathbb{Q})$
    Iwahori at $p$ of
    character $\underline{k}$
    to be the restricted tensor product
    \[
    H^{\mathrm{u},b,c}_{\underline{k}}
    =
    (\bigotimes_{\ell\neq p,\text{ split in } E}\,'
    H_\ell^\mathrm{u})
    \bigotimes
    H_{p,\underline{k}}^{\mathrm{u},b,c}.
    \]

    We briefly recall the notion of automorphic forms on
    $\mathrm{U}_4(\mathbb{A})$
    as in \cite[Section 2.2]{Geraghty2010}.
    As above, let
    $\underline{k}$
    be a dominant character of $T_{G'}$
    and we write
    $\widetilde{M}_{\underline{k}}$
    for the algebraic representation
    $\mathrm{Ind}_{B_{G'}}^{G'}
    (w_{G'}\underline{k})_{/\mathbb{Z}_p}$
    and
    $M_{\underline{k}}:=
    \widetilde{M}_{\underline{k}}(\mathbb{Z}_p)$.
    For any $\mathbb{Z}_p$-module $R$,
    we then write
    $S^\mathrm{u}_{\underline{k}}(R)$
    for the space of functions
    (\cite[Definition 2.2.4]{Geraghty2010})
    \[
    f\colon
    \mathrm{U}_4(\mathbb{Q})
    \backslash
    \mathrm{U}_4(\mathbb{A}^\infty)
    \rightarrow
    M_{\underline{k}}\otimes_{\mathbb{Z}_p}R
    \]
    such that there is a compact open subgroup
    $\widetilde{K}$ of
    $\mathrm{U}_4(\mathbb{A}^{\infty,p})\times
    \mathrm{U}_4(\mathbb{Z}_p)$
    and for any
    $u\in \widetilde{K}$ and 
    $g\in\mathrm{U}_4(\mathbb{A}^\infty)$,
    we have
    $u_p(f(gu))=f(g)$.
    The group
    $\mathrm{U}_4(\mathbb{A}^{\infty,p})\times
    \mathrm{U}_4(\mathbb{Z}_p)$
    acts on
    $S^\mathrm{u}_{\underline{k}}(R)$
    by
    $(g\cdot f)(g'):=g_p(f(g'g))$.
    For any subgroup
    $\widetilde{K}$ of
    $\mathrm{U}_4(\mathbb{A}^{\infty,p})
    \times\mathrm{U}_4(\mathbb{Z}_p)$,
    we then write
    $S^\mathrm{u}_{\underline{k}}(U,R)$
    for the
    $\widetilde{K}$-invariant submodule of
    $S^\mathrm{u}_{\underline{k}}(R)$.
    For any compact open subgroup
    $\widetilde{K}$ of
    $\mathrm{U}_4(\mathbb{A}^\infty)$
    such that
    $U_\ell=\mathrm{U}_4(\mathbb{Z}_\ell)$
    for any
    $\ell$ split in $E$
    (including the the case $p$),
    we set
    $\widetilde{K}^{b,c}
    =\widetilde{K}^p
    \times\mathrm{Iw}_p^{b,c}$.
    If $R$ is a $\mathbb{Z}_p$-algebra,
    we let the above defined global Hecke algebra 
    $H_{\underline{k}}^{\mathrm{u},b,c}$
    act on
    $S^\mathrm{u}_{\underline{k}}
    (\widetilde{K}^{b,c},R)$
    and write the $R$-subalgebra
    of
    $\mathrm{End}_R
    (S^\mathrm{u}_{\underline{k}}
    (\widetilde{K}^{b,c},R))$
    generated by the image of
    $H_{\underline{k}}^{\mathrm{u},b,c}$
    as
    $\mathbb{T}_{\underline{k}}^\mathrm{u}
    (\widetilde{K}^{b,c},R)$.
    If $R$ is only a $\mathbb{Z}_p$-module,
    we replace $R$-subalgebra by
    $\mathbb{Z}_p$-subalgebra.
    In the Hecke algebra
    $\mathbb{T}_{\underline{k}}^\mathrm{u}
    (\widetilde{K}^{b,c},R)$,
    we consider the
    idempotent element
    $e:=\lim\limits_{n\rightarrow\infty}
    (\prod_{i=1}^4U_{\underline{k},i}^{b,c})^{n!}$
    and we define the ordinary Hecke algebra
    and
    ordinary automorphic forms
    as
    (\cite[Definition 2.4.1]{Geraghty2010}):
    \[
    \mathbb{T}^\mathrm{u,ord}_{\underline{k}}
    (\widetilde{K}^{b,c},R)
    =
    e\mathbb{T}^\mathrm{u}_{\underline{k}}
    (\widetilde{K}^{b,c},R),
    \quad
    S_{\underline{k}}^\mathrm{u,ord}
    (\widetilde{K}^{b,c},R)
    =
    eS^\mathrm{u}_{\underline{k}}
    (\widetilde{K}^{b,c},R).
    \]   
    Finally we define the big ordinary Hecke algebra
    (\cite[Definition 2.4.5]{Geraghty2010}):
    \[
    \mathbb{T}^\mathrm{u,ord}_{\underline{k}}
    (\widetilde{K}^\infty,\mathcal{O})
    =
    \lim\limits_{\overleftarrow{c>0}}
    \mathbb{T}_{\underline{k}}^\mathrm{u,ord}
    (\widetilde{K}^{c,c},\mathcal{O}),
    \quad
    S_{\underline{k}}^{\mathrm{u,ord}}
    (\widetilde{K}^\infty,\mathbb{Q}_p/\mathbb{Z}_p)
    =
    \lim\limits_{\overrightarrow{c>0}}
    S_{\underline{k}}^\mathrm{u,ord}
    (\widetilde{K}^{c,c},\mathbb{Q}_p/\mathbb{Z}_p)
    \]
    Note that
    $\mathbb{T}^\mathrm{u,ord}_{\underline{k}}
    (\widetilde{K}^\infty,\mathcal{O})$
    is a finite faithful algebra over
    $\Lambda^\mathrm{u}$
    where
    \begin{align*}
    \Lambda^\mathrm{u}
    =
    \mathcal{O}[[T'(1+p\mathbb{Z}_p)]]
    &
    \simeq
    \mathcal{O}[[Y_1,Y_2,Y_3,Y_4]],
    \\
    \mathrm{diag}(1_{i-1},1+p,1_{4-i})
    &
    \mapsto
    1+Y_i,
    \quad
    (i=1,2,3,4).
    \end{align*}
    is the Iwasawa weight algebra
    (\cite[Corollary 2.5.4]{Geraghty2010}).

    We next recall some results on the 
    $p$-adic Galois representations
    associated to automorphic representations
    of $\mathrm{U}_4(\mathbb{A})$.
    Let $\pi'=\otimes_v'\pi'_v$ be a (cuspidal) irreducible automorphic
    representation of
    $\mathrm{U}_4(\mathbb{A})$
    of dominant weight
    $\underline{k}'=(k'_1,k'_2,k'_3,k'_4)$,
    unramified outside $Np$,
    ordinary at $p$
    (i.e., $\pi'_p$ is $\mathrm{Iw}^{0,1}$-spherical
    and $\pi'\subset
    S_{\underline{k}'}^\mathrm{u,ord}
    (\widetilde{K}^{0,1},\mathcal{O})$).
    We define a group scheme over $\mathbb{Z}$:
    \[
    \mathcal{G}_4
    =(\mathrm{GL}_4\times\mathrm{GL}_1)\rtimes\{1,j\}
    \]
    where
    $j(g,\mu)j^{-1}:=(\mu g^{-t},\mu)$
    for $(g,\mu)\in\mathrm{GL}_4\times\mathrm{GL}_1$.
    Then there exists a Galois representation
    $\rho_{\pi'}
    \colon
    \Gamma_\mathbb{Q}
    \rightarrow
    \mathcal{G}_4(\mathcal{O})$
    associate to $\pi'$
    (\cite[Section 2.7]{Geraghty2010})
    with certain properties introduced below.
    For any finite place $\ell\nmid Np$ of $\mathbb{Q}$
    split as $\mathfrak{l}\mathfrak{l}^c$ in $E$,
    we have an isomorphism of semi-simplified
    Weil-Deligne representations
    (\cite[Proposition 2.7.2(1)]{Geraghty2010})
    \begin{equation}\label{Weil-Deligne vs Galois, unitary group}
    \mathrm{WD}(\rho_{\pi'}|_{D_\mathfrak{l}})^\mathrm{ss}
    \simeq
    \iota(\mathrm{rec}(\pi'_\ell\otimes|\cdot|^{-3/2})^\mathrm{ss}).
    \end{equation}
    Here $\mathrm{rec}$ is the local 
    Langlands correspondence for
    $\mathrm{GL}_4$
    and $\iota\colon\mathbb{C}\simeq\overline{\mathbb{Q}}_p$.
    Moreover,
    suppose $p$ splits as $\mathfrak{p}\mathfrak{p}^c$
    in $E$
    and the Hecke operators
    $U_{\underline{k}',i}^{0,1}$
    acts on the space 
    $\pi'$ by scalars
    $u_{\underline{k}',i}\in\mathcal{O}^\times$
    ($i=1,2,3,4$),
    then by
    \cite[Corollary 2.7.8]{Geraghty2010},
    if $\underline{k}'$ is regular
    (i.e., $k_1>k_4$),
    we have that
    $\rho_{\pi',\mathfrak{p}}:=\rho_\pi|_{D_\mathfrak{p}}$
    is crystalline and is conjugate to an upper triangular 
    representation
    \begin{equation}\label{crystalline at p for U(4)}
    \rho_{\pi',\mathfrak{p}}
    \simeq
    \begin{pmatrix}
    \chi_1' & * & * & * \\
    0 & \chi_2' & * & * \\
    0 & 0 & \chi_3' & * \\
    0 & 0 & 0 & \chi_4'
    \end{pmatrix}
    \end{equation}
    whose diagonal entries are characters defined as
    $\chi_i'=\mathrm{ur}
    (u_{\underline{k}',i}/u_{\underline{k}',i-1})
    \chi_p^{-k'_{5-i}-i}$
    (we set
    $u_{\underline{k}',0}=1$).
    On the side of automorphic representations,
    $\pi_p$ is unramified
    and therefore
    (\cite[Lemma 2.7.5]{Geraghty2010}):
    \[
    \mathrm{det}(X-\mathrm{Fr}_p',
    \rho_{\pi',p}^\mathrm{crys})
    =
    \prod_{i=1}^4
    (X-p^{i-1+k'_{5-i}}
    \frac{u_{\underline{k}',i}}{u_{\underline{k}',i-1}}).
    \]
    Here $\mathrm{Fr}_p'$ and
    $\rho_{\pi',p}^\mathrm{cryst}$
    are defined as in the case of
    $\mathrm{GSp}_4$.

    Next we define deformation functors and
    universal deformation rings of the residual
    representation
    $\overline{\rho}_{\pi'}$ as follows
    (\cite[Section 2.3]{HidaTilouine16}):
    \[
    \mathcal{D}_{\pi'},
    \mathcal{D}_{\pi',\underline{k}''}
    \colon
    \mathrm{AL}_\mathcal{O}
    \rightarrow
    \mathrm{Sets}
    \]
    where
    $\mathcal{D}_{\pi'}(A)$
    is the set of equivalence classes of liftings
    $\rho_A
    \colon
    \Gamma_\mathbb{Q}
    \rightarrow
    \mathcal{G}_4(A)$
    such that
    \begin{enumerate}
    	\item 
    	$\rho_A$ is unramified outside $Np$;
    	
    	\item 
    	the projection of $\rho_A|_{D_p}$
    	to the factor
    	$\mathrm{GL}_4(A)$
    	is conjugate by an element
    	$g\in1+\mathfrak{m}_A\mathrm{M}_{4\times4}(A)$
    	to an upper triangular representation
    	whose diagonal entries are characters
    	$\widetilde{\chi}_i\chi_p^{-i}$
    	such that $\widetilde{\chi}_i$
    	lifts the character
    	$\mathrm{ur}(\overline{u_{\underline{k}',i}/u_{\underline{k}',i-1}})
    	\overline{\chi}_p^{k'_{5-i}}$.
    \end{enumerate}
    For $\mathcal{D}_{\pi',\underline{k}''}(A)$,
    we require moreover that in (2) above
    $\widetilde{\chi}_i$
    is of the form
    $\widetilde{\chi}_i(\sigma)=\mathrm{Art}(\sigma)^{k_{5-i}''}$
    for $\sigma\in I_p$
    and $\mathrm{Art}(\sigma)\in\mathbb{Z}_p^\times$
    by the local Artin map
    (of course
    $\mathcal{D}_{\pi',\underline{k}''}(A)$
    is non-empty only if
    $\underline{k}''\equiv\underline{k}'(\mathrm{mod}\, p-1)$).
    Two liftings 
    $\rho_A$ and $\rho_A'$
    are equivalent if there is an element
    $g\in\mathrm{Ker}(\mathrm{GL}_4(A)
    \rightarrow
    \mathrm{GL}_4(\mathbb{F}))$
    such that
    $\rho_A'\simeq g\rho_Ag^{-1}$.
    If we assume that
    the diagonal entries
    $\overline{\chi}_i'$ in
    $\overline{\rho}_{\pi',p}$
    are mutually distinct
    either on the Frobenius element
    $\mathrm{Fr}_p$ or the inertia subgroup
    $I_p$ of $D_p$,
    then by
    \cite[Lemma 2.6]{HidaTilouine16},
    $\mathcal{D}_{\pi'}$
    is pro-representable by
    $(\rho^\mathrm{u}_{\pi'},R^\mathrm{u}_{\pi'})$
    where $R^\mathrm{u}_{\pi'}$
    is in $\mathrm{CNL}_\mathcal{O}$
    and
    $\rho_{\pi'}^\mathrm{u}
    \colon
    \Gamma_\mathbb{Q}
    \rightarrow
    \mathcal{G}_4(R_{\pi'}^\mathrm{u})$.
    We then write
    $\mathfrak{m}_{\pi'}$
    for the maximal ideal of
    $\mathbb{T}_{\underline{k}'}^\mathrm{u,ord}
    (\widetilde{K}^\infty,\mathcal{O})$
    associated to the residual representation
    $\overline{\rho}_{\pi'}$
    and then set
    \[
    \mathbb{T}^\mathrm{u}_{\pi'}
    :=
    \mathbb{T}_{\underline{k}'}^\mathrm{u,ord}
    (\widetilde{K}^\infty,\mathcal{O})
    _{\mathfrak{m}_{\pi'}},
    \quad
    S^\mathrm{u}_{\pi'}
    :=
    S_{\underline{k}'}^{\mathrm{u,ord}}
    (\widetilde{K}^\infty,
    \mathbb{Q}_p/\mathbb{Z}_p)
    _{\mathfrak{m}_{\pi'}}.
    \]
    \[
    \mathbb{T}_{\underline{k}''}^\mathrm{u}
    :=
    \mathbb{T}_{\underline{k}''}^\mathrm{u,ord}
    (\widetilde{K}^{0,1},\mathcal{O}),
    \quad
    S_{\underline{k}''}^\mathrm{u}
    :=
    S_{\underline{k}''}^\mathrm{u,ord}
    (\widetilde{K}^{0,1},\mathbb{Q}_p/\mathbb{Z}_p).
    \]

    Then there is a Galois representation
    $\rho_{\mathbb{T}_{\pi'}^\mathrm{u}}
    \colon
    \Gamma_\mathbb{Q}
    \rightarrow
    \mathcal{G}_4(\mathbb{T}_{\pi'}^\mathrm{u})$
    lifting $\overline{\rho}_{\pi'}$
    and thus
    we have
    a morphism 
    $R_{\pi'}^\mathrm{u}
    \rightarrow
    \mathbb{T}_{\pi'}^\mathrm{u}$
    of $\Lambda^\mathrm{u}$-algebras.
    The method of Taylor-Wiles gives
    (\cite[Section 2.7]{HidaTilouine16}):
    \begin{theorem}
    	Assume Hypothesis
    	\ref{hypotheses for R=T theorems},
    	then the morphism
    	$R_{\pi'}^\mathrm{u}
    	\rightarrow
    	\mathbb{T}_{\pi'}^\mathrm{u}$
    	is an isomorphism and
    	$\mathbb{T}_{\pi'}^\mathrm{u}$
    	is a finite flat complete intersection over
    	$\Lambda^\mathrm{u}$.
    	The Pontryagin dual $(S^\mathrm{u}_{\pi'})^\vee$
    	is a finite free
    	$\mathbb{T}_{\pi'}^\mathrm{u}$-module.
    	Moreover, the specialization map
    	$\mathbb{T}_{\pi'}^\mathrm{u}
    	\otimes_{\Lambda^\mathrm{u},\underline{k}''}
    	\mathcal{O}
    	\rightarrow
    	\mathbb{T}_{\underline{k}''}^\mathrm{u}$
    	is an isomorphism for any
    	dominant weight
    	$\underline{k}''
    	\equiv
    	\underline{k}'(\mathrm{mod}\, p-1)$.
    	In particular, we have an isomorphism
    	$R_{\underline{k}''}^\mathrm{u}
    	\simeq
    	\mathbb{T}_{\underline{k}''}^\mathrm{u}$.
    \end{theorem}

    \subsection{Morphism of the transfer}
    Let $\pi$ be an irreducible automorphic representation
    of $\mathrm{GSp}_4(\mathbb{A})$
    of type $(\underline{k},\widehat{\Gamma})$
    (holomorphic or antiholomorphic).
    We write $\pi'=\Theta(\pi)$ for the theta lift of
    $\pi$ to
    $\mathrm{U}_4(\mathbb{A})$.
    Suppose that $\Theta(\pi)\neq0$.
    In this subsection we define morphisms
    $\Lambda^\mathrm{u}
    \rightarrow
    \Lambda^\mathrm{s}$,
    $\mathbb{T}_{\pi'}^\mathrm{u}
    \rightarrow
    \mathbb{T}_\pi^\mathrm{s}$
    and
    $R_{\pi'}^\mathrm{u}
    \rightarrow
    R_\pi^\mathrm{s}$.

    \subsubsection{Morphism of Iwasawa weight algebras}
    We define a map of tori
    $T^\mathrm{s}$ of
    $\mathrm{GSp}_4(\mathbb{Q}_p)$
    and
    $T^\mathrm{u}$ of
    $\mathrm{GL}_4(\mathbb{Q}_p)$
    as follows
    (natural inclusion):
    \[
    T^\mathrm{s}
    \rightarrow
    T^\mathrm{u},
    \quad
    \mathrm{diag}(t_0,t_0t_1,t_0t_2,t_0t_1t_2)
    \mapsto
    \mathrm{diag}(t_0,t_0t_1,t_0t_2,t_0t_1t_2).
    \]
    This induces a morphism of Iwasawa algebras
    \begin{equation}\label{morphism of Iwasawa weight algebras}
    \Lambda^\mathrm{u}
    =\mathcal{O}[[T^\mathrm{u}(1+p\mathbb{Z}_p)]]
    \rightarrow
    \Lambda^\mathrm{s}
    =\mathcal{O}[[T^\mathrm{s}(1+p\mathbb{Z}_p)]].
    \end{equation}

    \subsubsection{Morphism of Hecke algebras}
    We define a map of
    Hecke algebras 
    $\mathbb{T}_{\pi'}^\mathrm{u}
    \rightarrow
    \mathbb{T}_\pi^\mathrm{s}$
    as follows:
    for any prime $\ell\nmid N$ split as
    $\mathfrak{l}\mathfrak{l}^c$ in $E$,
    the morphism of local Hecke algebras
    $(\mathbb{T}_{\pi'}^\mathrm{u})_\mathfrak{l}
    \rightarrow
    (\mathbb{T}_\pi^\mathrm{s})_\mathfrak{l}$
    is given by
    \cite[Remark 4.4(A)]{Rallis1982}
    (in the notation of \textit{loc.cit.},
    for the case $i=3,n=2$).
    Note that in \cite{Rallis1982},
    only the case $\pi_\ell$
    unramified is treated.
    For the case of $\pi_p$ Steinberg,
    the same result follows using
    Theorem 
    \ref{Langlands parameters for GSp^+4 and GSO6}
    and the Satake isomorphism of Iwahori
    Hecke algebras.
    Moreover,
    for the diamond operators
    $\langle u\rangle$,
    we send the image 
    $\mathrm{Im}(\mathcal{O}[[T^\mathrm{u}
    (\mathbb{Z}_p)]]
    \xrightarrow{\langle\cdot\rangle}
    \mathbb{T}_{\pi'}^\mathrm{u})$
    to the image
    $\mathrm{Im}(\mathcal{O}[[T^\mathrm{s}
    (\mathbb{Z}_p)]]
    \xrightarrow{\langle\cdot\rangle}
    \mathbb{T}_\pi^\mathrm{s})$
    as induced by the map
    $T^\mathrm{s}(\mathbb{Z}_p)
    \rightarrow
    T^\mathrm{u}(\mathbb{Z}_p)$
    defined above.
    This finishes the definition of the map
    $\mathbb{T}_{\pi'}^\mathrm{u}
    \rightarrow
    \mathbb{T}_\pi^\mathrm{s}$.
    Similarly we define
    $\mathbb{T}_{\underline{k}''}^\mathrm{u}
    \rightarrow
    \mathbb{T}_{\underline{k}''}^\mathrm{s}$.

    \subsubsection{Morphism of Galois representations}
    Write
    $\Gamma_{E/\mathbb{Q}}=\{1,c\}$.
    Define a group scheme
    $\mathcal{G}_4'$ over $\mathbb{Z}$ by
    $(\mathrm{GL}_4\times\mathrm{GL}_1)\rtimes\{1,c\}$
    where 
    $c(g,\mu)c^{-1}
    :=(\mu J_4g^{-t}J_4^{-1},\mu)$.
    It is easy to verify that the following map
    is an isomorphism
    \[
    \mathcal{G}_4'
    \rightarrow
    \mathcal{G}_4,
    \quad
    (g,\mu,1)
    \mapsto
    (g,\mu,1),
    \quad
    (g,\mu,c)
    \mapsto
    (gJ_4,-\mu,j).
    \]
    Moreover, the following map is 
    a morphism of groups:
    \[
    \mathrm{GSp}_4\times\Gamma_{E/\mathbb{Q}}
    \rightarrow
    \mathcal{G}_4',
    \quad
    (g,x)
    \mapsto
    (g,\nu(g),x).
    \]
    Let $R$ be a topological ring.
    Then for any continuous Galois representation
    $\rho
    \colon
    \Gamma_\mathbb{Q}
    \rightarrow
    \mathrm{GSp}_4(R)$,
    we define
    \[
    \Theta'(\rho)
    \colon
    \Gamma_\mathbb{Q}
    \rightarrow
    \mathcal{G}_4'(R),
    \quad
    \sigma
    \mapsto
    (\rho(\sigma),\nu(\rho(\sigma)),\overline{\sigma})
    \]
    where $\overline{\sigma}$ is the projection 
    image of $\sigma$
    in $\Gamma_{E/\mathbb{Q}}$.
    Composing with the isomorphism
    $\mathcal{G}_4'\simeq\mathcal{G}_4$,
    we define the following
    (\textit{cf}. \cite[Lemma 2.1.2]{ClozelHarrisTaylor})
    \[
    \Theta(\rho)
    \colon
    \Gamma_\mathbb{Q}
    \xrightarrow{\Theta'(\rho)}
    \mathcal{G}_4'(R)
    \xrightarrow{\sim}
    \mathcal{G}_4(R).
    \]
    Thus $\Theta(\rho)$ and
    $\Theta'(\rho)$ are isomorphic as representations of
    $\Gamma_{E/\mathbb{Q}}$.

    We can relate the adjoint representations of
    $\rho$ and $\Theta(\rho)$
    as follows
    \begin{lemma}
    	\label{comparison of two adjoint representations}
    	Suppose that $\sqrt{2}\in R^\times$.
    	We have the following isomorphisms 
    	of representations of
    	$\Gamma_{\mathbb{Q}}$:
    	\[
    	\mathrm{ad}(\Theta(\rho))
    	\simeq
    	\mathrm{ad}(\Theta'(\rho))
    	\simeq
    	\mathrm{ad}(\rho)
    	\oplus
    	(\rho_\mathrm{st}\circ\rho\otimes\xi)
    	\]
    \end{lemma}
    \begin{proof}
    	The first isomorphism is clear.
    	For the second,
    	let $\mathcal{G}_4(R)$
    	act by adjoint action on $\mathfrak{sl}_4(R)$.
    	Note that
    	$c$ acts on $\mathfrak{sl}_4(R)$ by
    	$c(X):=-J_4X^tJ_4^{-1}$
    	while $\mathrm{GL}_1(R)$
    	acts trivially on $\mathfrak{sl}_4(R)$.
    	As in Remark 
    	\ref{orthogonal and special orthogonal, adjoint representation}(2),
    	we have the following decomposition
    	$\mathfrak{sl}_4(R)
    	=\mathbf{V}_1\oplus
    	\mathbf{V}_2$
    	as representations of
    	$\Gamma_{\mathbb{Q}}$,
    	where
    	\[
    	\mathbf{V}_1
    	=
    	\{X\in\mathfrak{sl}_4(R)|
    	J_4XJ_4^{-1}=-X
    	\},
    	\quad
    	\mathbf{V}_2
    	=
    	\{
    	X\in\mathfrak{sl}_4(R)|
    	J_4XJ_4^{-1}=X
    	\}.
    	\]
    	Note that $\mathbf{V}_1=\mathfrak{sp}_4(R)$,
    	thus $\mathbf{V}_1$ is isomorphic to
    	$\mathrm{ad}(\rho)$
    	as representations of $\Gamma_{\mathbb{Q}}$.
    	On the other hand,
    	by definition,
    	$c$ acts on $\mathbf{V}_2$
    	by $-1$,
    	so we see that
    	the action of
    	$\Gamma_{\mathbb{Q}}$
    	on $\mathbf{V}_2$
    	via the adjoint action of
    	$\Theta'(\rho)$
    	is isomorphic to
    	the action via
    	$\rho_\mathrm{st}\circ\rho\otimes\xi$
    	(in the definition of $\rho_\mathrm{st}$,
    	the element $\sqrt{2}$ is used,
    	that is why we suppose
    	$\sqrt{2}\in R^\times$).
    \end{proof}
    \begin{remark}
    	Note the difference of our decomposition
    	given in the lemma and the one given in
    	\cite[Before Theorem 7.3]{HidaTilouine16}.
    	This difference comes from the
    	fact that we use the theta correspondence
    	between $\mathrm{GSp}_4$ and $\mathrm{GSO}_6$
    	while
    	\cite{HidaTilouine16} uses the theta correspondence
    	between $\mathrm{GSp}_4$ and
    	$\mathrm{GSO}_{3,3}$.
    	More explicitly, in \cite{HidaTilouine16},
    	a Galois representation $\rho
    	\colon
    	\Gamma_{\mathbb{Q}}\rightarrow
    	\mathrm{GSp}_4(R)$
    	is taken to a Galois representation
    	$\Gamma_{\mathbb{Q}}\rightarrow\mathcal{G}'_4(R)$
    	by sending $g\in\mathrm{GSp}_4(R)$
    	to $(g,\nu(g),1)\in\mathcal{G}'_4(R)$.
    \end{remark}

    \subsubsection{Morphism of universal deformation rings}
    We apply the map $\Theta$ to
    the continuous $p$-adic Galois representation
    $\rho_\pi
    \colon
    \Gamma_{\mathbb{Q}}
    \rightarrow
    \mathrm{GSp}_4(\mathcal{O})$.
    \begin{lemma}
    	The representation
    	$\Theta(\rho_\pi)$ is a lift of
    	$\overline{\rho}_{\pi'}$.
    \end{lemma}
    \begin{proof}
    	For any finite place 
    	$\ell\nmid N$,
    	$\pi_\ell$ is unramified or
    	Steinberg.
    	By the correspondence of Satake parameters of
    	local theta correspondence 
    	(Theorem \ref{Satake parameter relation of theta lift}),
    	the relations of
    	Weil-Deligne representations and 
    	Galois representations
    	(\ref{Weil-Deligne vs. Galois, GSp(4)})
    	and
    	(\ref{Weil-Deligne vs Galois, unitary group}),
    	we see that
    	$\Theta(\rho_\pi)|_{D_l}$
    	is equal to
    	$\rho_{\pi'}|_{D_l}$
    	and thus lifts
    	$\overline{\rho}_{\pi'}|_{D_l}$.
    	Moreover, it is clear that
    	$\Theta(\rho_\pi)$ is unramified outside
    	$Np$.
    \end{proof}

    The same argument shows that
    $\Theta(\rho_\pi^\mathrm{s})$
    lifts
    $\overline{\rho}_{\pi'}$,
    thus by the universal property of
    $R_{\pi'}^\mathrm{u}$,
    there is a natural map
    of $\mathcal{O}$-algebras:
    \begin{equation}\label{morphism of universal deformation rings}
    R_{\pi'}^\mathrm{u}
    \rightarrow
    R_\pi^\mathrm{s},
    \text{ similarly, }
    R_{\underline{k}''}^\mathrm{u}
    \rightarrow
    R_{\underline{k}''}^\mathrm{s}.
    \end{equation}

    We have the following
    \begin{lemma}
    	The morphisms defined above are all surjective and
    	compatible with
    	each other:
    	\[
    	\begin{tikzcd}
    	\Lambda^\mathrm{u}
    	\ar[r]
    	\ar[d,twoheadrightarrow]
    	&
    	\mathbb{T}_{\pi'}^\mathrm{u}
    	\ar[r,"\sim"]
    	\ar[d,twoheadrightarrow]
    	&
    	R_{\pi'}^\mathrm{u}
    	\ar[d,twoheadrightarrow]
    	\ar[r]
    	&
    	R_{\underline{k}''}^\mathrm{u}
    	\ar[d,twoheadrightarrow]
    	\\
    	\Lambda^\mathrm{s}
    	\ar[r]
    	&
    	\mathbb{T}_\pi^\mathrm{s}
    	\ar[r,"\sim"]
    	&
    	R_\pi^\mathrm{s}
    	\ar[r]
    	&
    	R_{\underline{k}''}^\mathrm{s}
    	\end{tikzcd}
    	\]    	
    \end{lemma}
    \begin{proof}
    	The commutativity of the left square
    	follows from the definition of
    	the map
    	$\mathbb{T}_{\pi'}^\mathrm{u}
    	\rightarrow
    	\mathbb{T}_\pi^\mathrm{s}$.
    	The commutativity of the right square is clear.   	
    	For the middle square, it suffices to look at
    	their respective Galois representations:
    	for each $\ell\nmid Np$
    	split in $E$,
    	under the morphism
    	$\mathbb{T}_{\pi'}^\mathrm{u}
    	\rightarrow
    	\mathbb{T}_\pi^\mathrm{s}$,
    	the unramified Galois representation
    	$\rho_{\mathbb{T}_{\pi'}^\mathrm{u}}|_{D_\ell}$
    	is mapped to
    	$\Theta(\rho_{\mathbb{T}_\pi^\mathrm{s}})|_{D_\ell}$
    	by \cite[Section 7 ]{Rallis1982}.
    	Similarly, by the definition of diamond operators
    	$\langle u\rangle$,
    	the ordinary Galois representation
    	$\rho_{\mathbb{T}_{\pi'}^\mathrm{u}}|_{D_p}$
    	is mapped to
    	$\Theta(\rho_{\mathbb{T}_\pi^\mathrm{s}})|_{D_p}$.
    	This shows the compatibility of
    	$\mathbb{T}_{\pi'}^\mathrm{u}
    	\rightarrow
    	\mathbb{T}_\pi^\mathrm{s}$
    	with
    	$R_{\pi'}^\mathrm{u}
    	\rightarrow
    	R_\pi^\mathrm{s}$.

    	We next show that
    	$\Lambda^\mathrm{u}
    	\rightarrow
    	\Lambda^\mathrm{s}$
    	is surjective.
    	Since
    	$\mathcal{O}[[1+p\mathbb{Z}_p]]$
    	is the projective limit of the group rings
    	$\{
    	\mathcal{O}
    	[(1+p\mathbb{Z}_p)/(1+p^{n+1}\mathbb{Z}_p)]
    	\simeq\mathcal{O}[\mathbb{Z}/p^n\mathbb{Z}]
    	\}_n$,
    	it suffices to show that
    	the inclusion
    	of the finite group
    	$T^\mathrm{s}_n:=
    	{T^\mathrm{s}(1+p\mathbb{Z}_p)
    		/T^\mathrm{s}(1+p^{n+1}\mathbb{Z}_p)}$
    	in
    	$T^\mathrm{u}_n:=
    	{T^\mathrm{u}(1+p\mathbb{Z}_p)
    		/T^\mathrm{u}(1+p^{n+1}\mathbb{Z}_p)}$
    	induces
    	a surjective morphism of group rings
    	$\mathcal{O}[T^\mathrm{u}_n]
    	\rightarrow
    	\mathcal{O}[T^\mathrm{s}_n]$.
    	This is clear since $T^\mathrm{s}_n$
    	is a direct summand of
    	$T^\mathrm{u}_n$.

    	Finally we show that
    	$R_{\pi'}^\mathrm{u}
    	\rightarrow
    	R_\pi^\mathrm{s}$
    	is surjective,
    	which also gives the surjectivity of
    	the other two vertical arrows.
    	The argument is similar to
    	\cite[Lemma 4.10]{Zhang2018}.
    	For this it suffices to show that for any object
    	$A$ in $\mathrm{AL}_\mathcal{O}$,
    	for any two lifts $\rho_A$ and $\rho_A'$
    	of $\overline{\rho}_\pi$,
    	if $\Theta(\rho_A)$
    	and $\Theta(\rho_A')$
    	are equivalent
    	(conjugate by an element in
    	$\mathrm{GL}^1_4(A)
    	:=\mathrm{Ker}(\mathrm{GL}_4(A)
    	\rightarrow
    	\mathrm{GL}_4(\mathbb{F}))$), 
    	then
    	$\rho_A$ and $\rho_A'$
    	are equivalent
    	(conjugate by an element in
    	$\mathrm{GSp}^1_4(A)
    	:=\mathrm{Ker}(\mathrm{GSp}_4(A)
    	\rightarrow
    	\mathrm{GSp}_4(\mathbb{F}))$).
    	Suppose that there is some
    	$g\in\mathrm{GL}^1_4(A)$
    	such that
    	$\Theta'(\rho_A')=g\Theta'(\rho_A)g^{-1}$.
    	We write
    	$\rho_{A,p}=\rho_A|_{D_p}$
    	and
    	$\rho_{A,p}'=\rho_A'|_{D_p}$.
    	By definition there are elements
    	$h,h'\in\mathrm{GSp}_4(A)$
    	such that the representations of $D_p$,
    	$\widetilde{\rho}:=h^{-1}\rho_{A,p}h$
    	and 
    	$\widetilde{\rho}':=(h')^{-1}\rho_{A,p}'h'$
    	are upper triangular.
    	Recall that the diagonal entries of
    	$\widetilde{\rho}$, resp., $\widetilde{\rho}'$,
    	are mutually distinct characters,
    	so there is some element
    	$\sigma_0\in D_p$
    	such that the diagonal elements of
    	$\widetilde{\rho}(\sigma_0)$,
    	resp., $\widetilde{\rho}'(\sigma_0)$,
    	are mutually distinct.
    	We can choose $h$, resp., $h'$
    	such that
    	$\widetilde{\rho}(\sigma_0)$,
    	resp., $\widetilde{\rho}'(\sigma_0)$,
    	are diagonal matrices with mutually distinct diagonal 
    	entries and are
    	equal to each other.
    	Since
    	$\rho_{A,p}=g^{-1}\rho_{A,p}'g$
    	(note that $p$ splits in $E/\mathbb{Q}$),
    	we conclude that
    	$h'=gh$.
    	This gives
    	$g=h'h^{-1}\in\mathrm{GSp}_4(A)
    	\cap\mathrm{GL}^1_4(A)
    	=
    	\mathrm{GSp}^1_4(A)$.    	
    \end{proof}

    \subsection{Selmer groups and congruence ideals}
    In this subsection we define Selmer groups for the Galois
    representations
    $\rho_{\underline{k}''}^\mathrm{?}$
    with $?=\mathrm{s},\mathrm{u}$
    following the treatment in
    \cite[Section 1]{HidaTilouine16}.

    Let $R$ be an object in $\mathrm{CNL}_\mathcal{O}$
    and
    $\rho
    \colon
    \Gamma_{\mathbb{Q}}
    \rightarrow
    G(R)$
    a continuous Galois representation,
    unramified outside
    $Np$,
    ordinary at $p$.
    Here $G=\mathrm{GSp}_4,\mathcal{G}_4$
    or a classical subgroup of
    $\mathrm{GL}_n$.
    We say that $\Gamma_{\mathbb{Q}}$
    acts on the free $R$-module $M$
    with $M=R^4$, $R^n$
    according to $G=\mathrm{GSp}_4,\mathcal{G}_4$
    or a subgroup of $\mathrm{GL}_n$.
    By the ordinariness assumption,
    there is a filtration
    $(\mathrm{Fil}_i(M))_i$
    on $M$
    stable under
    $\rho_{D_p}$
    (or $\rho_{D_\mathfrak{p}}$
    if $G=\mathcal{G}_4$).
    Assume moreover that
    $\rho_{D_p}$
    acts by $\mathrm{ur}(a_i)\chi_p^{-i}$
    on $\mathrm{Fil}_i(M)/\mathrm{Fil}_{i-1}(M)$.
    We write
    $M^\vee
    =\mathrm{Hom}_\mathcal{O}(M,\mathcal{O}[1/p]/\mathcal{O})$
    for the Pontryagin dual of $M$.
    We define a subspace
    $L_p$ of
    $H^1(D_p,
    M\otimes_RR^\vee)$
    as the image of
    \[
    \mathrm{Ker}
    (
    H^1(D_p,\mathrm{Fil}_0(M)\otimes_RR^\vee)
    \rightarrow
    H^1(I_p,\mathrm{Fil}_0(M)/\mathrm{Fil}_{-1}
    \otimes_RR^\vee)
    ).
    \]
    Then we define the
    discrete Selmer group
    of $\rho$ as
    \[
    \mathrm{Sel}(\rho)
    =
    \mathrm{Ker}
    \bigg(
    H^1(\mathbb{Q},M\otimes_RR^\vee)
    \rightarrow
    \prod_{\ell\neq p}
    H^1(I_\ell,M\otimes_RR^\vee)
    \times
    H^1(D_p,
    M/\mathrm{Fil}_0(M)\otimes_RR^\vee)/L_p
    \bigg).
    \]

    Now let $\rho
    \colon
    \Gamma_{\mathbb{Q}}
    \rightarrow
    \mathrm{GSp}_4(R)$
    be a continuous Galois representation as above.
    Then the following Galois representations
    also satisfies the above assumptions,
    ${\mathrm{ad}\,\rho
    \colon
    \Gamma_{\mathbb{Q}}
    \rightarrow
    \mathrm{GL}_R(\mathfrak{sp}_4(R))}$,
    ${\Theta(\rho)
    \colon
    \Gamma_{\mathbb{Q}}
    \rightarrow
    \mathcal{G}_4(R)}$,
    ${\mathrm{ad}\,\Theta(\rho)
    \colon
    \Gamma_{\mathbb{Q}}
    \rightarrow
    \mathrm{GL}_R(\mathfrak{sl}_4(R))}$,
    ${\rho_\mathrm{st}
    \circ\rho
    \colon
    \Gamma_{\mathbb{Q}}
    \rightarrow
    \mathrm{SO}_5(R)}$.
    Moreover, since
    $\mathrm{ad}\,\Theta(\rho)
    \simeq\mathrm{ad}\,\rho
    \oplus(\rho_\mathrm{st}\circ\rho
    \otimes\xi)$,
    we have the following decomposition:
    \[
    \mathrm{Sel}(\mathrm{ad}\,\Theta(\rho))
    \simeq
    \mathrm{Sel}(\mathrm{ad}\,\rho)
    \bigoplus
    \mathrm{Sel}(\rho_\mathrm{st}\circ\rho
    \otimes\xi).
    \]

    Now we consider the morphism defined above
    $\vartheta
    \colon
    \mathbb{T}_{\underline{k}}^\mathrm{u}
    \twoheadrightarrow
    \mathbb{T}_{\underline{k}}^\mathrm{s}$
    where $\underline{k}$ is the weight of $\pi$.
    The Hida family of $\pi$ in
    $\mathbb{T}_{\underline{k}}^\mathrm{s}$
    also gives a surjective map
    of $\mathcal{O}$-algebras
    $\vartheta_\pi
    \colon
    \mathbb{T}_{\underline{k}}^\mathrm{s}
    \twoheadrightarrow
    \mathcal{O}$.
    Note that both
    $\mathbb{T}_{\underline{k}}^\mathrm{s}$
    and $\mathbb{T}_{\underline{k}}^\mathrm{u}$
    are complete intersections over $\mathcal{O}$
    by the $R=\mathbb{T}$ theorems,
    so we can apply
    \cite[Corollary 8.6]{HidaTilouine16}
    to conclude that
    \[
    \mathfrak{c}(\vartheta_\pi\circ\vartheta)
    =
    \mathfrak{c}(\vartheta_\pi)
    \vartheta_\pi(\mathfrak{c}(\vartheta)).
    \]

    Recall that for any $\mathcal{O}$-module
    of finite type $M$,
    $M$ is isomorphic to a unique
    $\mathcal{O}$-module of the form
    $\mathcal{O}^r\times
    \prod_{i=1}^j\mathcal{O}/\varpi^{r_i}$
    for some $r,j,r_i\geq0$.
    Then the Fitting ideal $\mathrm{Fit}(M)$
    of $M$ is $0$ if $r>0$
    and $\prod_{i=1}^j\varpi^{r_i}\mathcal{O}$
    otherwise.
    By \cite[Proposition 3.4]{HidaTilouine16},
    the Fitting ideal
    $\mathrm{Fit}
    \big(
    \mathrm{Sel}(\mathrm{ad}\,\Theta'(\rho_\pi)
    \big)$,
    resp.,
    $\mathrm{Fit}
    (\mathrm{Sel}(\mathrm{ad}\,\rho_\pi))$
    is equal to the congruence ideal
    $\mathfrak{c}(\vartheta_\pi\circ\vartheta)$,
    resp.,
    $\mathfrak{c}(\vartheta_\pi)$.
    Therefore we conclude
    \begin{proposition}
    	\label{congruence ideals and Selmer groups}
    	Assume
    	Hypothesis
    	\ref{hypotheses for R=T theorems},
    	then the Fitting ideal
    	$\mathrm{Fit}
    	(\mathrm{Sel}
    		(\xi\otimes
    		\rho_\mathrm{st}\circ\rho_\pi))$
    	is equal to
    	the congruence ideal
    	$\vartheta_\pi(\mathfrak{c}(\vartheta))$.
    \end{proposition}

    \subsection{Conclusion}
    Recall that we fix a
    $p$-ordinary antiholomorphic cohomological
    cuspidal irreducible automorphic representation
    $\pi$ of
    $[G_1]$,
    of trivial central character,
    of type $(\underline{k},\widehat{\Gamma})$
    with $k_1\geq k_2\geq3$
    ($\pi$ is anti-holomorphic on $G_1$
    while it is holomorphic on $G_2$).
    We assume that the Hida family
    $\mathbb{T}^\mathrm{s}_\pi$ in
    $\mathbb{T}^\mathrm{s}$
    passing through $\pi$
    is of weight
    $\underline{\kappa}=(\underline{\kappa}_\mathrm{f},
    \underline{k})\in
    \mathrm{Hom}(T_{G_1^1}(\mathbb{Z}_p),
    \overline{\mathbb{Q}}_p^\times)$.
    The $\mathbb{U}_p$-operators
    $U_{p,\underline{c}}$
    act on the ordinary subspace
    $e\pi$
    by the scalar
    $\mathfrak{a}_1^{c_1}\mathfrak{a}_2^{c_2}$
    for all $\underline{c}\in C^+$.   
    Moreover, in Section 4,
    we have constructed sections
    $\phi_{i,\underline{k}}\in
    \mathcal{S}(W_i(\mathbb{A}))$
    for $i=1,2$.

    We
    show the theta lift
    $\Theta_{\phi_{1,\underline{\kappa}}}(\varphi)$
    is a $p$-integral primitive modular form on 
    $\mathrm{U}_4(\mathbb{A})$
    up to a certain power of $p$
    if $\varphi\in\pi$ is 
    a $p$-integral primitive form.
    In
    Proposition 
    \ref{algebraicity of Fourier coefficients of theta series},
    we have shown that the Fourier coefficients of
    the theta series
    $\Theta_{\phi^+_{i,\underline{\kappa}}}$
    are all in
    $\mathcal{O}$,
    thus
    for each
    $h\in\mathrm{U}_4(\mathbb{A})$,
    $\Theta_{\phi^+_{i,\underline{\kappa}}}
    (\cdot,h)|\nu(\cdot)|^3$
    corresponds to
    $\Theta_{\phi^+_{i,\underline{\kappa}}}^\mathrm{alg}
    (\cdot,h)
    |\nu(\cdot)|^3\in
    H^0(\widetilde{\mathbf{A}}_{G_i,\widehat{\Gamma}/
    	\mathcal{O}},
    \mathcal{E}(V))$
    for some algebraic representation $W_i$ of
    $\mathrm{GL}_{2/\mathbb{Q}}$
    by the map
    $\Phi(\cdot,\mathfrak{e})$
    (\textit{cf}. Definition
    \ref{algebraic theta series},
    the similitude factor does not affect the
    $p$-integrality of the theta series).
    Thus for any
    anti-holomorphic
    $p$-integral cuspidal Siegel modular form
    $f_i\in
    H^3
    (G_i,\widehat{\Gamma},\mathcal{O})$,
    the Serre duality shows that
    for each fixed $h\in\mathrm{U}_4(\mathbb{A})$,
    we have
    \[
    \Theta^\ast_{\phi_{i,\underline{\kappa}}}(f_i)(h)
    =
    \langle
    \Theta_{\phi_{i,\underline{k}}}(\cdot,h)|\nu(\cdot)|^3,
    f_i(\cdot)
    \rangle^\mathrm{Ser}
    =
    \int_{[G_i]}
    f_i(g_i)\Theta_{\phi_{i,\underline{\kappa}}}
    (g_i,h)dg_i
    \in \mathcal{O}
    \]
    (we extend $\Theta_{\phi_{i,\underline{k}}}(f_i)$
    by zero from $\mathrm{SU}_4(\mathbb{A})$
    to $\mathrm{U}_4(\mathbb{A})$,
    \textit{cf.} Section
    \ref{lift from GSp4 to U4}).
    Therefore, we can choose some
    $\nu(f_i)\in\mathbb{Z}_{\geq0}$
    (which is clearly unique)
    such that
    $p^{-\nu(f_i)}
    \Theta^\ast_{\phi_{i,\underline{k}}}(f_i)$
    is primitive.
    We define new Schwartz functions
    $\widetilde{\phi}_{i,\underline{k}}
    =\otimes'_v\widetilde{\phi}_{i,\underline{k},v}$
    whose local components are given by
    \[
    \widetilde{\phi}_{i,\underline{k},\infty}
    =p^{-\nu(f_i)}
    \phi_{i,\underline{k},\infty},
    \quad
    \widetilde{\phi}_{i,\underline{k},\ell}
    =\phi_{i,\underline{k},\ell},
    \,
    \forall
    \ell.
    \]
    Thus
    $\Theta^\ast_{\widetilde{\phi}_{i,\underline{\kappa}}}
    (f_i)
    =
    p^{-\nu(f_i)}
    \Theta^\ast_{\phi_{i,\underline{k}}}(f_i)$
    is a $p$-integral primitive
    modular form on
    $\mathrm{U}_4(\mathbb{A})$.
    Now we take $f_1$ and $f_2$ to be
    cuspidal
    ordinary
    $p$-integral Siegel modular forms
    $\varphi_1\in
    H^{3,\mathrm{ord}}_{\underline{k}^D}
    (G_1,\widehat{\Gamma},\mathcal{O})[\pi]$
    and
    $\varphi_2\in
    H^{3,\mathrm{ord}}_{\underline{k}^D}
    (G_2,\widehat{\Gamma},\mathcal{O})
    [\overline{\pi}]$
    such that
    $\langle
    \varphi_1,\varphi_2
    \rangle$
    is equal to the period
    $\widehat{P}[\pi]$
    (\textit{cf.} Lemma \ref{anti-periods}).
    We have integers
    $\nu(\varphi_1)$ and $\nu(\varphi_2)$
    as above and also
    $\widetilde{\phi}_{i,\underline{\kappa}}
    =p^{-\nu(\varphi_i)}
    \phi_{i,\underline{\kappa}}$
    whose local components are given as above.
    As such we have two
    $p$-integral primitive modular forms
    $\Theta^\ast_{\widetilde{\phi}_{i,\underline{\kappa}}}
    (\varphi_i)$
    on $\mathrm{U}_4(\mathbb{A})$.
    Recall from Theorem 
    \ref{arithmetic Rallis inner product},
    we have defined the modified local Euler factor
    $L^\ast_v(1,\mathrm{St}(\pi)\otimes\xi)$
    for $v|Np\infty$
    which depends on
    $\phi_{i,\underline{k},v}$.
    Now we put
    \[
    \widetilde{L}_\infty(1,\mathrm{St}(\pi)\otimes\xi)
    =p^{-\nu(\varphi_1)-\nu(\varphi_2)}
    L^\ast(1,\mathrm{St}(\pi)\otimes\xi),
    \quad
    \widetilde{L}_\ell(1,\mathrm{St}(\pi)\otimes\xi)
    =L^\ast(1,\mathrm{St}(\pi)\otimes\xi),
    \,
    \forall \ell|Np.
    \]
    These are the local Euler factors given by the
    new Schwartz functions
    $\widetilde{\phi}_{i,\underline{k},v}$
    for $v|Np\infty$.
    Taking into account Remark
    \ref{hypothesis on Gorenstein is satisfied}
    (which implies $\mathfrak{c}(\vartheta_\pi)
    =\mathfrak{c}^\mathrm{coh}(\pi)$),
    the Rallis inner product formula now reads
    \begin{theorem}
    	Let the notations be as above,
    	then we have the following identity
    	up to a unit in $\mathcal{O}$:
    	\[
    	\langle
    	\Theta^\ast_{\widetilde{\phi}_{1,\underline{\kappa}}}
    	(\varphi_1),
    	\Theta^\ast_{\widetilde{\phi}_{2,\underline{\kappa}}}
    	(\varphi_2)
    	\rangle_{\mathrm{U}_4}
    	=
    	\mathfrak{c}(\vartheta_\pi)
    	\frac{L^{Np\infty}(1,\mathrm{St}(\pi)\otimes\xi)
    		\widetilde{L}_{Np\infty}(1,\mathrm{St}(\pi)\otimes\xi)}
    	{P_{\pi^\vee}}.
    	\]
    \end{theorem}

    By the same argument as in
    \cite[Lemma 23]{Berger2014},
    one can show that the theta series
    $\Theta^\ast_{\widetilde{\phi}_{i,\underline{\kappa}}}$
    are both eigenforms for the
    Hecke algebra
    $\mathbb{T}_{\underline{\kappa}}^\mathrm{u}$
    (we first show that they are eigenforms
    for the Hecke algebra of the orthogonal
    group $\mathrm{O}(U)$ and
    then use the exceptional isogeny
    between the orthogonal group and unitary group
    to identify their respective Hecke algebras).
    Therefore we see that the Siegel modular forms
    $\varphi_i$
    give rise to
    $\mathbb{T}_{\underline{\kappa}}^\mathrm{u}$-eigenforms
    $\Theta^\ast_{\widetilde{\phi}_{i,\underline{\kappa}}}
    (\varphi_i)$.
    By Lemma
    \ref{congruence ideals and Petersson products},
    the Petersson product
    $\langle
    \Theta^\ast_{\widetilde{\phi}_{1,\underline{\kappa}}}
    (\varphi_1),
    \Theta^\ast_{\widetilde{\phi}_{2,\underline{\kappa}}}
    (\varphi_2)
    \rangle_{\mathrm{U}_4}$
    generates the congruence ideal
    $\mathfrak{c}(\vartheta_\pi\circ\vartheta)$.
    Write
    $\chi(\mathrm{Sel}(\xi\otimes
    \rho_\mathrm{st}\circ\rho_\pi))$
    for a generator of the 
    Fitting ideal
    $\mathrm{Fit}(\mathrm{Sel}(\xi\otimes
    \rho_\mathrm{st}\circ\rho_\pi))$.
    Combining the above theorem with
    Proposition
    \ref{congruence ideals and Selmer groups},
    we get the following
    (for the local factors
    $L_{Np\infty}^\ast(1,\mathrm{St}(\pi)
    \otimes\xi)$,
    see Theorem
    \ref{arithmetic Rallis inner product}):

    \begin{theorem}\label{Selmer group and L-value}
    	Let the notations be as in the last theorem.
    	We make the assumptions 
    	(1) Hypothesis
    	\ref{hypotheses for R=T theorems};
    	(2)
    	the pairing 
    	$\langle\varphi_1,\varphi_2\rangle
    	=\widehat{P}[\pi]\neq0$
    	and $\varphi_1,\varphi_2$
    	are $p$-ordinary.  	
    	Then we have the following identity up to units in
    	$\mathcal{O}$:
    	\[
    	\frac{L^{Np\infty}(1,\mathrm{St}(\pi)\otimes\xi)
    	\widetilde{L}_{Np\infty}
    	(1,\mathrm{St}(\pi)\otimes\xi)}
        {P_{\pi^\vee}}
    	=
    	\chi(\mathrm{Sel}(\xi\otimes
    	\rho_\mathrm{st}\circ\rho_\pi)).
    	\]
    \end{theorem}

	\end{document}